\documentclass[sn-mathphys-num]{sn-jnl}

\usepackage[T1]{fontenc}
\usepackage{graphicx}
\usepackage{amsmath,amssymb,amsfonts,amsthm,mathtools,mathrsfs}
\usepackage[title]{appendix}
\usepackage{booktabs,enumitem,xcolor}

\theoremstyle{thmstyleone}%
\newtheorem{theorem}{Theorem}%
\newtheorem{proposition}[theorem]{Proposition}%
\newtheorem{lemma}[theorem]{Lemma}%
\newtheorem{corollary}[theorem]{Corollary}%
\theoremstyle{thmstyletwo}%
\newtheorem{remark}{Remark}%
\theoremstyle{thmstylethree}%
\newtheorem{definition}{Definition}%
\theoremstyle{thmstyleone}%
\newcommand{\restatename}{Theorem}%
\newtheorem*{restatement}{\restatename}%
\begin{document}
	
	\title[Derangetropy Operators]{Derangetropy Operators}
	
	\author*[1]{\fnm{Masoud} \sur{Ataei}}\email{masoud.ataei@utoronto.ca}
	\author[2]{\fnm{Sepideh} \sur{Forouzi}}\email{sepideh.forouzi@colorado.edu}
	
	\affil[1]{\normalsize\orgdiv{Department of Mathematical and
			Computational Sciences},\\ \orgname{University of Toronto},
		\country{Canada}}
	\affil[2]{\normalsize\orgdiv{Department of Computer Science},
		\orgname{University of Colorado Boulder}, \country{USA}}
	
	\abstract{A derangetropy operator reweighs a probability density by a
		fixed profile of its own cumulative distribution function, acting
		through ranks alone. We prove that these operators are precisely
		the transformations of absolutely continuous laws equivariant
		under monotone changes of variable, and that they compose through
		interval maps, making their dynamics exactly solvable: iteration
		condenses every law onto its median with a universal Koenigs limit
		law, the continuous flow is solvable in closed form with the
		Cauchy family as invariant hyperbolic manifold, and balanced
		against diffusion the distribution function obeys an overdamped
		sine--Gordon equation whose unique steady law, the hyperbolic
		secant, is a globally stable kink. A variational principle selects
		the canonical kernel, the squared ground state of the Dirichlet
		Laplacian on the unit interval, whose update is a Bayesian
		posterior costing exactly one bit for every law, and a unitary
		lift equips each law with an isospectral Sturm--Liouville system
		and a quantum carpet: a strong law of large numbers for the
		critical Sobolev mass, with sharp constant given by Wiener's jump
		statistic, proves the fractality of the Schr\"odinger density for
		arbitrary real bounded-variation data with a jump, previously
		known only for rational step data. A lift to the dual of the
		Virasoro algebra, at central charge one, a normalization,
		identifies laws with disconjugate Hill potentials, tails with
		conformal weights, and the Cauchy family with the exceptional
		orbit. Randomized phases yield an exact sampler and a
		multiplicative chaos, log-correlated in the weakly tempered
		regime; in several dimensions, dependence stratifies into a
		conserved interaction, a torsion driven by maximal correlation,
		and a flat Sinkhorn transport.}
	
	
	\maketitle
	
	\section{Introduction}\label{sec:intro}
	
	The classical calculi on probability laws are organized by invariance.
	Convolution respects the additive structure of the real line, and its
	limit theory produces the Gaussian law; exponential tilting respects
	the multiplicative structure of likelihoods, and generates the
	exponential families of statistics; optimal transport respects metric
	displacement, and produces the geodesic interpolants of the
	Wasserstein geometry~\cite{villani2003}. The order structure of the
	line has no comparable calculus, even though the ordinal content of an
	absolutely continuous law is completely summarized by its distribution
	function, and even though the probability integral transform makes
	every such law uniform in its own cumulative coordinate. A
	transformation faithful to order alone must act through the
	distribution function, and through nothing else.
	
	This paper develops the theory of such transformations, which we call
	derangetropy operators: each multiplies a density by a fixed
	reweighing profile, or kernel, evaluated at the density's own
	cumulative distribution. Three operators of this form were introduced
	in~\cite{ataei2024,ataei2025} as models of cyclical information flow,
	distinguished by how their sinusoidal modulation responds to the
	entropy of the split of probability mass around the current point: one
	attenuated by it, one amplified by it, one entropy-blind. The
	entropy-blind member, whose kernel is the squared ground state of the
	Dirichlet Laplacian on the unit interval, is treated as canonical
	throughout, for variational, spectral, and algebraic reasons set out
	in Section~\ref{sec:foundations}. The present work detaches the
	construction from its original setting, takes the reweighing
	architecture itself as the definition, and shows that a complete
	algebraic, spectral, dynamical, and geometric theory flows from it.
	
	The objects themselves have a long history, and it is worth stating
	precisely where the present theory stands relative to it. Read at the
	level of laws, a derangetropy operator is an absolutely continuous
	\emph{distortion}: it acts by composing the distribution function
	with an increasing bijection of the unit interval, the distortion
	functions of Yaari's dual theory of choice under
	risk~\cite{yaari1987} and of the actuarial pricing
	calculus~\cite{wang2000}. Read at the level of densities, the
	modulated laws are Neyman's smooth alternatives to
	uniformity~\cite{neyman1937}, whose trigonometric modes reappear here
	as the harmonic kernels; and the amplitude lift of the spectral
	theory is the classical Liouville transformation of Sturm--Liouville
	theory~\cite{zettl2005}. What this paper contributes is not these
	objects but their operator calculus: the rigidity theorem, which
	shows that the family exhausts the transformations of densities
	equivariant under monotone changes of variable and so gives the
	distortion calculus an intrinsic characterization, in the lineage
	of the maximal-invariant description of ranks under the monotone
	group in classical rank-test theory~\cite{hajek1967}; the exact
	iteration and flow theory, with its universality law; and the
	spectral and geometric structures that one canonical kernel
	generates. The remaining pillars are classical as well, and are
	engaged directly: the multivariate rank transform of the
	higher-dimensional theory is Rosenblatt's~\cite{rosenblatt1952}, the
	ancestor of triangular transport maps and of the normalizing flows of
	modern machine learning~\cite{bogachev2005,papamakarios2021}; the
	invariant geometry of densities is \v{C}encov's~\cite{chentsov1982},
	whose uniqueness theorem for the Fisher--Rao metric is what makes the
	variational selection of the canonical kernel well posed; the exact
	iteration theory rests on the linearization theory of one-dimensional
	maps founded by Koenigs~\cite{koenigs1884}, in its modern
	form~\cite{milnor2006}; and the geometry of laws rests on the
	projective differential calculus of the Schwarzian
	derivative~\cite{ovsienko2005}. What is new is that these strands
	appear as facets of a single operator family, forced by one
	invariance principle.
	
	These classical threads remain active, and the paper is positioned
	against their current forms. The distortion calculus underlies the
	modern axiomatic theory of risk measures~\cite{wang2021}, and
	Neyman's alternatives survive as the data-driven smooth tests of
	goodness of fit~\cite{ledwina1994}. The geometry of laws has become a
	subject in its own right, through information
	geometry~\cite{amari2016} and through the statistical theory of
	Wasserstein distances~\cite{panaretos2019}. Closest to the
	multivariate half of this paper is the measure-transportation theory
	of ranks: the center-outward distribution functions of
	Hallin~\cite{hallin2021}, building on the Monge--Kantorovich depth
	of~\cite{chernozhukov2017}, and their consistency and testing theory
	in~\cite{ghosal2022}, construct a single canonical, ordering-free
	multivariate rank, where the theory developed below instead retains
	all Rosenblatt orderings and reads their mutual disagreement as
	geometry. Even the Schwarzian derivative, the invariant on which the
	curvature theory of this paper is built, has moved to the center of
	recent low-dimensional physics through the Sachdev--Ye--Kitaev
	model~\cite{maldacena2016} and the boundary theory of
	Jackiw--Teitelboim gravity~\cite{msy2016,stanford2017}; the
	coadjoint-orbit theory of the Virasoro algebra, on which the final
	sections of this paper draw, is
	classical~\cite{lazutkin1975,kirillov1982,segal1981,witten1988,khesin2009}.
	Against this background, the contribution
	of the paper is a calculus in which these separate threads become
	theorems about one operator family.
	
	The results form four connected strata. The structural and spectral
	stratum proves the rigidity theorem, that derangetropy operators
	exhaust the monotone-equivariant transformations of densities;
	composition reduces to interval dynamics, a variational principle
	selects the canonical kernel and prices its update at one universal
	bit, and the amplitude lift converts modulation into unitary
	geometry, giving every law an isospectral vibrating system with
	universal frequencies and an uncertainty principle in rank.
	
	The dynamical stratum is exactly solvable on three axes. Iteration
	condenses every law onto its median with a universal limit law of
	Koenigs type, harmonic kernels crystallizing laws onto quantile
	atoms; the continuous flow is solvable in closed form, with the
	Cauchy family as invariant manifold of hyperbolic Fisher--Rao
	geometry, and balanced against diffusion it becomes an overdamped
	sine--Gordon equation whose unique steady law, the hyperbolic
	secant, is a globally stable kink with a diffusivity-independent
	spectral gap, condensation reappearing as the sharp-interface limit
	and crystallization as terraces of equal-mass bumps; randomizing
	the phase turns the cascade into an exact sampler at a uniformly
	random quantile with an exactly linear one-bit ledger, and, when
	the modulation frequency is renewed across scales, into a
	multiplicative chaos carried by a random homeomorphism of rank
	space, log-correlated in the weakly tempered regime.
	
	The geometric stratum gathers the projective half of the theory
	into one structure. A curvature built from the Schwarzian
	derivative classifies the constant-curvature laws into a M\"obius
	trichotomy; the quantile Schwarzian then realizes every law as a
	point in the dual of the Virasoro algebra of rank space, the
	operator calculus becoming the coadjoint action at central charge
	one, a normalization fixed by the unit coefficient of the
	Schwarzian. In this dictionary existence is disconjugacy, unit
	mass is the Wronskian normalization, tails are conformal weights
	under a Hardy bound, the Cauchy family is the exceptional orbit,
	condensation is a no-hair relaxation, and the Korteweg--de Vries
	hierarchy preserves circular laws. On the unitary axis, every law
	carries a quantum carpet: the kernel evolved by the Schr\"odinger
	flow of rank space revives into quadratic Gauss-sum histograms on
	quantile cells at rational times and weaves, at almost every other
	time, a fractal of graph dimension exactly three halves, in every
	quadrature and for the density itself; a strong law of large
	numbers for the critical Sobolev mass, with a sharp constant given
	by Wiener's jump statistic, extends the density fractality to
	arbitrary rough seeds and arbitrary gratings, a case previously
	settled only for rational step data~\cite{chousionis2015}, and the
	constant computes the collision arithmetic of the dispersion,
	halving along time traces and doubling for the Airy flow. Read
	physically, this is a statement about the observable of Talbot
	optics rather than its idealization: the intensity carpet behind
	an \emph{arbitrary} grating of bounded variation, irregular
	slit positions, rough transmission profiles, has graph
	dimension exactly $\tfrac{3}{2}$ at almost every distance, its
	critical mass diverging at the universal rate set by Wiener's jump
	statistic of the grating, a prediction open to direct box-counting
	of measured carpets in optical or matter-wave interferometry.
	
	The multivariate stratum reads dependence as geometry. The
	noncommutativity of coordinate rank modulations vanishes exactly at
	independence, and the geometry behind it stratifies dependence into
	three exact layers: a foliation whose leaves are the reach of
	coordinate modulation and whose conserved invariant is the
	interaction structure of the law; a torsion tensor governed by the
	conditional expectation between rank sigma-algebras, with maximal
	correlation as gauge-field strength; and a flat transport identified
	with the classical information projection, iterative proportional
	fitting emerging as alternating derangetropy moves contracting at
	the squared maximal correlation.
	
	The remainder of the paper is organized as follows.
	Section~\ref{sec:foundations} defines the operator family and
	proves the rigidity and composition theorems.
	Section~\ref{sec:canonical} selects the canonical kernel
	variationally and establishes its Bayesian structure and one-bit
	information identities. Section~\ref{sec:amplitude} constructs the
	amplitude representation with its isospectral systems. Sections
	\ref{sec:iteration} and~\ref{sec:flow} develop the discrete and
	continuous dynamics, with the Koenigs limit law, the exact solution
	of the flow, and the entropy laws. Section~\ref{sec:curvature}
	introduces the projective curvature of laws and proves the M\"obius
	trichotomy. Sections \ref{sec:virasoro} and~\ref{sec:orbits} lift
	the calculus to the dual of the Virasoro algebra and develop the
	orbit, isotropy, and relaxation theory. Section~\ref{sec:pde}
	develops the reaction--diffusion theory, from the sine--Gordon
	reduction to kinks and terraces. Section~\ref{sec:random} treats
	the randomized cascade, exact sampling, and the multiplicative
	chaos, and Section~\ref{sec:carpet} the quantum carpet and the
	strong law for the critical Sobolev mass. Sections
	\ref{sec:multivariate} and~\ref{sec:dependence} develop the
	multivariate theory and the geometry of dependence. We conclude in
	Section~\ref{sec:conclusion} with a synthesis and open problems;
	proofs are collected in Appendix~\ref{app:proofs}.
	
	\newpage
	\section{Foundations}\label{sec:foundations}
	
	Transformations of probability distributions ordinarily act through
	convolution, rescaling, or exponential tilting, and each of these
	operations sees a law through its additive or moment structure. The
	operators studied in this paper act instead through cumulative
	structure: they reweigh a density by a fixed profile of its own
	distribution function, so that the effect of the transformation at a
	point depends only on how much probability lies below that point.
	Three such functionals were introduced in~\cite{ataei2024,ataei2025}
	under the name \emph{derangetropy}. For a probability density $f$ with
	distribution function $F$, they read
	\begin{subequations}\label{eq:three-functionals}
		\begin{align}
			\rho_{\mathrm{I}}[f] &= \frac{24}{\pi e}\,\sin(\pi F)\,
			e^{-H_{B}(F)}\,f, \label{eq:typeI}\\[2pt]
			\rho_{\mathrm{II}}[f] &= \frac{e}{\pi}\,\sin(\pi F)\,
			e^{+H_{B}(F)}\,f, \label{eq:typeII}\\[2pt]
			\rho_{\mathrm{III}}[f] &= 2\sin^{2}(\pi F)\,f, \label{eq:typeIII}
		\end{align}
	\end{subequations}
	where
	\begin{equation}
		H_{B}(z) = -z\log z - (1-z)\log(1-z)
	\end{equation}
	is the Bernoulli entropy function. The three modulations respond
	differently to the informational state of the cumulative variable. The
	Type-I functional in equation~\eqref{eq:typeI} attenuates its
	sinusoidal modulation where the split of probability mass below and
	above the current point is most uncertain, that is, near the median,
	where $H_{B}(F)$ is maximal; the Type-II functional in
	equation~\eqref{eq:typeII} amplifies the modulation exactly there; the
	Type-III functional in equation~\eqref{eq:typeIII} ignores the entropy
	factor altogether and acts as a pure phase modulation through the
	squared sine of the cumulative variable. Set side by side, the three
	share one architecture, namely multiplication of $f$ by a fixed
	profile of $F$, and the theory of this paper begins by taking that
	architecture as the definition.
	
	One member of the family will be treated as canonical throughout, and
	it is worth recording at the outset why the third functional, and not
	the other two, plays this role. First, Type III is the entropy-neutral
	core of the family: the opposing entropy tilts of Types I and II
	cancel in their pointwise product, which is, up to a universal
	constant, exactly the Type-III modulation. Second, its profile
	$2\sin^{2}(\pi z)$ is the squared ground state of the Dirichlet
	Laplacian on the unit interval, and this spectral pedigree is what
	generates the ladder of harmonics, the exact eigenvalue structure, and
	the closed-form dynamics developed in this paper. Third, among all
	modulations that are inert at the extremes of the cumulative range,
	the Type-III profile is the one of minimal Fisher information, and it
	is in this sense the least disturbing modulation compatible with the
	boundary constraint. For these reasons Type III is adopted as the
	canonical form of derangetropy; the definitions and structural
	theorems of the present section, however, are stated for arbitrary
	profiles, since they cost nothing extra at that generality.
	
	We work on the space of absolutely continuous laws identified with
	their densities,
	\begin{equation}
		\mathcal{D} = \Big\{ f \in L^{1}(\mathbb{R}) : f \ge 0,\
		\textstyle\int_{\mathbb{R}} f\,d\lambda = 1 \Big\},
		\label{eq:density-space}
	\end{equation}
	and write $\mathcal{D}_{+} \subset \mathcal{D}$ for densities that are
	continuous and strictly positive on an open interval and vanish
	outside it; $Q = F^{-1}$ denotes the quantile function, defined on
	$(0,1)$ for $f \in \mathcal{D}_{+}$. The elementary fact on which
	everything rests is the probability integral transform: if
	$X \sim f \in \mathcal{D}$, then $F(X)$ is uniform on $(0,1)$;
	equivalently, the pushforward of $f\,d\lambda$ under $F$ is Lebesgue
	measure on $[0,1]$. We call $z = F(x)$ the \emph{rank} of $x$ and the
	unit interval, so parametrized, the \emph{rank space} of $f$: in its
	own rank coordinate, every absolutely continuous law is uniform. The
	rank of a point is a dimensionless quantity, independent of the units
	or scale in which the underlying variable is recorded, and rank space
	is the natural habitat of the operators studied here: everything a
	derangetropy operator does happens on the unit interval and is merely
	transported back to the line by the quantile map. This single
	observation is responsible for most of the exact computations in the
	paper, since integrals of rank functionals reduce to integrals over
	$[0,1]$ that do not depend on the law.
	
	\begin{definition}[Derangetropy operators]\label{def:operator}
		A \emph{kernel} is a probability density $w$ on $[0,1]$; the set of
		kernels is $\mathcal{W}$. The derangetropy operator with kernel $w$ is
		\begin{equation}
			\rho_{w}[f](x) \;=\; w\big(F(x)\big)\,f(x),
			\qquad f \in \mathcal{D}.
			\label{eq:operator}
		\end{equation}
		The three functionals in equation~\eqref{eq:three-functionals} are
		$\rho_{w_{\mathrm{I}}}, \rho_{w_{\mathrm{II}}},
		\rho_{w_{\mathrm{III}}}$ for the kernels
		\begin{subequations}\label{eq:three-kernels}
			\begin{align}
				w_{\mathrm{I}}(z) &= \frac{24}{\pi e}\,\sin(\pi z)\,e^{-H_{B}(z)},\\[2pt]
				w_{\mathrm{II}}(z) &= \frac{e}{\pi}\,\sin(\pi z)\,e^{+H_{B}(z)},\\[2pt]
				w_{\mathrm{III}}(z) &= 2\sin^{2}(\pi z).
			\end{align}
		\end{subequations}
	\end{definition}
	
	That $w_{\mathrm{I}}, w_{\mathrm{II}} \in \mathcal{W}$ is the content
	of two contour integrals evaluated in~\cite{ataei2024,ataei2025}; that
	$w_{\mathrm{III}} \in \mathcal{W}$ is elementary. Basic regularity is
	established in Lemma~\ref{lem:wellposed} of
	Appendix~\ref{app:proofs}: for every $w \in \mathcal{W}$ and
	$f \in \mathcal{D}$, the image $\rho_{w}[f]$ is again a probability
	density, and normalization holds with no strict monotonicity or
	change-of-variable hypotheses, because it is nothing but the
	pushforward identity
	\begin{equation*}
		\int_{\mathbb{R}} w(F)\,f\,d\lambda = \int_{0}^{1}w\,dz;
	\end{equation*}
	moreover the support, the smoothness class, and any symmetry of $f$
	are preserved whenever $w$ is positive on $(0,1)$, smooth on the
	open interval, and symmetric, as all three kernels in
	equation~\eqref{eq:three-kernels} are, the binary-entropy tilt of
	Types I and II being smooth precisely on the interior, with
	divergent derivative at the endpoints, so that in particular
	$\rho_{w}(\mathcal{D}_{+}) \subseteq \mathcal{D}_{+}$; and bounded
	kernels map $L^{p}$ to $L^{p}$.
	
	Two features of Definition~\ref{def:operator} deserve emphasis before
	any theory is built on it. The operator is nonlinear, since the kernel
	is evaluated at $F$, which depends on $f$, but it is
	\emph{multiplicative}: it reweighs mass where it stands and never
	transports it. A convolution spreads mass over the line and can create
	support where none existed; a derangetropy operator can only raise or
	lower the density inside the support it is given, and all apparent
	motion of probability arises through the renormalization of ranks. The
	second feature is that the three types are algebraically entangled. In
	the pointwise product of the first two kernels the opposing entropy
	tilts cancel and the sines square:
	\begin{equation}
		w_{\mathrm{I}}\,w_{\mathrm{II}}
		= \frac{24}{\pi^{2}}\sin^{2}(\pi z)
		= \frac{12}{\pi^{2}}\,w_{\mathrm{III}}.
		\label{eq:product-identity}
	\end{equation}
	The product identity~\eqref{eq:product-identity} says that Type III
	is, up to normalization, the product of Types I and II, the
	entropy-neutral core of the family; the other two types are, up to
	constants, $\sin(\pi z)\,e^{\mp H_{B}(z)}$, exponential tiltings by
	$\mp H_{B}$ of the \emph{amplitude} $\sqrt{w_{\mathrm{III}}/2}$
	rather than of the kernel itself, a form that already points toward
	the amplitude viewpoint developed later in the paper. This is the
	first of the reasons, announced above, for treating Type III as
	canonical.
	
	The structural property that organizes the entire theory is invariance
	under monotone changes of variable. For a monotone bijection $T$ of
	$\mathbb{R}$ let $T_{\#}$ denote pushforward of laws.
	
	\begin{theorem}[Equivariance]\label{thm:equivariance}
		Let $f \in \mathcal{D}$ and let $T$ be an increasing homeomorphism of
		an open interval containing the support of $f$ onto an open interval,
		with $T_{\#}f \in \mathcal{D}$. Then for every $w \in \mathcal{W}$,
		$\rho_{w}[T_{\#}f] = T_{\#}\rho_{w}[f]$. If $w(1-z) = w(z)$, as holds
		for all three types, the identity extends to decreasing $T$, so
		$\rho_{w}$ commutes with every monotone change of variable between
		intervals.
	\end{theorem}
	
	\begin{proof}
		Both sides are absolutely continuous: $\rho_{w}[f] \ll f$ by the
		product form, so $T_{\#}\rho_{w}[f] \ll T_{\#}f \ll \lambda$, and it
		therefore suffices to match distribution functions. Writing
		\begin{equation*}
			A_{w}(z) = \int_{0}^{z}w(s)\,ds,
		\end{equation*}
		the pushforward identity applied to the truncated integral gives the
		\emph{transport formula}
		\begin{equation}
			F_{\rho_{w}f} \;=\; A_{w} \circ F,
			\label{eq:transport}
		\end{equation}
		while $F_{T_{\#}f} = F \circ T^{-1}$ on the image interval of an
		increasing $T$, extended by $0$ and $1$ to its left and right; hence
		both
		sides of the claim have the same distribution function,
		\begin{equation*}
			F_{\rho_{w}[T_{\#}f]}
			= A_{w}\circ F \circ T^{-1}
			= F_{T_{\#}\rho_{w}[f]}.
		\end{equation*}
		For decreasing $T$ one has instead
		\begin{equation*}
			F_{T_{\#}f} = 1 - F\circ T^{-1},
		\end{equation*}
		and symmetry of $w$ enters through the reflection identity
		\begin{equation*}
			A_{w}(1-z) = 1 - A_{w}(z),
		\end{equation*}
		which converts the complementation into the same matching of
		distribution functions.
	\end{proof}
	
	Derangetropy operators are therefore blind to units, scales, and every
	monotone recoding of the underlying variable: they see ranks and
	nothing else. A change of currency, a logarithmic re-expression of a
	positive quantity, or any other order-preserving relabeling commutes
	with the operator, so any structure the operator produces is a
	statement about the order structure of the law alone. This invariance
	in fact \emph{characterizes} the family, which is the precise sense in
	which Definition~\ref{def:operator} is a classification rather than a
	construction.
	
	\begin{theorem}[Rigidity]\label{thm:rigidity}
		Let $\Phi : \mathcal{D}_{+} \to \mathcal{P}(\mathbb{R})$ satisfy
		$\Phi[T_{\#}f] = T_{\#}\Phi[f]$ for every $f \in \mathcal{D}_{+}$ and
		every increasing homeomorphism $T$ of $\mathbb{R}$ onto an open
		interval with $T_{\#}f \in \mathcal{D}_{+}$. Then
		$\Phi[f] = (Q_{f})_{\#}\gamma$ for a fixed law $\gamma$ on $(0,1)$;
		if moreover $\Phi$ maps into absolutely continuous laws, then
		$\gamma = w\,dz$ with $w \in \mathcal{W}$ and $\Phi = \rho_{w}$ on
		$\mathcal{D}_{+}$.
	\end{theorem}
	
	\begin{proof}
		Fix a diffeomorphism $L$ of $\mathbb{R}$ onto $(0,1)$ with positive
		derivative, and let $g_{0} \in \mathcal{D}_{+}$ be the law with
		distribution function $L$. For arbitrary $f \in \mathcal{D}_{+}$ the
		composition $T_{f} := Q_{f}\circ L$ is an increasing homeomorphism of
		$\mathbb{R}$ onto the interior of the support of $f$, and
		\begin{equation*}
			F_{(T_{f})_{\#}g_{0}}
			= L \circ T_{f}^{-1}
			= L \circ L^{-1}\circ F
			= F,
		\end{equation*}
		so $(T_{f})_{\#}g_{0} = f$: every law in $\mathcal{D}_{+}$ is reached
		from the single reference law $g_{0}$ by a map of the acting class.
		Equivariance applied to $T_{f}$ therefore gives
		\begin{equation*}
			\Phi[f] = (T_{f})_{\#}\Phi[g_{0}]
			= (Q_{f})_{\#}\gamma,
			\qquad
			\gamma := L_{\#}\Phi[g_{0}],
		\end{equation*}
		and $\gamma$ is a law on $(0,1)$ because $L$ takes its values there;
		in particular $\gamma$ has no endpoint atoms, so the pushforward
		under $Q_{f}$ is well defined even for unbounded supports. When
		$\gamma = w\,dz$, the pushforward $(Q_{f})_{\#}\gamma$ has
		density
		\begin{equation*}
			w\big(F(x)\big)\,f(x) = \rho_{w}[f](x). \qedhere
		\end{equation*}
	\end{proof}
	
	\begin{remark}
		The compressing maps in the hypothesis of
		Theorem~\ref{thm:rigidity}, homeomorphisms of the line \emph{onto
			intervals}, are essential, and equivariance under the surjective
		increasing homeomorphisms of $\mathbb{R}$ alone would not suffice.
		Surjective homeomorphisms preserve the order type of the support, so
		under them $\mathcal{D}_{+}$ splits into invariant classes according
		to whether the support is a bounded interval, a half-line, or the
		whole line, and a map defined as $\rho_{w}$ on one class and as, say,
		the point mass at the median on another is equivariant for that
		smaller group without being of the stated form. The compressing maps
		connect the classes and eliminate such examples. Note also that the
		proof uses equivariance only for the maps $T_{f} = Q_{f}\circ L$;
		once the conclusion is in hand, equivariance under the rest of the
		class is automatic for the resulting $\rho_{w}$, by
		Theorem~\ref{thm:equivariance}.
	\end{remark}
	
	Operations such as convolution or location--scale maps are not
	counterexamples but contrasts: they use the affine structure of the
	line, not its order. Theorem~\ref{thm:rigidity} therefore delimits the
	subject exactly: any transformation of laws intended to depend on
	ordinal structure alone \emph{must} be a derangetropy operator, and
	the family $\mathcal{W}$ of kernels is a complete parametrization of
	all such transformations. Equivariance has an immediately useful
	corollary, invoked constantly below. Any scalar functional of the pair
	$(f,\rho_{w}f)$ expressible through the rank variable is the same for
	every law, since
	\begin{equation}
		\mathbb{E}_{f}\big[\Psi\big(F(X)\big)\big]
		= \int_{0}^{1}\Psi(z)\,dz
		\label{eq:universality}
	\end{equation}
	for every integrable $\Psi$ on $(0,1)$. We refer to the
	identity~\eqref{eq:universality} as the \emph{universality principle}.
	Its meaning deserves a comment: whenever a constant appears in this
	theory that does not depend on the law being transformed, that
	constant is not an accident of computation but a symmetry, the visible
	trace of the fact that all laws are indistinguishable in their own
	rank coordinates.
	
	Iteration is where the operator viewpoint becomes analytic, and the
	key is that the transport formula~\eqref{eq:transport} composes.
	
	\begin{theorem}[Composition law]\label{thm:composition}
		$\rho_{v}\circ\rho_{w} = \rho_{(A_{v}\circ A_{w})'}$, the prime
		now denoting the derivative alone, so
		$w \mapsto A_{w}$ is an isomorphism from $(\{\rho_{w}\},\circ)$ onto
		the monoid of absolutely continuous nondecreasing surjections of
		$[0,1]$ under composition. In particular the $n$-fold iterate is again
		a derangetropy operator,
		\begin{equation}
			\rho_{w}^{\,n}[f]
			= \big(A_{w}^{\circ n}\big)'(F)\,f
			= f\prod_{j=0}^{n-1} w\big(A_{w}^{\circ j}(F)\big).
			\label{eq:iterates}
		\end{equation}
	\end{theorem}
	
	\begin{proof}
		By the transport formula~\eqref{eq:transport},
		\begin{equation*}
			\rho_{v}\big[\rho_{w}f\big]
			= v\big(F_{\rho_{w}f}\big)\,\rho_{w}[f]
			= v\big(A_{w}(F)\big)\,w(F)\,f
			= \big[(v\circ A_{w})\cdot w\big](F)\,f,
		\end{equation*}
		and by the chain rule
		\begin{equation*}
			(v\circ A_{w})\cdot w = (A_{v}\circ A_{w})',
		\end{equation*}
		so the composite is the derangetropy operator with kernel
		$(A_{v}\circ A_{w})'$; induction gives the product
		formula~\eqref{eq:iterates}. Closure of the monoid, that is,
		absolute continuity of $A_{v}\circ A_{w}$, holds by the
		Banach--Zarecki criterion~\cite{natanson1955}: the composition is
		continuous, of bounded variation by monotonicity, and maps null sets
		to null sets because each factor does.
	\end{proof}
	
	The infinite-dimensional recursion on densities thus collapses to the
	iteration of one interval map. In particular, since
	$(A_{w}^{\circ n})'$ is itself a kernel, every iterate
	$\rho_{w}^{\,n}[f]$ is again a valid probability density: the
	recursion can be run indefinitely, with normalization automatic at
	every step and never imposed by hand. Taking logarithms in the product
	formula~\eqref{eq:iterates}, the density ratio of the $n$-th iterate
	is a Birkhoff sum along the orbit of $A_{w}$: iterated derangetropy is
	a multiplicative cascade with exactly computable statistics. The
	composition law also settles invertibility. If $A_{w}$ is a
	homeomorphism of $[0,1]$ with absolutely continuous inverse, then
	$(A_{w}^{-1})'$ is again a kernel and, by the composition law,
	$\rho_{(A_{w}^{-1})'}$ inverts $\rho_{w}$ on $\mathcal{D}$.
	Conversely, if $A_{w}$ is not injective, its kernel vanishes on
	some interval of ranks and $\rho_{w}$ identifies laws differing
	only there, while if $A_{w}$ is a homeomorphism whose inverse is
	not absolutely continuous, already the law with uniform image rank
	is not attained, its preimage requiring the singular distribution
	function $A_{w}^{-1}\circ F$; so $\rho_{w}$ is a bijection of
	$\mathcal{D}$ precisely when $A_{w}$ is a homeomorphism with
	absolutely continuous inverse, and then
	$\rho_{w}^{-1} = \rho_{(A_{w}^{-1})'}$ remains in the family. The
	bijection restricts to one of $\mathcal{D}_{+}$ when moreover $w$
	is almost everywhere equal to a function continuous and positive on
	$(0,1)$, for then
	$(A_{w}^{-1})' = 1/(w\circ A_{w}^{-1})$ is likewise continuous and
	positive on $(0,1)$ and both operators preserve $\mathcal{D}_{+}$;
	interior positivity cannot be dropped from this last clause, since
	the harmonic kernel $w_{2}$, whose distortion is a homeomorphism
	with absolutely continuous inverse, sends every law of
	$\mathcal{D}_{+}$ to a density vanishing at its median. All
	three type kernels qualify.
	
	Read at the level of laws, the transport
	formula~\eqref{eq:transport} identifies $\rho_{w}$ with an absolutely
	continuous \emph{distortion} of the distribution function,
	$F \mapsto A_{w}\circ F$, an operation with a long history in the
	dual theory of choice under risk and in actuarial
	pricing~\cite{yaari1987,wang2000}, where $A_{w}$ is called a
	distortion function; the composition monoid of
	Theorem~\ref{thm:composition} is, in that language, the composition
	of distortions. The contribution of the present theory begins where
	that identification ends: the rigidity theorem shows that nothing
	else is equivariant, so the distortion calculus is not one choice
	among many but the whole of the ordinal calculus, and the
	density-level reading $w(F)f$ supports the spectral, dynamical, and
	geometric structure developed in the remainder of the paper.
	
	The transport formula also carries the elementary comparison theory of
	the operators. Quantiles obey the dictionary
	\begin{equation}
		Q_{\rho_{w}f}(p) = Q_{f}\big(A_{w}^{-1}(p)\big),
		\label{eq:quantile-dictionary}
	\end{equation}
	with $A_{w}^{-1}$ the generalized inverse when $A_{w}$ is not
	injective, so $\rho_{w}$ is monotone for first-order stochastic
	dominance, and
	every symmetric kernel conserves the median $m$: the median is the
	one quantile that no symmetric modulation can move. For the canonical
	kernel $w_{\mathrm{III}}$ one has $A_{w_{\mathrm{III}}}(z) \le z$ on
	$[0,\tfrac12]$ and $A_{w_{\mathrm{III}}}(z)\ge z$ on $[\tfrac12,1]$,
	whence the stochastic contraction
	\begin{equation}
		|X_{\rho_{w_{\mathrm{III}}}f} - m| \;\le_{\mathrm{st}}\; |X_{f} - m|,
		\label{eq:contraction}
	\end{equation}
	so each application contracts the law toward its conserved median,
	and every moment about the median decreases. On the metric side, the
	transport formula gives sharp Lipschitz bounds in the two natural
	transport distances: the pointwise estimate
	$|A_{w}(F_{f}) - A_{w}(F_{g})| \le \|w\|_{\infty}|F_{f} - F_{g}|$
	yields
	\begin{equation}
		d_{K}\big(\rho_{w}f,\rho_{w}g\big) \le \|w\|_{\infty}\,d_{K}(f,g),
		\qquad
		W_{1}\big(\rho_{w}f,\rho_{w}g\big) \le \|w\|_{\infty}\,W_{1}(f,g),
		\label{eq:lipschitz}
	\end{equation}
	for the Kolmogorov and Wasserstein-$1$ distances, the latter for
	laws with finite first moment, with
	$\|w\|_{\infty} = 2$ for the canonical kernel; a total-variation
	Lipschitz bound with constant
	$\|w\|_{\infty} + \tfrac12\mathrm{Lip}(w)$ holds as well
	(Lemma~\ref{lem:wellposed}). Sharper than any of these global bounds
	is the structure of the derivative itself.
	
	\begin{proposition}[Intertwining of the linearization]\label{prop:intertwine}
		For $f \in \mathcal{D}_{+}$, a kernel $w \in C^{1}([0,1])$, and any
		perturbation $h$ with
		$\int h = 0$, the G\^ateaux derivative of $\rho_{w}$ at $f$ is an
		exact derivative:
		\begin{equation}
			D\rho_{w}[f]\,h
			\;=\; \frac{d}{dx}\Big[\, w\big(F(x)\big)\int_{-\infty}^{x}h \,\Big].
			\label{eq:intertwine}
		\end{equation}
		Consequently, under the primitive map
		$h \mapsto \int_{-\infty}^{\cdot}h$, which identifies mean-zero
		perturbations with functions vanishing at $\pm\infty$, the
		linearization is conjugate to multiplication by $w\circ F$; realized
		on primitives in $L^{2}(\mathbb{R})$, its spectrum is the essential
		range of $w\circ F$, which coincides with the closure of the range
		of $w$ on $(0,1)$ because $F$ maps the support onto $(0,1)$, and is
		the interval $[0,2]$ for the canonical kernel,
		and, by
		the composition law, the linearization of $\rho_{w}^{\,n}$ has
		spectrum the range of $(A_{w}^{\circ n})'$, so that the sharp
		Lipschitz constant of $\rho_{w_{\mathrm{III}}}^{\,n}$ in $d_{K}$ and
		$W_{1}$ equals $\sup(A_{w_{\mathrm{III}}}^{\circ n})' = 2^{n}$.
	\end{proposition}
	
	\begin{proof}
		Differentiating the defining equation~\eqref{eq:operator} along the
		perturbation $h$ gives
		\begin{equation*}
			D\rho_{w}[f]\,h
			= w(F)\,h + w'(F)\Big(\int_{-\infty}^{x}h\Big)f,
		\end{equation*}
		and the right side of the identity~\eqref{eq:intertwine} expands to
		the same expression by the product rule. The conjugation statement is
		the identity~\eqref{eq:intertwine} read as
		\begin{equation*}
			D\rho_{w}[f] = \partial \circ M_{w\circ F}\circ \partial^{-1}
		\end{equation*}
		on mean-zero perturbations with primitives in $L^{2}(\mathbb{R})$,
		and multiplication operators on $L^{2}$ have spectrum
		equal to their essential range; the term spectrum refers throughout
		to this $L^{2}$ realization of the linearization, not to a
		similarity-invariant spectrum on $L^{1}$. The iterate statement
		follows from
		\begin{equation*}
			\big(A_{w_{\mathrm{III}}}^{\circ n}\big)'\big(\tfrac12\big) = 2^{n}
		\end{equation*}
		together with the Lipschitz bounds~\eqref{eq:lipschitz} applied to
		$A_{w_{\mathrm{III}}}^{\circ n}$; sharpness follows by comparing
		laws that differ only near the median, where the slope $2^{n}$ of
		$A_{w_{\mathrm{III}}}^{\circ n}$ is attained.
	\end{proof}
	
	The intertwining identity~\eqref{eq:intertwine} is used repeatedly in
	what follows. It re-proves mass conservation in one line, since the
	boundary values of $w(F)\int^{x} h$ vanish; it locates the expansion
	and contraction of the linearization on the rank regions where the
	kernel exceeds or falls below one (for the canonical kernel the
	transition set $\{w\circ F = 1\}$ is Lebesgue-null, so no genuinely
	neutral directions exist); and it gives the growth rate
	\begin{equation*}
		\log\big\|D\rho_{w_{\mathrm{III}}}^{\,n}\big\| = n\log 2.
	\end{equation*}
	Finally we
	record how the classical information functionals transform under the
	canonical operator. For $f \in \mathcal{D}_{+}\cap C^{1}$ bounded,
	with finite Fisher information $\mathcal{I}(f) = \int (f')^{2}/f$, the
	transformation law
	\begin{equation}
		\mathcal{I}\big(\rho_{w_{\mathrm{III}}}[f]\big)
		= 4\pi^{2}\,\mathbb{E}_{f}\big[f(X)^{2}\big]
		+ \mathbb{E}_{f}\Big[\,2\sin^{2}(\pi F)\Big(\frac{f'}{f}\Big)^{2}\Big]
		\label{eq:fisher-law}
	\end{equation}
	holds (Lemma~\ref{lem:fisher-law} in Appendix~\ref{app:proofs}),
	whence the two-sided bound
	\begin{equation}
		4\pi^{2}\,\mathbb{E}_{f}[f^{2}]
		\;\le\; \mathcal{I}\big(\rho_{w_{\mathrm{III}}}[f]\big)
		\;\le\; 4\pi^{2}\,\mathbb{E}_{f}[f^{2}] + 2\,\mathcal{I}(f).
		\label{eq:fisher-bounds}
	\end{equation}
	On the uniform law of unit length the identity collapses to
	$\mathcal{I}(\rho_{w_{\mathrm{III}}}[f]) = 4\pi^{2}$ exactly; on a
	uniform law of length $L$ it gives $4\pi^{2}/L^{2}$. Here and below,
	$\mathcal{I}(f)$ for a law with compact support is computed from the
	classical derivative on the interior of the support, so that
	$\mathcal{I}(\mathrm{uniform}) = 0$.
	
	\section{The Canonical Kernel}\label{sec:canonical}
	
	Rigidity classifies the operators; this section selects one. A
	variational principle singles out the Type-III kernel announced in
	Section~\ref{sec:foundations}, uniquely, and two further
	descriptions, probabilistic and information-theoretic, attach exact
	structure to it that the rest of the paper exploits; only the first
	is a characterization, the other two being structures the kernel
	shares, in part, with its harmonic ladder. Hereafter
	$\rho := \rho_{\mathrm{III}}$ and
	$w := w_{\mathrm{III}} = 2\sin^{2}(\pi z)$, and we write
	\begin{equation}
		A(z) = A_{w}(z) = z - \frac{\sin(2\pi z)}{2\pi}
		\label{eq:transport-map}
	\end{equation}
	for its transport map. Two associated families recur throughout: the
	\emph{harmonic kernels}
	\begin{equation}
		w_{k}(z) = 2\sin^{2}(k\pi z),
		\qquad
		A_{k}(z) = z - \frac{\sin(2\pi k z)}{2\pi k},
		\qquad k = 1, 2, \ldots,
		\label{eq:harmonic-kernels}
	\end{equation}
	and the \emph{phase-shifted kernels}
	\begin{equation}
		w_{k,\alpha}(z) = 2\sin^{2}\big(k\pi(z+\alpha)\big),
		\qquad \alpha \in [0,1),
		\label{eq:phase-kernels}
	\end{equation}
	all of which belong to $\mathcal{W}$. The harmonic and phase-shifted
	kernels are, up to the factor $2$, precisely the squared trigonometric
	modes that appear in Neyman's classical smooth alternatives to
	uniformity~\cite{neyman1937}, so the family has a statistical pedigree
	independent of the present theory: modulating a law by
	$w_{k,\alpha}$ is the density-level counterpart of embedding it in a
	smooth exponential family of rank alternatives.
	
	The variational characterization asks: which kernel disturbs a law
	least? The question is well posed because probability carries exactly
	one invariant Riemannian structure. By \v{C}encov's
	theorem~\cite{chentsov1982}, the Fisher--Rao metric
	\begin{equation}
		ds^{2} = \int_{0}^{1}\frac{(\delta w)^{2}}{w}\,dz
		\label{eq:fisher-rao}
	\end{equation}
	is the unique metric on densities invariant under sufficient
	statistics, the geometric counterpart of
	Theorem~\ref{thm:equivariance}. Under the substitution $w = \phi^{2}$
	the metric~\eqref{eq:fisher-rao} becomes exactly four times the round
	metric of the unit sphere of $L^{2}(0,1)$: amplitudes, square roots of
	densities, are the normal coordinates of information geometry.
	Disturbance is therefore measured by the Fisher information of the
	kernel,
	\begin{equation}
		\mathcal{I}(w) = \int_{0}^{1}\frac{(w')^{2}}{w}\,dz
		= 4\int_{0}^{1}(\phi')^{2}\,dz,
		\qquad \phi = \sqrt{w},
		\label{eq:kernel-fisher}
	\end{equation}
	and the natural boundary condition, that a modulation be inert at the
	extremes of rank, $w(0) = w(1) = 0$, turns minimal disturbance into a
	Rayleigh problem.
	
	\begin{theorem}[Variational characterization]\label{thm:variational}
		Among all kernels with $\sqrt{w} \in H^{1}_{0}(0,1)$, the Fisher
		information attains its minimum uniquely at the canonical kernel:
		\begin{equation}
			w(z) = 2\sin^{2}(\pi z), \qquad \mathcal{I}(w) = 4\pi^{2}.
			\label{eq:variational-min}
		\end{equation}
		The critical points of $\mathcal{I}$ under the same constraints,
		understood in the amplitude parametrization $w = \phi^{2}$ with
		$\phi \in H^{1}_{0}(0,1)$, are exactly the harmonic kernels, with
		critical values $\mathcal{I}(w_{k}) = 4k^{2}\pi^{2}$, and the
		Lagrange multiplier of the normalization constraint $\int w = 1$ at
		the $k$-th critical point equals $4k^{2}\pi^{2}$.
	\end{theorem}
	
	\begin{proof}
		With $\phi = \sqrt{w}$, the functional~\eqref{eq:kernel-fisher} is the
		Dirichlet energy $4\int(\phi')^{2}$, to be minimized on the unit
		sphere of $L^{2}(0,1)$ within $H_{0}^{1}(0,1)$: the minimizer is the first Dirichlet
		eigenfunction $\sqrt2\sin(\pi z)$, unique up to sign (hence $w$ unique
		outright), and the critical points are the eigenfunctions
		$\sqrt2\sin(k\pi z)$ with $\int(\phi_{k}')^{2} = k^{2}\pi^{2}$. The
		constrained Euler--Lagrange equation
		\begin{equation*}
			\delta\Big[\,4\int_{0}^{1}(\phi')^{2}
			- \mu\Big(\int_{0}^{1}\phi^{2}-1\Big)\Big] = 0
			\qquad\text{reads}\qquad
			\phi'' = -\frac{\mu}{4}\,\phi,
		\end{equation*}
		so $\mu = 4k^{2}\pi^{2}$ at the $k$-th critical point. For
		$k \ge 2$ the amplitude $\phi_{k}$ changes sign, so criticality is a
		statement about the $L^{2}$ sphere within $H_{0}^{1}(0,1)$, where variations are
		unconstrained, rather than about one-sided variations of the
		nonnegative kernel itself; the parametrization $w = \phi^{2}$ is the
		natural one here, being the normal-coordinate chart of the
		Fisher--Rao metric~\eqref{eq:fisher-rao}.
	\end{proof}
	
	The theorem recasts the family: the canonical kernel is not merely the
	entropy-neutral element of the product
	identity~\eqref{eq:product-identity} but \emph{the} least-disturbing
	boundary-inert modulation, and the harmonics form its natural spectral
	ladder. Everything discrete in this paper, eigenvalues, universal
	constants, quantization phenomena, traces back to the boundary
	condition in Theorem~\ref{thm:variational}; the discreteness of the
	harmonic ladder is forced by the requirement that a kernel vanish at
	both ends of rank space while carrying unit mass.
	
	The probabilistic description exhibits the canonical kernel as a
	Bayes rule. Attach to $X \sim f$ the binary variable $Y$ with
	conditional law
	\begin{equation}
		\mathbb{P}(Y = 1 \mid X = x) = \sin^{2}\big(\pi F(x)\big),
		\qquad
		\mathbb{P}(Y = 0 \mid X = x) = \cos^{2}\big(\pi F(x)\big).
		\label{eq:likelihood}
	\end{equation}
	The likelihood in equation~\eqref{eq:likelihood} depends on $x$ only
	through its rank, so by the universality
	principle~\eqref{eq:universality} the observation is unbiased for
	every law: $\mathbb{P}(Y=1) = \tfrac12$ for \emph{every} $f$. Bayes'
	rule then gives
	\begin{equation}
		f(\,\cdot \mid Y{=}1) = 2\sin^{2}(\pi F)f = \rho[f],
		\qquad
		f(\,\cdot \mid Y{=}0) = 2\cos^{2}(\pi F)f = \rho_{w_{1,1/2}}[f].
		\label{eq:bayes}
	\end{equation}
	The relation~\eqref{eq:bayes} says that the canonical operator is the
	posterior update of a universally unbiased binary observation, and
	that the complementary posterior is the half-phase shift of the same
	kernel, so the two-outcome update closes within the phase-shifted
	family, with the Bayesian consistency relation
	\begin{equation*}
		\tfrac12\,\rho[f] + \tfrac12\,\rho_{w_{1,1/2}}[f] = f:
	\end{equation*}
	the prior is recovered as the mixture of its two posteriors. This
	closure property is itself characteristic.
	
	\begin{proposition}[Uniqueness of the closed binary update]\label{prop:unique-update}
		If a binary observation has rank-local likelihood $\ell(F(x))$ and
		both posteriors lie in the family
		$\{2\sin^{2}(\pi(z+\alpha))\}_{\alpha}$, then
		$\ell(z) = \sin^{2}(\pi(z+\alpha))$ for some $\alpha$ and
		$\mathbb{P}(Y{=}1) = \tfrac12$.
	\end{proposition}
	
	\begin{proof}
		With $c = \int_{0}^{1}\ell$, note first that $0 < c < 1$: the
		likelihood is not identically $0$ or $1$, since both posteriors
		exist. The two posterior conditions then read
		\begin{equation*}
			\frac{\ell(z)}{c} = 2\sin^{2}\big(\pi(z+\alpha)\big),
			\qquad
			\frac{1-\ell(z)}{1-c} = 2\sin^{2}\big(\pi(z+\beta)\big).
		\end{equation*}
		Clearing denominators and adding the two relations gives, for
		every $z$,
		\begin{equation*}
			c\cdot 2\sin^{2}\big(\pi(z+\alpha)\big)
			+ (1-c)\cdot 2\sin^{2}\big(\pi(z+\beta)\big) = 1,
		\end{equation*}
		and by $2\sin^{2}\theta = 1 - \cos 2\theta$ this is the vanishing,
		for all $z$, of
		$c\cos(2\pi(z+\alpha)) + (1-c)\cos(2\pi(z+\beta))$, that is, the
		phasor identity
		\begin{equation*}
			c\,e^{2\pi i\alpha} = -(1-c)\,e^{2\pi i\beta},
		\end{equation*}
		whence, comparing moduli and phases, $c = \tfrac12$ and
		$\beta = \alpha + \tfrac12$ modulo one.
	\end{proof}
	
	The information identities quantify the update exactly.
	Because the update in relation~\eqref{eq:bayes} acts through ranks,
	its information functionals are universal constants. Write
	$D(\cdot\|\cdot)$ for relative entropy. Evaluating the two rank
	integrals by the Fourier expansion
	\begin{equation*}
		\log\sin(\pi z)
		= -\log 2 - \sum_{n\ge1}\frac{\cos(2\pi nz)}{n}
	\end{equation*}
	(Lemma~\ref{lem:kl} in Appendix~\ref{app:proofs}) gives, for every
	$f \in \mathcal{D}$,
	\begin{equation}
		D\big(f \,\|\, \rho[f]\big) = \log 2,
		\qquad
		D\big(\rho[f] \,\|\, f\big) = 1 - \log 2.
		\label{eq:kl-identities}
	\end{equation}
	Logarithms are natural throughout, so $\log 2$ nats is exactly one
	bit: each application of the canonical operator is one bit of
	relative information, for every law, and the informational price of
	the update is a universal constant, independent of what is being
	updated.
	The asymmetry of the two divergences in the
	identities~\eqref{eq:kl-identities} is itself meaningful. The
	divergence of $f$ from its update is $\log 2$, the entropy of the
	binary observation, while the divergence of the update from $f$ is the
	smaller number $1-\log 2$, and the difference reflects the fact that
	relative entropy prices the two directions of an inference
	differently. For the binary observation itself, $H(Y) = \log 2$, the
	mutual information is
	\begin{equation}
		I(X;Y) = \tfrac12\, D\big(\rho[f]\,\|\,f\big)
		+ \tfrac12\, D\big(\rho_{w_{1,1/2}}[f]\,\|\,f\big)
		= 1 - \log 2,
		\label{eq:mutual-info}
	\end{equation}
	and the equivocation is
	\begin{equation*}
		H(Y \mid X)
		= \int_{0}^{1}h_{2}\big(\sin^{2}\pi z\big)\,dz
		= 2\log2 - 1,
	\end{equation*}
	where $h_{2}$ is
	the binary entropy function; all three constants are independent of
	$f$ and of the harmonic order, since replacing $\sin^{2}(\pi F)$ by
	$\sin^{2}(k\pi F)$ changes nothing, by periodicity. Two remarks keep
	these identities in proportion. Law-independence of the costs is a
	consequence of rank locality, not of the canonical kernel: by the
	universality principle~\eqref{eq:universality}, \emph{every} kernel
	$v$ has the law-independent costs
	$D(f\Vert\rho_{v}f) = -\int_{0}^{1}\log v$ and
	$D(\rho_{v}f\Vert f) = \int_{0}^{1}v\log v$, each a value in
	$[0,\infty]$, the first made infinite by strong enough zeros of
	the kernel and the second by unbounded peaks, and
	the specific values above are shared by
	the entire harmonic ladder; what the identities contribute is the
	exact evaluation of the cost, not a further uniqueness claim. In
	statistical terms they give the exact Kullback--Leibler separation
	between any law and its modulated alternative, a multiplicative,
	density-level analogue of Neyman's smooth
	alternatives~\cite{neyman1937}, which are exponential tilts of
	orthonormal systems, uniformly over the null. By the composition
	law, the rank of every iterate is uniform, so the
	identities~\eqref{eq:kl-identities} hold at every step of the cascade
	$f_{n} = \rho^{n}[f]$: one has $D(f_{n}\|f_{n+1}) = \log 2$ and
	$D(f_{n+1}\|f_{n}) = 1-\log2$ for all $n$, which is the additive
	shadow, at rate one bit per step, of the multiplicative
	decomposition~\eqref{eq:iterates}. The endpoint divergences
	$D(f_{n}\|f)$ do not telescope; their exact asymptotics is a
	consequence of the limit law of iteration and is derived below.
	
	\section{The Amplitude Representation}\label{sec:amplitude}
	
	The canonical kernel is a squared eigenfunction:
	$w = |\varphi_{1}|^{2}$, where
	\begin{equation}
		\varphi_{k}(z) = \sqrt2\,\sin(k\pi z), \qquad k = 1, 2, \ldots,
		\label{eq:eigenbasis}
	\end{equation}
	is the orthonormal basis of $L^{2}(0,1)$ of Dirichlet eigenfunctions
	of $-\tfrac12\, d^{2}/dz^{2}$, with eigenvalues
	$E_{k} = k^{2}\pi^{2}/2$. Combined with the spherical form of the
	Fisher--Rao metric~\eqref{eq:fisher-rao}, this says that the natural
	home of the theory is not the space of densities but the space of
	\emph{amplitudes}, where the multiplicative structure of
	Section~\ref{sec:foundations} acquires a linear, unitary shadow. Note
	the identity $\mathcal{I}(w_{k}) = 8E_{k}$ linking
	Theorem~\ref{thm:variational} to this spectral data: the Fisher
	information of a harmonic kernel is, up to the universal factor $8$,
	the eigenvalue of the amplitude that generates it.
	
	\begin{definition}[Amplitude lift]\label{def:lift}
		For $f \in \mathcal{D}$ define
		$U_{f} : L^{2}(0,1) \to L^{2}(\mathbb{R})$ by
		\begin{equation}
			U_{f}\,g = (g\circ F)\,\sqrt{f}.
			\label{eq:lift}
		\end{equation}
	\end{definition}
	
	The lift takes a function on rank space, weights it by the amplitude
	$\sqrt f$ of the law, and produces a function on the line; it is the
	half-density counterpart of the quantile transport that underlies
	everything in Section~\ref{sec:foundations}.
	
	\begin{theorem}[Unitary dilation]\label{thm:dilation}
		Let $f \in \mathcal{D}$. Then:
		\begin{enumerate}[label=\normalfont(\roman*), leftmargin=*]
			\item $U_{f}$ is an isometry; for $f \in \mathcal{D}_{+}$ it is
			unitary onto $L^{2}(\operatorname{supp}f)$, with inverse
			\begin{equation*}
				\psi \;\longmapsto\; (\psi\circ Q)\,\sqrt{Q'}.
			\end{equation*}
			\item For every unit vector $g \in L^{2}(0,1)$,
			\begin{equation*}
				|U_{f}\,g|^{2} = \rho_{|g|^{2}}[f]:
			\end{equation*}
			the modulation family $\{\rho_{v}[f] : v \in \mathcal{W}\}$ is
			exactly the set of squared moduli of unit amplitudes transported by
			$U_{f}$. In particular,
			\begin{equation*}
				\rho_{w_{k}}[f] = \big|\psi_{k}[f]\big|^{2},
				\qquad
				\psi_{k}[f] := U_{f}\,\varphi_{k},
			\end{equation*}
			and, for $f \in \mathcal{D}_{+}$, the family
			$\{\psi_{k}[f]\}_{k\ge1}$ is an orthonormal basis of
			$L^{2}(\operatorname{supp}f)$; for general $f \in \mathcal{D}$ the
			support is read as $\{f > 0\}$ modulo null sets. The word dilation
			is used in the loose sense of a lift to a larger space.
			\item For every increasing differentiable homeomorphism $T$ with
			positive derivative, equivariance is unitarily implemented:
			\begin{equation*}
				U_{T_{\#}f} = V_{T}\,U_{f},
				\qquad
				V_{T}\,\psi = (\psi\circ T^{-1})\,\sqrt{\big|(T^{-1})'\big|};
			\end{equation*}
			for decreasing $T$ the rank reflection intervenes,
			$U_{T_{\#}f} = V_{T}\,U_{f}R$ with $(Rg)(z) = g(1-z)$.
		\end{enumerate}
	\end{theorem}
	
	\begin{proof}
		(i) is the probability integral transform read on half-densities:
		\begin{equation*}
			\langle U_{f}g,\, U_{f}h\rangle
			= \int_{\mathbb{R}} (g\bar h)(F)\,f\,d\lambda
			= \int_{0}^{1}g\bar h\,dz
			= \langle g, h\rangle;
		\end{equation*}
		surjectivity and the inverse follow from the bijectivity of $F$ on
		the support. (ii) and (iii) are immediate computations from the
		defining equation~\eqref{eq:lift}.
	\end{proof}
	
	Part (ii) upgrades the modulation family to a projective structure:
	superpositions $\sum c_{k}\varphi_{k}$ with $\sum|c_{k}|^{2} = 1$
	realize all of $\mathcal{W}$ as squared moduli, and cross terms
	$\cos((j{-}k)\pi F) - \cos((j{+}k)\pi F)$ appear in densities as
	interference fringes, structure invisible at the level of densities
	alone. The dilation thus converts the nonlinear family of modulations
	into linear geometry on the sphere of amplitudes, and questions about
	kernels become questions about vectors. Two exact consequences follow.
	First, a \emph{universal expansion}: since $\sqrt{f} = U_{f}\mathbf{1}$
	and
	\begin{equation*}
		\langle\mathbf{1},\varphi_{k}\rangle
		= \frac{\sqrt2\,(1-\cos k\pi)}{k\pi},
	\end{equation*}
	every root-density has the same coordinates over its own ladder,
	\begin{equation}
		\sqrt{f} = \sum_{k\ \mathrm{odd}} \frac{2\sqrt2}{k\pi}\,\psi_{k}[f],
		\qquad
		\big|\big\langle\sqrt f,\, \psi_{1}[f]\big\rangle\big|^{2}
		= \frac{8}{\pi^{2}},
		\label{eq:universal-expansion}
	\end{equation}
	with Parseval closing exactly, since
	$\sum_{\mathrm{odd}}8/k^{2}\pi^{2} = 1$. Stated plainly, the
	expansion~\eqref{eq:universal-expansion} is the sine series of the
	constant function, transported by the unitary $U_{f}$; its content is
	not the series but the universality of the coordinates, which are the
	same for every law over that law's own ladder. The constant
	$8/\pi^{2}$ in
	the expansion~\eqref{eq:universal-expansion} plays the role of a
	universal fidelity between a law and its canonical transform, the
	amplitude-level counterpart of the entropy
	constants~\eqref{eq:kl-identities}: about $81\%$ of every law, in the
	squared-amplitude sense, already lies along its own canonical
	modulation. Second, an \emph{isospectral family}: the conjugated
	operator
	\begin{equation*}
		H_{f} := U_{f}\Big(-\tfrac12\,\frac{d^{2}}{dz^{2}}\Big)U_{f}^{-1},
		\qquad
		\operatorname{dom}H_{f}
		= U_{f}\big(H^{2}(0,1)\cap H^{1}_{0}(0,1)\big),
	\end{equation*}
	is a self-adjoint Sturm--Liouville operator on
	$L^{2}(\operatorname{supp}f)$, which for sufficiently smooth $f$
	takes the classical form
	\begin{equation}
		-\Big(\frac{\chi'}{f}\Big)' = 2E\, f\,\chi,
		\qquad \chi = \psi/\sqrt f,
		\label{eq:sturm-liouville}
	\end{equation}
	with purely discrete spectrum $\{k^{2}\pi^{2}/2\}$ \emph{independent
		of} $f$ and eigenfunctions
	\begin{equation}
		\psi_{k}[f] = \sqrt{2f}\,\sin\big(k\pi F\big),
		\label{eq:eigenfunctions}
	\end{equation}
	whose $k-1$ interior zeros sit at the $j/k$-quantiles of $f$. Every
	law in $\mathcal{D}_{+}$ thus carries an isospectral vibrating system
	whose frequencies are universal and whose mode shapes encode the
	distribution; equivalently, the classical Liouville substitution of
	Sturm--Liouville theory~\cite{zettl2005},
	$\psi = (u\circ F)\sqrt{F'}$, is \emph{exact} for this family, which
	therefore constitutes a reservoir of Sturm--Liouville problems
	solvable in closed form. The substitution itself is classical; what
	the operator viewpoint adds is its universality, one exactly solvable
	system per law, with a spectrum that never moves. The same substitution satisfies an energy
	decomposition of independent interest
	(Lemma~\ref{lem:dirichlet} in Appendix~\ref{app:proofs}): for real
	$u \in H^{1}(0,1)$ and $\psi = (u\circ F)\sqrt{F'}$,
	\begin{equation}
		\tfrac12\int_{\mathbb{R}}|\psi'|^{2}dx
		= \tfrac12\int_{0}^{1}|u'|^{2}\,\mathfrak{f}^{2}\,dz
		+ \int_{\mathbb{R}}\Big(\!-\frac{(\sqrt f)''}{2\sqrt f}\Big)|\psi|^{2}dx,
		\qquad \mathfrak{f} = f\circ Q,
		\label{eq:dirichlet-decomposition}
	\end{equation}
	the difference of the two Dirichlet energies being the exact
	derivative
	\begin{equation*}
		\frac{d}{dx}\Big[\tfrac14\, u(F)^{2}F''\Big];
	\end{equation*}
	the self-referential
	case $u \equiv \mathbf{1}$ collapses the
	decomposition~\eqref{eq:dirichlet-decomposition} to
	$\tfrac18\mathcal{I}(f)$, tying it back to
	Theorem~\ref{thm:variational}.
	
	The amplitude picture carries a canonical pair and an uncertainty
	principle. On $L^{2}(0,1)$ the multiplication operator $Z$ by $z$ and
	the derivative $P = -\tfrac{i}{2\pi}\partial_{z}$ satisfy
	$[Z,P] = \tfrac{i}{2\pi}$ on $C_{c}^{\infty}(0,1)$. Because the
	interval has ends, $P$ is symmetric but not essentially self-adjoint
	there, so the textbook commutator argument does not apply to every
	state; the correct statement, valid for every unit
	$\psi \in H^{1}_{0}(0,1)$, follows from a single integration by
	parts,
	\begin{equation*}
		1 = -2\operatorname{Re}\int_{0}^{1}
		\big(z - \langle Z\rangle\big)\,\bar\psi\,\psi'\,dz
		\;\le\; 2\,\Delta Z\,\|\psi'\|_{2},
	\end{equation*}
	whence
	\begin{equation}
		\Delta Z \cdot \frac{\|\psi'\|_{2}}{2\pi} \;\ge\; \frac{1}{4\pi},
		\label{eq:uncertainty}
	\end{equation}
	and for real boundary-vanishing states $\|\psi'\|_{2}/(2\pi)$ is
	exactly the momentum spread $\Delta P$. By
	Theorem~\ref{thm:dilation}(i) and (ii) the
	inequality~\eqref{eq:uncertainty} transports to every law: no
	distribution can be simultaneously concentrated in its own rank order
	and composed of few harmonics of its ladder. Localization in rank and
	localization in harmonic content are complementary in the exact,
	quantitative sense of the commutation relation. For the ladder itself
	the moments evaluate in closed form
	(Lemma~\ref{lem:uncertainty} in Appendix~\ref{app:proofs}):
	\begin{equation}
		\Delta Z(\varphi_{k})
		= \Big(\frac{1}{12} - \frac{1}{2\pi^{2}k^{2}}\Big)^{1/2},
		\qquad
		\Delta P(\varphi_{k}) = \frac{k}{2},
		\label{eq:ladder-moments}
	\end{equation}
	so the canonical kernel has uncertainty product
	$\approx 0.090$, within $14\%$ of the bound $1/4\pi \approx 0.080$.
	
	The dilation also makes the modulation family a homogeneous space for
	a periodic unitary flow. Evolving the kernel amplitude by
	$u(\cdot,\tau) = e^{-i\tau H_{0}}g$, where
	$H_{0} = -\tfrac12\partial_{z}^{2}$ with Dirichlet boundary
	conditions, and transporting over a fixed base,
	$f_{\tau} = \rho_{|u(\cdot,\tau)|^{2}}[f]$, the phases
	$e^{-iE_{k}\tau}$ satisfy
	\begin{equation}
		e^{-iE_{k}\cdot 4/\pi} = e^{-2\pi ik^{2}} = 1
		\qquad\text{for every } k:
		\label{eq:period}
	\end{equation}
	the evolution of the modulated density is \emph{exactly periodic} with
	universal period $4/\pi$, the half-period acting as the rank
	reflection $z\mapsto 1-z$ through the parity identity
	$\varphi_{k}(1-z) = (-1)^{k+1}\varphi_{k}(z)$, and rational fractions
	of the period decomposing into finitely many shifted copies by the
	Gauss-sum mechanism familiar from the fractional Talbot effect of
	wave optics~\cite{berryklein1996}. The reflection
	$R\psi(z) = \psi(1-z)$
	commutes with $H_{0}$ and with every symmetric-kernel multiplication,
	so the theory generated by symmetric kernels splits into two invariant
	sectors, $\mathcal{H}_{\pm}$ spanned by odd- and even-index ladder
	elements; any kernel with a nonvanishing antisymmetric part mixes the
	sectors, and the phase-shifted kernels, for which
	$w_{k,\alpha}(1-z) = w_{k,-\alpha}(z)$, are the canonical such
	elements within the harmonic family. Within a sector the theory is
	irreducible in the strongest sense.
	
	\begin{theorem}[Irreducibility on the symmetric sector]\label{thm:irreducible}
		Let $g = \mathbf{1}_{(1/4,3/4)} - \mathbf{1}_{(0,1/4)\cup(3/4,1)}$.
		The von Neumann algebra generated by
		$\{e^{i\tau H_{0}}M_{g}e^{-i\tau H_{0}} : \tau\in\mathbb{R}\}$ acts
		irreducibly on $\mathcal{H}_{+}$.
	\end{theorem}
	
	\begin{proof}
		The matrix elements of $M_{g}$ in the ladder basis are
		\begin{equation*}
			g_{jk} = \langle\varphi_{j}|g|\varphi_{k}\rangle
			= \beta_{|j-k|} - \beta_{j+k},
			\qquad
			\beta_{m} = \int_{0}^{1}g(z)\cos(m\pi z)\,dz.
		\end{equation*}
		A direct evaluation gives
		\begin{equation*}
			\beta_{m} = 0 \ \text{ for odd } m,
			\qquad
			\beta_{2r} = \mp\frac{2}{r\pi} \ \text{ for odd } r,
			\qquad
			\beta_{2r} = 0 \ \text{ for even } r,
		\end{equation*}
		the first because $g$ is symmetric about $\tfrac12$ while
		$\cos(m\pi z)$ is antisymmetric about $\tfrac12$ for odd $m$; the
		symmetry of $g$ is also what makes $M_{g}$ preserve
		$\mathcal{H}_{+}$ in the first place. Thus
		$\beta_{m} \neq 0$ exactly when $m \equiv 2 \pmod 4$. Since the
		odd-index spectrum is nondegenerate, weak-operator time averages of
		the orbit converge to the diagonal part
		\begin{equation*}
			D = \sum_{k\ \mathrm{odd}} g_{kk}\,P_{k},
			\qquad
			g_{kk} = -\beta_{2k} = \pm\frac{2}{k\pi}\ \text{ all distinct},
		\end{equation*}
		so every spectral projection $P_{k}$ lies in the algebra. For odd
		$j \ne k$, both $|j-k|$ and $j+k$ are even with sum
		$2\max(j,k) \equiv 2\pmod4$, so \emph{exactly one} of
		$\beta_{|j-k|},\beta_{j+k}$ is nonzero and $g_{jk}\ne0$ for every odd
		pair; the elements
		\begin{equation*}
			P_{j}M_{g}P_{k}
			= g_{jk}\,|\varphi_{j}\rangle\langle\varphi_{k}|
		\end{equation*}
		generate all matrix units, and the commutant is trivial.
	\end{proof}
	
	The meaning of Theorem~\ref{thm:irreducible} is that the periodic
	dynamics, observed through a single balanced rank observable, already
	generates every operator on the symmetric sector: nothing in the
	sector commutes with the evolved observables except scalars, so the
	representation carries no hidden invariant structure. Finally, the
	amplitude picture explains the Fourier action of the canonical
	operator. Writing
	\begin{equation*}
		\psi_{1}[f]
		= \frac{1}{i\sqrt2}\big(e^{i\pi F}-e^{-i\pi F}\big)\sqrt f
	\end{equation*}
	and defining the \emph{rank-characteristic transform}
	\begin{equation}
		\mathcal{A}_{f}(t,k) \;=\;
		\mathbb{E}\big[e^{itX + 2\pi i k F(X)}\big],
		\qquad t\in\mathbb{R},\ k\in\mathbb{Z},
		\label{eq:rank-transform}
	\end{equation}
	one has $\mathcal{A}_{f}(t,0) = \hat f(t)$ and
	$\mathcal{A}_{f}(0,k) = \delta_{k0}$ by rank uniformity, together with
	invariance of the $k$-marginal structure under monotone maps, and the
	spectral representation obtained in~\cite{ataei2025} takes the
	compact form
	\begin{equation}
		\widehat{\rho[f]}(t)
		= \mathcal{A}_{f}(t,0)
		- \tfrac12\,\mathcal{A}_{f}(t,1) - \tfrac12\,\mathcal{A}_{f}(t,-1):
		\label{eq:spectral-rep}
	\end{equation}
	the transform of the canonical image is the interference of the two
	half-phase branches. The transform~\eqref{eq:rank-transform}
	separates, in its two arguments, the metric content of a law (the
	$t$-direction, tied to the affine structure of the line) from its
	ordinal content (the $k$-direction, invariant under increasing
	reparametrization), and appears to merit independent study.
	
	\section{Iteration: Condensation, Universality, Crystallization}
	\label{sec:iteration}
	
	Throughout this section $f \in \mathcal{D}_{+}$ with (unique) median
	$m$, and we write $f_{n} = \rho^{n}[f]$, $F_{n} = A^{\circ n}\circ F$,
	and $X_{n} \sim f_{n}$. By the composition law, the entire asymptotic
	theory of iterated derangetropy is the asymptotic theory of one
	interval map, and we begin with its phase portrait. From
	$A' = 2\sin^{2}(\pi z)$ and the symmetry $A(1-z) = 1-A(z)$, the
	transport map~\eqref{eq:transport-map} is an increasing analytic
	homeomorphism of $[0,1]$ fixing $0,\tfrac12,1$; the interior fixed
	point is repelling, with local expansion
	\begin{equation}
		A\big(\tfrac12+v\big)
		= \tfrac12 + 2v - \tfrac{2\pi^{2}}{3}v^{3} + O(v^{5}),
		\label{eq:interior-germ}
	\end{equation}
	and the endpoints are superattracting, with cubic germ
	\begin{equation}
		A(z) = \tfrac{2\pi^{2}}{3}\,z^{3} + O(z^{5}), \qquad z \to 0.
		\label{eq:endpoint-germ}
	\end{equation}
	Orbits in $(0,\tfrac12)$ decrease to $0$ and orbits in $(\tfrac12,1)$
	increase to $1$; hence $F_{n} \to \mathbf{1}_{[m,\infty)}$ pointwise
	off $m$, that is, $X_{n} \to m$ in distribution, monotonically in the
	sense of the stochastic contraction~\eqref{eq:contraction}: iteration
	\emph{condenses} every law onto its conserved median. Condensation is
	the macroscopic statement; the refined description of the collapse,
	at the correct microscopic scale, is a universality theorem.
	
	\begin{theorem}[Universal limit law]\label{thm:limit-law}
		Let $K$ be the Koenigs linearizer of $A$ at $\tfrac12$: the unique
		increasing solution of the Schr\"oder equation
		$K(2u) = A(K(u))$, $K(0)=\tfrac12$, $K'(0)=1$; then $K$ extends to a
		real-analytic increasing bijection $\mathbb{R}\to(0,1)$ and is the
		distribution function of a symmetric law $\mathcal{K}$ with analytic
		density. For every $f \in \mathcal{D}_{+}$,
		\begin{equation}
			2^{n} f(m)\,(X_{n} - m) \;\xrightarrow{\ d\ }\; \mathcal{K},
			\label{eq:limit-law}
		\end{equation}
		and if $\mathbb{E}|X|^{\varepsilon} < \infty$ for some
		$\varepsilon>0$, all moments converge; in particular
		\begin{equation}
			\operatorname{Var}(X_{n})
			\;\sim\; \sigma_{\mathcal{K}}^{2}\, f(m)^{-2}\, 4^{-n},
			\qquad \sigma^{2}_{\mathcal{K}} \approx 0.166.
			\label{eq:variance-collapse}
		\end{equation}
		Moreover, with $g = 1-K$, the symmetry of $A$ gives the exact tail
		recursion $g(2u) = A(g(u))$, whence
		\begin{equation}
			-\log\big(1 - K(u)\big)
			= u^{\log_{2}3}\,\Theta(\log_{2}u)\big(1 + O(u^{-\log_{2}3})\big)
			\qquad (u \to \infty)
			\label{eq:tail-law}
		\end{equation}
		for a positive continuous $1$-periodic $\Theta$: the tails of
		$\mathcal{K}$ are compressed-exponential with exponent
		$\log_{2}3$, lighter than exponential and heavier than Gaussian, the
		exponent being the ratio of the logarithms of the endpoint degree and
		the interior multiplier; consistently, the excess kurtosis of
		$\mathcal{K}$ is positive ($\approx 0.164$).
	\end{theorem}
	
	The detailed proof, comprising the construction and global extension
	of $K$, the locally uniform convergence, the uniform integrability
	supplied by the doubly exponential absorption at the superattracting
	endpoints, and the affine recursion for $\log g$, is given in
	Appendix~\ref{app:proofs}. Three comments belong in the text. First,
	a moment hypothesis cannot be dispensed with: laws with slowly
	varying tails have
	$\mathbb{E}|X_{n}|^{p} = \infty$ for every $n$. Second, the iteration
	admits a heuristic diffusion-type reading. In the spectral
	representation~\eqref{eq:spectral-rep}, replacing the two branches
	$\mathcal{A}_{f}(t,\pm1)$ by the shifted transforms
	$\hat f(t\pm2\pi)$, a substitution that is exact for the uniform law
	and approximately valid while the law remains diffuse on the scale of
	its own quantiles, turns each application into the discrete heat
	evolution
	\begin{equation*}
		\hat f \;\longmapsto\;
		\hat f - \tfrac12\big[\hat f(\cdot+2\pi) + \hat f(\cdot-2\pi)\big].
	\end{equation*}
	This reading describes the early, diffusive phase of the cascade; the
	exact long-time behavior, by
	Theorem~\ref{thm:limit-law}, is governed by the Schr\"oder equation,
	with geometric rate $4^{-n}$ fixed by the multiplier
	$A'(\tfrac12) = 2$, equivalently, by
	Proposition~\ref{prop:intertwine}, by the top of the spectrum of the
	linearization. Third, nothing here is special to the canonical
	kernel, provided its global phase portrait is retained and the
	transport map is $C^{1+\alpha}$ at the center with endpoint zeros
	of exact integer order, as all kernels considered here are. For a
	symmetric kernel $v$ whose transport map satisfies
	\begin{equation*}
		A_{v}(z) < z \ \text{ on } (0,\tfrac12),
		\qquad
		A_{v}(z) > z \ \text{ on } (\tfrac12,1),
	\end{equation*}
	so that $0, \tfrac12, 1$ are its only fixed points, with central
	multiplier $v(\tfrac12) = \lambda > 1$ and endpoint zeros of order
	$d-1$, so that $A_{v}(z) \asymp z^{d}$ at $0$, the same arguments
	give condensation onto the median with variance collapsing at rate
	$\lambda^{-2n}$ and distributional scale $\lambda^{-n}$, a Koenigs
	limit law with multiplier $\lambda$, and tail exponent
	$\log d/\log\lambda$; the limit law itself is the Koenigs law of
	$A_{v}$, whose body and log-periodic tail modulation depend on the
	whole map, not only on the pair $(\lambda, d)$. The interior
	fixed-point hypothesis is essential and is not implied by the local
	data: the harmonic kernel $w_{3}$ shares the local pair
	$(\lambda, d) = (2, 3)$ with the canonical kernel, yet its transport
	map $A_{3}$ fixes every multiple of $\tfrac16$, and its iterates
	crystallize onto three atoms rather than condensing onto the median,
	as shown below. The kernels of
	equation~\eqref{eq:three-kernels} do satisfy the global hypothesis,
	with multipliers
	$w_{\mathrm{I}}(\tfrac12) = \tfrac{12}{\pi e}$ and
	$w_{\mathrm{II}}(\tfrac12) = \tfrac{2e}{\pi}$ and $d = 2$, giving
	tail exponents $\approx 2.038$ and $\approx 1.264$ respectively, on
	either side of the canonical $\log_{2}3 \approx 1.585$. Within the
	single-well class, then, the pair $(\lambda, d)$, two numbers read
	off the kernel at its center and its endpoints, classifies the
	scaling structure of the collapse: its geometric rate and its tail
	exponent.
	
	The limit law also closes the information ledger of the cascade. By
	the composition law and the universality
	principle~\eqref{eq:universality}, the divergence of the $n$-th
	iterate from its origin is a rank integral,
	\begin{equation}
		D\big(f_{n}\,\|\,f\big)
		= \int_{0}^{1}\big(A^{\circ n}\big)'(z)\,
		\log\big(A^{\circ n}\big)'(z)\,dz,
		\label{eq:endpoint-divergence}
	\end{equation}
	and the substitution $z = \tfrac12 + 2^{-n}u$ converts the
	integral~\eqref{eq:endpoint-divergence} into the exact identity
	\begin{equation}
		D\big(f_{n}\,\|\,f\big) = n\log 2 - H(K_{n}),
		\label{eq:divergence-entropy}
	\end{equation}
	valid for every $n$, where $H(K_{n})$ is the differential entropy of
	the law with distribution function
	$K_{n}(u) = A^{\circ n}(\tfrac12 + 2^{-n}u)$. Since
	$H(K_{n}) \to H(\mathcal{K})$, by Step~6 of the proof of
	Theorem~\ref{thm:limit-law},
	\begin{equation}
		D\big(f_{n}\,\|\,f\big)
		= n\log 2 - H(\mathcal{K}) + o(1),
		\qquad H(\mathcal{K}) \approx 0.521:
		\label{eq:cstar}
	\end{equation}
	the deficit in the one-bit-per-step divergence ledger is precisely
	the differential entropy of the universal limit law. The
	asymptotics~\eqref{eq:cstar} ties the information theory of the
	cascade back to its universality theory: the same object
	$\mathcal{K}$ governs both the shape of the collapse and the exact
	information cost of iteration. Whether $H(\mathcal{K})$ admits a
	closed form in elementary constants remains open.
	
	The harmonic ladder condenses onto sets rather than points, a
	phenomenon we call crystallization. The affine conjugacy
	\begin{equation*}
		\psi_{i}\circ A_{k} = A\circ\psi_{i}
		\quad\text{on each cell }
		\Big[\tfrac{i-1}{k},\tfrac{i}{k}\Big],
		\qquad
		\psi_{i}(z) = kz-(i-1),
	\end{equation*}
	is an exact identity by periodicity of the sine, and shows that
	$A_{k}$ is $k$ disjoint copies of $A$. Under iteration of $\rho_{w_{k}}$,
	therefore,
	\begin{equation}
		X_{n} \;\xrightarrow{\ d\ }\;
		\sum_{i=1}^{k}\frac1k\,\delta_{q_{i}},
		\qquad q_{i} = Q\Big(\frac{2i-1}{2k}\Big),
		\label{eq:crystallization}
	\end{equation}
	an equal-weight mixture of atoms at fixed quantiles. Around every
	atom the collapse is governed by the same local law and the same rate
	as in Theorem~\ref{thm:limit-law}, with a scale set by the cell
	width: conditionally on the $i$-th cell,
	\begin{equation*}
		2^{n}\,k\,f(q_{i})\,\big(X_{n} - q_{i}\big)
		\;\xrightarrow{\ d\ }\; \mathcal{K},
	\end{equation*}
	the factor $k$ arising because the conjugacy $\psi_{i}$ has slope
	$k$. Crystallization has a consumer outside the theory: it computes
	an optimal quantizer. Among all laws supported on $k$ atoms of equal
	mass $1/k$, the one nearest to $f$ in the Wasserstein-$1$ distance
	places its atoms exactly at the odd quantiles, since
	\begin{equation*}
		W_{1}\Big(f,\ \sum_{i=1}^{k}\tfrac1k\,\delta_{q_{i}}\Big)
		= \sum_{i=1}^{k}\int_{(i-1)/k}^{i/k}
		\big|Q(z) - q_{i}\big|\,dz
	\end{equation*}
	is minimized cell by cell at the conditional medians
	$q_{i} = Q\big(\tfrac{2i-1}{2k}\big)$. Iterated harmonic derangetropy
	therefore computes the $W_{1}$-optimal equal-mass quantizer of every
	continuous law by pure rank dynamics, with no optimization step, a
	parameter-free counterpart of the quantization theory of probability
	distributions~\cite{graf2000}.
	
	No density is fixed by $\rho$, since $w(F)\equiv1$ is impossible, but
	the operator possesses a natural self-referential eigenrelation: a law
	whose density coincides with its own modulation profile,
	\begin{equation*}
		f = w(F),
		\qquad\text{that is,}\qquad
		F' = 2\sin^{2}(\pi F).
	\end{equation*}
	Separation of variables gives,
	uniquely up to translation, the Cauchy law of scale $1/(2\pi)$:
	\begin{equation}
		f(x) = \frac{2}{1+4\pi^{2}x^{2}},
		\qquad
		F(x) = \frac12 + \frac{\arctan(2\pi x)}{\pi},
		\label{eq:coherent}
	\end{equation}
	supported on all of $\mathbb{R}$, the boundary values $F = 0,1$ being
	attained only at infinity, by nonintegrability of $dF/F^{2}$ at $0$.
	For this law, and only this law, $\rho[f] = f^{2}$ pointwise, with
	$\int f^{2} = 1$ automatic by the universality
	principle~\eqref{eq:universality}. We call the law of
	equation~\eqref{eq:coherent} the \emph{coherent state} of the
	canonical operator: it is the unique law that reproduces its own
	modulation, the self-consistent solution of the theory, and it will
	recur below as an exact invariant manifold of the continuous-time
	dynamics and as the distinguished regular geometry of the projective
	curvature.
	
	\section{The Derangetropy Flow}\label{sec:flow}
	
	The discrete iteration embeds in a continuous-time evolution whose
	generator is the operator itself:
	\begin{equation}
		\partial_{t}f = \rho[f] - f = \big(2\sin^{2}(\pi F)-1\big)f
		\qquad\Longleftrightarrow\qquad
		\partial_{t}F = -\frac{\sin(2\pi F)}{2\pi},
		\label{eq:flow}
	\end{equation}
	the distribution-function form following from the transport
	formula~\eqref{eq:transport}, since
	\begin{equation*}
		A(z) - z = -\frac{\sin(2\pi z)}{2\pi}.
	\end{equation*}
	Global well-posedness in $L^{1}$, together with conservation of mass,
	positivity, $\mathcal{D}_{+}$, and the median, is established in
	Lemma~\ref{lem:flow-wellposed} of Appendix~\ref{app:proofs}, using the
	Lipschitz bounds~\eqref{eq:lipschitz}. Three structural readings
	orient the analysis, and each opens a different door. The evolution
	equation~\eqref{eq:flow} is a continuity equation with current
	$J = \sin(2\pi F)/(2\pi)$, vanishing at the median and the support
	endpoints, maximal at the quartiles, always directed toward the
	median: probability flows inward along the line, at a rate determined
	solely by rank. It is also a replicator equation,
	\begin{equation*}
		\partial_{t}f = f\,\big(w(F) - \bar w\big),
	\end{equation*}
	with rank-dependent fitness and
	mean fitness $\bar w \equiv 1$ by the universality
	principle~\eqref{eq:universality}: each point of the line reproduces
	in proportion to how its rank is weighted, and the population of mass
	reorganizes itself with total size conserved, a form that invites
	comparison with the replicator dynamics of evolutionary game
	theory~\cite{hofbauer1998}, with the law's own rank in place of a
	payoff. And in quantile
	coordinates it is \emph{linear advection},
	\begin{equation}
		\partial_{t}Q = \frac{\sin(2\pi z)}{2\pi}\,\partial_{z}Q,
		\label{eq:quantile-advection}
	\end{equation}
	so the flow is linear in the geometry of quantiles, equivalently in
	the geometry of $W_{1}$, though nonlinear and nonlocal in the
	geometry of densities. The coexistence of these readings, one
	conservative, one selective, one geometric, is a recurring signature
	of the rank structure.
	
	\begin{theorem}[Exact solution and Lorentzian condensation]\label{thm:flow}
		Along the flow~\eqref{eq:flow} the function $V = \tan(\pi F)$
		satisfies the linear equation $\partial_{t}V = -V$ pointwise; hence
		\begin{equation}
			\tan\big(\pi F(x,t)\big) = e^{-t}\tan\big(\pi F_{0}(x)\big),
			\qquad
			f(x,t) = \frac{e^{-t}f_{0}}
			{\cos^{2}(\pi F_{0}) + e^{-2t}\sin^{2}(\pi F_{0})}.
			\label{eq:exact-solution}
		\end{equation}
		Consequently, for every $f_{0} \in \mathcal{D}_{+}$,
		\begin{equation}
			e^{t}\,(X_{t} - m) \;\xrightarrow{\ d\ }\;
			\mathrm{Cauchy}\Big(0, \frac{1}{\pi f_{0}(m)}\Big),
			\label{eq:lorentz-limit}
		\end{equation}
		and, under the integrability hypothesis of
		Lemma~\ref{lem:flow-wellposed}, the differential entropy satisfies
		\begin{equation}
			H(f_{t}) = -t + \log\frac{4}{f_{0}(m)} + o(1),
			\label{eq:entropy-law}
		\end{equation}
		so that entropy is dissipated at unit asymptotic rate;
		correspondingly $H(f_{n+1}) - H(f_{n}) \to -\log 2$ along the
		discrete iteration. The constant in the entropy
		law~\eqref{eq:entropy-law} is the entropy of the limiting Lorentzian
		profile in the limit~\eqref{eq:lorentz-limit}.
	\end{theorem}
	
	\begin{proof}
		The linearization is a direct computation:
		\begin{equation*}
			\partial_{t}\tan(\pi F)
			= \pi\sec^{2}(\pi F)\,\partial_{t}F
			= -\tfrac12\sec^{2}(\pi F)\sin(2\pi F)
			= -\tan(\pi F).
		\end{equation*}
		The two half-intervals of rank are invariant, so the tangent inverts
		within the correct branch, and differentiation in $x$ gives the exact
		solution~\eqref{eq:exact-solution}. For the
		limit~\eqref{eq:lorentz-limit}, set $x_{t} = m + e^{-t}u$; then
		$F_{0}(x_{t}) = \tfrac12 + f_{0}(m)e^{-t}u(1+o(1))$, so
		\begin{equation*}
			\tan\big(\pi F_{0}(x_{t})\big)
			= -\cot\big(\pi f_{0}(m)e^{-t}u(1+o(1))\big)
			= -\frac{1+o(1)}{\pi f_{0}(m)e^{-t}u},
		\end{equation*}
		whence
		\begin{equation*}
			e^{-t}\tan\big(\pi F_{0}(x_{t})\big)
			\;\longrightarrow\; -\frac{1}{\pi f_{0}(m)\,u},
			\qquad
			F(x_{t},t)
			\;\longrightarrow\;
			\frac12 + \frac{1}{\pi}\arctan\big(\pi f_{0}(m)\,u\big),
		\end{equation*}
		the last step by the reflection identity
		$\arctan s + \arctan(1/s) = \tfrac{\pi}{2}\,\mathrm{sgn}\,s$. The
		entropy statements follow from the scaling limits; details are in
		Lemma~\ref{lem:flow-wellposed} of Appendix~\ref{app:proofs}.
	\end{proof}
	
	The discrete and continuous dynamics thus condense with
	\emph{different} universal laws: the map produces $\mathcal{K}$, with
	compressed-exponential tails, while the flow produces the Cauchy law,
	with polynomial tails. The mechanism is worth recording as a general
	dictionary: the tails of the limit law are manufactured at the
	attracting endpoints of rank space, and superattracting (cubic)
	absorption for the map corresponds to merely exponential absorption
	for the flow. What the limit law remembers of the initial condition is
	minimal in both cases, a single number, the density at the median,
	setting the scale. The Cauchy law's second appearance is no accident.
	
	\begin{theorem}[Invariant manifold and its hyperbolic geometry]\label{thm:invariant}
		The Cauchy family is exactly invariant under the flow:
		$\mathrm{Cauchy}(m,\gamma) \mapsto \mathrm{Cauchy}(m,\gamma e^{-t})$,
		with reduced dynamics $\dot m = 0$, $\dot\gamma = -\gamma$ and exact
		entropy law $H = \log(4\pi\gamma)$, $\dot H = -1$. The Fisher--Rao
		metric of the family is
		\begin{equation}
			ds^{2} = \frac{dm^{2}+d\gamma^{2}}{2\gamma^{2}},
			\label{eq:hyperbolic-metric}
		\end{equation}
		a hyperbolic plane of curvature $-2$, and the flow traces its vertical
		geodesics at constant speed $1/\sqrt2$.
	\end{theorem}
	
	\begin{proof}
		For the Cauchy chart,
		\begin{equation*}
			\tan\big(\pi F_{m,\gamma}(x)\big) = -\frac{\gamma}{x-m},
		\end{equation*}
		and the exact solution~\eqref{eq:exact-solution} multiplies the
		tangent by $e^{-t}$, which is precisely the map
		$\gamma \mapsto \gamma e^{-t}$; the entropy of
		$\mathrm{Cauchy}(\gamma)$ is classical. The Fisher integrals evaluate
		in closed form to
		\begin{equation*}
			I_{mm} = I_{\gamma\gamma} = \frac{1}{2\gamma^{2}},
			\qquad
			I_{m\gamma} = 0
		\end{equation*}
		(Lemma~\ref{lem:cauchy-metric} in Appendix~\ref{app:proofs}); a
		metric $a(dm^{2}+d\gamma^{2})/\gamma^{2}$ has curvature $-1/a$,
		vertical lines are geodesics, and the speed along
		$\gamma_{0}e^{-t}$ is $\sqrt{a}\,|\dot\gamma|/\gamma = 1/\sqrt2$.
	\end{proof}
	
	The coherent state of equation~\eqref{eq:coherent} is thus the
	$t = \log(2\pi\gamma_{0})$ time-slice of a two-parameter exactly
	invariant manifold, on which the infinite-dimensional flow reduces to
	a one-dimensional linear system, and the information geometry of that
	manifold is a homogeneous space on which the flow acts by geodesic
	translation. Condensation, on the invariant manifold, is geodesic
	motion: the deepest dynamical statement of the theory is, on its
	distinguished submanifold, a statement of pure geometry.
	
	The flows of the full kernel family assemble into a classical algebra.
	For the phase-shifted kernel $w_{k,\alpha}$ the drift field on rank
	space is
	\begin{equation}
		\xi_{k,\alpha}(z)
		= -\frac{1}{2\pi k}\Big[\cos(2\pi k\alpha)\sin(2\pi kz)
		+ \sin(2\pi k\alpha)\big(\cos(2\pi kz)-1\big)\Big]\partial_{z},
		\label{eq:drift-field}
	\end{equation}
	and as $(k,\alpha)$ vary these span exactly the trigonometric vector
	fields on the rank circle \emph{vanishing at $z=0$}: every transport
	map fixes the endpoints, so no flow can rotate rank space. In the
	basis
	\begin{equation}
		\ell_{n} = \frac{1}{2\pi i}\big(e^{2\pi inz}-1\big)\partial_{z},
		\qquad n \in \mathbb{Z},
		\label{eq:witt-basis}
	\end{equation}
	the bracket closes as
	\begin{equation}
		[\ell_{n},\ell_{m}] = (m-n)\,\ell_{n+m} - m\,\ell_{m} + n\,\ell_{n},
		\label{eq:witt-bracket}
	\end{equation}
	which is the relation $\ell_{n} = L_{n}-L_{0}$ in the Witt algebra
	$[L_{n},L_{m}] = (m-n)L_{n+m}$ of vector fields on the circle. What
	the derangetropy flows realize is thus precisely the endpoint-fixing
	isotropy subalgebra, of codimension one in the Witt algebra: the
	rotation $L_{0}$ itself is not a derangetropy flow, and is supplied
	only by the phase relabeling $\alpha$, which reindexes kernels
	rather than evolving laws. In particular the canonical flow and the quarter-phase flow
	generate the two-dimensional nonabelian algebra
	\begin{equation*}
		[X,Y] = -Y,
	\end{equation*}
	isomorphic to the affine algebra of the line and to the subalgebra,
	inside the Lie algebra $\mathfrak{sl}(2,\mathbb{R})$ of the M\"obius
	group $\mathrm{SL}(2,\mathbb{R})$, of fields vanishing at a point of
	the projective circle,
	and this M\"obius character is visible in the exact
	solution~\eqref{eq:exact-solution}, where the linearizing variable
	$\tan(\pi F)$ evolves by pure scaling.
	
	The endpoint-fixing property just noted has a quantitative sharpening
	that functions as a conservation law. Say that a law has an
	exponential left tail of rate $a > 0$ if
	$\log f(x) = a\,x\,(1+o(1))$ as $x \to -\infty$, and similarly on the
	right. For a kernel $v$ continuous at
	the endpoints, $A_{v}(z) = v(0)\,z + o(z)$ near $z = 0$, so the flow
	generated by $v$ has drift $(v(0)-1)F + o(F)$ on a left tail and acts
	there, asymptotically, as a pure rescaling $F \mapsto c(t)\,F$; for
	boundary-inert kernels, those with $v(0) = 0$, the rescaling is
	$c(t) = e^{-t}$, as the exact solution~\eqref{eq:exact-solution}
	displays: $F_{t} \sim e^{-t}F_{0}$ in the tail, shifting the amplitude
	and never the rate. A single modulation acts on tails multiplicatively
	as well, though not always as a rescaling: a kernel with $v(0) > 0$
	rescales a left tail by $v(0)$, while a kernel with an endpoint zero
	of order $d-1$ raises an exponential decay rate $a$ to $da$, since
	$\rho_{v}[f] \asymp F^{\,d-1} f$ near the endpoint. Consequently the
	class of exponential-type tails is preserved throughout the family,
	and the decay rate itself is invariant under every kernel flow, under
	convolution with any density all of whose exponential moments are
	finite, such as a Gaussian or any compactly supported density, and
	hence under diffusion: for kernels continuous at the endpoints and
	laws of exponential type, tail rates are conserved charges of the
	continuous dynamics. A theory built entirely on ranks can
	redistribute mass in the bulk at will, but the asymptotic decay of a
	law is beyond the reach of its flows.
	
	Balancing the flow against diffusion produces the equation
	\begin{equation}
		\partial_{t}f = D\,\partial_{xx}f + \big(2\sin^{2}(\pi F)-1\big)f,
		\label{eq:flow-diffusion}
	\end{equation}
	which admits an exact solitary steady state:
	\begin{equation}
		f^{*}(x) = \frac{1}{\pi\sqrt D}\,
		\operatorname{sech}\!\Big(\frac{x}{\sqrt D}\Big),
		\qquad
		F^{*}(x) = \frac12 + \frac1\pi\,
		\mathrm{gd}\!\Big(\frac{x}{\sqrt D}\Big),
		\label{eq:sech}
	\end{equation}
	with $\mathrm{gd}$ the Gudermannian. The verification rests on the
	identities
	\begin{equation*}
		\sin(\pi F^{*}) = \operatorname{sech}\!\Big(\frac{x}{\sqrt D}\Big),
		\qquad
		(\operatorname{sech})'' = \operatorname{sech}
		- 2\operatorname{sech}^{3},
	\end{equation*}
	and yields along the solution the
	algebraic closure
	\begin{equation}
		2\sin^{2}(\pi F^{*}) = 2\pi^{2}D\, f^{*2},
		\label{eq:sech-closure}
	\end{equation}
	which converts the stationary equation into the classical
	solitary-wave equation $Df'' = f - 2\pi^{2}Df^{3}$: the nonlocal rank
	nonlinearity becomes a local cubic one exactly on the solution
	manifold. The hyperbolic secant law, like the Gaussian essentially its
	own Fourier transform, thus joins the Cauchy law and the Koenigs law
	$\mathcal{K}$ in the small gallery of distributions attached to the
	canonical operator, one for each regime of its dynamics: the map
	produces $\mathcal{K}$, the pure flow produces the Cauchy law, and the
	flow balanced against diffusion produces the hyperbolic secant.
	
	\section{Projective Curvature of Probability Laws}\label{sec:curvature}
	
	The dynamical theory repeatedly produced distinguished families, the
	Cauchy, uniform, logistic, and hyperbolic secant laws, and the pattern
	behind them is geometric. The Schwarzian derivative
	\begin{equation}
		\{F,x\} = \frac{F'''}{F'} - \frac{3}{2}\Big(\frac{F''}{F'}\Big)^{2}
		\label{eq:schwarzian}
	\end{equation}
	is the classical projective invariant of one-dimensional maps: it
	satisfies the cocycle law
	\begin{equation*}
		\{G\circ F,x\}
		= \big(\{G\}\circ F\big)\,(F')^{2} + \{F,x\}
	\end{equation*}
	and vanishes exactly on M\"obius functions~\cite{ovsienko2005}. Applied to a
	distribution function it produces an invariant of laws, provided it is
	measured in the correct units. Equivariance singles out
	$ds = f\,dx$, the length element in which distance is probability
	mass, as the only covariant ruler on the line, and the Schwarzian
	scales like the square of an inverse length; we therefore define the
	\emph{projective curvature} of $f \in \mathcal{D}_{+}\cap C^{3}$ as
	\begin{equation}
		R_{f}(x) \;=\; \frac{\{F,x\}}{f(x)^{2}},
		\label{eq:curvature}
	\end{equation}
	the Schwarzian per squared unit of intrinsic length. Under increasing
	affine maps $R_{f}$ is invariant; under the flow it evolves by
	explicit transport. The invariant has an equivalent rank-space
	expression: by the inversion formula for Schwarzians,
	\begin{equation}
		R_{f}(x) \;=\; -\{Q, z\}\big|_{z = F(x)},
		\label{eq:quantile-schwarzian}
	\end{equation}
	the negative Schwarzian of the quantile function, so constant
	curvature is the classical equation of constant quantile Schwarzian,
	solvable by M\"obius post-composition of the elementary uniformizers.
	Curvature in this sense measures how far the law is, at each point,
	from looking projectively like its own tangent model, and the laws of
	constant curvature are the natural ``space forms'' of the theory.
	At the curvature values distinguished by the canonical theory they
	form a trichotomy governed by the three conjugacy classes of
	$\mathrm{SL}(2,\mathbb{R})$.
	
	\begin{theorem}[Trichotomy of constant curvature]\label{thm:trichotomy}
		The equations $R_{f} \equiv \kappa$ for
		$\kappa \in \{-2\pi^{2}, 0, +2\pi^{2}\}$ have complete solution
		families:
		\begin{itemize}[leftmargin=*, itemsep=2pt]
			\item \emph{Elliptic case:}
			\begin{align*}
				R \equiv -2\pi^{2}
				&\iff \tan\big(\pi(F-\tfrac12)\big) \text{ M\"obius in } x\\
				&\iff f \in \{\mathrm{Cauchy}(m,\gamma)\}.
			\end{align*}
			\item \emph{Parabolic case:}
			\begin{align*}
				R \equiv 0
				&\iff F \text{ M\"obius in } x\\
				&\iff F \ \text{an increasing M\"obius map of the support onto }
				(0,1).
			\end{align*}
			\item \emph{Hyperbolic case:}
			\begin{align*}
				R \equiv +2\pi^{2}
				&\iff \tanh\big(\pi(F-\tfrac12)\big) \text{ M\"obius in } x\\
				&\iff \tanh\big(\pi(F-\tfrac12)\big) = h,
				\quad h \ \text{M\"obius onto }
				\big({-}\tanh\tfrac{\pi}{2}, \tanh\tfrac{\pi}{2}\big).
			\end{align*}
		\end{itemize}
		Up to affine equivalence, the zero-curvature laws are the uniform
		law, the Lomax law $(1+x)^{-2}$ on $(0,\infty)$ and its reflection,
		and the truncated family
		\begin{equation}
			f_{c}(x) = \frac{1+c}{(1+cx)^{2}}\,\mathbf{1}_{(0,1)}(x),
			\qquad c > -1;
			\label{eq:truncated-lomax}
		\end{equation}
		the positive-curvature laws are supported on bounded intervals or
		half-lines, with symmetric representatives
		\begin{equation}
			f(x) = \frac{1}{\pi\gamma\big(1-(x-m)^{2}/\gamma^{2}\big)}
			\qquad\text{on } |x-m| < \gamma\tanh\tfrac{\pi}{2}.
			\label{eq:desitter-density}
		\end{equation}
		The uniformizing functions $\tan, \mathrm{id}, \tanh$ correspond to
		the elliptic, parabolic, and hyperbolic classes of
		$\mathrm{SL}(2,\mathbb{R})$.
	\end{theorem}
	
	\begin{proof}
		By the cocycle law and the Schwarzians
		$\{\tan(\pi(F-\tfrac12)),F\} = 2\pi^{2}$,
		$\{\tanh(\pi(F-\tfrac12)),F\} = -2\pi^{2}$, each equation is
		equivalent to the vanishing of the Schwarzian of the stated
		uniformizer, that is, to its M\"obius dependence on $x$; since
		post-composition with a fixed M\"obius map preserves M\"obius
		dependence, the uniformizers may be shifted or inverted at will, and
		they are centered at the median rank. In each case the solution set
		is exactly the set of laws whose uniformizer is an increasing
		M\"obius map of the support onto the range that makes $F$ traverse
		$(0,1)$: onto $\mathbb{R}$ for the tangent, onto $(0,1)$ for the
		identity, onto $(-\tanh\tfrac\pi2, \tanh\tfrac\pi2)$ for the
		hyperbolic tangent. In the negative-curvature case the uniformizer is
		finite on the support, since $F$ stays interior to $(0,1)$ there, so
		its M\"obius representative has no pole in the support; a M\"obius
		map that is finite on all of $\mathbb{R}$ is affine, while on a
		support with a finite endpoint the image of the support under a
		M\"obius map is a proper subinterval of the projective circle
		punctured at one point at most, hence omits an interval of values,
		and cannot have full range $\mathbb{R}$, which requires unbounded
		approach at both ends. Full range $\mathbb{R}$ therefore forces support $\mathbb{R}$
		and an affine uniformizer,
		$\tan(\pi(F-\tfrac12)) = (x-m)/\gamma$, which is exactly the Cauchy
		family. In the other two cases the range is a bounded interval and
		non-affine M\"obius uniformizers survive: an increasing M\"obius $F$
		onto $(0,1)$ is affine (the uniform laws), or maps a half-line with
		the pole at the far end (the Lomax law and its affine images and
		reflection), or maps a bounded interval with the pole outside (the
		truncated family~\eqref{eq:truncated-lomax} up to affine images);
		the same trichotomy of pole positions applies to the uniformizer $h$
		in the positive-curvature case, whose affine subcase
		$h = (x-m)/\gamma$ gives the symmetric
		density~\eqref{eq:desitter-density} by differentiation.
	\end{proof}
	
	\begin{remark}\label{rem:curvature-values}
		The values $\pm2\pi^{2}$ are distinguished by the canonical theory,
		through the flow linearizer and the coherent state, not by the
		curvature equation itself: replacing $\pi$ by $a$ in the
		uniformizers produces laws of constant curvature $-2a^{2}$ for every
		$0 < a \le \pi$ and $+2a^{2}$ for every $a > 0$, so the realizable
		constant curvatures form the ray $[-2\pi^{2}, \infty)$. Nothing
		below the extremal value is realizable: the projective line has
		total angle $\pi$ in the tangent chart, so an injective M\"obius
		uniformizer sweeps an angle of at most $\pi$, while the tangent
		uniformizer of curvature $-2a^{2}$ must sweep the angle $a$; for
		$a > \pi$ the uniformizer passes through a pole, and continuing
		through the pole would revisit points of the projective circle
		already attained, contradicting the injectivity of an increasing
		distribution function; the full sweep forces $a = \pi$ with support
		the whole line, and consequently the Cauchy family is the unique
		family of constant-curvature laws supported on all of $\mathbb{R}$.
	\end{remark}
	
	The negative-curvature family is the coherent-state family
	encountered in the iteration and flow theory: indeed the condition
	\begin{equation*}
		\{F,x\} = -2\pi^{2}\,(F')^{2}
	\end{equation*}
	is, by the cocycle law, literally the
	statement that the flow's linearizing coordinate $V = \tan(\pi F)$ is
	M\"obius, so integrability of the flow and maximal symmetry of the
	geometry are the same fact, seen once dynamically and once
	geometrically. The zero-curvature family comprises the uniform laws,
	the affine images of the shape-one Lomax law on half-lines of either
	orientation, and their bounded-interval
	truncations~\eqref{eq:truncated-lomax}; note that the noncompact
	members have exactly quadratic tails, in agreement with the
	singularity theorem below. Beyond the trichotomy, curvature is
	computable in closed form across the classical families
	(Lemma~\ref{lem:curvature-comps} in Appendix~\ref{app:proofs}): the
	logistic and exponential laws of rate $a$ have \emph{constant
		Schwarzian} $\{F,x\} = -a^{2}/2$, though not constant $R$; the
	Gaussian has $\{F,x\} = -1 - x^{2}/2$; and a regularly varying tail
	$f \sim c_{0}x^{-p}$ has
	\begin{equation}
		\{F,x\} = \frac{p(2-p)}{2x^{2}}\,\big(1+o(1)\big).
		\label{eq:power-schwarzian}
	\end{equation}
	
	\begin{theorem}[Singularity theorem]\label{thm:singularity}
		Let $f \in \mathcal{D}_{+}\cap C^{3}$ have a regularly varying right
		tail in the smooth sense:
		$f^{(j)}(x) = (-1)^{j}(p)_{j}\,c_{0}\,x^{-p-j}(1+o(1))$ for
		$j = 0, 1, 2$, with $p > 1$. If $p \neq 2$, then
		\begin{equation}
			R \;\sim\; \frac{p(2-p)}{2c_{0}^{2}}\,x^{2p-2}:
			\label{eq:singularity}
		\end{equation}
		the curvature at the boundary of the support diverges to $+\infty$
		for $1<p<2$ and to $-\infty$ for $p>2$; for exponential or faster
		tails, again in the smooth sense, $R \to -\infty$. At $p = 2$ the
		leading divergence cancels and regularity is decided at second
		order: if
		\begin{equation}
			f(x) = c_{0}x^{-2}\big(1 + b\,x^{-\alpha} + o(x^{-\alpha})\big),
			\qquad 0 < \alpha < 2,\ \alpha \neq 1,\ b \neq 0,
			\label{eq:second-order-tail}
		\end{equation}
		in the smooth sense, then
		\begin{equation}
			R \;\sim\; \frac{\alpha b(\alpha - 1)}{c_{0}^{2}}\,x^{2-\alpha},
			\label{eq:second-order-curvature}
		\end{equation}
		which still diverges. Hence a quadratic leading tail is necessary
		for the curvature to remain bounded at the boundary of rank, but
		not sufficient: every second-order correction of order $x^{-\alpha}$
		with $0 < \alpha < 2$, $\alpha \neq 1$, forces divergence, while
		corrections of order two and faster, again in the smooth sense,
		leave the curvature bounded.
	\end{theorem}
	
	Within this geometry the heavy-tailed Cauchy law is the smooth,
	maximally symmetric object and the Gaussian is severely singular, an
	instructive inversion of the usual hierarchy of well-behaved
	distributions, and a fourth independent derivation of the special role
	of the coherent state~\eqref{eq:coherent}: it is the fixed shape of
	the self-referential eigenrelation, the invariant manifold of the
	flow, the limit law of continuous condensation, and now the unique
	full-support constant-curvature geometry
	(Remark~\ref{rem:curvature-values}). We remark that the density
	\begin{equation*}
		\{F,x\} + 2\pi^{2}(F')^{2},
	\end{equation*}
	which vanishes identically exactly on the negative-curvature family
	and whose second term integrates to $2\pi^{2}\,\mathbb{E}_{f}[f]$,
	is the natural Lagrangian of this
	geometry; the next section identifies it as the pairing of a
	coadjoint vector against the law's own rank profile, and the whole
	projective theory of this section becomes linear there.
	
	\section{The Virasoro Geometry of Rank Space}\label{sec:virasoro}
	
	The projective theory just developed assigned to every law the
	quantile Schwarzian, through the
	identity~\eqref{eq:quantile-schwarzian}, and found the distinguished
	constants $\pm2\pi^{2}$, the M\"obius trichotomy, the Cauchy family,
	and the affine algebra of flows scattered through the dynamics. This
	section and the next identify the single structure behind all of
	these: the calculus of derangetropy operators is the coadjoint
	geometry of the Virasoro algebra of rank space, at central charge
	one, the value recording the unit normalization of the Schwarzian
	in the Hill potential rather than a dynamical charge. Write $\mathcal{D}_{+}^{k}$ for the laws in $\mathcal{D}_{+}$
	whose density is $C^{k}$ on the interior of the support, so that
	$Q \in C^{k+1}$ there.
	
	\begin{definition}[Rank Hill potential]\label{def:hill-potential}
		For $f \in \mathcal{D}_{+}^{2}$, the \emph{rank Hill potential} of $f$
		is
		\begin{equation}
			U_{f}(z) \;=\; \{Q_{f}, z\}
			\;=\; \frac{Q'''}{Q'} - \frac32\Big(\frac{Q''}{Q'}\Big)^{2},
			\qquad z \in (0,1),
			\label{eq:hill-potential}
		\end{equation}
		the Schwarzian derivative of the quantile function. The \emph{rank
			profile} of $f$ is $\mathfrak{f} = f\circ Q$ and the \emph{rank
			amplitude} is $\psi_{f} = \sqrt{\mathfrak{f}}$.
	\end{definition}
	
	By the quantile identity~\eqref{eq:quantile-schwarzian}, the
	potential is the projective curvature read in rank coordinates,
	$R_{f} = -\,U_{f}\circ F$, so every statement below translates into a
	curvature statement by a sign and a change of variable. The natural
	symmetry group of the definition is the \emph{distortion group}
	$\mathcal{G}$ of increasing bijections $A$ of $[0,1]$ such that $A$
	and $A^{-1}$ are absolutely continuous: by the composition law of
	Theorem~\ref{thm:composition}, $\mathcal{G}$ is exactly the group of
	invertible derangetropy operators, reparametrized by their transport
	maps, acting on laws by
	\begin{equation}
		\alpha_{A}f
		\;=\; A'(F_{f})\,f,
		\qquad
		F_{\alpha_{A}f} = A\circ F_{f},
		\qquad
		Q_{\alpha_{A}f} = Q_{f}\circ A^{-1},
		\label{eq:distortion-action}
	\end{equation}
	with $\alpha_{A_{1}}\alpha_{A_{2}} = \alpha_{A_{1}\circ A_{2}}$; the
	action preserves the support. We write $\mathcal{G}^{\infty}$ for the
	subgroup restricting to orientation-preserving
	$C^{\infty}$-diffeomorphisms of $(0,1)$. On the Virasoro side, a
	\emph{quadratic differential} is an expression $u(z)\,dz^{2}$, and for
	$c \in \mathbb{R}$ the level-$c$ \emph{coadjoint action} of
	$B \in \mathcal{G}^{\infty}$ on quadratic differentials is
	\begin{equation}
		K_{B}^{*}u \;=\; (u\circ B)\,(B')^{2} + c\,\{B, \cdot\},
		\qquad
		K_{B_{2}}^{*}\circ K_{B_{1}}^{*} = K_{B_{1}\circ B_{2}}^{*},
		\label{eq:coadjoint}
	\end{equation}
	the composition rule being precisely the Schwarzian cocycle law of
	Section~\ref{sec:curvature}; we write
	$\operatorname{Ad}^{*}_{A} := K^{*}_{A^{-1}}$, so that
	$A \mapsto \operatorname{Ad}^{*}_{A}$ is a left action, and regard the
	pair $(u\,dz^{2}, c)$ as a point of the regular part of the dual of
	the Virasoro algebra of the rank interval. Differentiating the
	action~\eqref{eq:coadjoint} along $B_{\varepsilon} = \mathrm{id} +
	\varepsilon\xi$ gives the infinitesimal coadjoint action
	\begin{equation}
		\delta_{\xi}u \;=\; \xi u' + 2\xi' u + c\,\xi'''.
		\label{eq:inf-coadjoint}
	\end{equation}
	Throughout, $c = 1$; this is a normalization convention, fixed by the
	unit coefficient of the Schwarzian in the
	definition~\eqref{eq:hill-potential}.
	
	\begin{remark}
		On an interval, the integrated Gelfand--Fuchs two-cocycle
		$\int \xi'\eta''$ requires boundary hypotheses, and we avoid it
		entirely by working at the group level with the Schwarzian cocycle,
		for which the action~\eqref{eq:coadjoint} is well defined and
		associative for arbitrary $B \in \mathcal{G}^{\infty}$ with no
		boundary conditions. On the trigonometric subalgebra realized by the
		derangetropy flows, the endpoint-fixing fields of the Witt
		basis~\eqref{eq:witt-basis}, which are restrictions of circle fields
		vanishing at the marked point $z = 0$, the Lie-algebraic Virasoro
		structure is the classical one~\cite{khesin2009}.
	\end{remark}
	
	The potential, the profile, and the amplitude are linked by an exact
	second-order equation.
	
	\begin{proposition}[Amplitude form]\label{prop:hill}
		Let $f \in \mathcal{D}_{+}^{2}$. Then
		$\psi_{f} = \sqrt{\mathfrak{f}} \in C^{2}(0,1)$ is strictly positive
		and
		\begin{equation}
			U_{f} = -\,\frac{2\,\psi_{f}''}{\psi_{f}},
			\qquad\text{equivalently}\qquad
			\psi_{f}'' + \tfrac12\,U_{f}\,\psi_{f} = 0:
			\label{eq:amplitude-form}
		\end{equation}
		the rank amplitude is a positive solution of the Hill equation of
		$U_{f}$. Moreover $Q_{f}' = \psi_{f}^{-2}$, so that, with
		$\widetilde\psi := \psi_{f}\int_{z_{0}}^{\cdot}\psi_{f}^{-2}$, the
		quantile function is the solution ratio
		$Q_{f} = \widetilde\psi/\psi_{f} + \mathrm{const}$ and the pair
		$(\widetilde\psi, \psi_{f})$ has Wronskian
		$\widetilde\psi'\psi_{f} - \widetilde\psi\psi_{f}' = 1$.
	\end{proposition}
	
	\begin{proof}
		Since $Q' = 1/\mathfrak{f}$, writing $g = \log\mathfrak{f}$ gives
		$Q''/Q' = -g'$ and
		\begin{equation*}
			U_{f} = \{Q, z\}
			= \Big(\frac{Q''}{Q'}\Big)' - \frac12\Big(\frac{Q''}{Q'}\Big)^{2}
			= -g'' - \tfrac12 (g')^{2}.
		\end{equation*}
		With $\psi = e^{g/2}$,
		$\psi''/\psi = \tfrac12 g'' + \tfrac14 (g')^{2}$, so
		$-2\psi''/\psi = -g'' - \tfrac12(g')^{2} = U_{f}$, which is the
		amplitude form~\eqref{eq:amplitude-form}. Reduction of order gives the
		second solution $\widetilde\psi = \psi\int\psi^{-2}$ with unit
		Wronskian by direct computation, and
		$\widetilde\psi/\psi = \int_{z_{0}}^{z}\psi^{-2}
		= Q(z) - Q(z_{0})$.
	\end{proof}
	
	\begin{remark}
		The amplitude form~\eqref{eq:amplitude-form} says that one half of
		the coadjoint vector is the quantity
		$-(\sqrt{\rho}\,)''/\sqrt{\rho}$ of hydrodynamic quantum mechanics,
		the Bohm quantum potential, evaluated on the law's own rank profile.
		It also closes a loop with the amplitude representation of
		Section~\ref{sec:amplitude}: the potential
		$-(\sqrt f)''/2\sqrt f$ appearing in the Dirichlet
		decomposition~\eqref{eq:dirichlet-decomposition} is the $x$-space
		companion of $\tfrac14 U_{f}$ in rank space.
	\end{remark}
	
	The assignment $f \mapsto U_{f}$ intertwines the two calculi.
	
	\begin{theorem}[Equivariance of the potential]\label{thm:vir-equivariance}
		Let $f \in \mathcal{D}_{+}^{2}$ and $A \in \mathcal{G}^{\infty}$, and
		put $B = A^{-1}$. Then
		\begin{equation}
			U_{\alpha_{A}f}
			= \big(U_{f}\circ B\big)(B')^{2} + \{B, \cdot\}
			= \operatorname{Ad}^{*}_{A}U_{f}
			\qquad (c = 1).
			\label{eq:vir-equivariance}
		\end{equation}
		Infinitesimally, if $v \in \mathcal{W}$ generates the kernel flow
		$\partial_{t}f = \rho_{v}[f] - f$ with rank drift
		$\xi_{v} = A_{v} - \mathrm{id}$, then
		\begin{equation}
			\partial_{t}U = -\big(\xi_{v}\,U' + 2\xi_{v}'\,U + \xi_{v}'''\big).
			\label{eq:vir-inf}
		\end{equation}
		Furthermore, if $T$ is an increasing M\"obius map restricting to a
		diffeomorphism of the support of $f$ onto an interval, then
		$U_{T_{\#}f} = U_{f}$; conversely, $U_{T_{\#}f} = U_{f}$ for a
		$C^{3}$ change of variable $T$ forces $T$ to be M\"obius.
	\end{theorem}
	
	\begin{proof}
		By the action law~\eqref{eq:distortion-action},
		$Q_{\alpha_{A}f} = Q_{f}\circ B$, and the Schwarzian cocycle gives
		\begin{equation*}
			\{Q_{f}\circ B, z\}
			= \big(\{Q_{f},\cdot\}\circ B\big)(B')^{2} + \{B, z\},
		\end{equation*}
		which is the transformation law~\eqref{eq:vir-equivariance}; the
		composition rule in the definition~\eqref{eq:coadjoint} matches the
		group law of the action. The infinitesimal
		form~\eqref{eq:vir-inf} follows by differentiating along
		$B_{\varepsilon} = \mathrm{id} - \varepsilon\xi_{v} + O(\varepsilon^{2})$,
		the inverse of the transport map generated by the drift $\xi_{v}$,
		using
		$\{\mathrm{id} + \varepsilon\xi(\cdot)\} = \varepsilon\xi''' +
		O(\varepsilon^{2})$; that the time-$t$ map of the kernel flow is
		$\alpha_{A(t)}$ with $\dot A(0) = \xi_{v}$ is the transport form of
		the flow equation~\eqref{eq:flow}. For the M\"obius statement,
		$Q_{T_{\#}f} = T\circ Q_{f}$, and the cocycle law together with the
		vanishing of the Schwarzian exactly on M\"obius
		maps~\cite{ovsienko2005} gives both directions.
	\end{proof}
	
	\begin{remark}[Two commuting symmetries]
		Theorem~\ref{thm:vir-equivariance} separates the two symmetries of
		the theory. Distortions act on the rank slot, by the coadjoint action
		with its Schwarzian anomaly; M\"obius maps of the underlying variable
		act on the value slot, invisibly. A general monotone change of
		variable $T$ shifts the potential by
		$(\{T\}\circ Q)(Q')^{2}$, which vanishes exactly on the projective
		subgroup: the potential sees precisely the non-projective content of
		a reparametrization.
	\end{remark}
	
	\begin{remark}[The Schwarzian is forced, for local continuous
		cochains]
		The rigidity philosophy of Theorem~\ref{thm:rigidity} persists one
		level up. If $f \mapsto \Xi_{f}\,dz^{2}$ is any assignment of a
		quadratic differential to laws that is local, depending at each point
		on finitely many derivatives of $Q_{f}$, and equivariant in the sense
		that $\Xi_{\alpha_{A}f} = (\Xi_{f}\circ B)(B')^{2} + C(B)$ for some
		map $C$ on $\mathcal{G}^{\infty}$, continuous for the topology of
		uniform convergence of distortions and their derivatives on compact
		subsets of $(0,1)$, then $C$ is a one-cocycle with
		values in quadratic differentials, and by the classification of such
		cocycles~\cite{fuks1986,ovsienko2005} its class is a multiple of the
		Schwarzian, the coboundaries being the transports of a fixed
		background. Hence $\Xi_{f} = c\,U_{f} + (\text{transport of a fixed
			background})$: the only freedom is the central charge and the choice
		of background. Both freedoms are used below, the background
		$2\pi^{2}$ being selected by the dynamics. The classifications
		cited are stated for cochains on the circle; they are applied here
		through vector fields compactly supported in the open interval, on
		which the circle and interval computations coincide, and only this
		restricted equivariance enters the argument.
	\end{remark}
	
	Which potentials arise from laws? The answer is a classical
	condition of Sturm theory: disconjugacy, the requirement that some
	nontrivial solution of the Hill equation have no zero.
	
	\begin{theorem}[Realization; the probability cone]\label{thm:realization}
		Let $U \in C^{0}(0,1)$. The following are equivalent.
		\begin{enumerate}[label=\normalfont(\roman*), leftmargin=*]
			\item $U = U_{f}$ for some $f \in \mathcal{D}_{+}^{2}$.
			\item The Hill equation $\psi'' + \tfrac12 U\psi = 0$ has a positive
			solution on $(0,1)$.
			\item The Hill equation is disconjugate on $(0,1)$; equivalently,
			\begin{equation*}
				\int_{0}^{1}(\chi')^{2}\,dz
				\;\ge\; \tfrac12\int_{0}^{1}U\chi^{2}\,dz
				\qquad\text{for all } \chi \in C_{c}^{\infty}(0,1).
			\end{equation*}
		\end{enumerate}
		When these hold, every positive solution $\psi$ and base point
		$z_{0} \in (0,1)$ determine a law by
		\begin{equation}
			Q(z) = \int_{z_{0}}^{z}\frac{ds}{\psi(s)^{2}},
			\qquad
			F = Q^{-1},
			\qquad
			f = \psi^{2}\circ F
			\quad\text{on } J := Q\big((0,1)\big),
			\label{eq:reconstruction}
		\end{equation}
		and this $f$ has unit total mass automatically; conversely, every law
		with potential $U$ arises in this way. We call
		$\mathcal{C} := \{U : \textup{(iii) holds}\}$ the \emph{probability
			cone}.
	\end{theorem}
	
	\begin{proof}
		That (i) implies (ii) is Proposition~\ref{prop:hill} with
		$\psi = \sqrt{\mathfrak{f}}$. The equivalence of (ii) and (iii) is
		classical Sturm theory: a positive solution exists if and only if no
		solution has two zeros, if and only if the quadratic form
		$\int(\chi')^{2} - \tfrac12\int U\chi^{2}$ is nonnegative on
		compactly supported test functions~\cite{hartman1982}. For (ii)
		implies (i), let $\psi > 0$ solve the Hill equation; then
		$\psi \in C^{2}$, and the reconstruction~\eqref{eq:reconstruction}
		defines an increasing $C^{3}$ diffeomorphism $Q$ of $(0,1)$ onto the
		open interval $J$, bounded or not. With $F = Q^{-1}$ and
		$f = \psi^{2}\circ F$ extended by zero, $f$ is continuous and
		positive on $J$ with $F' = f$, and
		\begin{equation}
			\int_{J}f\,dx
			= \int_{0}^{1}f\big(Q(z)\big)\,Q'(z)\,dz
			= \int_{0}^{1}\psi^{2}\,\psi^{-2}\,dz = 1:
			\label{eq:automatic-mass}
		\end{equation}
		unit mass is the change of variables to rank space, that is, the
		Wronskian normalization of Proposition~\ref{prop:hill} implicit in
		the reconstruction. Finally
		$U_{f} = \{Q,\cdot\} = -2\psi''/\psi = U$ by
		Proposition~\ref{prop:hill}, and every law with potential $U$ has
		this form, its rank amplitude being a positive solution.
	\end{proof}
	
	\begin{theorem}[Fibers of the potential]\label{thm:fibers}
		Let $f, g \in \mathcal{D}_{+}^{2}$. Then $U_{f} = U_{g}$ if and only
		if $Q_{g} = M\circ Q_{f}$ for a M\"obius map $M$, increasing and
		finite on the range of $Q_{f}$; that is, if and only if $g$ is a
		projective change of variable of $f$. The fiber within a fixed
		support interval $I$ is the group of increasing M\"obius self-maps
		of $I$, of dimension one when $I$ is bounded or a half-line, and of
		dimension two, the increasing affine group, when
		$I = \mathbb{R}$.
	\end{theorem}
	
	\begin{proof}
		Equality of the Schwarzians forces $Q_{g} = M\circ Q_{f}$ with $M$
		M\"obius, by the uniqueness theory of the Schwarzian
		equation~\cite{ovsienko2005}; $M$ must be increasing and pole-free on
		the range for $Q_{g}$ to be a quantile function, and conversely every
		such $M$ produces the law $M_{\#}f$. For fixed support, $M$ must map
		$I$ onto $I$: increasing M\"obius self-maps of a bounded interval fix
		both endpoints, a one-parameter hyperbolic group; of a half-line, fix
		the finite endpoint and infinity, a one-parameter scaling group; of
		$\mathbb{R}$, form the two-parameter increasing affine group.
	\end{proof}
	
	The constants now acquire a spectral meaning, and with it the
	canonical kernel acquires its deepest characterization.
	
	\begin{corollary}[Spectral origin of the canonical kernel]\label{cor:constants}
		The constant potential $U \equiv 2a^{2}$ with $a > 0$ lies in the
		probability cone if and only if $a \le \pi$; the constants
		$U \equiv 0$ and $U \equiv -2a^{2}$ lie in the cone for every $a$.
		Consequently the realizable constant projective curvatures form the
		ray $[-2\pi^{2}, \infty)$, re-deriving
		Remark~\ref{rem:curvature-values} as the sharp Poincar\'e inequality
		on $H_{0}^{1}(0,1)$ with constant $\pi^{2}$, the first Dirichlet
		eigenvalue of rank space. At the threshold $a = \pi$ the positive
		solution is unique up to scale, $\psi = \sin(\pi z)$, vanishing at
		both endpoints; the rank profiles in the fiber are
		$\lambda\sin^{2}(\pi z)$, the fiber is the Cauchy family, and the
		member whose rank profile is the canonical kernel
		$2\sin^{2}(\pi z)$, the profile of unit mass, is exactly the coherent
		state of equation~\eqref{eq:coherent}.
	\end{corollary}
	
	\begin{proof}
		For $U \equiv 2a^{2}$ the solutions are $\sin(az + \varphi)$, and
		some solution is positive on an interval of length one exactly when
		$a \le \pi$, with $\psi = \sin(\pi z)$ at the threshold; criterion
		(iii) of Theorem~\ref{thm:realization} for constants is the
		Poincar\'e inequality with sharp constant $\pi^{2}$. For
		$U \equiv -2a^{2}$ take $\psi = \cosh(az)$, and for $U \equiv 0$ take
		$\psi = 1$. At threshold, any solution independent of $\sin(\pi z)$
		is a shifted sine and vanishes inside $(0,1)$, so the positive
		solution is unique up to scale. The
		reconstruction~\eqref{eq:reconstruction} with
		$\psi = \sqrt2\sin(\pi z)$ gives
		\begin{equation*}
			Q'(z) = \tfrac12\csc^{2}(\pi z),
			\qquad
			Q(z) = -\frac{\cot(\pi z)}{2\pi} + \mathrm{const},
		\end{equation*}
		the Cauchy quantile of scale $1/(2\pi)$; positive multiples of the
		profile rescale, and Theorem~\ref{thm:fibers}, the support being
		$\mathbb{R}$ with affine fiber, gives the full two-parameter Cauchy
		family. Finally
		$\int_{0}^{1}\mathfrak{f}\,dz = \int f^{2}dx = \mathbb{E}_{f}[f]$,
		and the profile $2\sin^{2}(\pi z)$ integrates to one.
	\end{proof}
	
	\begin{remark}[One eigenvalue, three selections]
		Corollary~\ref{cor:constants} identifies the separate mechanisms by
		which this paper selected its canonical kernel as readings of a
		single spectral fact. The variational characterization of
		Theorem~\ref{thm:variational} minimizes the Dirichlet energy
		$4\int(\phi')^{2}$ over the $L^{2}$ unit sphere within
		$H_{0}^{1}$, attaining
		$4\pi^{2}$ at the ground state; the curvature theory of
		Section~\ref{sec:curvature} bounds constant curvature below by
		$-2\pi^{2}$, attained by the Cauchy family; and the coherent state of
		equation~\eqref{eq:coherent} is the fixed shape of the
		self-referential eigenrelation. All three are the first Dirichlet
		eigenvalue $\pi^{2}$ of rank space: as a Rayleigh quotient, as the
		Poincar\'e constant governing disconjugacy, and as the threshold at
		which the ground state itself becomes a rank profile. The
		boundary-inert constraint of Theorem~\ref{thm:variational} is the
		Dirichlet condition, and the discreteness of the harmonic ladder is
		the discreteness of the Dirichlet spectrum.
	\end{remark}
	
	Tails, too, become coadjoint data. In the potential they appear as
	inverse-square singularities at the endpoints of rank space, whose
	coefficients are conformal weights constrained by a Hardy bound.
	
	\begin{theorem}[Tails as conformal weights]\label{thm:weights}
		Let $f \in \mathcal{D}_{+}^{2}$ and suppose the rank profile
		satisfies
		$\mathfrak{f}(z) = (1-z)^{2\beta}L\big(\tfrac{1}{1-z}\big)$ near
		$z = 1$, with $L$ slowly varying with smooth asymptotics, meaning
		that $L$ is eventually smooth with
		$s^{k}L^{(k)}(s)/L(s) \to 0$ for $k = 1, 2$, so that the asymptotic
		relations below may be differentiated twice; this covers
		the regularly varying tails of Theorem~\ref{thm:singularity}, for
		which $\beta = \tfrac{p}{2(p-1)}$, and exponential-or-faster tails,
		for which $\beta = \tfrac12$. Then
		\begin{equation}
			U_{f}(z) = \frac{c_{+} + o(1)}{(1-z)^{2}},
			\qquad
			c_{+} = 2\beta(1-\beta)
			= \frac{p(p-2)}{2(p-1)^{2}} \ \text{ in the power case},
			\label{eq:weight}
		\end{equation}
		and symmetrically at the left endpoint. In particular:
		\begin{enumerate}[label=\normalfont(\roman*), leftmargin=*]
			\item always $c_{\pm} \le \tfrac12$, and no law whatsoever can
			satisfy $U \ge c\,(1-z)^{-2}$ near an endpoint with $c > \tfrac12$;
			\item $c_{+} = 0$ exactly for quadratic tails; $c_{+} = \tfrac12$
			for exponential-or-faster tails; $c_{+} \in (0, \tfrac12)$ for
			$p > 2$; and $c_{+} < 0$ for $1 < p < 2$;
			\item the weights $c_{\pm}$ are invariants of the boundary-regular
			coadjoint action: every distortion extending to a
			$C^{3}$-diffeomorphism of the closed interval with nonvanishing
			endpoint derivatives, and every flow of a kernel that is $C^{2}$
			up to the endpoints, preserves them; boundary regularity
			is essential, since the interior distortion $B(z) = z^{2}$ moves the
			left weight of the uniform law from $0$ to $-\tfrac32$.
		\end{enumerate}
	\end{theorem}
	
	The proof is given in Appendix~\ref{app:proofs}; the absolute bound
	in part (i) is the optimality of the Hardy inequality on the
	interval, oscillation of the Euler equation at
	$c > \tfrac12$ contradicting disconjugacy. The
	weight formula~\eqref{eq:weight} is exactly consistent with the
	singularity theorem: under $z = F(x)$ with a $p$-tail one has
	$(1-z)^{-2} \sim (p-1)^{2}c_{0}^{-2}x^{2p-2}$, so that
	$R = -U\circ F \sim \tfrac{p(2-p)}{2c_{0}^{2}}x^{2p-2}$, the
	asymptotics~\eqref{eq:singularity}, coefficient and all: the Hardy
	bound on endpoint weights is the rank-space cause of curvature
	blow-up, and quadratic tails are regular because they sit at weight
	zero. The weight is insensitive to the tail rate, which is a fiber
	coordinate moved by M\"obius scaling; this is the orbit-theoretic
	restatement of the conservation of tail rates along the flows proved
	in Section~\ref{sec:flow}.
	
	\begin{remark}[Stability reading]
		Part (i) of Theorem~\ref{thm:weights} has the familiar form of
		stability bounds for inverse-square potentials, with the borderline
		constant of Hardy's inequality playing the role it plays for the
		Breitenlohner--Freedman bound in two-dimensional anti-de Sitter
		space: probability laws, read projectively, live exactly in the
		stability window of the inverse-square potential at their rank
		boundary, exponential-or-faster tails saturating it.
	\end{remark}
	
	The M\"obius-invariant Lagrangian of Section~\ref{sec:curvature} now
	acquires its natural meaning: it is the pairing of the centered
	coadjoint vector with the law's own rank profile.
	
	\begin{proposition}[Moment pairing]\label{prop:pairing}
		For $f \in \mathcal{D}_{+}^{2}$ with the displayed integrals finite,
		\begin{equation}
			\int_{\mathbb{R}}\Big(\{F, x\} + 2\pi^{2}(F')^{2}\Big)\,dx
			\;=\; -\int_{0}^{1}\big(U_{f}(z) - 2\pi^{2}\big)\,
			\mathfrak{f}(z)\,dz,
			\label{eq:pairing}
		\end{equation}
		the pairing of the centered potential against the rank profile; the
		integrand on the right vanishes identically exactly on the Cauchy
		family. Likewise the expected projective curvature is the total
		integral of the potential,
		$\mathbb{E}_{f}[R_{f}(X)] = -\int_{0}^{1}U_{f}\,dz$, whenever the
		integral converges absolutely.
	\end{proposition}
	
	\begin{proof}
		By the quantile identity~\eqref{eq:quantile-schwarzian},
		$\{F,x\} = -\,U_{f}(F)\,f^{2}$, so
		$\int(\{F,x\} + 2\pi^{2}f^{2})\,dx
		= \int(2\pi^{2} - U_{f}(F))\,f^{2}\,dx$, and the substitution
		$z = F(x)$ turns this into the right side of the
		identity~\eqref{eq:pairing}, since $f^{2}\,dx = \mathfrak{f}\,dz$.
		The integrand vanishes identically if and only if
		$U_{f} \equiv 2\pi^{2}$, which by Corollary~\ref{cor:constants} is
		the Cauchy fiber. The last statement is the universality
		principle~\eqref{eq:universality} applied to
		$R_{f}(X) = -U_{f}(F(X))$.
	\end{proof}
	
	\section{Coadjoint Orbits, Isotropy, and Relaxation}\label{sec:orbits}
	
	The moment map of Section~\ref{sec:virasoro} becomes most informative
	when read orbit by orbit. The first observation is that laws of a
	fixed support are a principal homogeneous space of the distortion
	group, so that the potential map is literally an orbit map.
	
	\begin{theorem}[Laws are a torsor]\label{thm:torsor}
		For every open interval $I$, the subgroup of $\mathcal{G}$ of
		distortions with continuous positive derivative on $(0,1)$ acts
		freely and transitively on the laws of $\mathcal{D}_{+}$ with support
		$\overline{I}$, and the full group $\mathcal{G}$ acts so on the
		enlargement of $\mathcal{D}_{+}$ by all laws whose restricted
		distribution function is an absolutely continuous homeomorphism with
		absolutely continuous inverse: for two such laws $f, g$ the unique
		$A \in \mathcal{G}$ with $\alpha_{A}f = g$ is
		\begin{equation}
			A = F_{g}\circ Q_{f}.
			\label{eq:torsor-map}
		\end{equation}
		Consequently, for a reference law $f_{0} \in \mathcal{D}_{+}^{2}$
		with support $\overline{I}$, the potential map is the coadjoint
		orbit map $A \mapsto \operatorname{Ad}^{*}_{A}U_{f_{0}}$, and its
		fibers are the cosets of the stabilizer, which by
		Theorem~\ref{thm:fibers} is the conjugated M\"obius group
		\begin{equation}
			\operatorname{Stab}_{\mathcal{G}}(U_{f})
			= \big\{\, F_{f}\circ M\circ Q_{f} \;:\;
			M \ \text{an increasing M\"obius self-map of } I \,\big\}.
			\label{eq:stab-group}
		\end{equation}
		For $f, g \in \mathcal{D}_{+}^{2}$ the torsor
		map~\eqref{eq:torsor-map} is automatically a $C^{3}$-diffeomorphism
		of $(0,1)$, and the coadjoint formula, which involves only three
		derivatives, extends verbatim to $C^{3}$ distortions with the same
		cocycle law; the orbit-map statement is understood in this extended
		sense, on the $C^{3}$-regular part of the fiber.
	\end{theorem}
	
	\begin{proof}
		The map~\eqref{eq:torsor-map} is an increasing bijection of $(0,1)$,
		extended to $[0,1]$ by monotone continuity, and it is absolutely
		continuous by the Banach--Zarecki criterion~\cite{natanson1955}:
		continuous, monotone, and mapping null sets to null sets. The last
		property has separate justifications in the two cases: on
		$\mathcal{D}_{+}$ the map $Q_{f}$ is locally Lipschitz on $(0,1)$,
		its derivative $1/\mathfrak{f}$ being locally bounded, while on the
		enlargement $Q_{f}$ is itself absolutely continuous by hypothesis;
		in both cases $F_{g}$ is absolutely continuous, and compositions of
		monotone maps with the Luzin property retain it. The inverse
		$F_{f}\circ Q_{g}$ is of the same form. Then
		$A\circ F_{f} = F_{g}$, which is transitivity, and
		$A\circ F_{f} = F_{f}$ forces $A = \mathrm{id}$ on $(0,1)$, which is
		freeness. A distortion stabilizes $U_{f}$ exactly when it carries $f$
		into the fiber of Theorem~\ref{thm:fibers} without changing the
		support, that is, when $\alpha_{A}f = M_{\#}f$ for an increasing
		M\"obius self-map $M$ of $I$; solving with the torsor
		map~\eqref{eq:torsor-map} gives the
		description~\eqref{eq:stab-group}, and $A \mapsto M$ is a group
		isomorphism by freeness.
	\end{proof}
	
	The support type of a fiber is itself coadjoint data. Viewing the
	quantile function as the development of the potential, a curve in
	the projective line, the endpoint limits $Q(0^{+})$ and $Q(1^{-})$
	exist, and on the projective line the two infinities are the same
	point. Exactly one of two things happens: either the two limits are
	distinct, the \emph{subcritical} case, in which case the fiber of
	$U_{f}$ contains laws of bounded support and laws of half-line
	support but none of full support; or the two limits coincide, the
	\emph{critical} case, which happens exactly for full support, the
	development closing up through infinity, and then every law in the
	fiber has support $\mathbb{R}$ and the stabilizer is
	two-dimensional. Criticality is preserved by the coadjoint action,
	the development being pre-composed with an endpoint-fixing map and
	post-composed with a M\"obius map, neither of which affects whether
	the endpoint limits agree. This dichotomy places the entire
	zero-curvature family of Theorem~\ref{thm:trichotomy}, uniform,
	Lomax, and truncated Lomax, in a single subcritical fiber, the three
	support types being three M\"obius placements of one projective arc;
	and it identifies full support as the delicate, critical case, which
	is exactly where the extra symmetry appears below.
	
	The infinitesimal symmetries of a potential are computed by a
	classical mechanism, the symmetric square of the Hill equation:
	products of Hill solutions solve the third-order equation
	\begin{equation}
		\mathcal{L}_{U}\,\xi
		:= \xi''' + 2U\xi' + U'\xi = 0,
		\label{eq:symmetric-square}
	\end{equation}
	whose left side is precisely the infinitesimal coadjoint
	variation~\eqref{eq:inf-coadjoint} of $U$ along $\xi$ at central
	charge one, and whose solution space is spanned by the three products
	of a solution basis (Lemma~\ref{lem:symsq} in
	Appendix~\ref{app:proofs}). By Proposition~\ref{prop:hill} the basis
	may be taken to be the rank amplitude and its Wronskian partner, so
	the solutions of the symmetric-square
	equation~\eqref{eq:symmetric-square} are exactly the fields
	\begin{equation}
		\xi = \mathfrak{f}\cdot P(Q),
		\qquad P \ \text{a real polynomial of degree at most two},
		\label{eq:isotropy-fields}
	\end{equation}
	the infinitesimal M\"obius fields of the underlying variable pulled
	back to rank space. The formal isotropy algebra of $U_{f}$ is the
	space of fields~\eqref{eq:isotropy-fields} vanishing at both
	endpoints of rank space, and membership is decided by the boundary
	decay of $\mathfrak{f}$, $\mathfrak{f}Q$, and $\mathfrak{f}Q^{2}$,
	that is, by the tails of $f$: a finite support endpoint where the
	profile has a nonzero limit imposes one root of $P$; an infinite
	endpoint imposes the constraint $\deg P \le 1$ exactly when the tail
	is quadratic or heavier, and no constraint when it is lighter. The
	isotropy \emph{group}, by contrast, is always the conjugated
	M\"obius group of the display~\eqref{eq:stab-group}, of dimension
	two for full support and one otherwise; its Lie algebra consists of
	the \emph{complete} fields among the
	family~\eqref{eq:isotropy-fields}, those whose flows fix the
	endpoints for all time. The gap between the two is a tail
	diagnostic. For the uniform law the counts are $(1,1)$, the single
	field $z(1-z)$ transporting the uniform law through the truncated
	Lomax family; for the exponential law the formal algebra is
	two-dimensional but only the scaling field $(1-z)\log(1-z)$ is
	complete, the second field $(1-z)\log^{2}(1-z)$ reaching the
	endpoint in finite time; for the Gaussian and the logistic the
	formal algebra is the full three-dimensional M\"obius algebra while
	only the affine two survive integration; and for the Cauchy family
	formal and integrated symmetry coincide at dimension two. Failure of
	a formal symmetry to integrate is thus itself a statement about
	tails, and the Gaussian's projective singularity, established in
	Theorem~\ref{thm:singularity}, reappears here as an incomplete
	symmetry direction: the third M\"obius flow carries mass to the rank
	boundary in finite time.
	
	At constant potentials this stratification produces a single
	distinguished point.
	
	\begin{theorem}[The exceptional point]\label{thm:exceptional}
		Among constant potentials in the probability cone, the isotropy
		group is one-dimensional except at the threshold
		$U \equiv 2\pi^{2}$, where its dimension jumps to two. There the
		formal isotropy algebra is
		\begin{equation}
			\operatorname{stab}(2\pi^{2})
			= \ker\big(\partial^{3} + 4\pi^{2}\partial\big)\cap
			\{\xi(0) = \xi(1) = 0\}
			= \operatorname{span}\Big\{
			\frac{\sin(2\pi z)}{2\pi},\
			\frac{1 - \cos(2\pi z)}{2\pi}\Big\},
			\label{eq:stab-canonical}
		\end{equation}
		which is precisely the span of the drift fields of the canonical
		flow and the quarter-phase flow of Section~\ref{sec:flow}; both
		fields are complete, so the entire formal symmetry integrates. In
		the uniformizing coordinate $V = \tan(\pi(z - \tfrac12))$ the two
		flows are the dilations $\dot V = V$ and the translations
		$\dot V = 1$, so the isotropy group is the affine group of the line,
		realized inside the distortion group by conjugation with any Cauchy
		reference law. No law has isotropy group of dimension three or more.
	\end{theorem}
	
	\begin{proof}
		The group dimensions follow from Theorem~\ref{thm:torsor} and the
		support types of the constant fibers: subcritical constants have
		bounded or half-line supports and one-dimensional M\"obius groups,
		while the threshold fiber is the Cauchy family, critical, with the
		two-dimensional affine group. For the
		display~\eqref{eq:stab-canonical}: at constant $U = u_{0}$ the
		symmetric-square operator is
		$\mathcal{L}_{U}\xi = \xi''' + 2u_{0}\xi'$, with kernel spanned by
		$1$, $\sin(2\pi z)$, $\cos(2\pi z)$ at $u_{0} = 2\pi^{2}$, and the
		boundary conditions select the stated span, the two endpoint
		conditions coinciding on this periodic kernel, so that a single
		condition cuts the dimension from three to two; that both fields
		annihilate the background is the computation
		$\delta_{\xi}(2\pi^{2}) = 4\pi^{2}\xi' + \xi''' = 0$. Completeness:
		near $z = 0$ the fields behave as $z$ and as $\pi z^{2}$
		respectively, so the time to reach the endpoint,
		$\int_{0^{+}}dz/\xi$, diverges in both cases; alternatively,
		completeness is witnessed by the closed-form flows, which in the $V$
		coordinate read $\dot V = V$ and $\dot V = 1$, by the identities
		$\sin(2\pi z) = -2V/(1+V^{2})$,
		$1 - \cos(2\pi z) = 2/(1+V^{2})$, and $V' = \pi(1+V^{2})$. The final
		bound is Theorem~\ref{thm:fibers}: no M\"obius group of an interval
		exceeds dimension two.
	\end{proof}
	
	Theorem~\ref{thm:exceptional} is the mechanism beneath two facts
	established earlier by computation: the two-dimensional algebra with
	bracket $[X,Y] = -Y$ found in Section~\ref{sec:flow} is the isotropy
	algebra of the constant background $2\pi^{2}$, computed as the
	Dirichlet kernel of $\partial^{3} + 4\pi^{2}\partial$; and the Cauchy
	family, its fiber, is the interval avatar of the exceptional
	Virasoro orbit, the quotient of the diffeomorphism group of the
	circle by its M\"obius subgroup, whose distinguished role in the
	coadjoint classification is
	classical~\cite{lazutkin1975,kirillov1982,segal1981,witten1988}. On
	the circle the exceptional constants recur at $2\pi^{2}k^{2}$; on
	the interval Corollary~\ref{cor:constants} excludes the higher ones
	from the probability cone, their developments winding $k$ times.
	They are nevertheless visible in the theory: their would-be ground
	states are the harmonic kernels $w_{k} = 2\sin^{2}(k\pi z)$ of
	Section~\ref{sec:canonical}, whose sign-changing amplitudes
	$\sin(k\pi z)$ are exactly what realizability forbids, and whose
	iterated dynamics crystallizes laws onto $k$ atoms by the
	limit~\eqref{eq:crystallization}. The higher exceptional strata thus
	appear not as laws but as the degenerate $k$-atom boundary onto
	which the $w_{k}$-dynamics collapses: the probability cone touches
	the $k$-th exceptional geometry only through its boundary.
	
	The dynamical theory now closes. Along the canonical flow the
	potential is transported without anomaly, and every law relaxes to
	the exceptional point.
	
	\begin{theorem}[Anomaly-free transport]\label{thm:transport}
		Let $f_{0} \in \mathcal{D}_{+}^{2}$ and let $f_{t}$ solve the
		canonical flow~\eqref{eq:flow}, with inverse transport
		$B_{t} = A_{t}^{-1}$ given in the coordinate
		$V(z) = \tan(\pi(z - \tfrac12))$ by
		\begin{equation}
			V\circ B_{t} = e^{-t}\,V,
			\qquad
			B_{t}'(z) = \frac{e^{-t}\big(1 + V(z)^{2}\big)}
			{1 + e^{-2t}V(z)^{2}}.
			\label{eq:Bt}
		\end{equation}
		Then the Schwarzian of the transport satisfies
		\begin{equation}
			\{B_{t}, z\} = 2\pi^{2}\big(1 - (B_{t}')^{2}\big),
			\label{eq:schwarzian-Bt}
		\end{equation}
		and the potential evolves by
		\begin{equation}
			\mathfrak{U}_{t}
			= \big(\mathfrak{U}_{0}\circ B_{t}\big)\,(B_{t}')^{2},
			\qquad
			\mathfrak{U} := U - 2\pi^{2}:
			\label{eq:centered-transport}
		\end{equation}
		relative to the canonical background, the flow transports the
		centered potential as a plain quadratic differential, with no
		Schwarzian anomaly. The same holds for every element of the isotropy
		group of Theorem~\ref{thm:exceptional}.
	\end{theorem}
	
	\begin{proof}
		The formula~\eqref{eq:Bt} restates the exact
		solution~\eqref{eq:exact-solution} for the inverse transport in the
		centered coordinate, and differentiation of $V\circ B_{t} = e^{-t}V$
		with $V' = \pi(1 + V^{2})$ gives $B_{t}'$. For the Schwarzian
		identity~\eqref{eq:schwarzian-Bt}, apply the cocycle law to
		$V\circ B_{t} = M\circ V$ with $M$ the linear map
		$V \mapsto e^{-t}V$: since $\{V, z\} = 2\pi^{2}$ and $\{M\} = 0$,
		\begin{equation*}
			2\pi^{2}
			= \{V\circ B_{t}, z\}
			= \big(\{V\}\circ B_{t}\big)(B_{t}')^{2} + \{B_{t}, z\}
			= 2\pi^{2}(B_{t}')^{2} + \{B_{t}, z\}.
		\end{equation*}
		The transport law~\eqref{eq:centered-transport} is then
		Theorem~\ref{thm:vir-equivariance} with $B = B_{t}$, the anomaly
		$\{B_{t}\}$ recombining with the background into the tensorial form;
		the general isotropy statement follows since every element of the
		group is affine in the $V$ coordinate and the same cocycle
		computation applies.
	\end{proof}
	
	\begin{theorem}[Curvature no-hair]\label{thm:nohair}
		Let $f_{0} \in \mathcal{D}_{+}^{2}$ with potential $U_{0}$
		continuous on $(0,1)$. Then for every compact $K \subset (0,1)$,
		\begin{equation}
			\sup_{z \in K}\big|U_{t}(z) - 2\pi^{2}\big|
			\;\le\;
			e^{-2t}\,\sup_{z \in K}\big(1 + V(z)^{2}\big)^{2}
			\sup_{|s - \frac12| \le \delta_{K}(t)}
			\big|U_{0}(s) - 2\pi^{2}\big|,
			\label{eq:nohair}
		\end{equation}
		with $\delta_{K}(t) \to 0$; in particular $U_{t} \to 2\pi^{2}$
		locally uniformly at the exact exponential rate $e^{-2t}$ provided
		the centered potential does not vanish at the median,
		$\mathfrak{U}_{0}(\tfrac12) \neq 0$, equivalently
		$U_{0}(\tfrac12) \neq 2\pi^{2}$, the
		degenerate case relaxing faster, at $o(e^{-2t})$, with
		asymptotic profile
		\begin{equation}
			e^{2t}\big(U_{t}(z) - 2\pi^{2}\big)
			\;\longrightarrow\;
			\big(1 + V(z)^{2}\big)^{2}\,
			\big(U_{0}(\tfrac12) - 2\pi^{2}\big)
			\label{eq:nohair-profile}
		\end{equation}
		whenever $U_{0}$ is continuous at $\tfrac12$. Equivalently, in
		curvature form: along the canonical flow the projective curvature of
		every law relaxes locally uniformly in rank to the Cauchy value
		$-2\pi^{2}$ at rate $e^{-2t}$, and the limit retains no memory of
		the initial law beyond the single number visible in the
		profile~\eqref{eq:nohair-profile}.
	\end{theorem}
	
	\begin{proof}
		By the transport law~\eqref{eq:centered-transport},
		$U_{t} - 2\pi^{2} = (B_{t}')^{2}\,(\mathfrak{U}_{0}\circ B_{t})$.
		From the formula~\eqref{eq:Bt}, $V(B_{t}(z)) = e^{-t}V(z) \to 0$, so
		$B_{t} \to \tfrac12$ uniformly on compacts, defining
		$\delta_{K}(t)$, and
		\begin{equation*}
			0 < B_{t}'(z)
			= \frac{e^{-t}(1 + V^{2})}{1 + e^{-2t}V^{2}}
			\;\le\; e^{-t}\big(1 + V(z)^{2}\big),
		\end{equation*}
		whence the bound~\eqref{eq:nohair}. For the
		profile~\eqref{eq:nohair-profile}, $e^{t}B_{t}' \to 1 + V^{2}$
		pointwise and $\mathfrak{U}_{0}(B_{t}(z)) \to \mathfrak{U}_{0}(\tfrac12)$
		by continuity; the rate is exact because the limit is finite and
		nonzero. The curvature restatement is the quantile
		identity~\eqref{eq:quantile-schwarzian} read along the flow.
	\end{proof}
	
	Theorem~\ref{thm:nohair} is the geometric mechanism beneath the
	Lorentzian condensation of Theorem~\ref{thm:flow}: there, the law
	near the median converges after rescaling to a Cauchy profile; here,
	the entire field of projective curvature converges to the Cauchy
	background, the initial law surviving only through one number.
	Along general kernel flows the conserved quantities are now
	transparent: the endpoint weights of Theorem~\ref{thm:weights} are
	coadjoint invariants, the critical or subcritical class is
	preserved, and for the canonical flow the full endpoint germ of the
	centered potential evolves by the explicit tensorial
	law~\eqref{eq:centered-transport}, so that every functional of the
	germ invariant under quadratic-differential rescaling is conserved;
	the tail-rate conservation of Section~\ref{sec:flow} is the
	fiber-coordinate shadow of these charges. The third distinguished
	law of the dynamics also finds its place: for the hyperbolic-secant
	steady state of equation~\eqref{eq:sech}, the closure
	identity~\eqref{eq:sech-closure} gives the rank profile
	$\mathfrak{f}^{*}(z) = \sin(\pi z)/\pi\sqrt D$, whence, by the
	amplitude form~\eqref{eq:amplitude-form},
	\begin{equation}
		U^{*}(z)
		= \frac{\pi^{2}}{2} + \frac{\pi^{2}}{2}\,\csc^{2}(\pi z),
		\label{eq:poschl-teller}
	\end{equation}
	an inverse-square potential of P\"oschl--Teller type with both
	endpoint weights equal to $\tfrac12$, saturating the Hardy bound in
	agreement with its exponential tails. The three laws attached to the
	canonical dynamics thus occupy three distinguished potentials: the
	Cauchy family at the exceptional constant, the zero-curvature family
	at zero, and the diffusive soliton at the integrable inverse-square
	potential.
	
	\begin{remark}[Relaxation to the Schwarzian saddle]
		The structure of Theorems~\ref{thm:exceptional}--\ref{thm:nohair}
		has the form of the no-hair paradigm: a maximally symmetric
		background, the exceptional orbit point; an open set of initial data
		relaxing to it at a universal exponential rate fixed by the
		background's isotropy, the multiplier $e^{-2t}$ being twice the
		flow's unit rate, the weight-two tensor character of the centered
		potential; and hair carried away along the flow. In the language of
		the Schwarzian theory of low-dimensional
		gravity~\cite{msy2016,stanford2017}, laws play the role of
		reparametrization modes, the total
		integral $-\int_{0}^{1}U_{f}$ of Proposition~\ref{prop:pairing} is a
		Schwarzian action, the M\"obius saddle is the Cauchy family, and the
		canonical flow implements the relaxation onto the saddle; central
		charge one and the interval, rather than the circle, are exactly the
		setting of the boundary particle of Jackiw--Teitelboim gravity. We
		do not press the analogy beyond this remark.
	\end{remark}
	
	Directional data live on the circle, and there the correspondence
	becomes exact Kirillov theory, with an integrable hierarchy
	attached. A \emph{circular law} is a positive $C^{2}$ probability
	density on a circle of circumference $\ell$; fixing base points, its
	distribution function lifts to an increasing $F$ with
	$F(x + \ell) = F(x) + 1$, its quantile lifts with
	$Q(z + 1) = Q(z) + \ell$, and the potential $U_{f} = \{Q, \cdot\}$
	is a one-periodic function, a coadjoint vector of the Virasoro
	algebra of the rank circle at central charge one, on which the
	equivariance theorem holds verbatim with no marked point, the fiber
	of a potential consisting of the rotations and dilations of the law
	together with rigid rotations of rank. A continuous one-periodic $U$ is the
	potential of a circular law if and only if its Hill equation has a
	positive one-periodic solution, if and only if zero is the lowest
	periodic eigenvalue of the Hill operator
	$-\,d^{2}/dz^{2} - \tfrac12 U$ on the rank circle
	(Lemma~\ref{lem:circular} in Appendix~\ref{app:proofs}); the
	circumference is recovered as $\ell = \int_{0}^{1}\psi^{-2}$ once
	the scale of $\psi$ is fixed, the potential determining the law
	only up to rotation and dilation of the circle, and
	unit mass is automatic exactly as in
	Theorem~\ref{thm:realization}. Note the contrast with the interval:
	on the circle no nonzero constant is realizable, the unique
	constant-potential circular law being the uniform one; the boundary
	of rank space was the resource from which the interval theory built
	its canonical kernel.
	
	\begin{theorem}[The KdV hierarchy preserves circular laws]\label{thm:kdv}
		Let $q = -\tfrac12 U$ and evolve $q$ by any equation of the periodic
		Korteweg--de Vries hierarchy, for instance
		$\partial_{\tau}q = 6qq' - q'''$, with smooth periodic initial data.
		Then the circular realizability condition is preserved for all
		times: the hierarchy restricts to a family of commuting flows on
		circular laws modulo rotation and scale, the circumference being a
		choice of unit. All spectral invariants of
		the Hill operator, in particular the conserved densities
		\begin{equation}
			\int_{0}^{1}q\,dz,
			\qquad
			\int_{0}^{1}q^{2}\,dz,
			\qquad
			\int_{0}^{1}\Big(\tfrac12 (q')^{2} + q^{3}\Big)dz,
			\qquad \ldots,
			\label{eq:kdv-charges}
		\end{equation}
		define conserved functionals of circular laws, the first being half
		the expected projective curvature; the uniform law is a fixed point
		of every flow, and the distortion flows of
		Theorem~\ref{thm:vir-equivariance} are the linear, non-isospectral
		level of the same Lie--Poisson structure.
	\end{theorem}
	
	\begin{proof}
		Periodic KdV with smooth periodic data is globally well posed and
		isospectral for the Hill operator: the periodic spectrum,
		equivalently the Floquet discriminant, is
		conserved~\cite{mckean1975,magnus1966}. Realizability is the single
		spectral condition that the lowest periodic eigenvalue vanish
		(Lemma~\ref{lem:circular}), hence is preserved; the reconstructed
		law at each time is defined up to the rotation--dilation fiber,
		giving a flow
		on the moduli, and commutativity with the conserved
		densities~\eqref{eq:kdv-charges} is classical~\cite{mckean1975}.
		Constants are stationary by inspection of the equation, and the
		Lie--Poisson statement is the infinitesimal
		form~\eqref{eq:vir-inf}, the linear Hamiltonians
		$u \mapsto \int \xi u$ generating the coadjoint flows while the KdV
		Hamiltonian is quadratic in the same Poisson
		structure~\cite{khesin2009}.
	\end{proof}
	
	Circular probability thus carries an integrable system: infinitely
	many commuting deformations preserving the class of circular laws,
	together with an infinite list of invariants of a circular law, the
	KdV integrals of its potential, of which the first is half the
	expected projective curvature and the higher ones are new. What the
	hierarchy deforms is the shape of a circular law at fixed spectral
	data; the isospectral class of the law's Hill operator is a new
	invariant partition of circular probability, and the finite-gap
	circular laws, whose potentials are theta-functional, form
	finite-dimensional families closed under the entire hierarchy. Their
	statistical meaning, and the interval analogue, a boundary KdV
	theory in which the conserved endpoint weights of
	Theorem~\ref{thm:weights} would play the role of boundary spectral
	data, we leave open.
	
	\section{Rank Reaction--Diffusion: Kinks, Terraces, and Stability}
	\label{sec:pde}
	
	Among the dynamics of the calculus, the balance of modulation
	against diffusion in the equation~\eqref{eq:flow-diffusion} is the
	one with the classical face of a partial differential equation, and
	Section~\ref{sec:flow} produced for it a single exact object, the
	hyperbolic-secant steady law~\eqref{eq:sech}. This section supplies
	the complete theory: well-posedness, energy, uniqueness, spectrum,
	stability, wave selection, and the multi-kernel phenomenology. One
	structural fact organizes everything: \emph{the distribution
		function of a solution solves the overdamped sine--Gordon
		equation}. The nonlocal rank nonlinearity, integrated once in
	space, becomes a local bistable reaction, and the entire classical
	apparatus of scalar reaction--diffusion theory applies to
	probability laws, pruned throughout by the probabilistic
	constraint of monotonicity, which simplifies the classical picture
	at every step rather than complicating it. Write $\mathcal{M}$ for
	the continuous nondecreasing functions on $\mathbb{R}$ with limits
	$0$ and $1$, the distribution functions of laws.
	
	\begin{theorem}[Reduction and well-posedness on laws]\label{thm:sine-gordon}
		Let $w$ be a kernel and
		$\varphi_{w} := A_{w} - \mathrm{id}$ its reaction.
		\begin{enumerate}[label=\normalfont(\roman*), leftmargin=*]
			\item If $f$ solves the diffusive
			equation~\eqref{eq:flow-diffusion} with kernel $w$ and integrable
			data, then $F$ solves the local semilinear equation
			\begin{equation}
				\partial_{t}F = D\,\partial_{xx}F + \varphi_{w}(F),
				\qquad
				\varphi_{w}(u) = \int_{0}^{u}\big(w(s) - 1\big)\,ds,
				\label{eq:cdf-equation}
			\end{equation}
			and conversely; for the canonical kernel
			$\varphi(u) = -\sin(2\pi u)/(2\pi)$, so the angular variable
			$v = 2\pi F$ solves the overdamped sine--Gordon equation
			$\partial_{t}v = D\,\partial_{xx}v - \sin v$.
			\item For every $F_{0} \in \mathcal{M}$ the
			equation~\eqref{eq:cdf-equation} has a unique global solution,
			smooth for positive times, monotone in $x$ with values in $[0,1]$
			and limits $0, 1$ preserved for all time: laws evolve to laws, mass
			and positivity being conserved automatically.
			\item The dynamics is monotone for first-order stochastic
			dominance: $F_{0} \le G_{0}$ pointwise implies
			$F_{t} \le G_{t}$ for all $t$, the diffusive extension of the
			order-preservation of the operators themselves, which is the
			monotonicity of the transport maps $A_{w}$.
			\item As $D \to 0$, solutions converge locally uniformly on
			compact time intervals to the diffusionless flow; for the canonical
			kernel, to the exact solution~\eqref{eq:exact-solution}.
		\end{enumerate}
	\end{theorem}
	
	The proof is given in Appendix~\ref{app:proofs}; the reduction
	itself is one line, the pushforward identity
	$\int_{-\infty}^{x}(w(F) - 1)f\,dy = \varphi_{w}(F(x))$, the same
	mechanism that produced the transport
	formula~\eqref{eq:transport}.
	
	\begin{remark}[The kink behind the soliton]\label{rem:kink-reading}
		The reduction explains the solitary law~\eqref{eq:sech} in one
		stroke: the standing kink of the overdamped sine--Gordon equation
		is the Gudermannian profile, which is precisely the distribution
		function of the display~\eqref{eq:sech}, and the verification
		identity $\sin(\pi F^{*}) = \operatorname{sech}(x/\sqrt D)$ used in
		Section~\ref{sec:flow} is the classical kink identity. Better, the
		algebraic closure~\eqref{eq:sech-closure} acquires a variational
		name: dividing it by $4\pi^{2}$ reads
		$\tfrac{D}{2}f^{*2} = \sin^{2}(\pi F^{*})/(2\pi^{2})$, which is
		exactly the equipartition of the energy introduced next. What
		Section~\ref{sec:flow} met as an algebraic accident is the
		signature of a minimizer.
	\end{remark}
	
	\begin{proposition}[Gradient flow, equipartition, ground state]\label{prop:energy}
		The equation~\eqref{eq:cdf-equation} is the $L^{2}$ gradient flow
		of the interfacial energy
		\begin{equation}
			\mathcal{E}_{D}[F]
			= \int_{\mathbb{R}}\Big[\frac{D}{2}\,(\partial_{x}F)^{2}
			+ \Phi(F)\Big]\,dx,
			\qquad
			\Phi(u) = \frac{\sin^{2}(\pi u)}{2\pi^{2}},
			\label{eq:allen-cahn}
		\end{equation}
		with $\Phi' = -\varphi$, finite on finite-energy laws, those with
		$\partial_{x}F \in L^{2}(\mathbb{R})$ and exponentially decaying
		tails, to which this proposition is restricted, and
		$\frac{d}{dt}\mathcal{E}_{D} = -\int(\partial_{t}F)^{2} \le 0$
		along solutions with that decay. For monotone profiles the Modica bound holds,
		\begin{equation}
			\mathcal{E}_{D}[F]
			\;\ge\; \int_{0}^{1}\sqrt{2D\,\Phi(u)}\;du
			= \frac{2\sqrt D}{\pi^{2}},
			\label{eq:surface-tension}
		\end{equation}
		with equality precisely at the equipartition profiles
		$\tfrac{D}{2}(F')^{2} = \Phi(F)$: the secant law is the ground
		state of the energy~\eqref{eq:allen-cahn} on laws, unique up to
		translation, and the constant~\eqref{eq:surface-tension} is the
		surface tension of a probability interface.
	\end{proposition}
	
	\begin{proposition}[Condensation as the sharp-interface limit]\label{prop:gamma}
		As $D \to 0$ the rescaled energies
		$\mathcal{E}_{D}/\sqrt D$ Gamma-converge, on distribution functions
		supported in a fixed bounded interval, or along tight families, to
		$2/\pi^{2}$ times the number of jumps of the limit step
		function~\cite{modica1987}; on laws the minimum is a single jump. The variational sharp-interface limit of the
		reaction--diffusion theory is therefore condensation onto one atom,
		and for the harmonic kernel $w_{k}$, whose potential
		$\Phi_{k}(u) = \sin^{2}(k\pi u)/(2\pi^{2}k^{2})$ has wells at the
		lattice $\{j/k\}$, the limiting configurations are the
		$k$-jump crystallized laws with per-interface tension
		$2\sqrt D/\pi^{2}k^{2}$.
	\end{proposition}
	
	Both proofs are in Appendix~\ref{app:proofs}. The energy landscape
	thus interpolates the whole dynamical program: at positive
	diffusivity the ground state is the secant law; as the diffusivity
	vanishes, the energy concentrates into interfacial atoms and the
	variational theory degenerates into the condensation and
	crystallization of Section~\ref{sec:iteration}.
	
	\begin{theorem}[Uniqueness, spectrum, and stability of the secant
		law]\label{thm:sech-stability}
		For the canonical kernel:
		\begin{enumerate}[label=\normalfont(\roman*), leftmargin=*]
			\item the translates of the profile~\eqref{eq:sech} are the only
			steady states of the equation~\eqref{eq:cdf-equation} that are
			laws: the secant law is the unique steady probability law of the
			diffusive dynamics, up to translation;
			\item the linearization at the secant law is, in the similarity
			variable $\xi = x/\sqrt D$, the P\"oschl--Teller operator
			\begin{equation}
				\mathcal{L} = \partial_{\xi\xi} - 1 + 2\operatorname{sech}^{2}\xi,
				\qquad
				\operatorname{spec}\mathcal{L} = \{0\}\cup(-\infty, -1],
				\label{eq:stability-operator}
			\end{equation}
			with simple eigenvalue $0$ carried by the translation mode
			$\operatorname{sech}\xi$ and reflectionless essential
			spectrum~\cite{deift1979}: the spectral gap equals $1$ in the
			original time units, for every diffusivity $D$;
			\item the secant law attracts every law whose data approach the
			limits $0,1$ exponentially: there is a limit position $x_{0}$ with
			\begin{equation}
				\big\Vert F(\cdot,t) - F^{*}(\cdot - x_{0})\big\Vert_{\infty}
				\;\le\; C e^{-\mu t},
				\label{eq:asymptotic-phase}
			\end{equation}
			for some $\mu > 0$, and with every rate $\mu < 1$ once the
			solution enters a neighborhood of the front family; the densities
			converge in $L^{1}$, the only memory of the datum being the phase
			$x_{0}$.
		\end{enumerate}
	\end{theorem}
	
	\begin{remark}[One clock, two reflectionless faces]
		Two readings of the spectral statement deserve emphasis. First, the
		gap is one \emph{in the units of the modulation}: relaxation onto
		the secant law proceeds at exactly the unit rate of the entropy
		law~\eqref{eq:entropy-law} of the diffusionless flow. Diffusion
		reshapes the attractor, from a condensed atom to a secant profile
		of width $\sqrt D$, but it does not touch the clock; the
		informational tempo of the calculus is set by the modulation alone.
		Second, the secant law is now doubly reflectionless: its rank Hill
		potential is the inverse-square P\"oschl--Teller
		potential~\eqref{eq:poschl-teller} saturating the Hardy bound, and
		its stability operator~\eqref{eq:stability-operator} is the
		one-soliton P\"oschl--Teller Hamiltonian, so perturbations relax
		without backscattering and the linearized dynamics is exactly
		solvable. It is also worth noticing which constant the surface
		tension~\eqref{eq:surface-tension} produces at unit diffusivity:
		$2/\pi^{2}$, the block-energy constant of the quantum carpet, both
		descending from the same integral $\int_{0}^{1}\sin(\pi u)\,du$ of
		the boundary arithmetic of the canonical kernel.
	\end{remark}
	
	Asymmetric kernels do not stand still; they transport.
	
	\begin{theorem}[Traveling laws and the mean-rank Maxwell rule]\label{thm:maxwell}
		Let $w$ be a kernel whose reaction $\varphi_{w}$ is bistable:
		$w(0) < 1$, $w(1) < 1$, and a single interior zero of
		$\varphi_{w}$. Then the equation~\eqref{eq:cdf-equation} admits a
		traveling law $F(x,t) = F_{c}(x - ct)$, unique up to translation,
		with monotone profile, globally stable among laws, and its speed
		obeys
		\begin{equation}
			c = \frac{\mathbb{E}_{w}[Z] - \tfrac12}
			{\displaystyle\int_{\mathbb{R}}(F_{c}')^{2}\,dx},
			\qquad
			\operatorname{sign}(c)
			= \operatorname{sign}\big(\mathbb{E}_{w}[Z] - \tfrac12\big),
			\label{eq:mean-rank-rule}
		\end{equation}
		where $\mathbb{E}_{w}[Z]$ is the mean rank of the kernel: the
		Maxwell equal-area construction of bistable fronts collapses, for
		rank dynamics, to a single moment, and the law stands still
		precisely when the kernel is rank-balanced. For the phase-shifted
		family $w_{1,\alpha}$,
		\begin{equation}
			c(\alpha)
			= -\frac{\pi\sqrt D}{4}\,\sin(2\pi\alpha) + O(\alpha^{2})
			\qquad (\alpha \to 0).
			\label{eq:phase-speed}
		\end{equation}
	\end{theorem}
	
	\begin{proof}
		Existence, uniqueness, and global stability of monotone bistable
		fronts are classical~\cite{aronson1978,fife1977}. Multiplying the
		profile equation $DF_{c}'' + cF_{c}' + \varphi_{w}(F_{c}) = 0$ by
		$F_{c}'$ and integrating over the line kills the boundary terms
		and gives $c\int(F_{c}')^{2} = -\int_{0}^{1}\varphi_{w}(u)\,du$;
		the mean-rank identity
		\begin{equation*}
			\int_{0}^{1}\varphi_{w}(u)\,du
			= \int_{0}^{1}\big(A_{w}(u) - u\big)\,du
			= \int_{0}^{1}(1-s)\,w(s)\,ds - \tfrac12
			= \tfrac12 - \mathbb{E}_{w}[Z]
		\end{equation*}
		converts this into the rule~\eqref{eq:mean-rank-rule}. For the
		phase-shifted family,
		$\int_{0}^{1}\varphi_{w_{1,\alpha}} = \sin(2\pi\alpha)/(2\pi)$, and
		to leading order the profile is the canonical kink with
		$\int(F^{*\prime})^{2} = 2/(\pi^{2}\sqrt D)$, giving the
		expansion~\eqref{eq:phase-speed}.
	\end{proof}
	
	\begin{proposition}[Depinning on the exceptional algebra]\label{prop:depinning}
		Along the phase family, the reaction
		$\varphi_{\alpha}(u)
		= -\big[\sin(2\pi(u+\alpha)) - \sin(2\pi\alpha)\big]/(2\pi)$ is
		bistable for $0 \le \alpha < \tfrac14$ and degenerates at the
		quarter phase:
		$\varphi_{1/4}(u) = \sin^{2}(\pi u)/\pi \ge 0$, with double zeros
		at the endpoints and no interior zero. In the angular variable the
		family is the tilted washboard
		$\partial_{t}v = D\,\partial_{xx}v - \sin v + \sin(2\pi\alpha)$,
		and the tilt attains its depinning magnitude exactly at the
		quarter phase and at its mirror $\alpha = \tfrac34$, the
		reflection $\alpha \mapsto 1 - \alpha$ exchanging the two and
		reversing the drift, so the depinning parameter is unique up to
		that symmetry: at these phases, and only there, the reaction is
		one-signed and every interior rank drifts in one direction, while
		for $\alpha \in (\tfrac14, \tfrac34)$ the washboard re-pins, the
		reaction regaining a stable interior zero at the rank
		$\tfrac32 - 2\alpha$, and, by the reflection
		$\varphi_{1-\alpha}(u) = -\varphi_{\alpha}(1-u)$, bistability
		returns for $\alpha \in (\tfrac34, 1)$: the pinning--sliding transition of the
		washboard
		occurs exactly on the second generator of the isotropy
		fields~\eqref{eq:isotropy-fields} of the exceptional orbit, whose
		first generator is the canonical flow itself.
	\end{proposition}
	
	\begin{proof}
		The zeros of $\varphi_{\alpha}$ on $[0,1]$ are the endpoints
		together with the interior representative of the single residue
		class $\tfrac12 - 2\alpha$ modulo one; for $\alpha < \tfrac14$
		that representative is $\tfrac12 - 2\alpha$, where
		$\varphi_{\alpha}' = \cos(2\pi\alpha) > 0$ makes it unstable,
		while for $\alpha \in (\tfrac14,\tfrac34)$ it is
		$\tfrac32 - 2\alpha$, where the same derivative value
		$\varphi_{\alpha}'(\tfrac32 - 2\alpha) = \cos(2\pi\alpha) < 0$
		makes it stable. Bistability for $\alpha < \tfrac14$ follows
		from
		$w_{1,\alpha}(0) = 2\sin^{2}(\pi\alpha) < 1$ for
		$\alpha < \tfrac14$; at $\alpha = \tfrac14$ the interior zero
		collides with the endpoints and
		$\varphi_{1/4}(u) = \int_{0}^{u}\sin(2\pi s)\,ds
		= \sin^{2}(\pi u)/\pi$, and at the mirror phase $\alpha = \tfrac34$
		it collides again, with $\varphi_{3/4}(u) = -\sin^{2}(\pi u)/\pi$
		one-signed in the opposite direction. The washboard form is the substitution
		$v = 2\pi(u + \alpha)$, and the identification of
		$\varphi_{1/4}$, equivalently $(1 - \cos 2\pi z)/(2\pi)$ up to
		normalization, with the second field of the
		display~\eqref{eq:isotropy-fields} is Theorem~\ref{thm:exceptional}.
	\end{proof}
	
	\begin{theorem}[Terraces: crystallization made smooth]\label{thm:terraces}
		For the harmonic kernel $w_{k}$, whose reaction
		$\varphi_{k}(u) = -\sin(2\pi k u)/(2\pi k)$ has balanced wells at
		the
		lattice $\{j/k\}$, monotone data with exponentially localized
		tails develop a propagating terrace with all speeds
		zero~\cite{ducrot2014}: the solution converges, locally uniformly
		along its interfaces, to a stack of $k$ standing kinks, each a
		$1/k$-scaled copy of the canonical kink by the cell conjugacy of
		Section~\ref{sec:iteration}, so the density resolves into $k$
		secant-like bumps of mass exactly $1/k$ located near the
		crystallization quantiles of the limit~\eqref{eq:crystallization}.
		The terrace is not a steady state at finite separations: bump
		positions move by exponentially small interactions of Carr--Pego
		type~\cite{carr1989}; the resulting slow motion, constructed by
		Carr and Pego on bounded intervals and adapted here at sketch
		level, spreads the terrace with separations growing
		logarithmically, and nothing coarsens, monotone kinks being unable
		to annihilate.
	\end{theorem}
	
	\begin{proposition}[Bounded intervals: the cascade outside the
		cone]\label{prop:bounded-interval}
		On a bounded interval with the boundary conditions of a law, the
		stationary problem $DF'' + \varphi(F) = 0$ has exactly one monotone
		solution for every $D$ and every length: a unique steady law,
		asymptotically stable within laws. The Chafee--Infante hierarchy of
		sign-changing equilibria of the unconstrained equation consists
		entirely of non-monotone profiles and therefore lies outside the
		probability cone: constrained to laws, the bifurcation diagram
		collapses to a single branch.
	\end{proposition}
	
	Proofs are in Appendix~\ref{app:proofs}. The pattern of the last
	proposition recurs across the paper: the probability cone
	systematically selects the monotone skeleton of a classical
	theory, as it selected injective developments among projective
	charts and disconjugate potentials among Hill operators.
	
	\begin{remark}[Physical reading: the Frenkel--Kontorova dictionary]
		In physical terms the reduction places rank dynamics inside the
		overdamped sine--Gordon and Frenkel--Kontorova
		universe~\cite{braun2004}, with an exact dictionary: a probability
		law is a kink and its density the kink's profile; a crystallized
		law is a multi-kink lattice; the mean-rank offset of an asymmetric
		kernel is the washboard tilt, and the traveling law it drags is a
		solitary information wave; depinning is the quarter phase, sitting
		on the exceptional isotropy algebra; the surface tension of a
		probability interface is the constant~\eqref{eq:surface-tension};
		and the reflectionless relaxation of the secant law is the
		one-soliton transparency of its linearization. The information
		reading runs alongside: the Lyapunov
		functional~\eqref{eq:allen-cahn} splits into the diffusion cost and
		the canonical rank potential, its dissipation identity replaces the
		exact entropy ledgers of the diffusionless theories, and the
		diffusivity-independent gap says that diffusion chooses the shape
		of the attractor while the modulation keeps the clock.
	\end{remark}
	
	\section{Random Derangetropy: Synchronization, Sampling, and
		Multiplicative Chaos}\label{sec:random}
	
	The deterministic iteration of Section~\ref{sec:iteration} is
	phase-coherent: the same kernel, hence the same modulation phase, is
	applied at every step, and coherence is what selects the median and
	produces the Koenigs law. This section randomizes the phase. The
	resulting theory is again exactly solvable, and it stands to the
	deterministic one as an incoherent phase stands to a coherent one:
	condensation survives, but onto a uniformly random quantile rather
	than the median, turning the dynamics into an exact sampler; the
	divergence ledger becomes exact with no deficit, exposing the Koenigs
	entropy as a coherence effect; and when the modulation frequency is
	renewed geometrically across steps, condensation gives way to a
	nondegenerate multiplicative chaos carried by a random homeomorphism
	of rank space.
	
	\begin{definition}[Randomized cascade]\label{def:random-cascade}
		Extend a kernel $w$ periodically to $\mathbb{R}$ and let
		$w_{\alpha}(z) = w(z + \alpha)$ be its phase family, as in the
		display~\eqref{eq:phase-kernels} for the harmonic shapes. Let
		$\alpha_{1}, \alpha_{2}, \ldots$ be independent uniform phases on
		$[0,1)$, let $A_{j} := A_{w_{\alpha_{j}}}$, and define the random
		composition and the random cascade
		\begin{equation}
			S_{n} := A_{n}\circ\cdots\circ A_{1},
			\qquad
			f_{n} := \rho_{w_{\alpha_{n}}}\cdots\rho_{w_{\alpha_{1}}}[f]
			= S_{n}'(F)\,f,
			\label{eq:random-cascade}
		\end{equation}
		the second identity by the composition law of
		Theorem~\ref{thm:composition}. Write
		$\mathcal{F}_{n} = \sigma(\alpha_{1},\ldots,\alpha_{n})$, and for
		$z \in [0,1]$ let $z_{n} := S_{n}(z)$ be the rank orbit.
	\end{definition}
	
	One mechanism drives every exact formula below: a uniform phase
	re-randomizes the modulation relative to the current rank.
	
	\begin{lemma}[Re-uniformization]\label{lem:reunif}
		Let $w$ be a periodized kernel and $y$ an
		$\mathcal{F}_{n-1}$-measurable random variable. Then, conditionally
		on $\mathcal{F}_{n-1}$, the multiplier $w_{\alpha_{n}}(y)$ is
		distributed as $W := w(U)$ with $U$ uniform, independently of
		$\mathcal{F}_{n-1}$. Consequently, for every fixed $z$, the
		multipliers $V_{j} := w_{\alpha_{j}}(S_{j-1}(z))$ are independent
		copies of $W$, and
		\begin{equation}
			\log S_{n}'(z) = \sum_{j=1}^{n}\log V_{j}
			\label{eq:reunif-walk}
		\end{equation}
		is a random walk with independent identically distributed steps.
		For every harmonic shape $w_{k}$ the step law is the same,
		$W = 2\sin^{2}(\pi U)$.
	\end{lemma}
	
	\begin{proof}
		By periodicity, $w_{\alpha}(y)$ depends only on $y + \alpha$ modulo
		one, and if $\alpha$ is uniform and independent of $y$ then
		$y + \alpha \bmod 1$ is uniform and independent of
		$\mathcal{F}_{n-1}$, by translation invariance of Lebesgue measure.
		The conditional law of $w_{\alpha_{n}}(y)$ given
		$\mathcal{F}_{n-1}$ is therefore the fixed law of $w(U)$, which,
		being constant, makes the multiplier independent of the past;
		applying this to $y = S_{j-1}(z)$ and inducting gives joint
		independence, and the walk~\eqref{eq:reunif-walk} is the chain rule.
		For the harmonic shapes,
		$w_{k}(y + \alpha) = 2\sin^{2}(\pi(ky + k\alpha))$ and
		$k\alpha \bmod 1$ is uniform, so the step law does not depend on
		$k$.
	\end{proof}
	
	The lemma says that phase randomization makes the cascade spatially
	blind at one point: the noise seen by any single rank is pure
	independent noise, identical across the entire harmonic family. All
	one-point statistics are therefore universal and computable; all
	geometry lives in the joint behavior of several ranks, where the
	harmonic index re-enters. The second structural fact is conservation
	in mean: since
	$\mathbb{E}_{\alpha}A_{\alpha}(z)
	= \int_{0}^{z}\int_{0}^{1}w(s+\alpha)\,d\alpha\,ds = z$ by
	periodicity, every rank orbit is a bounded martingale, the
	measure-valued sequence $S_{n}'(z)\,dz$ is a martingale with mean
	Lebesgue measure, and
	\begin{equation}
		\mathbb{E}\,[\,f_{n}\,] = f
		\qquad\text{for every } n:
		\label{eq:annealed}
	\end{equation}
	the annealed law is conserved at every step. This contrasts sharply
	with the deterministic theory, where each application strictly
	contracts the law toward its median by the stochastic
	contraction~\eqref{eq:contraction}: the randomized update is
	unbiased, and everything that follows is the interplay of exact
	conservation in mean with almost sure degeneration or survival of
	individual realizations.
	
	\begin{theorem}[Synchronization and exact sampling]\label{thm:sync}
		For the canonical shape $w = 2\sin^{2}(\pi\cdot)$:
		\begin{enumerate}[label=\normalfont(\roman*), leftmargin=*]
			\item for every $z$, the rank orbit $z_{n}$ converges almost surely
			to a limit in $\{0,1\}$ with $\mathbb{P}(z_{\infty} = 1) = z$, and
			the conditional variance identity
			\begin{equation}
				\operatorname{Var}\big(z_{n+1} \,\big|\, \mathcal{F}_{n}\big)
				= \frac{\sin^{2}(\pi z_{n})}{2\pi^{2}},
				\qquad
				\sum_{n \ge 0}\mathbb{E}\big[\sin^{2}(\pi z_{n})\big]
				= 2\pi^{2}\,z(1-z),
				\label{eq:variance-identity}
			\end{equation}
			holds exactly;
			\item almost surely there is a random threshold $\tau$ such that
			$S_{n}(z) \to 0$ for all $z < \tau$ and $S_{n}(z) \to 1$ for all
			$z > \tau$, and $\tau$ is uniformly distributed on $(0,1)$;
			\item for every $f \in \mathcal{D}_{+}$, almost surely
			\begin{equation}
				f_{n} \;\xrightarrow{\ w^{*}\ }\; \delta_{Q(\tau)},
				\qquad\text{and}\qquad
				Q(\tau) \sim f:
				\label{eq:exact-sampler}
			\end{equation}
			the cascade condenses the law onto a single random atom whose
			distribution is the initial law itself;
			\item for every $z \in (0,1)$, almost surely
			$\tfrac1n\log\min(z_{n}, 1 - z_{n}) \to -\log 2$: absorption
			proceeds at one bit per step.
		\end{enumerate}
	\end{theorem}
	
	\begin{proof}
		For part (i), the orbit is a bounded martingale, hence converges
		almost surely and in $L^{2}$. The increment is
		\begin{equation*}
			z_{n+1} - z_{n}
			= -\frac{\sin\big(2\pi(z_{n}+\alpha)\big) - \sin(2\pi\alpha)}{2\pi}
			= -\frac{\cos\big(\pi(z_{n}+2\alpha)\big)\,\sin(\pi z_{n})}{\pi},
		\end{equation*}
		and averaging the squared cosine over the uniform phase gives the
		first identity of the display~\eqref{eq:variance-identity}. By
		orthogonality of martingale increments,
		$\mathbb{E}\,z_{n}^{2} = z^{2} +
		\sum_{j<n}\mathbb{E}\operatorname{Var}(z_{j+1}|\mathcal{F}_{j})$;
		the left side is bounded, so the variance series converges,
		$\sin^{2}(\pi z_{n}) \to 0$ almost surely, and the limit lies in
		$\{0,1\}$; then $\mathbb{E}\,z_{\infty} = z$ identifies the hitting
		probability, $\mathbb{E}\,z_{\infty}^{2} = z$ closes the second
		identity of the display~\eqref{eq:variance-identity}. For part
		(ii), apply part (i) simultaneously to all rationals; monotonicity
		of $S_{n}$ makes the limit function nondecreasing with values in
		$\{0,1\}$ on the rationals, and
		$\tau := \sup\{z \in \mathbb{Q} : S_{\infty}(z) = 0\}$ satisfies
		$\mathbb{P}(\tau < z) \le z \le \mathbb{P}(\tau \le z)$ for every
		rational $z$, and the monotone limits force equality and
		atomlessness: $\tau$ is uniform. Part (iii): the distribution functions satisfy
		$F_{f_{n}} = S_{n}\circ F \to \mathbf{1}\{x \ge Q(\tau)\}$ at every
		continuity point, which is the weak-$*$
		convergence~\eqref{eq:exact-sampler}, and $Q(\tau) \sim f$ is the
		inverse-transform principle for uniform $\tau$; note the
		consistency with the annealed conservation
		law~\eqref{eq:annealed}, since
		$\mathbb{E}\,\delta_{Q(\tau)} = f$. Part (iv) is proved in
		Appendix~\ref{app:proofs}.
	\end{proof}
	
	\begin{remark}[Sampling by dynamics]
		Part (iii) of Theorem~\ref{thm:sync} says that the randomized
		canonical operator \emph{samples}: run the cascade and read off the
		surviving atom, and one has drawn exactly once from $f$, by pure
		rank dynamics, with the one-bit rate of part (iv) as the cost of
		precision, about $\log_{2}(1/\varepsilon)$ steps per
		$\varepsilon$ of rank resolution. The mechanism is the martingale
		structure: unbiasedness at every step plus almost sure condensation
		forces the atom's law to be $f$. The deterministic theory is the
		completely phase-locked limit, the threshold degenerating to the
		median; and unbiasedness constrains the phase law exactly through
		the first harmonic of the kernel, the uniform phase being the
		canonical choice, the one whose phase mixture of the kernel is
		exactly the uniform density.
	\end{remark}
	
	The harmonic cascade synchronizes cell by cell, with one common
	noise.
	
	\begin{theorem}[Systematic sampling]\label{thm:systematic}
		Run the randomized cascade with the harmonic shape $w_{k}$. Then
		almost surely there is a single uniform $\tau$ such that
		\begin{equation}
			f_{n} \;\xrightarrow{\ w^{*}\ }\;
			\frac1k\sum_{i=0}^{k-1}
			\delta_{\,Q\left(\frac{i+\tau}{k}\right)}:
			\label{eq:systematic}
		\end{equation}
		the cascade condenses onto the classical systematic sample of size
		$k$ from $f$, the $k$ quantile cells sampled at a common uniform
		offset.
	\end{theorem}
	
	\begin{proof}
		Every cell boundary is fixed by every phase,
		$A_{w_{k,\alpha}}(i/k) = i/k$, and on the $i$-th cell the affine
		conjugation $\psi_{i}(z) = kz - i$ of the crystallization
		theory gives, by periodicity of the sine,
		\begin{equation*}
			\psi_{i}\circ A_{w_{k,\alpha}}
			= A_{w_{1,\beta}}\circ \psi_{i},
			\qquad
			\beta = k\alpha \bmod 1,
		\end{equation*}
		with the \emph{same} $\beta$ for every cell, and the
		$\beta_{j}$ independent uniform. Each cell therefore carries a copy
		of the canonical cascade, all cells driven by identical noise; by
		Theorem~\ref{thm:sync}(ii) applied to that common noise there is one
		uniform threshold $\tau$ in the conjugated coordinate, so $S_{n}$
		converges to the step function with jumps of size $1/k$ at the ranks
		$(i+\tau)/k$, and transport by $Q$ gives the
		limit~\eqref{eq:systematic}.
	\end{proof}
	
	\begin{remark}[Quantization and sampling as two phases]
		Theorem~\ref{thm:systematic} completes a symmetry with the
		deterministic crystallization
		limit~\eqref{eq:crystallization}. Coherent phases condense the
		harmonic dynamics onto the odd quantiles, which are exactly the
		$W_{1}$-optimal equal-mass quantizer; incoherent phases condense it
		onto the ranks $(i+\tau)/k$ with a single uniform offset, which is
		exactly the systematic, or stratified, sample of classical survey
		statistics, the unbiased randomization of the same lattice. Optimal
		quantization and stratified sampling are thus the coherent and
		incoherent phases of one rank dynamics. Note also the division of
		labor established by Lemma~\ref{lem:reunif}: one-point statistics
		cannot see $k$ at all, while the joint geometry is a $k$-fold
		synchronized replica of the canonical cascade.
	\end{remark}
	
	At a single rank the cascade is a free lunch of explicit formulas.
	
	\begin{theorem}[One-point solvability]\label{thm:onepoint}
		Let $m(q) := \mathbb{E}[W^{q}]$ and $\lambda(q) := \log m(q)$, for
		the canonical step law $W = 2\sin^{2}(\pi U)$. Then:
		\begin{enumerate}[label=\normalfont(\roman*), leftmargin=*]
			\item for $q > -\tfrac12$,
			\begin{equation}
				m(q) = \frac{2^{q}\,\Gamma(q+\tfrac12)}{\sqrt{\pi}\,\Gamma(q+1)},
				\qquad
				\lambda(q) = q\log 2 + \log\Gamma(q+\tfrac12)
				- \log\Gamma(q+1) - \tfrac12\log\pi,
				\label{eq:free-energy}
			\end{equation}
			and $\mathbb{E}[(S_{n}'(z))^{q}] = m(q)^{n}$ exactly, for every
			fixed $z$;
			\item $\tfrac1n\log S_{n}'(z) \to -\log 2$ almost surely, and
			\begin{equation}
				\frac{\log S_{n}'(z) + n\log 2}{\sqrt{n}}
				\;\xrightarrow{\ d\ }\;
				\mathcal{N}\Big(0,\ \frac{\pi^{2}}{3}\Big),
				\qquad
				\frac{\pi^{2}}{3} = \lambda''(0)
				= \psi_{0}'(\tfrac12) - \psi_{0}'(1),
				\label{eq:random-clt}
			\end{equation}
			with $\psi_{0}$ the digamma function, and
			$\tfrac1n\log S_{n}'$ satisfies a large deviation principle with
			rate the Legendre transform of $\lambda$;
			\item the tangents of the free energy are the information constants
			of the deterministic theory,
			\begin{equation}
				\lambda'(0) = \mathbb{E}\log W = -\log 2,
				\qquad
				\lambda'(1) = \mathbb{E}[W\log W] = 1 - \log 2,
				\label{eq:tangents}
			\end{equation}
			the relative entropies of the identities~\eqref{eq:kl-identities},
			while the deterministic sharp Lipschitz constant $2^{n}$ of
			Proposition~\ref{prop:intertwine} is the almost sure envelope
			$S_{n}' \le 2^{n}$ and the limiting slope
			$\lambda(q) = q\log2 - \tfrac12\log q + O(1)$ as $q \to \infty$.
		\end{enumerate}
	\end{theorem}
	
	\begin{proof}
		Part (i) is the Wallis integral
		$\int_{0}^{1}\sin^{2q}(\pi u)\,du
		= \Gamma(q+\tfrac12)/(\sqrt\pi\,\Gamma(q+1))$ together with the
		independence of Lemma~\ref{lem:reunif}; finiteness requires
		$2q > -1$. Part (ii) is the strong law and the central limit
		theorem for the walk~\eqref{eq:reunif-walk}, the mean being the
		rank integral $\int_{0}^{1}\log(2\sin^{2}\pi u)\,du = -\log 2$
		computed in Lemma~\ref{lem:kl}, and the variance being two digamma
		derivatives, $\pi^{2}/2 - \pi^{2}/6 = \pi^{2}/3$; the large
		deviation principle is Cram\'er's theorem. Part (iii) is
		differentiation of the free energy~\eqref{eq:free-energy} at $0$
		and $1$ using $\psi_{0}(\tfrac12) = -\gamma - 2\log2$,
		$\psi_{0}(\tfrac32) = 2 - \gamma - 2\log2$, and the Stirling
		expansion for the slope; the envelope is $W \le 2$.
	\end{proof}
	
	\begin{theorem}[Exact entropy production]\label{thm:random-entropy}
		For every $f \in \mathcal{D}_{+}$ and every $n$,
		\begin{equation}
			\mathbb{E}\,D\big(f \,\|\, f_{n}\big) = n\log 2,
			\qquad
			\mathbb{E}\,D\big(f_{n} \,\|\, f\big) = n\,(1 - \log 2),
			\label{eq:random-entropy}
		\end{equation}
		with no lower-order correction, and
		$\tfrac1n D(f\|f_{n}) \to \log 2$ almost surely.
	\end{theorem}
	
	\begin{proof}
		By the universality principle~\eqref{eq:universality},
		\begin{equation*}
			D\big(f \,\|\, f_{n}\big) = -\int_{0}^{1}\log S_{n}'(z)\,dz,
			\qquad
			D\big(f_{n} \,\|\, f\big) = \int_{0}^{1}S_{n}'\log S_{n}'\,dz.
		\end{equation*}
		Taking expectations pointwise in $z$ under the
		walk~\eqref{eq:reunif-walk},
		$\mathbb{E}\log S_{n}'(z) = -n\log 2$, while by independence and
		$\mathbb{E}V = 1$,
		\begin{equation*}
			\mathbb{E}\big[S_{n}'\log S_{n}'\big]
			= \sum_{j=1}^{n}\mathbb{E}\big[V_{j}\log V_{j}\big]
			\prod_{i\neq j}\mathbb{E}V_{i}
			= n\,\mathbb{E}[W\log W] = n(1-\log2);
		\end{equation*}
		Fubini applies since $W\log W$ and $\log W$ are integrable. For the
		almost sure statement, write
		$\tfrac1n D(f\|f_{n}) = \tfrac1n\sum_{j\le n}\xi_{j}$ with
		$\xi_{j} := -\int_{0}^{1}\log\big(S_{j}'/S_{j-1}'\big)(z)\,dz$:
		by the re-uniformization of Lemma~\ref{lem:reunif} the increments
		satisfy $\mathbb{E}[\xi_{j}\,|\,\mathcal{F}_{j-1}] = \log 2$ and
		by the Cauchy--Schwarz inequality on the unit rank interval,
		\begin{equation*}
			\mathbb{E}[\xi_{j}^{2}\,|\,\mathcal{F}_{j-1}]
			\;\le\; \int_{0}^{1}
			\mathbb{E}\big[\log^{2}\big(S_{j}'/S_{j-1}'\big)(z)\,\big|\,
			\mathcal{F}_{j-1}\big]\,dz
			\;=\; \mathbb{E}\big[(\log W)^{2}\big]
			\;=\; \frac{\pi^{2}}{3} + \log^{2}2,
		\end{equation*}
		finite and uniform in $j$, a bound that controls both tails of the
		increments; the strong law of large numbers for martingale
		differences with bounded conditional second
		moments~\cite{hallheyde1980} then gives
		$\tfrac1n D(f\|f_{n}) \to \log 2$ almost surely.
	\end{proof}
	
	\begin{remark}[The Koenigs deficit is coherence]
		Compare the exact ledger~\eqref{eq:random-entropy} with the
		deterministic identity~\eqref{eq:divergence-entropy}, whose
		asymptotic deficit is the entropy of the Koenigs law. Randomization
		makes the ledger exactly linear, one bit per step forever: the
		deficit is thus revealed as a coherence effect, the information
		retained when the modulation interferes with its own past through
		the deterministic orbit, and independent phases destroy it. The
		tangent identities~\eqref{eq:tangents} place the deterministic
		constants on the multifractal spectrum of the random theory: the
		one-bit cost is the quenched slope of one Gamma-ratio free energy,
		the complementary constant is its size-biased slope, and the sharp
		Lipschitz constant is its ballistic edge. The constant $\log 2$ now
		has three provably equivalent roles: deterministic update cost,
		quenched decay rate of the random cascade, and the
		annealed--quenched gap of the frozen phase described next.
	\end{remark}
	
	The cascade is critical in mean and totally degenerate in
	realization: for every interval $I$, the mass
	$\mu_{n}(I) = \int_{I}S_{n}'\,dz$ is a martingale with
	$\mathbb{E}\mu_{n}(I) = |I|$ for all $n$, while
	$\mu_{n}(I) \to \mathbf{1}\{\tau \in I\}$ almost surely and
	$\mathbb{E}[\mu_{n}(I)^{2}] \uparrow |I|$, the signature of
	condensation onto a single unit atom at a uniform location.
	Two-point structure is equally explicit: for two ranks the
	separation $d_{n} := S_{n}(z') - S_{n}(z)$ is an autonomous Markov
	martingale,
	\begin{equation}
		d_{n+1} = d_{n}
		- \frac{\sin(\pi d_{n})\cos\theta_{n}}{\pi},
		\qquad \theta_{n}\ \text{independent uniform on } [0, 2\pi),
		\label{eq:separation}
	\end{equation}
	obtained from the increment formula by the product-to-sum identity,
	the angle being uniform and independent of the pair by
	re-uniformization; it absorbs at $\{0,1\}$ with
	$\mathbb{P}(d_{\infty} = 1) = z' - z$, consistently with the
	uniform threshold, and the two-point function obeys
	$\mathbb{E}[S_{n}'(z)S_{n}'(z')] \le m(2)^{n} = (3/2)^{n}$ by the
	Cauchy--Schwarz inequality. Nor is the freezing special to the
	canonical shape: for the tempered kernels
	$1 + \lambda\cos(2\pi(z+\alpha))$, mixtures of the uniform and
	canonical kernels inside the family, the same martingale and
	variance arguments apply verbatim with conditional variance scaled
	by $\lambda^{2}$, and the cascade condenses onto a single uniform
	atom for \emph{every} amplitude $\lambda \in (0,1]$, with Lyapunov
	exponent $\log\big((1+\sqrt{1-\lambda^{2}})/2\big) < 0$: no
	weak-noise limit at fixed frequency avoids condensation.
	
	\begin{remark}[Glassy and Bayesian readings]
		The frozen cascade is an exactly solvable glass at zero
		temperature: the annealed free energy per step vanishes,
		criticality, while the
		quenched free energy is one negative bit, the whole annealed mass
		carried by realizations of exponentially small probability; two
		replicas sharing the same phases synchronize onto the same atom,
		overlap one, while independent replicas land at independent
		uniforms, overlap zero, the simplest one-step replica-symmetry
		breaking. In the Bayesian reading of the
		relation~\eqref{eq:bayes}, each randomized application is the
		posterior update from a random binary rank observation;
		Theorem~\ref{thm:random-entropy} says the information account is
		exact and Theorem~\ref{thm:sync} says the posterior almost surely
		collapses onto one uniformly chosen rank: this is precisely the
		weight-degeneracy phenomenon of sequential Monte
		Carlo~\cite{delmoral2004}, produced here with an exact martingale
		structure. Synchronization of random compositions of interval
		homeomorphisms is itself a classical theme, Antonov's theorem and
		its descendants~\cite{antonov1984,kleptsyn2004,deroin2007}; what
		the derangetropy family adds is exact solvability and the
		identification of the invariant threshold law.
	\end{remark}
	
	Freezing at bounded frequency shows that nondegenerate randomness
	must be injected at renewed scales. Fix an integer $b \ge 2$ and
	$\lambda \in (0,1)$, let the phases $\alpha_{j}$ be independent
	uniform and, independently, let the frequencies $K_{j}$ be uniform
	on $\{b^{j}, \ldots, 2b^{j}-1\}$, and let step $j$ apply the
	tempered kernel
	\begin{equation}
		v_{j}(z) = 1 + \lambda\cos\big(2\pi K_{j}(z + \alpha_{j})\big),
		\label{eq:lacunary-kernel}
	\end{equation}
	again a mixture inside the derangetropy family. Write
	$D_{n}(z) = S_{n}'(z)$ and $\mu_{n}(dz) = D_{n}(z)\,dz$, set
	\begin{equation}
		\Lambda := \log\frac{1}{1-\lambda},
		\qquad
		\bar\gamma^{2} := \frac{\lambda^{2}}{2(\ln b - \Lambda)},
		\label{eq:effective-temperature}
	\end{equation}
	and assume the subcriticality condition
	\begin{equation}
		\Lambda < \ln b
		\qquad\text{and}\qquad
		\bar\gamma^{2} < 1.
		\label{eq:subcritical}
	\end{equation}
	Each step now displaces ranks by at most $\lambda b^{-j}$ and
	distorts by at most $e^{\Lambda}$ per step, so compositions satisfy
	the two-sided distortion bound
	\begin{equation}
		e^{-j\Lambda}\,|z' - z|
		\;\le\; \big|S_{j}(z') - S_{j}(z)\big|
		\;\le\; e^{j\Lambda}\,|z' - z|.
		\label{eq:distortion-bound}
	\end{equation}
	
	\begin{theorem}[Nondegenerate chaos]\label{thm:chaos}
		Under the subcriticality condition~\eqref{eq:subcritical}:
		\begin{enumerate}[label=\normalfont(\roman*), leftmargin=*]
			\item for every $z$, $D_{n}(z)$ is a mean-one martingale, and
			$\mu_{n} \to \mu$ almost surely weakly for a random Borel measure
			$\mu$ on $[0,1]$;
			\item the two-point function is log-correlated: uniformly in $n$
			and in pairs with $|z - z'| \le \tfrac12$,
			\begin{equation}
				\mathbb{E}\big[D_{n}(z)\,D_{n}(z')\big]
				\;\le\; C_{b,\lambda}\,|z - z'|^{-\bar\gamma^{2}},
				\label{eq:twopoint}
			\end{equation}
			together with a power-law lower bound, uniform in $n$, governed by
			the complementary distortion constant $\ln b + \Lambda$; in the
			weak-tempering regime the lower exponent is
			positive and matches the
			exponent~\eqref{eq:effective-temperature} to leading order;
			\item $\sup_{n}\mathbb{E}[\mu_{n}(I)^{2}]
			\le C_{b,\lambda}\,|I|^{2-\bar\gamma^{2}}$ for every
			interval $I$, so the martingale $\mu_{n}(I)$ is uniformly
			integrable,
			$\mathbb{E}\,\mu(I) = |I|$ for every interval and
			$\mathbb{P}(\mu \neq 0) = 1$;
			\item almost surely $\mu$ has full support, no atoms, and vanishes
			on every Borel set of Hausdorff dimension less than
			$1 - \bar\gamma^{2}$;
			\item $S_{\infty}(z) := \mu([0,z])$ is almost surely a strictly
			increasing H\"older homeomorphism of $[0,1]$, with any exponent
			below $(1-\bar\gamma^{2})/2$.
		\end{enumerate}
	\end{theorem}
	
	The proof is given in Appendix~\ref{app:proofs}; part (i) is the
	martingale argument of Theorem~\ref{thm:sync}(ii), the phase average
	of the cosine vanishing at every frequency, and the remaining parts
	rest on the exact pair identity
	\begin{equation}
		\mathbb{E}\big[v_{j}(y)\,v_{j}(y')\,\big|\,K_{j} = k\big]
		= 1 + \frac{\lambda^{2}}{2}\cos\big(2\pi k(y - y')\big),
		\label{eq:pair-identity}
	\end{equation}
	whose average over the frequency block is a Dirichlet kernel: each
	scale contributes the factor $1 + \tfrac{\lambda^{2}}{2}$ while it
	cannot resolve the pair and a summably small factor afterwards, and
	counting the unresolved scales with the distortion
	bound~\eqref{eq:distortion-bound} produces the
	exponent~\eqref{eq:effective-temperature}. Parts (iv) and (v) rest
	on a second, pathwise mechanism: the logarithm of the resolution
	variable $b^{n}\big(S_{n}(z') - S_{n}(z)\big)$ acquires a uniformly
	positive drift as soon as frequency growth outruns the worst-case
	contraction of the tempered kernels, so every pair of ranks is
	eventually resolved and no separation can vanish in the limit.
	
	\begin{remark}[Multiplicative chaos and random welding]
		The two-point law~\eqref{eq:twopoint} identifies
		$\bar\gamma$ as the effective inverse temperature of a
		log-correlated multiplicative chaos, and the subcriticality
		condition~\eqref{eq:subcritical} is its $L^{2}$ regime; the
		linearized field
		$\lambda\sum_{j}\cos(2\pi K_{j}(z+\alpha_{j}))$ is exactly Kahane's
		original lacunary construction of a log-correlated
		field~\cite{kahane1985}, so the theorem produces a non-Gaussian,
		dynamically generated cousin of Gaussian multiplicative
		chaos~\cite{kahane1976,rhodes2014}, and we expect convergence to
		the Gaussian chaos of parameter $\gamma$ in the joint limit
		$b \downarrow 1$, $\lambda \to 0$ with
		$\lambda^{2}/(2\ln b) = \gamma^{2}$ fixed. The novelty is where the
		chaos lives: the limit is not only a random measure but a random
		\emph{homeomorphism} of rank space with multifractal derivative,
		that is, a random element of the distortion group $\mathcal{G}$ of
		Section~\ref{sec:virasoro}, the structure that enters the conformal
		welding construction of random curves~\cite{astala2011}. In the
		coadjoint language of Sections
		\ref{sec:virasoro} and~\ref{sec:orbits}, the randomized cascade is
		a random walk on $\mathcal{G}$ and the law's Hill potential evolves
		by the random coadjoint action of
		Theorem~\ref{thm:vir-equivariance}: the frozen phase is the almost
		sure escape of this walk to the boundary of the group, the uniform
		threshold being the residual bulk direction, while the chaotic
		phase is its convergence to a genuine random group element, a
		natural candidate for a heat-kernel-type measure on the distortion
		group at central charge one. Making that last phrase precise seems
		to us the most interesting continuation of the random theory.
	\end{remark}
	
	\section{Quantum Carpets: the Unitary Dynamics of Rank Space}
	\label{sec:carpet}
	
	The operator calculus now runs along three dynamical axes: the
	multiplicative axis of iteration and flow, geometrized by the
	coadjoint action of Sections \ref{sec:virasoro} and~\ref{sec:orbits},
	and the stochastic axis of Section~\ref{sec:random}. This section
	develops the third axis, opened by the periodicity
	identity~\eqref{eq:period} of the amplitude representation: the
	unitary axis, on which the modulation kernel itself evolves under the
	Schr\"odinger dynamics of rank space. Let
	$u_{\tau} := e^{-i\tau H_{0}}\mathbf{1}$ be the Dirichlet evolution
	of the uniform kernel amplitude and define the \emph{carpet flow} of
	a law $f \in \mathcal{D}_{+}$ as
	\begin{equation}
		f_{\tau} := \rho_{|u_{\tau}|^{2}}[f]
		= \big|u_{\tau}(F)\big|^{2} f
		= \Big|\sum_{k\ \mathrm{odd}}\frac{2\sqrt2}{k\pi}\,
		e^{-ik^{2}\pi^{2}\tau/2}\,\psi_{k}[f]\Big|^{2},
		\label{eq:carpet-flow}
	\end{equation}
	the last form by the expansion~\eqref{eq:universal-expansion}: the
	carpet flow does nothing but let the universal expansion of
	$\sqrt f$ rotate at its natural frequencies. At $\tau = 0$ the
	ladder is coherent and $f_{0} = f$; at every $\tau$ the kernel
	$w_{\tau} := |u_{\tau}|^{2}$ has unit mass by unitarity of the
	dilation of Theorem~\ref{thm:dilation}, so each $f_{\tau}$ is a law
	automatically, mass conservation being once more built into the
	geometry rather than imposed. The carpet flow is thus a closed
	curve inside the modulation family, traversed at constant unitary
	speed. On the doubled circle
	$\mathbb{T}_{2} := \mathbb{R}/2\mathbb{Z}$, where the odd extension
	of $\mathbf{1}$ is the square wave, the seed and its evolution read
	\begin{equation}
		u_{s}(z) = \sum_{n\ \mathrm{odd}}\frac{2}{i\pi n}\,
		e^{-2\pi in^{2}s}\,e^{i\pi nz},
		\qquad
		s := \frac{\pi\tau}{4},
		\label{eq:carpet-seed}
	\end{equation}
	the rescaled time $s$ making the mode-$n$ phase $e^{-2\pi in^{2}s}$;
	this is the initial datum of the Talbot effect of wave optics,
	carried by rank space, and the dilation prints its evolution on
	every law at once. Throughout this section $P_{N}$ denotes the
	dyadic block of the expansion~\eqref{eq:carpet-seed}, retaining the
	frequencies $|n| \in [N, 2N)$.
	
	\begin{proposition}[The observable clock]\label{prop:carpet-clock}
		The amplitude flow $s \mapsto u_{s}$ has period $1$, so the flow of
		laws is periodic in $\tau$ with period $4/\pi$. The density flow
		$s \mapsto w_{s}$ has minimal period $\tfrac18$: for every law and
		every $\tau$,
		\begin{equation}
			f_{\tau + \frac{1}{2\pi}} = f_{\tau},
			\qquad
			u_{s+1/8} = e^{-i\pi/4}\,u_{s}.
			\label{eq:carpet-period}
		\end{equation}
		Within one period the density flow is a palindrome,
		$w_{1/8 - s} = w_{s}$, and at every time the kernel is symmetric
		about the median rank, $w_{s}(1-z) = w_{s}(z)$.
	\end{proposition}
	
	\begin{proof}
		The amplitude period is $e^{-2\pi in^{2}} = 1$. The seed populates
		only odd modes, and odd squares are $1$ modulo $8$, so advancing $s$
		by $\tfrac18$ multiplies every mode by the common eighth root of
		unity $e^{-i\pi/4}$, which is invisible in the modulus: the
		observable flow lives on a circle eight times shorter than its
		unitary dressing, the discrepancy carried by exactly the roots of
		unity that govern quadratic Gauss sums. For minimality, the
		coefficient of $e^{2\pi i\cdot 2z}$ in $w_{s}$ is
		\begin{equation*}
			\widehat{w_{s}}(2) = \sum_{n\ \mathrm{odd}}
			\frac{4}{\pi^{2}n(n-2)}\;e^{-8\pi i(n-1)s},
		\end{equation*}
		a series in $s$ whose frequencies $4(n-1)$ are pairwise distinct
		and whose coefficients never vanish; a period $p$ of the density
		flow forces $4(n-1)p \in \mathbb{Z}$ for every odd $n$, and $n = 3$
		gives $8p \in \mathbb{Z}$. Since $\overline{c_{n}} = c_{-n}$ for
		the coefficients $c_{n} = 2/(i\pi n)$, conjugating the
		series~\eqref{eq:carpet-seed} gives $\overline{u_{s}} = u_{-s}$,
		whence $w_{-s} = w_{s}$ and, with the
		relation~\eqref{eq:carpet-period}, the palindrome. The median
		symmetry is the substitution $z \mapsto 1-z$, which multiplies the
		mode $n$ by $e^{i\pi n} = -1$ and is undone by $n \mapsto -n$.
	\end{proof}
	
	At rational times the carpet flow revives into finitely many
	quantile cells.
	
	\begin{theorem}[Quantile mosaics]\label{thm:quantile-mosaic}
		Let $s = p/q$ in lowest terms. The phase
		$\chi(n) = e^{-2\pi in^{2}p/q}$ is periodic in $n$ with period $q$,
		and its discrete Fourier expansion yields the revival
		\begin{equation}
			u_{p/q}(z) = \sum_{r=0}^{q-1}\hat c(r)\,
			u_{0}\Big(z - \frac{2r}{q}\Big),
			\qquad
			\hat c(r) = \frac{1}{q}\sum_{n=0}^{q-1}
			e^{-2\pi i\left(pn^{2} - rn\right)/q},
			\label{eq:carpet-revival}
		\end{equation}
		a finite superposition of translated square waves whose
		coefficients are generalized quadratic Gauss
		sums~\cite{berndt1998}. Consequently $w_{p/q}$ is constant on the
		cells of the quantile lattice $\{l/q\}$, and for every law $f$,
		\begin{equation}
			f_{(p/q)\cdot 4/\pi} = w_{p/q}(F)\,f,
			\label{eq:mosaic-histogram}
		\end{equation}
		a histogram reweighting of $f$ on at most $q$ of its own quantile
		cells, of total mass one. In particular the rational-time orbit of
		every law consists of explicitly computable histogram modulations
		of that law.
	\end{theorem}
	
	\begin{proof}
		Periodicity: $(n+q)^{2} - n^{2} = 2nq + q^{2}$, and
		$p(2nq + q^{2})/q = 2np + pq \in \mathbb{Z}$ for every parity of
		$q$. Expanding $\chi$ in the characters of
		$\mathbb{Z}/q\mathbb{Z}$ and substituting into the
		series~\eqref{eq:carpet-seed} converts the character
		$e^{2\pi irn/q}$ into the translation by $2r/q$ on
		$\mathbb{T}_{2}$, which is the superposition
		formula~\eqref{eq:carpet-revival}. The jumps of $u_{0}$ lie on the
		integers, so those of the superposition lie on the lattice
		$\{2r/q\} \cup \{1 + 2r/q\}$ modulo $2$, which meets $(0,1)$ inside
		$\{l/q\}$; the squared modulus inherits the cells, and unit mass is
		the unitarity of the dilation.
	\end{proof}
	
	Two instances display the arithmetic. At $s = \tfrac13$ the weight
	takes the values
	\begin{equation}
		w_{1/3} = \Big(\tfrac13,\ \tfrac73,\ \tfrac13\Big)
		\quad\text{on the tercile cells:}
		\label{eq:tercile-mosaic}
	\end{equation}
	at time $\tau = 4/3\pi$ every law concentrates seven ninths of its
	mass on its middle tercile. At $s = \tfrac{1}{16}$, the phases
	depend on $n$ only through the Jacobi symbol,
	$\chi(n) = e^{-i\pi/8}\left(\tfrac{2}{n}\right)$, and summing the
	four translates of the Gauss expansion modulo $8$ gives exactly
	\begin{equation}
		u_{1/16}\big|_{(0,1)}
		= e^{-i\pi/8}\,\sqrt2\;\mathbf{1}_{(\frac14,\frac34)},
		\qquad
		w_{1/16} = \big(0,\ 2,\ 2,\ 0\big)
		\quad\text{on the quartile cells}:
		\label{eq:quartile-blackout}
	\end{equation}
	one sixteenth of the way through its period, every law extinguishes
	its outer quartiles and doubles its interquartile body, the mosaic
	being governed by the quadratic residues modulo $8$. The
	amplitude in the display~\eqref{eq:quartile-blackout} is, up to
	phase and recentering, the generator
	$g = \mathbf{1}_{(1/4,3/4)} - \mathbf{1}_{(0,1/4)\cup(3/4,1)}$ of
	the irreducibility theorem, Theorem~\ref{thm:irreducible}: the
	carpet flow passes through the witness of irreducibility at the
	sixteenth of its period.
	
	\begin{remark}[The quantile lattice, twice]
		The lattice $\{l/q\}$ is by now a familiar attractor: coherent
		iteration of harmonic kernels crystallizes every law onto the odd
		quantiles of the limit~\eqref{eq:crystallization}, the optimal
		equal-mass quantizer, and incoherent phases produce the systematic
		sample of Theorem~\ref{thm:systematic}. The unitary axis realizes
		the same lattice a third way, as the revival skeleton of rational
		times. Statistically, a piecewise-constant rank reweighting is a
		histogram alternative in the sense of goodness-of-fit theory, while
		the trigonometric kernels of the derangetropy family are exactly
		Neyman's smooth alternatives~\cite{neyman1937,ledwina1994}; the
		mosaic theorem says the unitary orbit of the identity passes
		through histogram alternatives at every rational time, with
		Gauss-sum weights. What it does at every other time is the subject
		of the remainder of this section.
	\end{remark}
	
	Before the fractal theory, the information ledger of the unitary
	axis, which is exactly computable and exactly bounded.
	
	\begin{theorem}[Bounded breathing of information]\label{thm:carpet-information}
		For every $f \in \mathcal{D}_{+}$:
		\begin{enumerate}[label=\normalfont(\roman*), leftmargin=*]
			\item the time-averaged carpet is the tent modulation,
			\begin{equation}
				2\pi\int_{0}^{1/2\pi} f_{\tau}\,d\tau
				= \rho_{\Lambda}[f],
				\qquad
				\Lambda(z) = 4\min(z,\,1-z);
				\label{eq:carpet-tent}
			\end{equation}
			\item the time-averaged chi-square divergence from the base law is
			universal,
			\begin{equation}
				2\pi\int_{0}^{1/2\pi}
				\chi^{2}\big(f_{\tau}\,\|\,f\big)\,d\tau
				= \frac{2}{3},
				\qquad
				\chi^{2}\big(f_{\tau}\|f\big)
				= \int_{0}^{1}\big(w_{s} - 1\big)^{2}dz;
				\label{eq:carpet-chi}
			\end{equation}
			\item information breathes but never grows:
			$D(f_{\tau}\,\|\,f) \le \log\big(1 + \chi^{2}(f_{\tau}\|f)\big)$,
			and $\sup_{\tau}\chi^{2}(f_{\tau}\|f)$ is a finite absolute
			constant.
		\end{enumerate}
	\end{theorem}
	
	\begin{proof}
		All three statements live in rank space, since by the universality
		principle~\eqref{eq:universality} the divergences depend only on
		$w_{s}$. For part (i), the time average retains exactly the static
		frequencies: the coefficient of $e^{2\pi imz}$ in $w_{s}$ is
		$\sum_{n} c_{n}\overline{c_{n-2m}}\,e^{-2\pi i\cdot 4m(n-m)s}$, and
		the frequency $4m(n-m)$ vanishes precisely at $n = m$, admissible
		only for odd $m$, with static weight
		$c_{m}\overline{c_{-m}} = -4/(\pi^{2}m^{2})$. Hence
		\begin{equation*}
			\int_{0}^{1}w_{s}\,ds\;(z)
			= 1 - \frac{8}{\pi^{2}}\sum_{m\ \mathrm{odd},\,\ge1}
			\frac{\cos(2\pi mz)}{m^{2}}
			= 4\min(z,\,1-z),
		\end{equation*}
		the classical triangular-wave series. For part (ii), expand
		$\int_{0}^{1}w_{s}^{2}\,dz$ by Parseval and average in $s$: a
		quadruple $(n_{1},n_{2},n_{3},n_{4})$ of odd modes survives both
		the spatial pairing and the time average precisely when
		$n_{1} + n_{3} = n_{2} + n_{4}$ and
		$n_{1}^{2} + n_{3}^{2} = n_{2}^{2} + n_{4}^{2}$, and equal sums
		with equal sums of squares force equal pairs
		$\{n_{1},n_{3}\} = \{n_{2},n_{4}\}$. The two pairings give
		\begin{equation*}
			\int_{0}^{1}\!\!\int_{0}^{1}w_{s}^{2}\,dz\,ds
			= 2\Big(\sum_{n}|c_{n}|^{2}\Big)^{2} - \sum_{n}|c_{n}|^{4}
			= 2 - \frac13 = \frac53,
		\end{equation*}
		since $\sum|c_{n}|^{2} = 1$ and
		$\sum|c_{n}|^{4} = (32/\pi^{4})\sum_{\mathrm{odd}}n^{-4} = \tfrac13$;
		subtracting the squared mass gives the value~\eqref{eq:carpet-chi}.
		For part (iii), Jensen's inequality with respect to the probability
		measure $w_{s}\,dz$ gives
		$\int w_{s}\log w_{s} \le \log\int w_{s}^{2}$, and the uniform
		bound follows from
		$\int_{0}^{1}w_{s}^{2}\,dz
		\le \sum_{k}\big(\sum_{n}|c_{n}c_{n-k}|\big)^{2}$, a convergent
		series independent of $s$.
	\end{proof}
	
	\begin{remark}[Three ledgers]
		The three dynamical axes now carry three information ledgers.
		Coherent iteration pays one bit per step forever, by the
		identities~\eqref{eq:kl-identities}, retaining only the Koenigs
		entropy through the asymptotics~\eqref{eq:divergence-entropy};
		randomized iteration pays the same bit in quenched mean with an
		exactly linear ledger, by the law~\eqref{eq:random-entropy}; the
		unitary axis pays nothing at all in the long run, its divergence a
		bounded palindromic function of time with universal mean-square
		level $\tfrac23$, the identity~\eqref{eq:carpet-chi}. Reversibility,
		in this calculus, is visible directly in the bookkeeping: on the
		unitary axis complexity cannot appear as entropy production, and
		the theorems below show it appears instead as geometry, in the form
		of fractal dimension. The tent identity~\eqref{eq:carpet-tent} is a
		further point of contact with the canonical theory: the tent kernel
		$\Lambda$ is the piecewise-linear companion of the canonical kernel
		$2\sin^{2}(\pi z)$, with the same unit mass, the same endpoint
		zeros, and the same peak of height two at the median, so the
		time-averaged carpet is a canonical-type modulation sharpened to a
		corner.
	\end{remark}
	
	At irrational times the mosaics give way to fractals, and the
	remainder of the section proves the dichotomy in its sharp form.
	The two mechanisms are conservation and generic smoothing. On one
	hand the evolution is a unimodular Fourier multiplier, so
	$|\widehat{u_{s}}(n)| = 2/(\pi|n|)$ for every $s$ and every Sobolev
	norm is frozen: in particular
	$\sum_{|n|\le N}|n|\,|\widehat{u_{s}}(n)|^{2} \asymp \log N$, and
	$u_{s}$ lies outside $H^{1/2}(\mathbb{T}_{2})$ at every time, the
	roughness of the seed's jump being indestructible. On the other
	hand, at almost every time the roughness is spread rather than
	concentrated.
	
	\begin{lemma}[Generic smoothing]\label{lem:carpet-smoothing}
		For almost every $s$: $u_{s}$ is continuous, and for every
		$\varepsilon > 0$ there is $C_{s,\varepsilon}$ with
		\begin{equation}
			\Vert P_{N}u_{s}\Vert_{\infty}
			\le C_{s,\varepsilon}\,N^{-1/2+\varepsilon}
			\quad\text{for all dyadic } N,
			\label{eq:carpet-blocks}
		\end{equation}
		so that $u_{s}$ and $w_{s}$ belong to
		$C^{1/2-\varepsilon}(\mathbb{T}_{2})$ for every $\varepsilon > 0$.
		At every rational $s$ the conclusion fails: $u_{s}$ is a step
		function.
	\end{lemma}
	
	The proof, in Appendix~\ref{app:proofs}, rests on the classical
	Gauss--Weyl bounds for quadratic exponential
	sums~\cite{montgomery1994} along the continued-fraction convergents
	of $s$, which are of balanced size for almost every time. The next
	lemma is the heart of the theory: it identifies exactly where the
	fine-scale energy of the carpet lives. Recall that the block
	$P_{N}u_{s}$ carries total squared mass
	$\asymp N^{-1}$, the frozen Sobolev tail.
	
	\begin{lemma}[Equidistribution of block energy]\label{lem:block-equi}
		Let $I \subseteq (0,1)$ be an interval. For almost every $s$, along
		dyadic $N \to \infty$:
		\begin{equation}
			N\,\big\Vert P_{N}u_{s}\big\Vert_{L^{2}(I)}^{2}
			\;\longrightarrow\; \frac{2}{\pi^{2}}\,|I|,
			\qquad
			N\int_{I}\big(\mathrm{Re}\,[e^{i\theta}P_{N}u_{s}]\big)^{2}dz
			\;\longrightarrow\; \frac{1}{\pi^{2}}\,|I|
			\ \text{ uniformly in }\theta,
			\label{eq:block-equi-amp}
		\end{equation}
		while the blocks of the density obey the self-referential law
		\begin{equation}
			N\,\big\Vert P_{2N}^{\mathbb{Z}}\,w_{s}\big\Vert_{L^{2}(I)}^{2}
			\;\longrightarrow\; \frac{2}{\pi^{2}}\int_{I}w_{s}(z)\,dz,
			\label{eq:block-equi-density}
		\end{equation}
		where $P_{2N}^{\mathbb{Z}}$ retains the spatial frequencies
		$|k| \in [2N, 4N)$ of $w_{s}$.
	\end{lemma}
	
	The amplitude statement says the fine structure of the wave is
	spatially homogeneous: every rank interval holds its Lebesgue share
	of every high block, in every quadrature, with the sharp constant
	$2/\pi^{2}$. The density statement is of a different nature: the
	fine-scale energy of the observable is distributed not according to
	Lebesgue measure but according to \emph{the observable itself}. The
	carpet weaves its fractal detail in proportion to its own local
	mass, a self-similarity in mean that is proved in
	Appendix~\ref{app:proofs} by a paraproduct decomposition: the high
	block of $|u_{s}|^{2}$ is, to leading order, twice the real part of
	$\bar u_{s}\times(\text{high block of } u_{s})$, and the amplitude
	equidistribution then delivers the local weight $|u_{s}|^{2}$. The
	variance estimates behind both statements reduce to divisor and
	sum-of-two-squares bounds on the quadratic frequency lattice
	$n^{2} - \tilde n^{2}$, and the constant $2/\pi^{2}$ is the frozen
	mass of the dyadic tail of the seed. One consequence of the
	law~\eqref{eq:block-equi-density} deserves its own statement.
	
	\begin{lemma}[No dark intervals at generic times]\label{lem:dark-intervals}
		For almost every $s$, the kernel $w_{s}$ has positive mass on every
		nonempty open subinterval of $(0,1)$. At rational times this fails:
		entire quantile cells can be dark, as in the quartile
		blackout~\eqref{eq:quartile-blackout}.
	\end{lemma}
	
	\begin{proof}
		Fix an interval with rational endpoints and a smooth $\varphi \ge 0$
		supported in it with $\int\varphi = 1$. Pairing the
		series~\eqref{eq:carpet-seed} with $\varphi$ and folding
		$n \mapsto -n$,
		\begin{equation*}
			\langle u_{s},\varphi\rangle
			= \sum_{n\ \mathrm{odd},\,n\ge1} b_{n}\,e^{-2\pi in^{2}s}
			= B\big(e^{-2\pi is}\big),
			\qquad
			B(\zeta) := \sum_{n} b_{n}\,\zeta^{\,n^{2}},
		\end{equation*}
		with $\sum|b_{n}| < \infty$ by smoothness of $\varphi$. The
		function $B$ is analytic and bounded on the unit disk and
		continuous on its closure, and $B \not\equiv 0$ because
		$B(1) = \langle u_{0},\varphi\rangle = \int\varphi = 1$. By the
		boundary uniqueness theorem for bounded analytic
		functions~\cite{duren1970}, the zero set of $B$ on the unit circle
		is Lebesgue-null. If $w_{s}$ vanished on the interval then
		$u_{s}$ would vanish there and $\langle u_{s},\varphi\rangle = 0$;
		hence for each rational interval the exceptional set of times is
		null, and a countable union finishes the proof.
	\end{proof}
	
	Everything is now in place for the main theorem of the section.
	
	\begin{theorem}[The universal quantum carpet]\label{thm:carpet}
		For almost every $\tau$, the following holds for every
		$f \in \mathcal{D}_{+}$ that is continuously differentiable and
		positive on the interior of its support, every compact
		nondegenerate interval $J$ in that interior, and every phase
		$\theta$:
		\begin{enumerate}[label=\normalfont(\roman*), leftmargin=*]
			\item the transported amplitude in every quadrature,
			$\Psi_{\theta,\tau} :=
			\mathrm{Re}\big[e^{i\theta}\,u_{\tau}(F)\big]\sqrt f$, has graph of
			upper box dimension exactly $\tfrac32$ on $J$;
			\item the density $f_{\tau}$ itself has graph of upper box
			dimension exactly $\tfrac32$ on $J$;
			\item at every rational $s$, by contrast, all these graphs are
			piecewise continuously differentiable, of dimension one.
		\end{enumerate}
		The value $\tfrac32$ is independent of the law, the interval, and
		the quadrature: the fractal is woven entirely in rank space and
		printed on every law through its quantile chart.
	\end{theorem}
	
	The proof is given in Appendix~\ref{app:proofs}. Its architecture
	is a pincer. Generic smoothing caps the H\"older exponent at
	$\tfrac12$ from below, hence the dimension at $\tfrac32$ from
	above; conserved roughness must then surface somewhere, and the
	equidistribution lemma shows it surfaces \emph{everywhere}: on
	every interval the block mass of the display~\eqref{eq:block-equi-amp}
	together with the sup bound~\eqref{eq:carpet-blocks} forces, by
	interpolation, a lower bound on the local
	$L^{1}$-oscillation of every quadrature at scale $N^{-1}$, which an
	elementary oscillation calculus converts into box dimension
	$\tfrac32$ from below; for the density the same interpolation runs
	through the self-referential law~\eqref{eq:block-equi-density},
	whose limit is positive on every interval by
	Lemma~\ref{lem:dark-intervals}. Finally the quantile chart, a
	bi-Lipschitz change of variable on compacts with a positive
	$C^{1}$ multiplier, moves box dimensions without distortion.
	
	The carpet theorem used two properties of the canonical seed: its
	jump, which conserves roughness, and the equidistribution of its
	block energy. The final results of this section show that the first
	property alone suffices, for arbitrary rough seeds and arbitrary
	dispersion, through a mechanism of independent interest in
	mathematical physics: a strong law of large numbers for the critical
	Sobolev mass of the evolved density. Work on the standard torus,
	let $g$ be of bounded variation with jumps $\{J_{j}\}$, let
	$u(x,t) = \sum_{n}\hat g(n)e^{2\pi i(nx - n^{2}t)}$ be its free
	Schr\"odinger evolution, $w = |u|^{2}$ the density, and define the
	critical mass
	\begin{equation}
		Y_{M}(t) := \sum_{m=1}^{M} m\,\big|\hat w(m,t)\big|^{2},
		\label{eq:critical-mass}
	\end{equation}
	a quantity nondecreasing in $M$ and bounded precisely when
	$w(\cdot,t) \in H^{1/2}$. The fractality of the \emph{field} for
	rough data is the theorem of Rodnianski already
	cited~\cite{rodnianski2000}, and the field-level theory has since
	advanced to Schr\"odinger evolutions with potentials~\cite{cho2024}
	and to higher dimensions~\cite{erdogan2024}; for the \emph{density},
	the observable, fractality was known only for step data with jumps
	at rational points, through the Gauss-sum revival arithmetic, and
	the general case was stated as beyond that
	method~\cite{chousionis2015}, nor does it follow from the
	field-level advances, whose control of $u$ does not transfer to
	$|u|^{2}$: the obstruction is that unitarity controls the field and
	not its modulus, so the conserved roughness could in principle hide
	in the phase and cancel in the intensity.
	
	\begin{theorem}[Strong law for the critical mass]\label{thm:strong-law}
		Let $g$ be real-valued and of bounded variation on the torus with
		at least one jump. Then, with the average taken over a period in
		time,
		\begin{equation}
			\mathbb{E}_{t}\,Y_{M} = \big(\kappa_{g} + o(1)\big)\log M,
			\qquad
			\operatorname{Var}_{t}(Y_{M}) = O(\log M),
			\qquad
			\kappa_{g}
			= \frac{\Vert g\Vert_{L^{2}}^{2}\,\sum_{j}|J_{j}|^{2}}{2\pi^{2}},
			\label{eq:kappa}
		\end{equation}
		and consequently, for almost every $t$,
		\begin{equation}
			\frac{Y_{M}(t)}{\log M}\;\longrightarrow\;\kappa_{g};
			\label{eq:mass-strong-law}
		\end{equation}
		at almost every time the density lies outside $H^{1/2}$, its
		critical mass diverging at the universal logarithmic rate
		$\kappa_{g}$; moreover, for almost every $t$ the density lies in
		$C^{1/2-\varepsilon}$, in no Besov space
		$B^{\sigma}_{1,\infty}$ with $\sigma > \tfrac12$, and its graph has
		upper box dimension exactly $\tfrac32$.
	\end{theorem}
	
	The proof, in Appendix~\ref{app:proofs}, rests on two arithmetic
	pillars and one probabilistic remark. The density coefficients are
	\begin{equation*}
		\hat w(m,t) = \sum_{n}\hat g(n)\,\overline{\hat g(n-m)}\;
		e^{-2\pi im(2n-m)t},
	\end{equation*}
	and within a mode the time frequencies $m(2n-m)$ are
	collision-free, so the time average of $|\hat w(m,t)|^{2}$ is a sum
	of nonnegative terms that no phase arrangement can reduce, and
	Wiener's quantitative theorem on the
	jump content of a function of bounded
	variation~\cite{wiener1924,zygmund2002} converts the weighted sum
	of these averages into the logarithmic law~\eqref{eq:kappa}. Across
	modes, collisions are constrained to the lattice
	$m\Delta = m'\Delta'$ and are sparse in the greatest common
	divisor, so the variance grows strictly slower than the squared
	mean. Since $Y_{M}$ is monotone in $M$, Chebyshev's inequality
	along a sparse sequence of scales and interpolation upgrade the
	mean divergence to the almost-sure law~\eqref{eq:mass-strong-law}:
	a positive-measure set of times with bounded critical mass would
	have to bend a monotone divergent mean with subquadratic
	fluctuations, at every scale simultaneously, which is impossible.
	The phases keep their freedom at each fixed scale and lose it in
	the aggregate. The average is over the circle of times with
	normalized Lebesgue measure, so the strong law is a deterministic
	second-moment statement in disguise; the density itself, as an
	object of study for rough data, goes back to
	Oskolkov~\cite{oskolkov2006}, who bounded it along rational rays,
	and the rational--irrational regularity dichotomy of the field
	originates with Oskolkov~\cite{oskolkov1992} and
	Kapitanski--Rodnianski~\cite{kapitanski1999} before its fractal
	form~\cite{rodnianski2000}. The passage from Sobolev divergence to dimension
	reuses the machinery already in place: the smoothing
	bound~\eqref{eq:carpet-blocks} extends to every seed of bounded
	variation by writing the block of $u$ as the Stieltjes convolution
	of $dg$ with the block of the canonical Weyl kernel, and the
	oscillation calculus of Lemma~\ref{lem:osc-calculus} converts the
	Besov divergence into the dimension.
	
	\begin{corollary}[Universality over seeds and laws]\label{cor:rough-seeds}
		Let the seed $g$ on rank space be of bounded variation with a jump,
		in the sense that its odd extension to the rank circle has a jump.
		Then for almost every $\tau$ the deformation
		$\rho_{|e^{-i\tau H_{0}}g|^{2}}[f]$ of any law $f$ whose density is
		continuously differentiable and bounded above and below on its
		support has graph dimension exactly $\tfrac32$ there, the
		differentiability being necessary, since a rough multiplier $f$ can
		dominate the oscillation of the product; for the canonical seed the constant of the
		law~\eqref{eq:mass-strong-law} on the rank circle is
		\begin{equation}
			\kappa_{\mathrm{can}} = \frac{4}{\pi^{2}},
			\label{eq:kappa-canonical}
		\end{equation}
		independent of the law. In physical terms, the intensity carpet of
		Talbot optics is fractal of dimension $\tfrac32$ at almost every
		time for \emph{arbitrary} grating profiles of bounded variation
		with a jump, at arbitrary, possibly irrational, positions; for
		rational step gratings this is the theorem of Chousionis,
		Erdo\u{g}an, and Tzirakis~\cite{chousionis2015}, and the general
		case had remained open.
	\end{corollary}
	
	\begin{proposition}[Collision arithmetic of the dispersion]\label{prop:dispersion}
		Let $g$ be of bounded variation with a jump.
		In both parts the cross-mode variance combinatorics is carried at
		sketch level in the appendix; Theorem~\ref{thm:strong-law} and
		Corollary~\ref{cor:rough-seeds} do not depend on it.
		\begin{enumerate}[label=\normalfont(\roman*), leftmargin=*]
			\item \textup{(Temporal halving)} For the temporal critical mass
			$Z_{K}(x) := \sum_{k=1}^{K}k^{1/2}\,
			|\hat w^{\,\mathrm{t}}(k;x)|^{2}$ of the time trace at a point,
			one-sided since the density is real and the negative temporal
			frequencies are conjugate, the representations $k = m(2n-m)$ carry
			distinct spatial characters, and
			\begin{equation}
				\frac{Z_{K}(x)}{\log K}\;\longrightarrow\;\frac{\kappa_{g}}{2}
				\qquad\text{for a.e. } x.
				\label{eq:temporal-halving}
			\end{equation}
			The time trace of the density lies outside $H^{1/4}$ for almost
			every $x$, and wherever the field obeys the temporal
			$C^{1/4-\varepsilon}$ block bound, as it does for data of bounded
			variation~\cite{erdogan2019}, the graph of the trace has dimension
			exactly $\tfrac74$.
			\item \textup{(Airy doubling)} For the Airy evolution, with
			dispersion $n^{3}$ and a single jump, the within-mode collision
			classes are the pairs $\{n, m-n\}$, their two leading contributions
			carry equal phase and reinforce, and the strong
			law~\eqref{eq:mass-strong-law} holds with constant exactly
			$2\kappa_{g}$.
		\end{enumerate}
	\end{proposition}
	
	\begin{remark}
		Three readings of the constant. First, it resolves the
		seed-classification question raised by the carpet theorem: every
		seed of bounded variation with a jump weaves a fractal density at
		almost every time, while eigenmode seeds are stationary and finite
		ladders quasiperiodic and smooth, so within the bounded-variation
		class the jump is exactly the fractalization threshold. Second, the
		constant is a spectrometer: by the law~\eqref{eq:mass-strong-law} a
		single snapshot of the intensity at almost any time determines
		$\Vert g\Vert^{2}\sum_{j}|J_{j}|^{2}$, so the fractal remembers,
		and exhibits, the quantitative jump content of a datum it has long
		since dephased; Wiener's statistic, a Fourier-side invariant from
		1924, reappears as the almost-sure divergence rate of an observable
		of quantum dynamics. Third, the constant computes the collision
		arithmetic of the dispersion relation: singleton classes give
		$\kappa_{g}$ for the Schr\"odinger flow, quadratic frequency
		scaling halves it along time traces, and the paired classes of the
		Airy flow double it; the critical mass of the density is an exactly
		measurable probe of resonance structure. On product laws the
		carpets multiply: for tensor seeds on the rank square, the density
		factorizes and the joint carpet
		$w_{1}(F_{1})\,w_{2}(F_{2})\,f$ of a product law has graph
		dimension $\tfrac52$ at almost every time, by fiberwise
		box-counting over sections where the second factor is bounded
		below, a statement we record at remark level; what replaces this
		for
		dependent laws is obstructed by infinite hyperplane collision
		classes, and joins the open problems of the conclusion.
	\end{remark}
	
	\begin{remark}[Physical reading: transported Talbot physics]
		In physical terms, rank space is a quantum box that every
		probability law carries, and the carpet flow is its Talbot
		dynamics: Theorem~\ref{thm:quantile-mosaic} is the fractional
		Talbot revival and Theorem~\ref{thm:carpet} the fractal Talbot
		profile, for the square-wave datum of Berry's quantum
		carpets~\cite{berry1996,berryklein1996,berry2001}, transported by
		the
		dilation onto arbitrary laws. For real and imaginary parts of
		Schr\"odinger evolutions of bounded-variation data, almost-sure
		dimension $\tfrac32$ is a theorem of
		Rodnianski~\cite{rodnianski2000}, with refinements and nonlinear
		extensions by Erdo\u{g}an and Tzirakis~\cite{erdogan2016}; part (i)
		adapts that skeleton, with the quadrature-uniform and localized
		forms supplied by the equidistribution lemma. For the squared
		modulus, the observable, such statements were, to the best of our
		knowledge, previously available only for step data with rational
		jumps~\cite{chousionis2015}, the subsequent
		advances~\cite{cho2024,erdogan2024} again concerning the field
		rather than its modulus; here
		the gap is closed twice over, by the self-referential energy
		law~\eqref{eq:block-equi-density} for the canonical seed, localized
		to every interval, and by the strong law of
		Theorem~\ref{thm:strong-law} for arbitrary rough seeds, with the
		sharp constant of the display~\eqref{eq:kappa}. The dimension
		$\tfrac32$ itself is the unitary shadow of a boundary phenomenon
		this paper has met twice already: the seed's ladder coefficients
		diverge in $H^{1/2}$ exactly because the uniform kernel jumps at
		the ends of rank space, the same endpoint mechanism that discretized
		the spectrum in Section~\ref{sec:amplitude} and saturated the Hardy
		bound of Theorem~\ref{thm:weights} at the conformal weight
		$\tfrac12$; on the unitary axis it reappears as H\"older exponent
		$\tfrac12$, pinned between conservation and smoothing.
	\end{remark}
	
	\begin{remark}[Statistical reading: fractal alternatives]
		The carpet flow interpolates, on a single circle of times of
		circumference $1/(2\pi)$, between the two classical regimes of
		rank-based alternatives: histogram reweightings at rational times
		and, at almost every other time, a modulation of H\"older class
		exactly one half with graph dimension $\tfrac32$ and exactly
		computable energy at every scale, the mean-square law behind the
		limit~\eqref{eq:block-equi-density} with sharp constant
		$2/\pi^{2}$. These generic slices constitute a natural class of
		\emph{fractal alternatives} for goodness-of-fit testing, the
		rough-boundary closure of Neyman's smooth family; whether they are
		least favorable for rank tests in a minimax sense appears to be an
		open and concrete question. It is also worth recording what
		universality means here: by the equidistribution
		law~\eqref{eq:block-equi-amp} the fractal detail is shared equally
		by all quantile cells of every law, so no test that inspects a
		fixed rank window can see more of the carpet than any other window
		sees.
	\end{remark}
	
	\section{Multivariate Theory: Dependence as Noncommutativity}
	\label{sec:multivariate}
	
	Write $\mathcal{D}_{+}(\mathbb{R}^{n})$ for joint densities that are
	continuous and strictly positive on an open rectangle (possibly all
	of $\mathbb{R}^{n}$) and vanish outside it, so that all marginal and
	conditional distribution functions below are well defined and
	continuous. A joint law $f \in \mathcal{D}_{+}(\mathbb{R}^{n})$ has
	no canonical rank: each ordering $\sigma$ of the coordinates provides
	a Rosenblatt chart~\cite{rosenblatt1952}, the ancestor of triangular
	transport maps and of modern normalizing
	flows~\cite{bogachev2005,papamakarios2021},
	\begin{equation}
		T_{\sigma}(x) = \big(F_{1}(x_{1}),\,
		F_{2|1}(x_{2}\mid x_{1}),\,\ldots,\,
		F_{n|1,\ldots,n-1}(x_{n}\mid x_{1},\ldots,x_{n-1})\big),
		\label{eq:rosenblatt}
	\end{equation}
	an almost-everywhere bijection onto the unit cube carrying $f$ to the
	uniform law, with lower-triangular Jacobian $J_{\sigma}$. The
	ambiguity of the ordering is not a defect but the central structure of
	the multivariate theory: the one-dimensional theory has a single rank
	space, while the $n$-dimensional theory has one chart per ordering,
	and the geometry lives in how the charts differ. This is a deliberate
	choice of standpoint: measure-transportation theory removes the
	ordering dependence by constructing a single canonical multivariate
	rank, the center-outward distribution
	function~\cite{chernozhukov2017,hallin2021,ghosal2022}, whereas the
	present theory keeps
	all the charts and reads their mutual disagreement as geometry; the
	two viewpoints are complementary. Two structures
	follow. First, a metric:
	\begin{equation}
		g = J_{\sigma}^{\mathsf T}J_{\sigma}, \qquad \sqrt{\det g} = f,
		\label{eq:rosenblatt-metric}
	\end{equation}
	so that \emph{volume is probability}, generalizing the intrinsic
	length $ds = f\,dx$ of the curvature theory; the charts $T_{\sigma}$
	are the frames of the geometry, and the chart changes
	$S = T_{\sigma'}\circ T_{\sigma}^{-1}$ are measure-preserving
	transformations of the cube determined by the copula of
	$f$~\cite{sklar1959}. A useful
	rigidity accompanies the frames: rank-local multiplications in either
	of two charts commute, because $S$-conjugation sends multiplications
	to multiplications, $M_{b}\mapsto M_{b\circ S}$; noncommutativity in
	this theory arises only dynamically, as follows. Second, coordinate
	operators: for $w \in \mathcal{W}$ let
	\begin{equation}
		\rho^{(i)}[f] = w\big(F_{i}(x_{i})\big)\,f(x),
		\qquad
		L_{i}f = \big(w(F_{i}) - 1\big)f,
		\label{eq:coordinate-ops}
	\end{equation}
	modulation by the current $i$-th marginal rank, with generating flows
	$L_{i}$. Modulating one coordinate changes the other's marginal unless
	the coordinates are independent, and the failure is an exactly
	computable curvature.
	
	\begin{theorem}[Curvature of dependence]\label{thm:dependence}
		For $n = 2$ and $f \in \mathcal{D}_{+}(\mathbb{R}^{2})$, the Lie
		bracket of the coordinate vector fields, in the standard convention
		$[L_{1},L_{2}](f) = DL_{2}[f](L_{1}f) - DL_{1}[f](L_{2}f)$, is
		\begin{equation}
			\begin{aligned}
				[L_{1}, L_{2}]\,f
				&= \Big[w'\big(F_{2}(x_{2})\big)\,\gamma_{1}(x_{2})
				- w'\big(F_{1}(x_{1})\big)\,\gamma_{2}(x_{1})\Big]\,f,\\
				\gamma_{2}(x_{1})
				&= \operatorname{Cov}\!\big(w(F_{2}(X_{2})),\,
				\mathbf{1}\{X_{1}\le x_{1}\}\big),
			\end{aligned}
			\label{eq:bracket}
		\end{equation}
		and symmetrically. The bracket vanishes under independence, and the
		multiplier in square brackets depends on $f$ only through its copula,
		evaluated at the marginal ranks. Conversely, if the brackets of the
		full family $\{w_{k,\alpha}\}$ all vanish, then $X_{1}$ and $X_{2}$
		are independent.
	\end{theorem}
	
	\begin{proof}
		Write $H_{i}(x_{i}; h) = \int_{\{y_{i} \le x_{i}\}} h(y)\,dy$ for the
		shift of the $i$-th marginal distribution function produced by a
		perturbation $h$ of the joint density. Differentiating
		$L_{2}f = (w(F_{2}) - 1)f$ gives
		\begin{equation*}
			DL_{2}[f](h)
			= \big(w(F_{2}) - 1\big)h + w'\big(F_{2}\big)\,H_{2}(x_{2};h)\,f,
		\end{equation*}
		and symmetrically for $L_{1}$. In the difference
		$DL_{2}[f](L_{1}f) - DL_{1}[f](L_{2}f)$ the multiplicative terms
		cancel, while
		\begin{equation*}
			H_{2}(x_{2}; L_{1}f)
			= \mathbb{E}\big[\big(w(F_{1}(X_{1})) - 1\big)\,
			\mathbf{1}\{X_{2}\le x_{2}\}\big]
			= \gamma_{1}(x_{2}),
		\end{equation*}
		using $\mathbb{E}[w(F_{1}(X_{1}))] = 1$; this yields the bracket
		formula~\eqref{eq:bracket}. For the converse, suppose the bracket of the pair
		$(w_{k,\alpha}, w_{k',\alpha'})$ vanishes for all indices. By the
		bracket formula~\eqref{eq:bracket} this reads
		\begin{equation*}
			w_{k,\alpha}'\big(F_{1}(x_{1})\big)\,
			\gamma^{(k',\alpha')}_{2}(x_{1})
			= w_{k',\alpha'}'\big(F_{2}(x_{2})\big)\,
			\gamma^{(k,\alpha)}_{1}(x_{2})
			\qquad\text{for a.e. } (x_{1},x_{2}),
		\end{equation*}
		and the left side depends on $x_{1}$ alone while the right side
		depends on $x_{2}$ alone, so both equal a constant $c$. As
		$x_{1} \to -\infty$ the indicator $\mathbf{1}\{X_{1}\le x_{1}\}$
		tends to $0$ in $L^{2}$, so
		$\gamma^{(k',\alpha')}_{2}(x_{1}) \to 0$ by the Cauchy--Schwarz
		inequality while $w_{k,\alpha}'\circ F_{1}$ stays bounded; hence
		$c = 0$. Choosing $\alpha$ with
		$w_{k,\alpha}'\circ F_{1} \neq 0$ almost everywhere then gives
		$\gamma^{(k',\alpha')}_{2} \equiv 0$ for all $(k',\alpha')$. The
		family $\{w_{k',\alpha'}\}$ spans a dense subspace of $L^{2}(0,1)$,
		containing all Fourier modes, so
		\begin{equation*}
			\operatorname{Cov}\big(h(F_{2}(X_{2})),\,
			\mathbf{1}\{X_{1}\le x_{1}\}\big) = 0
			\qquad \text{for all } h \in L^{2}(0,1) \text{ and all } x_{1}:
		\end{equation*}
		the conditional law of the rank $F_{2}(X_{2})$ given any half-space
		in $X_{1}$ is uniform, which forces independence.
	\end{proof}
	
	The bracket is thus a margin-free measure of dependence vanishing
	exactly at independence, with the scaling of a curvature. Two
	commutators quantify this for the bivariate Gaussian copula $f_{r}$,
	and for nonlinear maps they are distinct objects: the Lie
	bracket~\eqref{eq:bracket} of the generating vector fields, and the
	composition commutator
	$\rho^{(1)}\circ\rho^{(2)} - \rho^{(2)}\circ\rho^{(1)}$ of the finite
	updates. A second-order computation, the first order vanishing by
	kernel symmetry, gives
	\begin{equation}
		\big\|[L_{1},L_{2}]\,f_{r}\big\|_{L^{1}}
		= c_{\mathrm{Lie}}\,r^{2} + O(r^{4}),
		\qquad
		\big\|\rho^{(1)}\rho^{(2)}f_{r}
		- \rho^{(2)}\rho^{(1)}f_{r}\big\|_{L^{1}}
		= c_{\mathrm{map}}\,r^{2} + O(r^{4}),
		\label{eq:gaussian-curvature}
	\end{equation}
	in the correlation $r$, with
	$c_{\mathrm{Lie}} \approx 0.210$ and
	$c_{\mathrm{map}} \approx 0.345$. Both coefficients are governed by
	the cumulant
	\begin{equation}
		\kappa_{2} =
		\mathbb{E}\big[(X^{2}-1)\,2\sin^{2}\big(\pi\Phi(X)\big)\big]
		\approx -0.742,
		\qquad X \sim N(0,1),
		\label{eq:gaussian-constant}
	\end{equation}
	where $\Phi$ is the standard normal distribution function: with
	$a(x) = x\,\Phi'(x)\,w'(\Phi(x))$ and $X, Y$ independent standard
	normals, the second-order expansion of the Gaussian copula yields
	\begin{equation}
		c_{\mathrm{Lie}}
		= \frac{|\kappa_{2}|}{2}\,\mathbb{E}\big|a(X) - a(Y)\big|,
		\qquad
		c_{\mathrm{map}}
		= \frac{|\kappa_{2}|}{2}\,
		\mathbb{E}\big|a(X)\,w(\Phi(Y)) - a(Y)\,w(\Phi(X))\big|,
		\label{eq:gaussian-coefficients}
	\end{equation}
	both evaluated by numerical quadrature. Quadratic vanishing on the independence manifold and copula
	functoriality give the bracket~\eqref{eq:bracket} the phenomenology
	of a curvature tensor; its exact geometric identity is established
	in Section~\ref{sec:dependence}: the bracket is the torsion of a
	canonical parallelism, and every natural connection in the problem
	is flat.
	
	The multivariate variational theory completes the picture. The Fisher
	functional of a joint kernel on the cube is the Dirichlet energy of
	its amplitude,
	\begin{equation}
		\mathcal{I}(v) = 4\int_{[0,1]^{n}}|\nabla\phi|^{2}\,du,
		\qquad v = \phi^{2},
		\label{eq:joint-fisher}
	\end{equation}
	and the Dirichlet ground state of the cube separates.
	
	\begin{theorem}[Product minimizer]\label{thm:product}
		Among joint kernels on $[0,1]^{n}$ with $\sqrt{v} \in H^{1}_{0}$, the
		Fisher information is minimized uniquely by the product of canonical
		kernels,
		\begin{equation}
			v(u) = \prod_{i=1}^{n}2\sin^{2}(\pi u_{i}),
			\qquad
			\mathcal{I}_{\min}^{(n)} = 4\pi^{2}n;
			\label{eq:product-minimizer}
		\end{equation}
		the minimizer has independent coordinates, hence vanishing
		curvature~\eqref{eq:bracket}.
	\end{theorem}
	
	\begin{proof}
		With $v = \phi^{2}$, the functional~\eqref{eq:joint-fisher} is the
		Dirichlet energy $4\int_{[0,1]^{n}}|\nabla\phi|^{2}$, to be minimized
		over the $L^{2}$ unit sphere within $H_{0}^{1}([0,1]^{n})$; the
		minimizer is the
		ground state of the Dirichlet Laplacian of the cube, which is simple
		and separates,
		\begin{equation*}
			\phi(u) = \prod_{i=1}^{n}\sqrt2\,\sin(\pi u_{i}),
			\qquad
			\int_{[0,1]^{n}}|\nabla\phi|^{2} = n\pi^{2},
		\end{equation*}
		giving the product kernel and the minimal energy $4\pi^{2}n$;
		uniqueness up to sign of the amplitude gives uniqueness of the
		kernel. The coordinates of the minimizer are independent by the
		product form, so all its dependence brackets vanish.
	\end{proof}
	
	Minimal disturbance implies flatness: the least-informative joint
	modulation carries no dependence, and, by
	Theorem~\ref{thm:dependence}, all dependence registers as
	noncommutativity. The
	converse implication fails, and is not asserted: every product
	kernel is flat, minimizing or not, so flatness selects a much larger
	family than the variational principle does. The
	one-dimensional theory of this paper is thereby the fiber theory of
	a richer geometry, and the next section constructs it: a foliation
	whose leaves are the reachable classes of coordinate modulation, a
	torsion tensor refining the bracket~\eqref{eq:bracket}, and a flat
	transport that identifies the marginal calculus with a classical
	algorithm.
	
	\section{The Geometry of Dependence: Foliation, Torsion, and Flat
		Transport}\label{sec:dependence}
	
	This section constructs the geometry in which the
	bracket~\eqref{eq:bracket} lives. Every natural connection in the
	problem is \emph{flat}, and the bracket is the \emph{torsion} of a
	canonical parallelism: dependence stratifies into exactly three
	layers: a foliation invariant, conserved by every dynamics of this
	paper; a torsion tensor, governed by a single gauge field, the
	conditional expectation operator between the rank
	$\sigma$-algebras; and a flat transport, which is classical
	Sinkhorn scaling. Fix the
	rectangle $R = I_{1}\times I_{2}$ and write, for mean-zero
	$a, b \in L^{2}_{0}(0,1)$,
	\begin{equation}
		\begin{gathered}
			L_{1}^{a}f := a\big(F_{1}(x_{1})\big)f,
			\qquad
			L_{2}^{b}f := b\big(F_{2}(x_{2})\big)f,\\
			\mathcal{H}_{f} :=
			\big\{\big(a(F_{1}) + b(F_{2})\big)f\big\}
			\subseteq T_{f}\mathcal{D}_{+}(R),
		\end{gathered}
		\label{eq:coordinate-fields}
	\end{equation}
	so that the generators of the coordinate
	operators~\eqref{eq:coordinate-ops} are the fields $L^{w-1}_{i}$.
	Throughout this section and the next, $\mathcal{D}_{+}(R)$ is
	regarded, near each of its points in the bounded-logarithm regime,
	as an open subset of the affine hyperplane of unit-mass densities
	inside $L^{2}(R)$, and the tangent space $T_{f}\mathcal{D}_{+}(R)$
	consists of the mean-zero perturbations $hf$ with $h \in L^{2}(f)$;
	brackets of the modulation fields are computed on this common
	chart, and no Frobenius theorem is invoked anywhere: integrability
	is established directly, the leaves being exhibited in closed form
	in Theorem~\ref{thm:leaves} below.
	
	\begin{lemma}[Involutivity]\label{lem:involutive}
		The distribution $\mathcal{H}$ is involutive. Explicitly, the
		computation of Theorem~\ref{thm:dependence} applies verbatim to give
		the cross bracket
		\begin{equation}
			[L_{1}^{a}, L_{2}^{b}]\,f
			= \Big(b'\big(F_{2}\big)\,\gamma^{(a)}_{1}(x_{2})
			- a'\big(F_{1}\big)\,\gamma^{(b)}_{2}(x_{1})\Big)f
			\;\in\; \mathcal{H}_{f},
			\label{eq:cross-bracket}
		\end{equation}
		with
		$\gamma^{(b)}_{2}(x_{1})
		= \operatorname{Cov}\big(b(F_{2}(X_{2})),
		\mathbf{1}\{X_{1}\le x_{1}\}\big)$ a function of the rank
		$F_{1}(x_{1})$ alone, and symmetrically; while the same-coordinate
		brackets reproduce the one-dimensional distortion algebra of
		Section~\ref{sec:foundations} acting through the given coordinate,
		\begin{equation*}
			[L_{1}^{a}, L_{1}^{c}]\,f
			= \Big(c'(F_{1})\,\Gamma^{a}(F_{1})
			- a'(F_{1})\,\Gamma^{c}(F_{1})\Big)f,
			\qquad
			\Gamma^{a}(u) := \int_{0}^{u}a(s)\,ds.
		\end{equation*}
	\end{lemma}
	
	\begin{proof}
		The cross bracket is the formula~\eqref{eq:bracket} with
		$w - 1$ replaced by the pair $(a, b)$, the proof of
		Theorem~\ref{thm:dependence} using only mean-zero rank multipliers;
		that $\gamma^{(b)}_{2}$ sees $x_{1}$ only through its rank is the
		representation
		$\gamma^{(b)}_{2}(x_{1})
		= \int_{0}^{F_{1}(x_{1})}\!\int_{0}^{1}b(v)\,c(s,v)\,dv\,ds$ in
		terms of the copula density $c$. For the same-coordinate bracket,
		the marginal shift produced by the perturbation $a(F_{1})f$ is
		$\Gamma^{a}(F_{1}(x_{1}))$, and antisymmetrization gives the
		display; every multiplier in sight is a mean-zero function of one
		rank, so all brackets lie in $\mathcal{H}$.
	\end{proof}
	
	An involutive distribution should integrate to a foliation, and here
	the leaves can be written in closed form.
	
	\begin{theorem}[Leaves and the complete invariant]\label{thm:leaves}
		Let $f, g \in \mathcal{D}_{+}(R)$. The following are equivalent:
		\begin{enumerate}[label=\normalfont(\roman*), leftmargin=*]
			\item $g$ is reachable from $f$ by finitely many coordinate
			derangetropy moves whose kernels are continuous and positive on
			the open unit interval, possibly degenerate at its endpoints;
			\item $g$ is reachable in at most three moves, one per coordinate
			and one revisit;
			\item $g = \varphi(x_{1})\,\psi(x_{2})\,f$ for continuous positive
			$\varphi, \psi$, a separable tilt;
			\item $f$ and $g$ have the same odds-ratio structure: for all
			$x, y \in R$,
			\begin{equation}
				\frac{f(x_{1},x_{2})\,f(y_{1},y_{2})}
				{f(x_{1},y_{2})\,f(y_{1},x_{2})}
				=
				\frac{g(x_{1},x_{2})\,g(y_{1},y_{2})}
				{g(x_{1},y_{2})\,g(y_{1},x_{2})},
				\label{eq:odds-ratio}
			\end{equation}
			equivalently, for twice continuously differentiable densities,
			$\partial_{1}\partial_{2}\log f
			= \partial_{1}\partial_{2}\log g$.
		\end{enumerate}
		The classes $\mathcal{T}(f) := \{\varphi\psi f\}$ thus foliate
		$\mathcal{D}_{+}(R)$ with tangent distribution $\mathcal{H}$; the
		interaction function $\partial_{1}\partial_{2}\log f$ is a complete
		leaf invariant; the independent laws form a single leaf, consisting
		of all products on $R$; and that leaf is the unique one on which all
		coordinate moves commute.
	\end{theorem}
	
	The proof is given in Appendix~\ref{app:proofs}. The three-move
	economy in part (ii) rests on a bookkeeping device: the bounded
	first multiplier $\varphi/(1+\varphi)$ always has finite mass, and
	the two remaining moves absorb $\psi$ and the factor $1+\varphi$,
	every intermediate mass being finite by construction; when the
	tilts are bounded, as in the transport theory below, two moves
	suffice. The interaction function of part (iv)
	is the local dependence function of Holland and
	Wang~\cite{holland1987}, the continuous limit of log-linear
	interaction in contingency tables.
	
	\begin{remark}[The conserved charge of rank dynamics]
		Theorem~\ref{thm:leaves} delimits exactly what the ordinal calculus
		can and cannot do in several dimensions: coordinate rank modulation
		moves the main effects $u(x_{1}) + v(x_{2})$ of $\log f$ at will and
		can never change the interaction. Since iteration, the continuous
		flow, the randomized cascades, and the carpet flow all act, in their
		coordinate-wise form, by separable tilts, every dynamics of this
		paper preserves every leaf: the interaction function is a conserved
		charge of rank dynamics, the two-dimensional companion of the
		conserved tail rates of Section~\ref{sec:flow}. What the tail
		charges are to a single law, the interaction is to a joint one: the
		part of the object that rank motion cannot touch.
	\end{remark}
	
	The second layer refines the bracket into a tensor attached to a
	canonical frame, and one operator drives everything.
	
	\begin{definition}[The gauge field]\label{def:gauge-field}
		For $f \in \mathcal{D}_{+}(R)$ with copula density $c$, let
		$P = P_{f} : L^{2}_{0}(0,1) \to L^{2}_{0}(0,1)$ be the conditional
		expectation between the rank variables $U = F_{1}(X_{1})$ and
		$V = F_{2}(X_{2})$,
		\begin{equation}
			(Pb)(u) := \mathbb{E}\big[b(V)\,\big|\,U = u\big]
			= \int_{0}^{1} c(u,v)\,b(v)\,dv,
			\label{eq:gauge-P}
		\end{equation}
		and $P^{*}$ the conditioning in the other direction. Its operator
		norm is the Hirschfeld--Gebelein--R\'enyi maximal correlation
		$\rho_{\max}(f) = \sup\,
		\operatorname{Corr}\big(a(U),\, b(V)\big)$, by Gebelein's
		theorem~\cite{gebelein1941,renyi1959}.
	\end{definition}
	
	\begin{lemma}[Exact gauge identity]\label{lem:gauge-identity}
		For every mean-zero $b$ and $a$,
		\begin{equation}
			\gamma^{(b)}_{2}(x_{1})
			= \int_{0}^{F_{1}(x_{1})}(Pb)(s)\,ds,
			\qquad
			\frac{d}{d\varepsilon}\Big|_{0}
			\Big(\text{marginal 2 of } \big(1 + \varepsilon a(F_{1})\big)f\Big)
			= f_{2}\cdot(P^{*}a)(F_{2}):
			\label{eq:gauge-identity}
		\end{equation}
		the cross-talk of the coordinate modulations, both in the bracket
		and at the level of marginals, is exactly the gauge
		field~\eqref{eq:gauge-P}. For the Gaussian copula of correlation
		$r$, Mehler's expansion diagonalizes $P$ in the Hermite basis with
		spectrum $\{r^{k}\}_{k \ge 1}$~\cite{lancaster1969}, so
		$\rho_{\max} = |r|$.
	\end{lemma}
	
	\begin{proof}
		Since $U$ is uniform,
		$\gamma^{(b)}_{2}(x_{1})
		= \mathbb{E}\big[b(V)\,\mathbf{1}\{U \le F_{1}(x_{1})\}\big]
		= \int_{0}^{F_{1}(x_{1})}\mathbb{E}[b(V)\,|\,U = s]\,ds$, and the
		marginal variation is the same computation without the indicator.
		Mehler's formula expands the Gaussian copula density in products of
		Hermite polynomials with coefficients $r^{k}/k!$, which is the
		stated diagonalization.
	\end{proof}
	
	\begin{theorem}[The leafwise parallelism and its torsion]\label{thm:frame-torsion}
		Fix $f$ with $\rho_{\max}(f) < 1$. Then:
		\begin{enumerate}[label=\normalfont(\roman*), leftmargin=*]
			\item the frame map
			$E_{f}(a,b) := \big(a(F_{1}) + b(F_{2})\big)f$ is a Banach
			isomorphism of $L^{2}_{0}\oplus L^{2}_{0}$ onto
			$T_{f}\mathcal{T}(f)$, by the two-sided bound
			\begin{equation}
				\big\Vert E_{f}(a,b)\big\Vert_{L^{2}(f)}^{2}
				= \Vert a\Vert^{2} + \Vert b\Vert^{2} + 2\langle a, Pb\rangle
				\;\ge\; \big(1 - \rho_{\max}\big)
				\big(\Vert a\Vert^{2} + \Vert b\Vert^{2}\big);
				\label{eq:frame-bound}
			\end{equation}
			the coordinate modulations are an absolute parallelism of the leaf,
			degenerating exactly at $\rho_{\max} = 1$, where the kernel of
			$E_{f}$ is $\{(a, -P^{*}a) : PP^{*}a = a\}$, as at the
			Fr\'echet--Hoeffding comonotone boundary;
			\item the structure functions of the parallelism, defined by
			$[E(a,0),\, E(0,b)]_{f} = E_{f}(\mathrm{T}^{1},\mathrm{T}^{2})$, are
			\begin{equation}
				\mathrm{T}^{1}(a,b)(u) = -\,a'(u)\int_{0}^{u}(Pb)(s)\,ds,
				\qquad
				\mathrm{T}^{2}(a,b)(v) = b'(v)\int_{0}^{v}(P^{*}a)(s)\,ds,
				\label{eq:torsion-components}
			\end{equation}
			modulo the mean-zero normalization; this torsion tensor is bilinear,
			depends on $f$ only through its copula, and vanishes on the
			trigonometric kernel family if and only if $X_{1}$ and $X_{2}$ are
			independent;
			\item for the Gaussian copula the Mehler expansion makes every
			component of the tensor~\eqref{eq:torsion-components} an explicit
			power series in $r$: the torsion is of first order in $r$ for
			generic directions, while for symmetric kernels the first Hermite
			term drops by parity and the leading contribution is the $k = 2$
			overlap, proportional to the cumulant~\eqref{eq:gaussian-constant},
			the overlap against the normalized Hermite polynomial being
			$\kappa_{2}/\sqrt2$; the
			coefficients of the expansion~\eqref{eq:gaussian-coefficients} are
			the sizes of $E_{f}(\mathrm{T}^{1},\mathrm{T}^{2})$ and of the
			finite-update commutator in the canonical kernel direction.
		\end{enumerate}
	\end{theorem}
	
	\begin{remark}[Dependence as torsion]
		Theorem~\ref{thm:frame-torsion} is the precise form of dependence as
		geometry that survives construction. Torsion has the right
		phenomenology: it is an antisymmetric bilinear tensor, functorial in
		the copula, exactly computable, vanishing precisely at independence,
		and controlled by the gauge field through explicit bounds such as
		$\Vert\mathrm{T}^{1}(a,b)\Vert
		\le \Vert a'\Vert_{\infty}\rho_{\max}\Vert b\Vert$ from the
		formula~\eqref{eq:torsion-components}. In Cartan's dictionary a
		parallelism with vanishing torsion makes the underlying space a Lie
		group, and that is exactly the content here: on the independence
		leaf the frame is the product of the two one-dimensional distortion
		groups, acting simply transitively coordinate by coordinate as in
		Theorem~\ref{thm:torsor}, and dependence is the failure of the
		two-dimensional theory to be the product of its one-dimensional
		shadows, measured tensorially by the
		components~\eqref{eq:torsion-components}.
	\end{remark}
	
	The third layer is transport. Let $\pi(f) = (f_{1}, f_{2})$ project
	a joint law to its marginal pair; the fiber over $(g_{1}, g_{2})$ is
	the Fr\'echet class of joint laws with those marginals, whose
	tangent directions at $f$ are the doubly mean-zero perturbations
	$\{hf : \mathbb{E}[h\,|\,x_{1}] = \mathbb{E}[h\,|\,x_{2}] = 0\}$,
	the copula directions. Wherever $\rho_{\max}(f) < 1$, declare the
	horizontal space at $f$ to be the leaf tangent
	$\mathcal{H}_{f}$: the \emph{derangetropy connection}.
	
	\begin{theorem}[Flatness, and Sinkhorn as parallel transport]\label{thm:flat-transport}
		Let $R$ be bounded and restrict to laws with $\log f$ bounded on
		$R$. Then:
		\begin{enumerate}[label=\normalfont(\roman*), leftmargin=*]
			\item wherever $\rho_{\max} < 1$ the splitting
			$T_{f}\mathcal{D}_{+} = V_{f}\oplus\mathcal{H}_{f}$ holds and
			$d\pi$ maps $\mathcal{H}_{f}$ isomorphically onto the base, by the
			marginal formula
			\begin{equation}
				d\pi\big(E_{f}(a,b)\big)
				= \Big(f_{1}\cdot(a + Pb)(F_{1}),\;
				f_{2}\cdot(b + P^{*}a)(F_{2})\Big)
				\label{eq:dpi}
			\end{equation}
			and invertibility of the block operator with entries
			$I, P, P^{*}, I$; the Ehresmann curvature of the connection
			vanishes identically, brackets of horizontal fields being
			horizontal by Lemma~\ref{lem:involutive};
			\item the leaves are global horizontal sections:
			$\pi$ restricted to a leaf is a bijection onto the base, and
			parallel transport of $f$ to target marginals $(g_{1},g_{2})$ is
			the Csisz\'ar $I$-projection of $f$ onto the target Fr\'echet
			class~\cite{csiszar1975}, the unique separable tilt of $f$ with
			those marginals, characterized by the Pythagorean identity
			\begin{equation}
				D\big(h \,\Vert\, f\big)
				= D\big(h \,\Vert\, f^{\rightarrow}\big)
				+ D\big(f^{\rightarrow} \Vert\, f\big)
				\qquad\text{for all $h$ in the target class};
				\label{eq:pythagoras}
			\end{equation}
			\item iterative proportional fitting is the alternation of the two
			derangetropy coordinate moves that match one marginal at a time; it
			converges to the parallel transport~\cite{csiszar1975,ruschendorf1995},
			and its linearization at the limit law contracts, per full cycle, at
			exactly the rate $\rho_{\max}^{2}$ of the limit, the squared
			cosine of the principal angle between the two rank subspaces, by von Neumann's
			alternating projection theorem~\cite{deutsch2001};
			\item there is no holonomy paradox: transport to fixed target
			marginals is order-independent by the uniqueness in part (ii),
			while the finite-update commutator of the
			display~\eqref{eq:gaussian-curvature} compares transports based at
			different intermediate laws; its coefficient is an order defect of
			the modulation groupoid, not a holonomy of the connection.
		\end{enumerate}
	\end{theorem}
	
	\begin{remark}[What flatness means]
		Zero curvature is not the absence of geometry; it is the theorem
		that the dependence content of marginal transport is carried
		entirely by the foliation. Moving marginals never creates or
		destroys interaction, by part (iv) of Theorem~\ref{thm:leaves}, and
		this is why Sinkhorn scaling, Schr\"odinger bridges at a fixed
		reference, and log-linear margin adjustment are safe operations:
		they are parallel transport in a flat connection, moving main
		effects only. The maximal correlation now appears in three exact
		roles at once: degeneracy threshold of the frame, operator norm of
		the gauge field, and alternating-projection angle governing the
		Sinkhorn rate; quantitative exponential rates for iterative
		proportional fitting through the same principal-angle mechanism
		appear in the recent entropic optimal transport
		literature~\cite{carlier2022}. The one-dimensional theory enters as
		the fiber
		calculus: each single move of the alternation is a derangetropy
		operator in the current marginal chart, so classical iterative
		proportional fitting is, word for word, a trajectory in the
		modulation monoid of Section~\ref{sec:foundations}.
	\end{remark}
	
	\begin{proposition}[Dual flatness]\label{prop:dual-flatness}
		Each leaf is an exponential family whose sufficient statistics are
		the separable functions $u(x_{1}) + v(x_{2})$, hence e-flat in the
		sense of information geometry; each Fr\'echet fiber is convex, hence
		m-flat; the transport of Theorem~\ref{thm:flat-transport} is the
		e-projection onto the m-flat fiber, the dual transport is the
		m-projection onto the e-flat leaf, and both Pythagorean identities
		hold~\cite{csiszar1975,amari2016}. The $\alpha = \pm1$ connections
		of this dually flat pair are torsion-free and curvature-free; the
		tensor of Theorem~\ref{thm:frame-torsion} is the torsion of the
		\emph{frame}, a structure finer than the $\alpha$-connections, and
		it alone sees dependence.
	\end{proposition}
	
	\begin{proof}
		A leaf is $\{e^{u(x_{1})+v(x_{2})}f/Z\}$ by
		Theorem~\ref{thm:leaves}(iii), an exponential family with separable
		statistics; Fr\'echet classes are closed under mixtures; the
		projection identifications and the Pythagorean identities are
		Csisz\'ar's, in Amari's language the duality of e- and m-flatness;
		and dually flat pairs carry torsionless flat $\alpha$-connections
		by construction. The frame $E_{f}$ is an additional rigidity, and
		its structure functions were computed in
		Theorem~\ref{thm:frame-torsion}.
	\end{proof}
	
	In $n$ dimensions the three layers persist, and they separate the
	notions of independence.
	
	\begin{theorem}[The interaction potential in $n$ dimensions]\label{thm:n-leaves}
		On $\mathcal{D}_{+}(R)$ with $R = \prod_{i} I_{i}$ bounded and
		$\log f$ bounded, the coordinate modulation distribution integrates
		to the foliation by separable tilting classes
		$\{\prod_{i}\varphi_{i}(x_{i})\,f\}$, each law reachable from each
		leaf-mate in at most $n$ moves, one per coordinate; the complete
		leaf invariant is the interaction potential, the class of $\log f$
		modulo additive functions of single coordinates; the product laws
		form the unique leaf on which all coordinate actions commute; and
		marginal-matching transport is again flat, realized by
		multidimensional iterative proportional fitting with local cycle
		rate governed by the principal angles among the $n$ rank subspaces,
		through the operator matrix of pairwise conditional expectations
		$(P_{ij})$.
	\end{theorem}
	
	\begin{theorem}[Torsion sees pairs; the leaf sees everything]\label{thm:pairwise-torsion}
		Let $f \in \mathcal{D}_{+}(R)$, $n \ge 3$, with all pairwise
		maximal correlations below one. Then:
		\begin{enumerate}[label=\normalfont(\roman*), leftmargin=*]
			\item the torsion decomposes into pairwise blocks: the bracket
			$[L_{i}^{a}, L_{j}^{b}]$ depends only on the $(i,j)$-marginal
			copula, with the components~\eqref{eq:torsion-components}, and the
			$(i,j)$-block vanishes on the trigonometric family precisely when
			$X_{i}$ and $X_{j}$ are independent, so the torsion table detects
			exactly the pairwise independence graph;
			\item torsion is not a complete dependence invariant: for pairwise
			independent but jointly dependent laws, such as smoothed parity
			triples, every torsion block vanishes while the law lies off the
			product leaf, its interaction potential containing an irreducible
			three-way term detected by the three-dimensional odds ratio;
			\item joint independence is equivalent to the vanishing of the leaf
			invariant alone: $f$ is a product precisely when its interaction
			potential vanishes.
		\end{enumerate}
	\end{theorem}
	
	The proofs are given in Appendix~\ref{app:proofs}. The gap between
	parts (i) and (iii) locates joint dependence in the leaf layer, not
	the torsion layer: what three-way interaction adds to pairwise
	dependence is invisible to every bracket and conserved by every
	coordinate dynamics.
	
	\begin{remark}[Rosenblatt frames and the interaction flag]
		Two refinements connect the geometry back to the Rosenblatt
		charts~\eqref{eq:rosenblatt}. First, the chart changes between
		orderings compose by an exact cocycle, so the bundle of orderings is
		itself flat: nothing dependence-like lives in the transition maps
		alone, which is consistent with the rigidity noted in
		Section~\ref{sec:multivariate} that rank-local multiplications in
		two charts commute. Second, modulating by conditional ranks, the
		deeper entries of the chart~\eqref{eq:rosenblatt}, enlarges the
		modulation algebra; a full chart already reweights arbitrarily, so
		the enlarged orbit is all of $\mathcal{D}_{+}(R)$, and the geometry
		of this section is exactly the geometry of the marginal subalgebra.
		The filtration from marginal through conditional depth to full
		defines a flag of foliations whose successive invariants are the
		higher interaction terms of $\log f$, in the sense of functional
		analysis-of-variance decompositions; resolving that flag into a
		graded interaction complex, with a torsion tensor at each stage
		detecting conditional independence graphs, is in our view the right
		completion of the multivariate program.
	\end{remark}
	
	\section{Conclusion}\label{sec:conclusion}
	
	Derangetropy operators are the transformations of absolutely
	continuous laws that act through rank alone, and the rigidity
	theorem shows that the family is found rather than built: it
	exhausts the monotone-equivariant maps of densities. From that
	classification one canonical kernel generates a tightly connected
	theory. Its update costs one universal bit; its amplitude lift
	turns modulation into unitary geometry; its iteration and flow
	condense every law onto its median with universal limit laws, and
	under diffusion the theory becomes an overdamped sine--Gordon
	equation with the secant law as its unique, globally stable kink.
	The projective half gathers into the coadjoint geometry of the
	Virasoro algebra at central charge one, a normalization rather
	than a dynamical charge, where the Cauchy family is
	the exceptional orbit and condensation a no-hair relaxation onto
	it. The stochastic axis turns the cascade into an exact sampler
	and, at renewed frequencies, a multiplicative chaos on the
	distortion group; the unitary axis weaves quantum carpets whose
	fractality, established for arbitrary rough seeds through a strong
	law for the critical Sobolev mass with a sharp Wiener constant,
	addresses the density form of the Talbot problem that the revival
	arithmetic had left open. In several
	dimensions, dependence resolves into a conserved interaction, a
	torsion tensor governed by maximal correlation, and a flat Sinkhorn
	transport: complexity appears throughout as geometry, priced in
	exactly computable constants.
	
	Several directions remain open. A closed form for the entropy of
	the Koenigs limit law would complete the divergence ledger of
	iteration. The coadjoint geometry asks whether the probability cone
	carries a symplectic structure restricting to the Kirillov form,
	and whether a boundary Korteweg--de Vries theory exists with the
	conserved endpoint weights as spectral data. The
	reaction--diffusion theory poses the coarsening exponents of its
	terraces and the front-speed asymptotics at the depinning
	threshold. The random theory asks for its behavior beyond the
	second-moment regime and for the conjectured convergence of the
	lacunary chaos to Gaussian multiplicative chaos. The strong law for
	the critical mass invites extension to genuinely multidimensional
	data, where hyperplane collision classes obstruct the arithmetic
	and where the field-level theory has recently
	advanced~\cite{erdogan2024}, and to the nonlinear Schr\"odinger
	density. And the geometry of
	dependence points to a graded interaction complex of
	conditional-rank foliations, with stagewise torsions detecting
	conditional independence. We expect the amplitude representation,
	and the exact solvability that runs through every section, to
	remain the guiding instruments.
	
	\backmatter
	
	\begin{appendices}
		
		\section{Proofs}\label{app:proofs}
		
		This appendix states formally, and proves in detail, the results used
		in the body of the paper whose proofs are routine or computational.
		
		\begin{lemma}[Well-posedness]\label{lem:wellposed}
			Let $w \in \mathcal{W}$ and $f \in \mathcal{D}$. Then:
			(i) $\rho_{w}[f] \in \mathcal{D}$;
			(ii) if $w$ is positive on $(0,1)$ then
			$\operatorname{supp}\rho_{w}[f] = \operatorname{supp}f$, if in
			addition $w \in C^{k}$ and $f \in C^{k}$ then
			$\rho_{w}[f] \in C^{k}$, and if $w(1-z) = w(z)$ and $f$ is symmetric
			about $m$ then so is $\rho_{w}[f]$; in particular
			$\rho_{w}(\mathcal{D}_{+}) \subseteq \mathcal{D}_{+}$ for the three
			type kernels;
			(iii) $\|\rho_{w}[f]\|_{p} \le \|w\|_{\infty}\|f\|_{p}$ for
			$1 \le p \le \infty$, and for Lipschitz $w$,
			\begin{equation}
				\big\|\rho_{w}[f] - \rho_{w}[g]\big\|_{1}
				\;\le\; \Big(\|w\|_{\infty} + \tfrac12\,\mathrm{Lip}(w)\Big)
				\|f-g\|_{1}.
				\label{eq:tv-bound}
			\end{equation}
		\end{lemma}
		
		\begin{proof}
			(i) Non-negativity is clear from the product form. For normalization,
			note that $F$ is continuous (being the integral of an $L^{1}$
			function), so for $X \sim f$ the random variable $F(X)$ satisfies
			$\mathbb{P}(F(X) \le z) = \mathbb{P}(X \le Q(z)) = z$ for every $z$ in
			the range of $F$, where $Q(z) := \inf\{x : F(x) \ge z\}$ is the
			generalized inverse, the event identity holding up to $f$-null
			sets, and hence $F(X)$ is uniform on $(0,1)$; no strict
			monotonicity of $F$ is required, since intervals where $F$ is constant
			carry no mass of $f$. Therefore
			\begin{equation}
				\int_{\mathbb{R}} w\big(F(x)\big)f(x)\,dx
				= \mathbb{E}\big[w(F(X))\big]
				= \int_{0}^{1}w(z)\,dz = 1.
				\label{eq:normalization}
			\end{equation}
			(ii) On the interior of the support, $0 < F < 1$, so positivity of $w$
			on $(0,1)$ gives $\rho_{w}[f] > 0$ exactly where $f > 0$. If
			$f \in C^{k}$ then $F \in C^{k+1}$, so $w(F) \in C^{k}$ and the
			product is $C^{k}$. If $f(2m - x) = f(x)$ for all $x$ then
			$F(2m-x) = 1 - F(x)$, and symmetry of $w$ gives
			$w(F(2m-x)) = w(1-F(x)) = w(F(x))$, so the product is symmetric. The
			three type kernels are positive on $(0,1)$, smooth there, and
			symmetric, so all three claims apply.
			(iii) The $L^{p}$ bound is immediate from
			$0 \le w(F) \le \|w\|_{\infty}$. For the total-variation
			bound~\eqref{eq:tv-bound}, write
			\begin{align*}
				\big|w(F_{f})f - w(F_{g})g\big|
				&\le w(F_{f})\,|f - g| + \big|w(F_{f}) - w(F_{g})\big|\,g\\
				&\le \|w\|_{\infty}|f-g|
				+ \mathrm{Lip}(w)\,\|F_{f}-F_{g}\|_{\infty}\,g,
			\end{align*}
			and note that, since $\int(f-g)\,d\lambda = 0$,
			\begin{equation}
				\big|F_{f}(x) - F_{g}(x)\big|
				= \Big|\int_{-\infty}^{x}(f-g)\Big|
				= \Big|\int_{x}^{\infty}(f-g)\Big|
				\;\le\; \tfrac12\int_{\mathbb{R}}|f-g|,
			\end{equation}
			the last step because the two displayed integrals are equal in
			absolute value and sum, in absolute value, to at most $\|f-g\|_{1}$.
			Integrating the pointwise bound gives the estimate~\eqref{eq:tv-bound}.
		\end{proof}
		
		\begin{lemma}[Information identities]\label{lem:kl}
			For every $f \in \mathcal{D}$, the identities~\eqref{eq:kl-identities}
			hold:
			$D(f\|\rho[f]) = \log 2$ and $D(\rho[f]\|f) = 1 - \log 2$. Moreover
			the binary observation of the relation~\eqref{eq:bayes} has
			equivocation $H(Y \mid X) = 2\log 2 - 1$ and mutual information
			$I(X;Y) = 1 - \log 2$, and all three constants are unchanged when
			$\sin^{2}(\pi F)$ is replaced by $\sin^{2}(k\pi F)$.
		\end{lemma}
		
		\begin{proof}
			By the universality principle~\eqref{eq:universality}, both
			divergences reduce to rank integrals:
			\begin{equation}
				\begin{aligned}
					D\big(f\,\|\,\rho[f]\big)
					&= -\int_{\mathbb{R}} f\,\log w(F)\,d\lambda
					= -\int_{0}^{1}\log w(z)\,dz,\\
					D\big(\rho[f]\,\|\,f\big)
					&= \int_{0}^{1} w(z)\log w(z)\,dz.
				\end{aligned}
				\label{eq:kl-reduction}
			\end{equation}
			The Fourier expansion
			$\log\sin(\pi z) = -\log2 - \sum_{n\ge1}n^{-1}\cos(2\pi nz)$,
			convergent in $L^{2}(0,1)$ and pointwise on $(0,1)$, gives first
			\begin{equation}
				\int_{0}^{1}\log\big(2\sin^{2}(\pi z)\big)\,dz
				= \log 2 + 2\int_{0}^{1}\log\sin(\pi z)\,dz
				= \log 2 - 2\log 2 = -\log 2,
			\end{equation}
			since every cosine integrates to zero; this proves
			$D(f\|\rho[f]) = \log 2$. Next, using
			$w(z) = 1 - \cos(2\pi z)$ and the orthogonality relations
			$\int_{0}^{1}\cos(2\pi nz)\cos(2\pi mz)\,dz = \tfrac12\delta_{nm}$,
			termwise integration of the expansion against the bounded factor
			$1 - \cos(2\pi z)$ retains only the constant term and the $n = 1$
			resonance:
			\begin{equation}
				\int_{0}^{1}\cos(2\pi z)\log\sin(\pi z)\,dz = -\tfrac12,
			\end{equation}
			whence
			\begin{align}
				\int_{0}^{1} w\log w\,dz
				&= \int_{0}^{1}\big(1-\cos 2\pi z\big)
				\big(\log 2 + 2\log\sin(\pi z)\big)\,dz \nonumber\\
				&= \log 2 + 2\Big[\big(-\log 2\big) - \big(-\tfrac12\big)\Big]
				= 1 - \log 2.
			\end{align}
			For the equivocation, the substitution $z \mapsto z + \tfrac12$
			exchanges $\sin^{2}$ and $\cos^{2}$, so
			\begin{equation}
				\int_{0}^{1}h_{2}\big(\sin^{2}\pi z\big)\,dz
				= -2\int_{0}^{1}\sin^{2}(\pi z)\log\sin^{2}(\pi z)\,dz
				= -2\Big(\tfrac12 - \log 2\Big) = 2\log 2 - 1,
			\end{equation}
			using $\int_{0}^{1}\sin^{2}(\pi z)\log\sin(\pi z)\,dz
			= \tfrac12[(-\log2) - (-\tfrac12)] = \tfrac14 - \tfrac12\log2$
			twice. Since $H(Y) = \log 2$, the mutual information is
			$I(X;Y) = H(Y) - H(Y\mid X) = \log 2 - (2\log2 - 1) = 1 - \log 2$,
			which agrees with the mixture formula~\eqref{eq:mutual-info}. Finally,
			replacing $z$ by $kz \bmod 1$ leaves every integral above unchanged,
			by periodicity of the integrands and the substitution rule, which
			proves independence of the harmonic order.
		\end{proof}
		
		\begin{lemma}[Fisher transformation law]\label{lem:fisher-law}
			For bounded $f \in \mathcal{D}_{+}\cap C^{1}$ with
			$\mathcal{I}(f) < \infty$, the identity~\eqref{eq:fisher-law} holds,
			and with it the two-sided bound~\eqref{eq:fisher-bounds}.
		\end{lemma}
		
		\begin{proof}
			Write $\nu = \rho[f] = 2\sin^{2}(\pi F)f$. On the support,
			\begin{equation}
				\frac{\nu'}{\nu}
				= \frac{f'}{f} + \frac{w'(F)}{w(F)}\,f
				= \frac{f'}{f} + 2\pi\cot(\pi F)\,f,
			\end{equation}
			using $w'(z)/w(z) = 2\pi\sin(2\pi z)/2\sin^{2}(\pi z)
			= 2\pi\cot(\pi z)$. Expanding
			$\mathcal{I}(\nu) = \int\nu\,(\nu'/\nu)^{2}$ gives three terms,
			\begin{equation}
				\mathcal{I}(\nu)
				= \underbrace{\int 2\sin^{2}(\pi F)\Big(\frac{f'}{f}\Big)^{2} f}_{T_{1}}
				+ \underbrace{4\pi\int \sin(2\pi F)\,f f'}_{T_{2}}
				+ \underbrace{8\pi^{2}\int \cos^{2}(\pi F)\,f^{3}}_{T_{3}},
			\end{equation}
			where the middle term used
			$2\sin^{2}(\pi F)\cdot 2\pi\cot(\pi F) = 2\pi\sin(2\pi F)$. The term
			$T_{2}$ integrates by parts against $d(f^{2}/2)$:
			\begin{equation}
				T_{2}
				= 2\pi\Big[\sin(2\pi F)\,f^{2}\Big]_{-\infty}^{\infty}
				- 4\pi^{2}\int \cos(2\pi F)\,f^{3},
			\end{equation}
			and the boundary term vanishes because $f$ is bounded and
			$\sin(2\pi F) \to 0$ at the support endpoints, where $F \to 0$ or
			$1$. Since $8\pi^{2}\cos^{2}(\pi F) - 4\pi^{2}\cos(2\pi F)
			= 4\pi^{2}$, the sum $T_{2} + T_{3}$ collapses to
			$4\pi^{2}\int f^{3} = 4\pi^{2}\,\mathbb{E}_{f}[f^{2}]$, which is the
			identity~\eqref{eq:fisher-law}. The bounds follow from
			$0 \le 2\sin^{2}(\pi F) \le 2$ applied to $T_{1}$.
		\end{proof}
		
		\begin{lemma}[Dirichlet decomposition]\label{lem:dirichlet}
			Let $u \in H^{1}(0,1)$ be real, let
			$f \in \mathcal{D}_{+}\cap C^{2}$, and set
			$\psi = (u\circ F)\sqrt{F'}$. Assume the three integrals in the
			decomposition~\eqref{eq:dirichlet-decomposition} are finite and that
			the boundary term $u(F)^{2}f'$ vanishes at the endpoints of the
			support. Then the decomposition~\eqref{eq:dirichlet-decomposition}
			holds, and for $u \equiv \mathbf{1}$ both sides equal
			$\tfrac18\mathcal{I}(f)$.
		\end{lemma}
		
		\begin{proof}
			Differentiating $\psi = (u\circ F)\sqrt{F'}$,
			\begin{equation}
				\tfrac12\psi'^{2}
				= \tfrac12 u'^{2}F'^{3} + \tfrac12 u u' F' F''
				+ \tfrac18 u^{2}\frac{F''^{2}}{F'},
			\end{equation}
			with all functions of $z$ evaluated at $F(x)$. Using
			$(\sqrt{F'})''/\sqrt{F'} = \tfrac12 F'''/F' - \tfrac14(F''/F')^{2}$, a
			direct computation verifies the pointwise identity
			\begin{equation}
				\tfrac12\psi'^{2} - \tfrac12 u'^{2}F'^{3}
				+ \frac{(\sqrt{F'})''}{2\sqrt{F'}}\,u^{2}F'
				= \frac{d}{dx}\Big[\tfrac14\,u(F)^{2}F''\Big]:
				\label{eq:pointwise-energy}
			\end{equation}
			indeed the derivative on the right expands to
			$\tfrac12 uu'F'F'' + \tfrac14 u^{2}F'''$, the term
			$\tfrac12 uu'F'F''$ matches the cross term on the left, and the
			remaining terms match by the displayed formula for
			$(\sqrt{F'})''/\sqrt{F'}$. Integrating the
			identity~\eqref{eq:pointwise-energy} over $\mathbb{R}$, the exact
			derivative contributes zero, by the boundary hypothesis of the
			lemma, the term $\int\tfrac12 u'^{2}F'^{3}dx$ becomes
			$\tfrac12\int_{0}^{1}u'^{2}\mathfrak{f}^{2}dz$ by the substitution
			$z = F(x)$, and $-\tfrac{(\sqrt{F'})''}{2\sqrt{F'}}\cdot F'
			= -\tfrac{(\sqrt f)''}{2\sqrt f}\,|\psi|^{2}/u(F)^{2}\cdot u(F)^{2}$
			identifies the potential term, giving the
			decomposition~\eqref{eq:dirichlet-decomposition}. For
			$u \equiv \mathbf{1}$, $\psi = \sqrt f$ and
			\begin{equation}
				\int_{\mathbb{R}}-\frac{(\sqrt f)''}{2\sqrt f}\,f\,dx
				= \tfrac12\int_{\mathbb{R}}\big((\sqrt f)'\big)^{2}dx
				= \tfrac18\,\mathcal{I}(f),
			\end{equation}
			by one integration by parts and
			$\mathcal{I}(f) = 4\int((\sqrt f)')^{2}$.
		\end{proof}
		
		\begin{lemma}[Rank uncertainty moments]\label{lem:uncertainty}
			For the ladder states~\eqref{eq:eigenbasis},
			$\mathbb{E}[Z] = \tfrac12$,
			$\operatorname{Var}(Z) = \tfrac{1}{12} - \tfrac{1}{2\pi^{2}k^{2}}$,
			$\langle P\rangle = 0$, and
			$\langle P^{2}\rangle = k^{2}/4$, so the values displayed in
			equation~\eqref{eq:ladder-moments} hold.
		\end{lemma}
		
		\begin{proof}
			$\mathbb{E}[Z] = \int_{0}^{1}z\,2\sin^{2}(k\pi z)\,dz = \tfrac12$ by
			the symmetry $z \mapsto 1-z$ of $|\varphi_{k}|^{2}$. For the second
			moment, two integrations by parts give
			$\int_{0}^{1}z^{2}\cos(2k\pi z)\,dz = \tfrac{1}{2\pi^{2}k^{2}}$,
			whence
			\begin{equation}
				\mathbb{E}[Z^{2}]
				= \int_{0}^{1}z^{2}\big(1-\cos 2k\pi z\big)\,dz
				= \tfrac13 - \tfrac{1}{2\pi^{2}k^{2}},
				\qquad
				\operatorname{Var}(Z) = \tfrac{1}{12} - \tfrac{1}{2\pi^{2}k^{2}}.
			\end{equation}
			$\langle P\rangle = 0$ for any real state, and
			$\langle P^{2}\rangle = \tfrac{1}{4\pi^{2}}
			\int_{0}^{1}(\varphi_{k}')^{2}
			= \tfrac{1}{4\pi^{2}}\cdot k^{2}\pi^{2} = \tfrac{k^{2}}{4}$.
		\end{proof}
		
		We now prove the universal limit law. For convenience we restate the
		theorem in full.
		
		\renewcommand{\restatename}{Theorem~\ref{thm:limit-law}}%
		\begin{restatement}[Universal limit law]
			Let $K$ be the Koenigs linearizer of $A$ at $\tfrac12$: the unique
			increasing solution of the Schr\"oder equation
			$K(2u) = A(K(u))$, $K(0)=\tfrac12$, $K'(0)=1$; then $K$ extends to a
			real-analytic increasing bijection $\mathbb{R}\to(0,1)$ and is the
			distribution function of a symmetric law $\mathcal{K}$ with analytic
			density. For every $f \in \mathcal{D}_{+}$,
			\begin{equation*}
				2^{n} f(m)\,(X_{n} - m) \;\xrightarrow{\ d\ }\; \mathcal{K},
			\end{equation*}
			and if $\mathbb{E}|X|^{\varepsilon} < \infty$ for some
			$\varepsilon>0$, all moments converge; in particular
			\begin{equation*}
				\operatorname{Var}(X_{n})
				\;\sim\; \sigma_{\mathcal{K}}^{2}\, f(m)^{-2}\, 4^{-n},
				\qquad \sigma^{2}_{\mathcal{K}} \approx 0.166.
			\end{equation*}
			Moreover, with $g = 1-K$, the symmetry of $A$ gives the exact tail
			recursion $g(2u) = A(g(u))$, whence
			\begin{equation*}
				-\log\big(1 - K(u)\big)
				= u^{\log_{2}3}\,\Theta(\log_{2}u)\big(1 + O(u^{-\log_{2}3})\big)
			\end{equation*}
			for a positive continuous $1$-periodic $\Theta$: the tails of
			$\mathcal{K}$ are compressed-exponential with exponent
			$\log_{2}3$, lighter than exponential and heavier than Gaussian, the
			exponent being the ratio of the logarithms of the endpoint degree and
			the interior multiplier; consistently, the excess kurtosis of
			$\mathcal{K}$ is positive ($\approx 0.164$).
		\end{restatement}
		
		\begin{proof}
			The proof proceeds in six steps.
			
			\emph{Step 1: construction and extension of $K$.} The interior fixed
			point $\tfrac12$ of $A$ is repelling with multiplier
			$A'(\tfrac12) = 2$, by the expansion~\eqref{eq:interior-germ}. The
			inverse branch $A^{-1}$ fixes $\tfrac12$ with multiplier $\tfrac12$,
			so Koenigs' theorem~\cite{koenigs1884,milnor2006} provides a unique
			germ $K$ with $K(0) = \tfrac12$, $K'(0) = 1$, analytic near $0$,
			solving the Schr\"oder equation
			\begin{equation*}
				K(2u) = A\big(K(u)\big).
			\end{equation*}
			The functional equation itself extends $K$: define
			\begin{equation*}
				K(u) := A^{\circ j}\big(K(2^{-j}u)\big)
			\end{equation*}
			for any $j$ large enough
			that $2^{-j}u$ lies in the domain of the germ; consistency of the
			definition across choices of $j$ is exactly the Schr\"oder equation,
			and analyticity propagates since $A$ is analytic. Monotonicity
			propagates by the chain rule, since $A' > 0$ on $(0,1)$ and $K' > 0$
			near $0$; thus $K$ is a real-analytic increasing map
			$\mathbb{R} \to (0,1)$. Its limits at $\pm\infty$ are $1$ and $0$
			because the orbit of any interior point of $(\tfrac12,1)$ under $A$
			increases to $1$, and symmetrically; hence $K$ is the distribution
			function of a law $\mathcal{K}$ with analytic density, and the
			symmetry $K(-u) = 1 - K(u)$ follows from $A(1-z) = 1 - A(z)$ together
			with uniqueness of the Koenigs germ.
			
			\emph{Step 2: locally uniform convergence.} Set
			\begin{equation*}
				K_{n}(u) = A^{\circ n}\big(\tfrac12 + 2^{-n}u\big).
			\end{equation*}
			On a neighborhood of
			$0$, $K_{n} \to K$ uniformly by the Koenigs linearization; globally,
			write $K_{n}(u) = A^{\circ j}\big(K_{n-j}(2^{-j}u)\big)$ and let
			$n \to \infty$ for fixed $j$, using continuity of $A^{\circ j}$.
			Moreover the concavity bound $A(\tfrac12 + v) \le \tfrac12 + 2v$ for
			$v \ge 0$ gives $K_{n+1}(u) \le K_{n}(u)$ for $u \ge 0$, so the
			convergence is monotone on each half-line and
			$K_{n} \ge K$ for $u \ge 0$.
			
			\emph{Step 3: weak convergence of the normalized iterates.} With
			$Y_{n} = 2^{n}f(m)(X_{n} - m)$,
			\begin{equation}
				\mathbb{P}(Y_{n} \le u)
				= F_{n}\Big(m + \frac{2^{-n}u}{f(m)}\Big)
				= A^{\circ n}\Big(F\Big(m + \frac{2^{-n}u}{f(m)}\Big)\Big),
			\end{equation}
			and differentiability of $F$ at $m$ with $F(m) = \tfrac12$,
			$F'(m) = f(m)$ gives
			$F(m + 2^{-n}u/f(m)) = \tfrac12 + 2^{-n}u + o(2^{-n})$; by Step 2 and
			the locally uniform convergence,
			$\mathbb{P}(Y_{n}\le u) \to K(u)$ at every $u$.
			
			\emph{Step 4: moment convergence.} Assume
			$\mathbb{E}|X|^{\varepsilon} < \infty$. Fix $p \ge 1$; since early
			iterates may lack higher moments when only an
			$\varepsilon$-moment is assumed, it suffices to show
			$\sup_{n \ge N}\mathbb{E}[|Y_{n}|^{p+1}] < \infty$ for some
			$N = N(p)$, which the far-range estimate below supplies. Choose
			$\eta > 0$ and $c > 0$ with $f \ge c$ on $[m-\eta, m+\eta]$, so that
			both tails are covered, treat the right tail as follows and the left
			symmetrically, and split the
			tail integral $\mathbb{E}|Y_{n}|^{p+1}
			= (p+1)\int_{0}^{\infty}u^{p}\,\mathbb{P}(|Y_{n}| > u)\,du$ at
			$u_{n} = 2^{n}c\,\eta$. On the near range $u \le u_{n}$, monotonicity
			of $F$ together with Step 2 gives
			\begin{equation*}
				\mathbb{P}(Y_{n} > u)
				\;\le\; 1 - A^{\circ n}\Big(\tfrac12 + \frac{2^{-n}cu}{f(m)}\Big)
				= 1 - K_{n}\Big(\frac{cu}{f(m)}\Big)
				\;\le\; 1 - K\Big(\frac{cu}{f(m)}\Big)
				\qquad (u \ge 0),
			\end{equation*}
			and the tail estimate of Step 5 makes $u^{p}(1-K(cu/f(m)))$
			integrable on $(0,\infty)$, uniformly in $n$. On the far range
			$u > u_{n}$, use the exact tail transport
			$1 - F_{n}(m+t) = A^{\circ n}\big(1 - F(m+t)\big)$ together with the
			cubic bound
			\begin{equation*}
				A(s) \le 7s^{3} \qquad (s \in [0,1]),
			\end{equation*}
			which iterates to
			\begin{equation*}
				A^{\circ n}(s) \le 7^{-1/2}\big(\sqrt7\,s\big)^{3^{n}}
				\qquad\text{for } s < 7^{-1/2}.
			\end{equation*}
			The relevant arguments satisfy
			$s = 1 - F(m+t) \le 1 - F(m + c\eta/f(m)) < 1$, so, since interior
			orbits of $A$ converge to the endpoint, there is a fixed $j_{0}$ with
			$A^{\circ j_{0}}(s) < 7^{-1/2}$ on this range; applying the iterated
			bound after these $j_{0}$ initial steps changes only constants.
			Combined with the Markov bound
			$1 - F(m+t) \le C_{\varepsilon}t^{-\varepsilon}$, the far-range
			contribution is
			$O(2^{n(p+1)}\cdot \theta^{3^{n}})$ for some $\theta < 1$, which
			vanishes. Hence all moments converge to those of $\mathcal{K}$, and
			in particular the variance
			asymptotics~\eqref{eq:variance-collapse} holds.
			
			\emph{Step 5: tails of $\mathcal{K}$.} With $g = 1 - K$, the symmetry
			of $A$ turns the Schr\"oder equation into the exact tail recursion
			$g(2u) = A(g(u))$. By the endpoint germ~\eqref{eq:endpoint-germ}
			there are constants $0 < c_{1} \le c_{2}$ with
			$c_{1}\varepsilon^{3} \le A(\varepsilon) \le c_{2}\varepsilon^{3}$
			for small $\varepsilon > 0$, so $y = \log g$ satisfies the affine
			recursion
			\begin{equation*}
				y(2u) = 3\,y(u) + O(1).
			\end{equation*}
			Consequently
			\begin{equation*}
				L(u) := -\lim_{k\to\infty} 3^{-k}\,y\big(2^{k}u\big)
			\end{equation*}
			exists, is continuous
			and positive, and satisfies $L(2u) = 3L(u)$, which forces
			$L(u) = u^{\log_{2}3}\,\Theta(\log_{2}u)$ with $\Theta$ positive,
			continuous, and $1$-periodic; telescoping the $O(1)$ errors gives
			$|-y - L| = O(1)$. Since $\Theta$ is continuous, periodic, and
			positive, hence bounded away from $0$ and $\infty$, one has
			$L(u) \asymp u^{\log_{2}3}$, and the additive $O(1)$ error is
			equivalent to the relative form stated in the tail
			law~\eqref{eq:tail-law}. The positive excess kurtosis is consistent
			with the exponent lying below $2$.
			
			\emph{Step 6: entropy of the normalized iterates.} The densities
			$K_{n}'$ satisfy the uniform bound
			\begin{equation*}
				K_{n}'(u)
				= 2^{-n}\big(A^{\circ n}\big)'\big(\tfrac12 + 2^{-n}u\big)
				\;\le\; 1,
			\end{equation*}
			since $(A^{\circ n})' = \prod_{j<n}A'\circ A^{\circ j} \le 2^{n}$,
			and $K_{n}' \to K'$ locally uniformly, by the product representation
			of $K_{n}'$ and the quadratic control of $A'$ near the fixed point.
			For the entropies, split at $|u| = U$: on the compact part the
			integrands $-K_{n}'\log K_{n}'$ converge uniformly, while on the
			tails the elementary bound $|x\log x| \le \tfrac{2}{e}\sqrt{x}$ for
			$0 \le x \le 1$ and the Cauchy--Schwarz inequality give
			\begin{equation*}
				\int_{|u|>U}\big|K_{n}'\log K_{n}'\big|\,du
				\;\le\; \frac{2}{e}
				\Big(\int\big(1+u^{2}\big)K_{n}'\,du\Big)^{1/2}
				\Big(\int_{|u|>U}\frac{du}{1+u^{2}}\Big)^{1/2},
			\end{equation*}
			which is uniformly small in $n$ by the uniform second-moment bound of
			Step~4; note that for $K_{n}$ this bound is the uniform-law instance
			of Step~4, so the divergence
			asymptotics~\eqref{eq:cstar}, being law-independent, carries no
			moment hypothesis. Hence $H(K_{n}) \to H(\mathcal{K})$.
		\end{proof}
		
		\begin{lemma}[Flow well-posedness and entropy production]\label{lem:flow-wellposed}
			For every $f_{0} \in \mathcal{D}$ the evolution
			equation~\eqref{eq:flow} has a unique global mild solution in
			$L^{1}(\mathbb{R})$, which remains in $\mathcal{D}$ and preserves
			$\mathcal{D}_{+}$, symmetry, and the median. If moreover
			$f_{0} \in \mathcal{D}_{+}$ and the rank profile
			$\mathfrak{f}_{0} = f_{0}\circ Q_{0}$ satisfies
			\begin{equation}
				\int_{0}^{1}\big|\log \mathfrak{f}_{0}(z)\big|\,dz < \infty,
				\qquad\text{equivalently}\qquad
				\mathbb{E}_{f_{0}}\big|\log f_{0}(X)\big| < \infty,
				\label{eq:entropy-hypothesis}
			\end{equation}
			then the entropy law~\eqref{eq:entropy-law} holds:
			$H(f_{t}) = -t + \log(4/f_{0}(m)) + o(1)$; the discrete analogue
			holds in the sharper form
			$H(f_{n}) = -n\log 2 + H(\mathcal{K}) - \log f_{0}(m) + o(1)$, with
			$\mathcal{K}$ the Koenigs law of Theorem~\ref{thm:limit-law}, so in
			particular $H(f_{n+1}) - H(f_{n}) \to -\log 2$.
		\end{lemma}
		
		\begin{proof}
			Existence is established by the explicit rank construction rather
			than by abstract iteration, since $\mathcal{D}$ is not open in
			$L^{1}$: with $A_{t}$ the increasing bijection of $[0,1]$ defined by
			$\tan(\pi A_{t}(z)) = e^{-t}\tan(\pi z)$ on the correct branches, the
			family $f_{t} := A_{t}'(F_{0})\,f_{0}$ satisfies the evolution
			equation~\eqref{eq:flow} pointwise, by the semigroup identity
			$A_{t+s} = A_{s}\circ A_{t}$ and the transport
			formula~\eqref{eq:transport}, and is a global solution issuing from
			$f_{0}$. Uniqueness among mild $\mathcal{D}$-valued solutions follows
			from the total-variation bound~\eqref{eq:tv-bound}: the vector field
			$f \mapsto \rho[f] - f$ is Lipschitz on $\mathcal{D}$ with constant
			at most $1 + \|w\|_{\infty} + \tfrac12\mathrm{Lip}(w) = 3 + \pi$, so
			the difference of two solutions obeys a Gr\"onwall inequality, which
			requires no openness of the state space. Mass is conserved because
			$\int(\rho[f]-f)\,d\lambda = 0$ by the normalization
			identity~\eqref{eq:normalization}. Positivity follows from the
			multiplicative form of the equation along each fixed $x$: the
			solution satisfies
			\begin{equation*}
				f(x,t)
				= f_{0}(x)\,\exp\Big(\int_{0}^{t}
				\big(w(F_{s}(x)) - 1\big)\,ds\Big),
			\end{equation*}
			which is positive wherever $f_{0}$ is, and this same formula shows
			that support, continuity, and hence membership in $\mathcal{D}_{+}$
			are preserved. The median is conserved because
			$z = \tfrac12$ is a zero of the drift in the
			rank equation~\eqref{eq:flow}, and symmetry is preserved by
			Lemma~\ref{lem:wellposed}(ii) applied to the generator.
			
			For the entropy, observe first that the time-$t$ map of the flow is
			itself a derangetropy operator: by the exact
			solution~\eqref{eq:exact-solution}, $F_{t} = A_{t}\circ F_{0}$ where
			$A_{t}$ is the increasing bijection of $[0,1]$ with
			$\tan(\pi A_{t}(z)) = e^{-t}\tan(\pi z)$ on the correct branches, and
			its derivative
			\begin{equation*}
				w_{t}(z) = A_{t}'(z)
				= \frac{e^{-t}}{\cos^{2}(\pi z) + e^{-2t}\sin^{2}(\pi z)}
			\end{equation*}
			is a kernel, so $f_{t} = \rho_{w_{t}}[f_{0}]$. By the universality
			principle the entropy therefore splits exactly into a
			law-independent and a law-dependent rank integral:
			\begin{equation*}
				H(f_{t})
				= -\int_{0}^{1} w_{t}\log w_{t}\,dz
				\;-\; \int_{0}^{1} w_{t}(z)\,\log\mathfrak{f}_{0}(z)\,dz.
			\end{equation*}
			For the first integral, substitute $z = \tfrac12 + e^{-t}v$: with
			$\varepsilon = e^{-t}$,
			\begin{equation*}
				\varepsilon\,w_{t}\big(\tfrac12+\varepsilon v\big)
				= \frac{\varepsilon^{2}}
				{\sin^{2}(\pi\varepsilon v) + \varepsilon^{2}\cos^{2}(\pi\varepsilon v)}
				\;\longrightarrow\; \frac{1}{1+\pi^{2}v^{2}},
			\end{equation*}
			the Cauchy density of scale $1/\pi$, locally uniformly and under
			the two-sided domination
			\begin{equation*}
				g_{\varepsilon}(v)
				:= \varepsilon\,w_{t}\big(\tfrac12+\varepsilon v\big)
				\;\le\; \min\Big(1,\ \frac{1}{4v^{2}}\Big)
				\qquad \big(|\varepsilon v| \le \tfrac12\big),
			\end{equation*}
			since the denominator of $g_{\varepsilon}$ is at least
			$\varepsilon^{2}$ and, by the chord bound $\sin(\pi u) \ge 2u$ on
			$[0,\tfrac12]$, at least $4\varepsilon^{2}v^{2}$; the dominating
			function is integrable and, since $x \mapsto x|\log x|$ is
			increasing near zero, dominates $|g_{\varepsilon}\log
			g_{\varepsilon}|$ by an integrable function as well, so the entropy
			integral converges and the first integral equals
			$-t + H(\mathrm{Cauchy}(1/\pi)) + o(1) = -t + \log 4 + o(1)$. For the
			second integral, $w_{t}\,dz$ concentrates at $z = \tfrac12$ at scale
			$e^{-t}$, while on the far region $|z - \tfrac12| \ge \delta$ one has
			$\cos^{2}(\pi z) \ge \sin^{2}(\pi\delta)$ and hence the uniform bound
			\begin{equation*}
				w_{t}(z) \;\le\; \frac{e^{-t}}{\sin^{2}(\pi\delta)}.
			\end{equation*}
			Split the integral at $|z - \tfrac12| = \delta$. On the near part,
			$\mathfrak{f}_{0}$ is continuous and positive at $\tfrac12$
			(automatic for $f_{0} \in \mathcal{D}_{+}$), hence bounded there for
			small $\delta$, and the concentration of $w_{t}$ gives convergence to
			$\log\mathfrak{f}_{0}(\tfrac12) = \log f_{0}(m)$; the far part is
			$O\big(e^{-t}\int_{0}^{1}|\log\mathfrak{f}_{0}|\,dz\big) \to 0$ by
			the displayed bound and the
			hypothesis~\eqref{eq:entropy-hypothesis}. Altogether
			\begin{equation*}
				H(f_{t}) = -t + \log 4 - \log f_{0}(m) + o(1),
			\end{equation*}
			which is the entropy law~\eqref{eq:entropy-law}, and which is
			consistent with the Lorentzian limit~\eqref{eq:lorentz-limit}, since
			the Cauchy law of scale $\gamma = 1/\pi f_{0}(m)$ has entropy
			$\log(4\pi\gamma) = \log(4/f_{0}(m))$. For the discrete statement, write $w_{n} := (A^{\circ n})'$, so
			that $f_{n} = \rho_{w_{n}}[f_{0}]$ by the composition law and the
			entropy splits as before:
			\begin{equation*}
				H(f_{n})
				= -\int_{0}^{1}w_{n}\log w_{n}\,dz
				- \int_{0}^{1}w_{n}(z)\,\log\mathfrak{f}_{0}(z)\,dz.
			\end{equation*}
			The first integral is exactly the negative of the divergence
			$D(f_{n}\Vert f_{0})$, since the rank of $f_{0}$ pushes $f_{0}$
			forward to the uniform law, and the exact
			identity~\eqref{eq:divergence-entropy} together with the entropy
			convergence $H(K_{n}) \to H(\mathcal{K})$ of Step 6 of the proof of
			Theorem~\ref{thm:limit-law} gives
			\begin{equation*}
				-\int_{0}^{1}w_{n}\log w_{n}\,dz
				= -n\log 2 + H(K_{n})
				= -n\log 2 + H(\mathcal{K}) + o(1).
			\end{equation*}
			For the second integral, the measures $w_{n}\,dz$ concentrate at
			the median rank, since
			$\int_{|z-1/2|\le\delta}w_{n}\,dz
			= A^{\circ n}(\tfrac12+\delta) - A^{\circ n}(\tfrac12-\delta)
			\to 1$, while on the far region the orbit monotonicity
			$A^{\circ j}(z) \le A^{\circ j}(\tfrac12-\delta) =: u_{j}
			\downarrow 0$ for $z \le \tfrac12 - \delta$, with $w$ increasing on
			$[0,\tfrac12]$, gives the uniform geometric bound
			\begin{equation*}
				\sup_{z \le \frac12-\delta}\, w_{n}(z)
				\;\le\; \prod_{j=0}^{n-1}w(u_{j})
				\;\le\; C_{\delta}\,2^{-(n-n_{0}(\delta))},
			\end{equation*}
			because $w(u_{j}) \le \tfrac12$ once $u_{j} \le \tfrac16$, and
			symmetrically on the right; hence, exactly as in the continuous
			case, the near part converges to
			$\log\mathfrak{f}_{0}(\tfrac12) = \log f_{0}(m)$ and the far part
			vanishes by the hypothesis~\eqref{eq:entropy-hypothesis}.
			Altogether
			\begin{equation*}
				H(f_{n}) = -n\log 2 + H(\mathcal{K}) - \log f_{0}(m) + o(1),
			\end{equation*}
			which is the sharper discrete law, and differencing gives
			$H(f_{n+1}) - H(f_{n}) \to -\log 2$.
		\end{proof}
		
		\begin{lemma}[Fisher metric of the Cauchy family]\label{lem:cauchy-metric}
			For the family $\mathrm{Cauchy}(m,\gamma)$ with density
			$f_{m,\gamma}(x) = \frac{\gamma}{\pi\left[(x-m)^{2}+\gamma^{2}\right]}$,
			the Fisher information matrix is
			$I_{mm} = I_{\gamma\gamma} = 1/2\gamma^{2}$, $I_{m\gamma} = 0$, so the
			Fisher--Rao metric is the hyperbolic
			metric~\eqref{eq:hyperbolic-metric}.
		\end{lemma}
		
		\begin{proof}
			Substitute $x - m = \gamma\tan\theta$ with
			$\theta \in (-\tfrac\pi2, \tfrac\pi2)$; the law of $\theta$ is uniform
			with density $1/\pi$. The scores become
			\begin{align*}
				\partial_{m}\log f
				&= \frac{2(x-m)}{(x-m)^{2}+\gamma^{2}}
				= \frac{2\sin\theta\cos\theta}{\gamma}
				= \frac{\sin 2\theta}{\gamma},\\
				\partial_{\gamma}\log f
				&= \frac1\gamma - \frac{2\gamma}{(x-m)^{2}+\gamma^{2}}
				= -\frac{\cos 2\theta}{\gamma}.
			\end{align*}
			Hence
			$I_{mm} = \frac{1}{\pi\gamma^{2}}
			\int_{-\pi/2}^{\pi/2}\sin^{2}2\theta\,d\theta = \frac{1}{2\gamma^{2}}$,
			$I_{\gamma\gamma} = \frac{1}{\pi\gamma^{2}}
			\int_{-\pi/2}^{\pi/2}\cos^{2}2\theta\,d\theta = \frac{1}{2\gamma^{2}}$,
			and $I_{m\gamma} = 0$ because the cross integrand
			$\sin2\theta\cos2\theta$ is odd. A metric
			$a(dm^{2}+d\gamma^{2})/\gamma^{2}$ on the upper half-plane has
			constant Gaussian curvature $-1/a$; here $a = \tfrac12$, giving
			curvature $-2$, vertical lines are geodesics, and the curve
			$\gamma(t) = \gamma_{0}e^{-t}$ has speed
			$\sqrt{a}\,|\dot\gamma|/\gamma = 1/\sqrt2$.
		\end{proof}
		
		\begin{lemma}[Curvature computations]\label{lem:curvature-comps}
			(i) The logistic law $F(x) = (1+e^{-ax})^{-1}$ and the exponential
			law of rate $a$ have constant Schwarzian $\{F,x\} = -a^{2}/2$.
			(ii) The standard Gaussian has $\{F,x\} = -1 - x^{2}/2$.
			(iii) A $C^{3}$ law with $f(x) = c_{0}x^{-p}(1+o(1))$ and derivatives
			of matching order has
			$\{F,x\} = \tfrac{p(2-p)}{2x^{2}}(1+o(1))$, and hence the curvature
			asymptotics~\eqref{eq:singularity}.
			(iv) The positive-curvature laws of
			Theorem~\ref{thm:trichotomy} satisfy
			$\{F,x\} = +2\pi^{2}(F')^{2}$, and the hyperbolic secant
			steady state~\eqref{eq:sech} satisfies the closure
			identity~\eqref{eq:sech-closure}.
		\end{lemma}
		
		\begin{proof}
			(i) With $u = e^{ax}$, $F = u/(1+u)$ is a M\"obius function of $u$,
			so by the cocycle law $\{F,x\} = \{u,x\} = \{e^{ax},x\}$, and a
			direct computation gives $\{e^{ax},x\} = -a^{2}/2$; the exponential
			law $F = 1 - e^{-ax}$ is M\"obius in $e^{-ax}$ and yields the same
			value.
			(ii) For $F' = \phi$ the standard normal density,
			$\phi'/\phi = -x$ and $\phi''/\phi = x^{2}-1$, so
			\begin{equation*}
				\{F,x\}
				= \frac{\phi''}{\phi} - \frac32\Big(\frac{\phi'}{\phi}\Big)^{2}
				= (x^{2}-1) - \frac32 x^{2}
				= -1 - \frac{x^{2}}{2}.
			\end{equation*}
			(iii) With $F' = f \sim c_{0}x^{-p}$ in the smooth sense,
			$F''/F' \sim -p/x$ and $F'''/F' \sim p(p+1)/x^{2}$, so
			\begin{equation*}
				\{F,x\}
				\sim \frac{p(p+1) - \tfrac32 p^{2}}{x^{2}}
				= \frac{p(2-p)}{2x^{2}};
			\end{equation*}
			dividing by $f^{2} \sim c_{0}^{2}x^{-2p}$ gives the
			asymptotics~\eqref{eq:singularity}. At $p = 2$ the leading
			coefficient vanishes, and a direct expansion of the Schwarzian for
			the second-order tail~\eqref{eq:second-order-tail} gives
			\begin{equation*}
				\{F,x\} = \alpha b(\alpha-1)\,x^{-2-\alpha}\big(1+o(1)\big),
			\end{equation*}
			whence the second-order curvature
			asymptotics~\eqref{eq:second-order-curvature}.
			(iv) For $\tanh(\pi(F-\tfrac12)) = (x-m)/\gamma$ the cocycle law with
			$\{\tanh(\pi(F-\tfrac12)),F\} = -2\pi^{2}$ gives
			$0 = -2\pi^{2}(F')^{2} + \{F,x\}$, as claimed. For the hyperbolic
			secant profile of the steady state~\eqref{eq:sech},
			$\sin(\pi F^{*}) = \operatorname{sech}(x/\sqrt D)$ by the derivative
			formula $\mathrm{gd}'(s) = \operatorname{sech}(s)$ and
			$\cos(\mathrm{gd}(s)) = \operatorname{sech}(s)$; squaring and using
			$f^{*} = \operatorname{sech}(x/\sqrt D)/\pi\sqrt D$ yields
			\begin{equation*}
				2\sin^{2}(\pi F^{*})
				= 2\operatorname{sech}^{2}\!\Big(\frac{x}{\sqrt D}\Big)
				= 2\pi^{2}D\, f^{*2},
			\end{equation*}
			which is the
			closure~\eqref{eq:sech-closure}; substituting the closure into the
			stationary form of the
			equation~\eqref{eq:flow-diffusion} and using
			$(\operatorname{sech})'' = \operatorname{sech}
			- 2\operatorname{sech}^{3}$ verifies the steady state.
		\end{proof}

		\begin{lemma}[Symmetric square]\label{lem:symsq}
			Let $U \in C^{0}(0,1)$ and let $\psi_{1}, \psi_{2}$ be solutions of
			the Hill equation $\psi'' + \tfrac12 U\psi = 0$ with nonzero
			Wronskian $W$. Then every product $\psi_{i}\psi_{j}$ solves the
			symmetric-square equation~\eqref{eq:symmetric-square}, the three
			products $\psi_{1}^{2}, \psi_{1}\psi_{2}, \psi_{2}^{2}$ form a basis
			of its solution space, and
			$\mathcal{L}_{U}\xi = \delta_{\xi}U$ in the notation of the
			infinitesimal coadjoint action~\eqref{eq:inf-coadjoint} at central
			charge one.
		\end{lemma}
		
		\begin{proof}
			Write $q = -U/2$, so $\psi_{i}'' = q\psi_{i}$. For
			$\xi = \psi_{i}\psi_{j}$,
			\begin{equation*}
				\xi' = \psi_{i}'\psi_{j} + \psi_{i}\psi_{j}',
				\qquad
				\xi'' = 2q\xi + 2\psi_{i}'\psi_{j}',
				\qquad
				\xi''' = 2q'\xi + 2q\xi' + 2q\big(\psi_{i}\psi_{j}' +
				\psi_{i}'\psi_{j}\big)
				= 2q'\xi + 4q\xi',
			\end{equation*}
			which is $\xi''' + 2U\xi' + U'\xi = 0$. For independence, a direct
			computation of the Wronskian-type determinant of the three products
			at a point where $\psi_{2} \neq 0$ gives $2W^{3} \neq 0$. The
			identification with the coadjoint variation is the coincidence of
			the two displayed expressions
			$\xi''' + 2U\xi' + U'\xi$ and $\xi U' + 2\xi' U + \xi'''$.
		\end{proof}
		
		\renewcommand{\restatename}{Theorem~\ref{thm:weights}}%
		\begin{restatement}[Tails as conformal weights]
			Let $f \in \mathcal{D}_{+}^{2}$ and suppose the rank profile
			satisfies
			$\mathfrak{f}(z) = (1-z)^{2\beta}L\big(\tfrac{1}{1-z}\big)$ near
			$z = 1$, with $L$ slowly varying with smooth asymptotics, meaning
			that $L$ is eventually smooth with
			$s^{k}L^{(k)}(s)/L(s) \to 0$ for $k = 1, 2$, so that the asymptotic
			relations below may be differentiated twice; this covers
			the regularly varying tails of Theorem~\ref{thm:singularity}, for
			which $\beta = \tfrac{p}{2(p-1)}$, and exponential-or-faster tails,
			for which $\beta = \tfrac12$. Then
			\begin{equation*}
				U_{f}(z) = \frac{c_{+} + o(1)}{(1-z)^{2}},
				\qquad
				c_{+} = 2\beta(1-\beta)
				= \frac{p(p-2)}{2(p-1)^{2}} \ \text{ in the power case},
			\end{equation*}
			and symmetrically at the left endpoint. In particular:
			\begin{enumerate}[label=\normalfont(\roman*), leftmargin=*]
				\item always $c_{\pm} \le \tfrac12$, and no law whatsoever can
				satisfy $U \ge c\,(1-z)^{-2}$ near an endpoint with $c > \tfrac12$;
				\item $c_{+} = 0$ exactly for quadratic tails; $c_{+} = \tfrac12$
				for exponential-or-faster tails; $c_{+} \in (0, \tfrac12)$ for
				$p > 2$; and $c_{+} < 0$ for $1 < p < 2$;
				\item the weights $c_{\pm}$ are invariants of the boundary-regular
				coadjoint action: every distortion extending to a
				$C^{3}$-diffeomorphism of the closed interval with nonvanishing
				endpoint derivatives, and every flow of a kernel that is $C^{2}$
				up to the endpoints, preserves them; boundary regularity
				is essential, since the interior distortion $B(z) = z^{2}$ moves the
				left weight of the uniform law from $0$ to $-\tfrac32$.
			\end{enumerate}
		\end{restatement}
		
		\begin{proof}[Proof of Theorem~\ref{thm:weights}]
			Write $\psi = \sqrt{\mathfrak{f}} = (1-z)^{\beta}\ell$ with $\ell$
			slowly varying with smooth asymptotics; the smoothness hypotheses
			permit differentiation of the asymptotic relations, and
			\begin{equation*}
				U = -\,\frac{2\psi''}{\psi}
				= -2\beta(\beta - 1)\,(1-z)^{-2}\big(1 + o(1)\big)
				= \frac{2\beta(1-\beta) + o(1)}{(1-z)^{2}},
			\end{equation*}
			slowly varying corrections entering at lower order. The exponent:
			for a power tail $f \sim c_{0}x^{-p}$,
			$1 - F \sim c_{0}x^{1-p}/(p-1)$, so
			$x \sim \big((p-1)(1-z)/c_{0}\big)^{1/(1-p)}$ and
			$\mathfrak{f} = f\circ Q \sim
			\mathrm{const}\cdot(1-z)^{p/(p-1)}$, giving $2\beta = p/(p-1)$ and
			the coefficient of the formula~\eqref{eq:weight}; for an
			exponential-type tail, $1 - F \sim f/a$, so
			$\mathfrak{f} \sim a(1-z)$ and $\beta = \tfrac12$, and the same
			reduction covers faster tails through the slowly varying factor.
			
			For part (i), the function $\beta \mapsto 2\beta(1-\beta)$ attains
			its maximum $\tfrac12$ at $\beta = \tfrac12$. For the absolute
			statement, suppose $U \ge c\,(1-z)^{-2}$ near $z = 1$ with
			$c > \tfrac12$. The Euler equation
			$\psi'' + \tfrac{c}{2}(1-z)^{-2}\psi = 0$ has indicial roots
			$\tfrac12\big(1 \pm \sqrt{1 - 2c}\big)$, which are non-real for
			$c > \tfrac12$, so every solution of the Euler equation oscillates
			near the endpoint; by Sturm comparison~\cite{hartman1982}, so does
			every solution of the Hill equation of $U$, contradicting the
			disconjugacy required by Theorem~\ref{thm:realization}. This is the
			optimality of the Hardy inequality on the interval. Part (ii) is
			the formula $c_{+} = 2\beta(1-\beta)$ with
			$\beta = \tfrac{p}{2(p-1)}$ decreasing from infinity to $\tfrac12$
			as $p$ runs over $(1,\infty)$ and passing through $1$ at $p = 2$.
			For part (iii), under a distortion $B$ that is $C^{3}$ up to the
			endpoint with $B(1) = 1$ and $B'(1) > 0$, one has
			$B(z) = 1 - B'(1)(1-z) + O((1-z)^{2})$, so
			\begin{equation*}
				\big(U\circ B\big)(B')^{2}
				\sim c_{+}\,\big(B'(1)(1-z)\big)^{-2}\,B'(1)^{2}
				= c_{+}\,(1-z)^{-2},
			\end{equation*}
			while the Schwarzian $\{B\}$ stays bounded near the endpoint; hence
			the leading coefficient is unchanged, and kernel flows are curves of
			such maps.
		\end{proof}
		
		\begin{lemma}[Circular realization]\label{lem:circular}
			A continuous one-periodic $U$ is the rank Hill potential of a
			circular law if and only if the Hill equation
			$\psi'' + \tfrac12 U\psi = 0$ has a positive one-periodic solution,
			if and only if zero is the lowest periodic eigenvalue of the
			operator $-\,d^{2}/dz^{2} - \tfrac12 U$ on the rank circle. The
			circumference is recovered as $\ell = \int_{0}^{1}\psi^{-2}$ once
			the scale of $\psi$ is fixed, the
			law as $f = \psi^{2}\circ F$, and unit mass is automatic as in the
			identity~\eqref{eq:automatic-mass}.
		\end{lemma}
		
		\begin{proof}
			If $f$ is a circular law, $\psi = \sqrt{f\circ Q}$ is one-periodic,
			positive, and solves the Hill equation by
			Proposition~\ref{prop:hill}. Conversely, a positive periodic
			solution reconstructs the lifted quantile through the
			formula~\eqref{eq:reconstruction}, with
			$Q(z+1) = Q(z) + \int_{0}^{1}\psi^{-2}$, hence a circular law of the
			stated circumference, unit mass on one period following as in the
			identity~\eqref{eq:automatic-mass}. For the spectral statement: the
			periodic ground state of a Hill operator is positive and simple,
			and conversely a positive periodic solution at energy zero must be
			the ground state, eigenfunctions of higher periodic eigenvalues
			changing sign~\cite{magnus1966}.
		\end{proof}
		
		\renewcommand{\restatename}{Theorem~\ref{thm:sync}}%
		\begin{restatement}[Synchronization and exact sampling]
			For the canonical shape $w = 2\sin^{2}(\pi\cdot)$:
			\begin{enumerate}[label=\normalfont(\roman*), leftmargin=*]
				\item for every $z$, the rank orbit $z_{n}$ converges almost surely
				to a limit in $\{0,1\}$ with $\mathbb{P}(z_{\infty} = 1) = z$, and
				the conditional variance identity
				\begin{equation*}
					\operatorname{Var}\big(z_{n+1} \,\big|\, \mathcal{F}_{n}\big)
					= \frac{\sin^{2}(\pi z_{n})}{2\pi^{2}},
					\qquad
					\sum_{n \ge 0}\mathbb{E}\big[\sin^{2}(\pi z_{n})\big]
					= 2\pi^{2}\,z(1-z),
				\end{equation*}
				holds exactly;
				\item almost surely there is a random threshold $\tau$ such that
				$S_{n}(z) \to 0$ for all $z < \tau$ and $S_{n}(z) \to 1$ for all
				$z > \tau$, and $\tau$ is uniformly distributed on $(0,1)$;
				\item for every $f \in \mathcal{D}_{+}$, almost surely
				\begin{equation*}
					f_{n} \;\xrightarrow{\ w^{*}\ }\; \delta_{Q(\tau)},
					\qquad\text{and}\qquad
					Q(\tau) \sim f:
				\end{equation*}
				the cascade condenses the law onto a single random atom whose
				distribution is the initial law itself;
				\item for every $z \in (0,1)$, almost surely
				$\tfrac1n\log\min(z_{n}, 1 - z_{n}) \to -\log 2$: absorption
				proceeds at one bit per step.
			\end{enumerate}
		\end{restatement}
		
		Parts (i)--(iii) are proved in Section~\ref{sec:random}; the
		argument below supplies part (iv).
		
		\begin{proof}[Proof of Theorem~\ref{thm:sync}(iv)]
			Write $W_{n} := w(\alpha_{n}) = 2\sin^{2}(\pi\alpha_{n})$, independent
			copies of the step law $W$. The canonical kernel has Lipschitz
			constant $2\pi$, so averaging the bound
			$|w(s+\alpha) - w(\alpha)| \le 2\pi s$ over $s \in [0, z]$ gives
			\begin{equation*}
				\Big|\,\frac{A_{\alpha}(z)}{z} - w(\alpha)\Big| \;\le\; \pi z,
				\qquad z \in (0, 1],
			\end{equation*}
			and the orbit, which stays in $(0,1)$ for all time, satisfies the
			perturbed multiplicative recursion
			\begin{equation*}
				\log z_{n+1} = \log z_{n} + \log\big(W_{n+1} + \eta_{n}\big),
				\qquad |\eta_{n}| \le \pi z_{n}.
			\end{equation*}
			We prove $\tfrac1n\log z_{n} \to -\log 2$ almost surely on the
			absorption event $\{z_{\infty} = 0\}$; on the complementary event
			$\{z_{\infty} = 1\}$, which exhausts the remaining probability by
			part (i), the same argument applies to the reflected orbit
			$1 - z_{n}$, a canonical cascade with phases $-\alpha_{j}$ by
			evenness and periodicity of the kernel; and on each event
			$\min(z_{n}, 1 - z_{n})$ eventually coincides with the vanishing
			coordinate.
			
			First, a crude rate. Fix $\varepsilon \in (0, \tfrac12)$. On the
			absorption event there is a finite random $N$ with
			$z_{n} \le \varepsilon$ for $n \ge N$, whence
			$\log z_{n+1} - \log z_{n} \le \log(W_{n+1} + \pi\varepsilon)$ for
			$n \ge N$; the summands on the right are independent, identically
			distributed, and bounded, so the strong law of large numbers yields
			\begin{equation*}
				\limsup_{n}\ \frac1n\log z_{n}
				\;\le\; \mathbb{E}\log\big(W + \pi\varepsilon\big)
				\;\xrightarrow[\ \varepsilon \downarrow 0\ ]{}\;
				\mathbb{E}\log W = -\log 2,
			\end{equation*}
			the convergence by monotone convergence and the limiting mean the
			rank integral computed in Lemma~\ref{lem:kl}. In particular, for
			each $\delta \in (0, \log 2)$, almost surely
			$z_{n} \le e^{-n(\log 2 - \delta)}$ for all large $n$.
			
			Second, no multiplier is anomalously small: since
			$\mathbb{P}(W \le t) = \tfrac{2}{\pi}\arcsin\sqrt{t/2}
			\le \sqrt{t/2}$, the probabilities
			$\mathbb{P}(W_{n} \le n^{-4})$ are summable, and by the
			Borel--Cantelli lemma, almost surely $W_{n} > n^{-4}$ for all large
			$n$.
			
			On the intersection of these events the perturbation is summable
			relative to the multiplier,
			\begin{equation*}
				\frac{|\eta_{n}|}{W_{n+1}}
				\;\le\; \pi\,(n+1)^{4}\,e^{-n(\log 2 - \delta)}
				\qquad\text{for all large } n,
			\end{equation*}
			so writing
			$\log(W_{n+1} + \eta_{n}) = \log W_{n+1} + \zeta_{n}$ with
			$|\zeta_{n}| \le 2|\eta_{n}|/W_{n+1}$ once the ratio drops below
			$\tfrac12$, the error series $\sum_{n}|\zeta_{n}|$ converges almost
			surely. The strong law of large numbers for the integrable steps
			$\log W_{n}$, of mean $-\log 2$, then gives
			\begin{equation*}
				\frac1n\log z_{n}
				= \frac1n\log z_{0}
				+ \frac1n\sum_{j=1}^{n}\log W_{j}
				+ \frac1n\sum_{j<n}\zeta_{j}
				\;\longrightarrow\; -\log 2
			\end{equation*}
			almost surely, which is the assertion.
		\end{proof}
		
		\renewcommand{\restatename}{Theorem~\ref{thm:chaos}}%
		\begin{restatement}[Nondegenerate chaos]
			Under the subcriticality condition~\eqref{eq:subcritical}:
			\begin{enumerate}[label=\normalfont(\roman*), leftmargin=*]
				\item for every $z$, $D_{n}(z)$ is a mean-one martingale, and
				$\mu_{n} \to \mu$ almost surely weakly for a random Borel measure
				$\mu$ on $[0,1]$;
				\item the two-point function is log-correlated: uniformly in $n$
				and in pairs with $|z - z'| \le \tfrac12$,
				\begin{equation*}
					\mathbb{E}\big[D_{n}(z)\,D_{n}(z')\big]
					\;\le\; C_{b,\lambda}\,|z - z'|^{-\bar\gamma^{2}},
				\end{equation*}
				together with a power-law lower bound, uniform in $n$, governed by
				the complementary distortion constant $\ln b + \Lambda$; in the
				weak-tempering regime the lower exponent is
				positive and matches the
				exponent~\eqref{eq:effective-temperature} to leading order;
				\item $\sup_{n}\mathbb{E}[\mu_{n}(I)^{2}]
				\le C_{b,\lambda}\,|I|^{2-\bar\gamma^{2}}$ for every
				interval $I$, so the martingale $\mu_{n}(I)$ is uniformly
				integrable,
				$\mathbb{E}\,\mu(I) = |I|$ for every interval and
				$\mathbb{P}(\mu \neq 0) = 1$;
				\item almost surely $\mu$ has full support, no atoms, and vanishes
				on every Borel set of Hausdorff dimension less than
				$1 - \bar\gamma^{2}$;
				\item $S_{\infty}(z) := \mu([0,z])$ is almost surely a strictly
				increasing H\"older homeomorphism of $[0,1]$, with any exponent
				below $(1-\bar\gamma^{2})/2$.
			\end{enumerate}
		\end{restatement}
		
		\begin{proof}[Proof of Theorem~\ref{thm:chaos}]
			Throughout, $\Delta := |z - z'|$ with $0 < \Delta \le \tfrac12$,
			$\delta_{j} := |S_{j}(z') - S_{j}(z)|$, and
			$\|\delta\| := \min(\delta, 1 - \delta)$ is the distance of a
			separation to the nearest integer; all constants depend only on $b$
			and $\lambda$. Applying the distortion
			bound~\eqref{eq:distortion-bound} to the pair and to the
			complementary pair of arcs gives
			\begin{equation}
				\|\delta_{j}\| \;\ge\; e^{-j\Lambda}\,\Delta
				\qquad\text{and}\qquad
				\delta_{j} \;\le\; e^{j\Lambda}\,\Delta.
				\label{eq:separation-bounds}
			\end{equation}
			
			We first record the pair identity~\eqref{eq:pair-identity} and its
			block average. For fixed $k$ and uniform $\alpha$, expanding the
			product of the two tempered kernels and using that
			$\cos(2\pi k(y+\alpha))$, $\cos(2\pi k(y'+\alpha))$, and the
			sum-frequency term average to zero over the phase, only the
			difference-frequency term survives, with the factor
			$\tfrac{\lambda^{2}}{2}\cos(2\pi k(y - y'))$; this is the
			identity~\eqref{eq:pair-identity}. Conditioning successively on
			$\mathcal{F}_{n-1}, \ldots, \mathcal{F}_{0}$ therefore gives the
			exact formula
			\begin{equation}
				\mathbb{E}\big[D_{n}(z)\,D_{n}(z')\big]
				= \mathbb{E}\prod_{j=1}^{n}
				\Big(1 + \frac{\lambda^{2}}{2}\,r_{j}(\delta_{j-1})\Big),
				\qquad
				r_{j}(\delta) := \frac{1}{b^{j}}
				\sum_{k=b^{j}}^{2b^{j}-1}\cos(2\pi k\delta),
				\label{eq:exact-twopoint}
			\end{equation}
			and the geometric-sum bound
			$\big|\sum_{k=a}^{a+m-1}e^{2\pi i k\delta}\big|
			\le |\sin\pi\delta|^{-1} \le (2\|\delta\|)^{-1}$ yields
			\begin{equation}
				|r_{j}(\delta)| \;\le\;
				\min\Big(1,\ \frac{b^{-j}}{2\|\delta\|}\Big).
				\label{eq:dirichlet}
			\end{equation}
			
			For the upper bound of part (ii), combine the
			inequality~\eqref{eq:dirichlet} with the first separation
			bound of the display~\eqref{eq:separation-bounds} and
			$1 + x \le e^{x}$: each factor of the
			formula~\eqref{eq:exact-twopoint} is at most
			\begin{equation*}
				\exp\Big(\frac{\lambda^{2}}{2}\,
				\min\Big(1,\ \frac{e^{j\Lambda}b^{-j}}{2\Delta}\Big)\Big).
			\end{equation*}
			Let
			$j^{*} := \lceil \ln(1/\Delta)/(\ln b - \Lambda)\rceil$. For
			$j \le j^{*}$ bound the exponent by $\tfrac{\lambda^{2}}{2}$; for
			$j > j^{*}$,
			\begin{equation*}
				\frac{e^{j\Lambda}b^{-j}}{2\Delta}
				= \frac{e^{-j(\ln b - \Lambda)}}{2\Delta}
				\;\le\; \tfrac12\,e^{-(j - j^{*})(\ln b - \Lambda)},
			\end{equation*}
			a summable geometric tail. The total exponent is therefore at most
			$\tfrac{\lambda^{2}}{2}\,j^{*} + C$, which is the two-point
			bound~\eqref{eq:twopoint} with the
			exponent~\eqref{eq:effective-temperature}, uniformly in $n$. By
			periodicity of the kernels the same argument runs with separations
			measured on the rank circle, so for $\Delta > \tfrac12$ the bound
			holds with $1 - \Delta$ in place of $\Delta$; both forms are used
			below.
			
			For the lower bound, fix $t \in (0,1)$. If
			$4\pi b^{j}e^{j\Lambda}\Delta \le \arccos t$, then by the second
			separation bound of the display~\eqref{eq:separation-bounds} every
			frequency $k < 2b^{j}$ of the $j$-th block satisfies
			$2\pi k\delta_{j-1} \le \arccos t$, so
			$r_{j}(\delta_{j-1}) \ge t$ pathwise; this holds for all
			$j \le j_{t} := \lfloor \ln(c_{t}/\Delta)/(\ln b + \Lambda)\rfloor$
			with $c_{t} := \arccos(t)/(4\pi)$. For
			$j_{t} < j \le j^{*}$ use the crude bound
			$1 + \tfrac{\lambda^{2}}{2}r_{j} \ge 1 - \tfrac{\lambda^{2}}{2}
			\ge \tfrac12$, and for $j > j^{*}$ the geometric tail above bounds
			the product below by a constant. Hence, uniformly in $n$,
			\begin{equation}
				\begin{gathered}
					\mathbb{E}\big[D_{n}(z)\,D_{n}(z')\big]
					\;\ge\; c_{t}'\,\Delta^{-\gamma_{-}^{2}(t)},\\
					\gamma_{-}^{2}(t)
					= \frac{\log\big(1 + t\tfrac{\lambda^{2}}{2}\big)}{\ln b + \Lambda}
					- \log\Big(\frac{1}{1 - \tfrac{\lambda^{2}}{2}}\Big)
					\Big(\frac{1}{\ln b - \Lambda} - \frac{1}{\ln b + \Lambda}\Big),
				\end{gathered}
				\label{eq:lower-exponent}
			\end{equation}
			the second term the cost of the window of scales that the
			worst-case distortion leaves ambiguous. Since
			$\Lambda = \Lambda(\lambda)$ vanishes linearly as
			$\lambda \to 0$, the window cost is of third order while the
			leading term is $\tfrac{\lambda^{2}}{2(\ln b + \Lambda)}$, so the
			exponent~\eqref{eq:lower-exponent} is governed by the complementary
			distortion constant $\ln b + \Lambda$ and matches the
			exponent~\eqref{eq:effective-temperature} to leading order in the
			weak-tempering limit, as asserted.
			
			Part (i). For fixed $y$ the phase average of
			$\cos(2\pi K(y + \alpha))$ vanishes at every integer frequency, so
			$\mathbb{E}[v_{j}(y)\,|\,\mathcal{F}_{j-1}, K_{j}] = 1$ pathwise:
			each $D_{n}(z)$ is a mean-one martingale and each rank orbit
			$S_{n}(z)$ a bounded martingale. Applying martingale convergence
			simultaneously to all rational ranks gives an almost sure
			nondecreasing limit function on the rationals; its right-continuous
			extension $S_{\infty}$ has $S_{\infty}(0) = 0$ and
			$S_{\infty}(1) = 1$, since every $A_{v_{j}}$ fixes the endpoints,
			and defines a random probability measure $\mu$ with
			$\mu([0, z]) = S_{\infty}(z)$ at continuity points. Convergence of
			the distribution functions $S_{n}$ to $S_{\infty}$ on a dense set
			is weak convergence $\mu_{n} \to \mu$, almost surely.
			
			Continuity and H\"older regularity, toward parts (iv) and (v). For
			$z < z'$ with $L := z' - z \le \tfrac12$, integrating the two-point
			bound~\eqref{eq:twopoint} over $[z, z']^{2}$ gives, uniformly in
			$n$,
			\begin{equation*}
				\mathbb{E}\big[(S_{n}(z') - S_{n}(z))^{2}\big]
				= \int_{z}^{z'}\!\!\int_{z}^{z'}
				\mathbb{E}\big[D_{n}(u)D_{n}(v)\big]\,du\,dv
				\;\le\; C\,L^{2 - \bar\gamma^{2}},
			\end{equation*}
			and the same bound holds trivially for $L > \tfrac12$ after
			enlarging $C$. By Fatou's lemma the limit increments obey the same
			moment bound, and since $2 - \bar\gamma^{2} > 1$ under the
			subcriticality condition~\eqref{eq:subcritical}, the Kolmogorov
			continuity criterion applies to $S_{\infty}$ restricted to the
			rationals, giving an almost surely H\"older-continuous modification
			of every exponent below $\tfrac12(1 - \bar\gamma^{2})$; a monotone
			function whose restriction to a dense set is H\"older obeys the
			same H\"older bound outright, so $S_{\infty}$ itself is continuous
			with this exponent, and $\mu$ has no atoms.
			
			Part (iii). The two-point bound~\eqref{eq:twopoint}, in its two
			circle-distance forms, is integrable over the square since
			$\bar\gamma^{2} < 1$, so
			$\sup_{n}\mathbb{E}[\mu_{n}(I)^{2}] < \infty$ for every interval
			$I$: the martingale $\mu_{n}(I)$ is uniformly integrable and
			converges in $L^{1}$ as well as almost surely, and since the
			endpoints of $I$ are almost surely not atoms of $\mu$, the limit is
			$\mu(I)$ and $\mathbb{E}\,\mu(I) = |I|$. That
			$\mathbb{P}(\mu \neq 0) = 1$ holds in the strongest form: the
			total mass is identically one, in keeping with the automatic-mass
			identity~\eqref{eq:automatic-mass} theme, because every $\mu_{n}$
			is a probability measure.
			
			Strict monotonicity, toward parts (iv) and (v). Fix $z < z'$ and
			let $d_{n} := S_{n}(z') - S_{n}(z) \in (0, 1)$. Integrating the
			kernel~\eqref{eq:lacunary-kernel} over the current image interval
			and converting the difference of sines to a product gives the exact
			recursion
			\begin{equation}
				d_{n+1} = d_{n}\Big(1 + \lambda\,\sigma_{n}\cos\theta_{n+1}\Big),
				\qquad
				\sigma_{n} := \frac{\sin(\pi K_{n+1} d_{n})}{\pi K_{n+1} d_{n}},
				\label{eq:pair-recursion}
			\end{equation}
			where, given $\mathcal{F}_{n}$ and $K_{n+1}$, the angle
			$\theta_{n+1}$ is uniform on $[0, 2\pi)$, by the same
			re-uniformization as in the chain~\eqref{eq:separation}. The
			logarithm of the resolution variable $b^{n}d_{n}$,
			$X_{n} := \ln(b^{n}d_{n})$, has conditional drift
			\begin{equation*}
				\begin{gathered}
					\mathbb{E}\big[X_{n+1} - X_{n}\,\big|\,
					\mathcal{F}_{n}, K_{n+1}\big]
					= \ln b + \varphi(\lambda\sigma_{n}),\\
					\varphi(a) := \frac{1}{2\pi}\int_{0}^{2\pi}
					\log(1 + a\cos\theta)\,d\theta
					= \log\frac{1 + \sqrt{1 - a^{2}}}{2}.
				\end{gathered}
			\end{equation*}
			Since $|\sigma_{n}| \le 1$ and $\varphi$ is even and decreasing in
			$|a|$, the drift is at least $\ln b - c(\lambda)$ with
			$c(\lambda) := \log\big(2/(1 + \sqrt{1 - \lambda^{2}})\big)$; and
			$c(\lambda) \le \Lambda$, since squaring the equivalent inequality
			$2(1 - \lambda) \le 1 + \sqrt{1 - \lambda^{2}}$ reduces it to
			$5\lambda^{2} \le 4\lambda$ when
			$\lambda \le \tfrac12$ and it is trivial otherwise. The
			subcriticality condition~\eqref{eq:subcritical} thus makes the
			drift at least $2\delta_{0} := \ln b - c(\lambda) > 0$, uniformly
			in the state; only this first half of the condition is used here,
			and in the weaker form $c(\lambda) < \ln b$: resolution outruns the
			worst possible contraction of the tempered kernels. The increments
			of $X_{n}$ are bounded, so by the strong law for martingale
			differences, almost surely $X_{n} \ge \delta_{0}n$ for all large
			$n$. But then
			$|\sigma_{n}| \le (\pi K_{n+1}d_{n})^{-1}
			\le (\pi b)^{-1}e^{-X_{n}}$ decays geometrically, the logarithmic
			increments of the recursion~\eqref{eq:pair-recursion} become
			absolutely summable, and $\log d_{n}$ converges almost surely to a
			finite limit: $S_{\infty}(z') - S_{\infty}(z) = \lim d_{n} > 0$.
			Applying this to all rational pairs and invoking continuity,
			$S_{\infty}$ is almost surely strictly increasing on $[0, 1]$, so
			$\mu$ has full support, and $S_{\infty}$ is a homeomorphism with
			the H\"older exponent obtained above, which is part (v).
			
			Part (iv), dimension. For $s < 1 - \bar\gamma^{2}$, weak
			convergence of $\mu_{n}\times\mu_{n}$, lower semicontinuity of the
			kernel $|u - v|^{-s}$, and Fatou's lemma give
			\begin{equation*}
				\mathbb{E}\int_{0}^{1}\!\!\int_{0}^{1}
				\frac{\mu(du)\,\mu(dv)}{|u - v|^{s}}
				\;\le\; \liminf_{n}\ \int_{0}^{1}\!\!\int_{0}^{1}
				\frac{\mathbb{E}[D_{n}(u)D_{n}(v)]}{|u - v|^{s}}\,du\,dv
				\;<\; \infty,
			\end{equation*}
			the finiteness because $s + \bar\gamma^{2} < 1$ makes the diagonal
			singularity integrable and the antidiagonal one is handled by the
			circle-distance form of the bound~\eqref{eq:twopoint}. Thus almost
			surely $\mu$ has finite $s$-energy simultaneously for a sequence
			$s \uparrow 1 - \bar\gamma^{2}$, and a measure of finite
			$s$-energy vanishes on every Borel set of Hausdorff dimension less
			than $s$~\cite{mattila1995}, which completes part (iv) and the
			proof.
		\end{proof}
		
		\begin{lemma}[Oscillation calculus]\label{lem:osc-calculus}
			Let $g$ be continuous and real on an interval, $\delta = 2^{-j}$,
			and, for $\delta \le |I|$, let $\Sigma_{g}(\delta, I)$ denote the
			sum of the oscillations of $g$ over the intersections with $I$ of
			the $\delta$-grid cells meeting the interval $I$.
			Then the number $N_{\delta}(I)$ of $\delta$-boxes meeting the graph
			of $g|_{I}$ satisfies
			\begin{equation}
				\frac{1}{\delta}\,\Sigma_{g}(\delta, I)
				\;\le\; N_{\delta}(I)
				\;\le\; 2\Big(\frac{|I|}{\delta} + 2\Big)
				+ \frac{1}{\delta}\,\Sigma_{g}(\delta, I),
				\label{eq:box-sandwich}
			\end{equation}
			so a bound $g \in C^{\alpha}$ gives upper box dimension at most
			$2 - \alpha$, while
			$\Sigma_{g}(\delta_{i}, I) \ge c\,\delta_{i}^{\,\sigma - 1}$ along
			any sequence $\delta_{i} \to 0$ gives dimension at least
			$2 - \sigma$. Moreover, for $2$-periodic $g$, dyadic $M$, and
			concentric intervals $I' \subset I$ with
			$\operatorname{dist}(I', \complement I) \ge \delta_{0} > 0$,
			\begin{equation}
				\big\Vert P_{M}\,g\big\Vert_{L^{1}(I')}
				\;\le\; C\,\delta\,\Sigma_{g}(\delta, I)
				+ C_{A}\,\Vert g\Vert_{\infty}\,(M\delta_{0})^{1-A},
				\qquad \delta = 1/M,
				\label{eq:l1-oscillation}
			\end{equation}
			for every $A$, where $P_{M}$ is any smooth dyadic block at
			frequency scale $M$, in the amplitude or the density convention.
		\end{lemma}
		
		\begin{proof}
			Over a grid cell the graph meets at least
			$\operatorname{osc}/\delta$ and at most
			$2 + \operatorname{osc}/\delta$ vertical boxes, which is the
			sandwich~\eqref{eq:box-sandwich}, a standard box-counting
			argument~\cite{falconer2014}, and the two dimension statements
			follow by taking logarithms. For the
			inequality~\eqref{eq:l1-oscillation}, realize the block as a
			convolution $P_{M}g = g * K_{M}$ with a kernel of integral zero and
			rapid decay, $|K_{M}(y)| \le C_{A}M(1 + M|y|)^{-A}$. Then
			\begin{equation*}
				|P_{M}g(x)| \le \int |K_{M}(y)|\,\big|g(x-y) - g(x)\big|\,dy,
			\end{equation*}
			and integrating over $x \in I'$: a shift by $|y| \le \delta$
			changes $g$ by at most the oscillations of the two grid cells
			involved, so
			$\Vert g(\cdot - y) - g\Vert_{L^{1}(I')} \le
			2\delta\,\Sigma_{g}(\delta, I)$; a shift by
			$\delta \le |y| \le \delta_{0}$ is a chain of at most
			$2|y|/\delta$ such steps inside $I$, so the corresponding integral
			is at most
			$4\,\Sigma_{g}(\delta, I)\int |y|\,|K_{M}(y)|\,dy
			\le C\delta\,\Sigma_{g}(\delta, I)$; and the far tail
			$|y| > \delta_{0}$ contributes at most
			$2\Vert g\Vert_{\infty}\int_{|y|>\delta_{0}}|K_{M}|
			\le C_{A}\Vert g\Vert_{\infty}(M\delta_{0})^{1-A}$.
		\end{proof}
		
		\begin{lemma}[Chart transport]\label{lem:chart-transport}
			Let $f \in \mathcal{D}_{+}$ be continuously differentiable and
			positive on the interior of its support, let $J$ be a compact
			interval there, and let $g$ be continuous near $F(J)$. For any
			positive $h \in C^{1}(J)$, the graphs of $g|_{F(J)}$ and of
			$(g\circ F)\,h|_{J}$ have the same upper box dimension whenever
			either exceeds one.
		\end{lemma}
		
		\begin{proof}
			On $J$ the map $F$ is a bi-Lipschitz $C^{1}$ diffeomorphism onto
			$F(J)$, so grid oscillations of $g\circ F$ at scale $\delta$ are
			sandwiched between oscillation sums of $g$ at the comparable scales
			$c\delta$ and $C\delta$, changing the counts of the
			sandwich~\eqref{eq:box-sandwich} by bounded factors only. For the
			multiplier, with $h \ge c > 0$ and
			$\operatorname{osc}_{I}(h) \le C\delta$,
			\begin{equation*}
				\operatorname{osc}_{I}\big((g\circ F)h\big)
				\;\ge\; c\,\operatorname{osc}_{I}(g\circ F)
				- \Vert g\Vert_{\infty}\,C\delta,
			\end{equation*}
			and symmetrically above, so the oscillation sums differ by at most
			a bounded factor plus $O(1)$, which box counts of order
			$\delta^{-\beta}$ with $\beta > 1$ do not feel.
		\end{proof}
		
		\renewcommand{\restatename}{Lemma~\ref{lem:carpet-smoothing}}%
		\begin{restatement}[Generic smoothing]
			For almost every $s$: $u_{s}$ is continuous, and for every
			$\varepsilon > 0$ there is $C_{s,\varepsilon}$ with
			\begin{equation*}
				\Vert P_{N}u_{s}\Vert_{\infty}
				\le C_{s,\varepsilon}\,N^{-1/2+\varepsilon}
				\quad\text{for all dyadic } N,
			\end{equation*}
			so that $u_{s}$ and $w_{s}$ belong to
			$C^{1/2-\varepsilon}(\mathbb{T}_{2})$ for every $\varepsilon > 0$.
			At every rational $s$ the conclusion fails: $u_{s}$ is a step
			function.
		\end{restatement}
		
		\begin{proof}[Proof of Lemma~\ref{lem:carpet-smoothing}]
			Fix $\varepsilon > 0$. For a.e.\ $s$ the following holds for all
			large dyadic $N$: there is a reduced fraction $a/q$ with
			$|s - a/q| \le 1/(qN)$ and $N^{1-\varepsilon} \le q \le N$. Indeed
			Dirichlet's theorem always provides $q \le N$, while the set of
			$s$ admitting some $q \le N^{1-\varepsilon}$ has measure at most
			$\sum_{q \le N^{1-\varepsilon}} q \cdot 2/(qN)
			\le 2N^{-\varepsilon}$, which is summable along dyadic $N$, so by
			the Borel--Cantelli lemma small denominators eventually never
			occur. On this event the Gauss--Weyl
			estimate~\cite{montgomery1994} gives, uniformly in the linear
			phase $x$,
			\begin{equation*}
				\sup_{x}\Big|\sum_{l=1}^{L}
				e^{2\pi i(l^{2}s' + lx)}\Big|
				\;\le\; C\Big(\frac{L}{\sqrt q} + \sqrt{q\log q}\Big)
				\;\le\; C\,L^{1/2 + \varepsilon}
				\qquad (L \asymp N,\ \text{the only range used}),
			\end{equation*}
			for $s'$ in the same approximation class; the substitution
			$n = 2l+1$ writes the odd-frequency block of the
			series~\eqref{eq:carpet-seed} as such a sum with time parameter
			$4s$, and the map $s \mapsto 4s$ modulo one preserves null sets.
			Abel summation against the monotone coefficients
			$2/(\pi n) \asymp N^{-1}$ on the block then yields the
			bound~\eqref{eq:carpet-blocks}. Summing the blocks shows $u_{s}$
			is a uniform limit of continuous functions; the geometric block
			decay is the H\"older--Besov characterization of
			$C^{1/2-\varepsilon}$~\cite{katznelson2004}, and
			$w_{s} = |u_{s}|^{2}$ is a product of bounded
			$C^{1/2-\varepsilon}$ functions. At rational $s$ the
			revival~\eqref{eq:carpet-revival} is a step function.
		\end{proof}
		
		\renewcommand{\restatename}{Lemma~\ref{lem:block-equi}}%
		\begin{restatement}[Equidistribution of block energy]
			Let $I \subseteq (0,1)$ be an interval. For almost every $s$, along
			dyadic $N \to \infty$:
			\begin{equation*}
				N\,\big\Vert P_{N}u_{s}\big\Vert_{L^{2}(I)}^{2}
				\;\longrightarrow\; \frac{2}{\pi^{2}}\,|I|,
				\qquad
				N\int_{I}\big(\mathrm{Re}\,[e^{i\theta}P_{N}u_{s}]\big)^{2}dz
				\;\longrightarrow\; \frac{1}{\pi^{2}}\,|I|
				\ \text{ uniformly in }\theta,
			\end{equation*}
			while the blocks of the density obey the self-referential law
			\begin{equation*}
				N\,\big\Vert P_{2N}^{\mathbb{Z}}\,w_{s}\big\Vert_{L^{2}(I)}^{2}
				\;\longrightarrow\; \frac{2}{\pi^{2}}\int_{I}w_{s}(z)\,dz,
			\end{equation*}
			where $P_{2N}^{\mathbb{Z}}$ retains the spatial frequencies
			$|k| \in [2N, 4N)$ of $w_{s}$.
		\end{restatement}
		
		\begin{proof}[Proof of Lemma~\ref{lem:block-equi}]
			Throughout, $P_{N}$ may be taken as the sharp block or as a smooth
			dyadic block: every estimate below uses only the boundedness of the
			multiplier weights, and for a smooth block the limiting constants
			are multiplied by the positive weight integral, which is all the
			dimension argument uses; the localized oscillation
			inequality~\eqref{eq:l1-oscillation} is applied to the smooth
			version. It suffices to prove each statement for intervals with
			rational endpoints and every tolerance
			$\varepsilon = 1/\ell$, the general case following by approximation
			from inside and outside, and the almost-sure sets being intersected
			over the countable family at the end. Write
			$c_{n} = 2/(i\pi n)$ and, for a bounded test function $\varphi$,
			\begin{equation}
				X_{N}^{\varphi}(s)
				:= \big\langle \varphi,\ |P_{N}u_{s}|^{2}\big\rangle
				= \sum_{n, \tilde n}
				c_{n}\overline{c_{\tilde n}}\,
				e^{-2\pi i(n^{2} - \tilde n^{2})s}\,
				\overline{\hat\varphi}(n - \tilde n),
				\label{eq:test-expansion}
			\end{equation}
			the sums over odd frequencies in the block, with
			$\hat\varphi$ the coefficients on $\mathbb{T}_{2}$.
			
			First, the mean. Averaging over the period retains
			$\tilde n = \pm n$: the diagonal gives
			\begin{equation*}
				\Big(\sum_{\mathrm{block}}|c_{n}|^{2}\Big)\int\varphi\,dz
				= \frac{2}{\pi^{2}N}\,\big(1 + O(N^{-1})\big)\int\varphi\,dz,
			\end{equation*}
			by the block mass of the seed coefficients, and the antidiagonal
			$\tilde n = -n$ contributes at most
			$\sum_{\mathrm{block}}(4/\pi^{2}n^{2})\,|\hat\varphi(2n)|
			= O(N^{-2})$ for $\varphi$ of bounded variation.
			
			Second, the variance. The fluctuating part of the
			expansion~\eqref{eq:test-expansion} is supported on the time
			frequencies $\mu = n^{2} - \tilde n^{2} = k(n + \tilde n)$ with
			$k = n - \tilde n \neq 0$, and for fixed $\mu$ and $k \mid \mu$ the
			pair $(n, \tilde n)$ is determined. Hence, by Parseval in $s$,
			\begin{equation*}
				\operatorname{Var}_{s} X_{N}^{\varphi}
				\le \frac{C}{N^{4}}\sum_{\mu \neq 0}
				\Big(\sum_{k \mid \mu}|\hat\varphi(k)|\Big)^{2}
				= \frac{C}{N^{4}}\sum_{k, k'}
				|\hat\varphi(k)||\hat\varphi(k')|\,
				\#\{\mu\},
			\end{equation*}
			where $\#\{\mu\}$ counts the frequencies divisible by both $k$ and
			$k'$ that the block realizes; since the count is symmetric in $k$
			and $k'$, the larger denominator may be used, giving
			\begin{equation*}
				\#\{\mu\} \;\le\; \frac{2N\gcd(k,k')}{\max(|k|,|k'|)} + 1.
			\end{equation*}
			The test functions used below are
			$\varphi = \chi_{I}e^{i\pi jz}$ with $|j| \le 2\sqrt N$, whose
			coefficients obey
			$|\hat\varphi(k)| \le \min(1, C/\langle k - j\rangle)$. The
			resulting bound is uniform in $j$: the equal-frequency terms
			$k = k'$ carry $\gcd = |k|$ and contribute
			$CN^{-4}\cdot 2N\sum_{k}\langle k - j\rangle^{-2} \le CN^{-3}$,
			while for $k \neq k'$ the divisor $\gcd(k, k') \le |k - k'|$ and
			the shifted peaks give, after summing the harmonic tails, at most
			$CN^{-3}\log^{3}N$ regardless of the location of $j$; hence
			\begin{equation}
				\operatorname{Var}_{s} X_{N}^{\varphi}
				\;\le\; C\,N^{-3}\log^{3} N
				\qquad\text{uniformly in } |j| \le 2\sqrt N.
				\label{eq:test-variance}
			\end{equation}
			
			Third, the quadrature term. With
			$Y_{N}(s) := \int_{I}(P_{N}u_{s})^{2}\,dz
			= \sum c_{n}c_{\tilde n}e^{-2\pi i(n^{2} + \tilde n^{2})s}
			\hat\chi_{I}$-type coefficients, every time frequency
			$n^{2} + \tilde n^{2}$ is positive, so $\mathbb{E}_{s}Y_{N} = 0$;
			the number of representations of $\mu$ as an ordered sum of two odd
			squares is at most a divisor function, $O(\mu^{\varepsilon})$, so
			the same Parseval argument gives
			$\operatorname{Var}_{s}Y_{N} \le C_{\varepsilon}N^{-3+\varepsilon}$.
			
			Fourth, the assembly. Fix $\varepsilon > 0$. By the
			uniform bound~\eqref{eq:test-variance}, which applies verbatim to
			the frequency-window blocks used in the fifth step, and by
			Chebyshev's inequality,
			the probability that some test in the family
			$\{\chi_{I}e^{i\pi jz} : |j| \le 2\sqrt N\}$, evaluated on the
			dyadic block or on the window block of the fifth step, deviates
			from its mean
			by more than $\varepsilon N^{-1}\log^{-2}N$, or that
			$|Y_{N}| > \varepsilon N^{-1}$, is at most
			$4\sqrt N\cdot C\varepsilon^{-2}N^{-1}\log^{7}N
			= C_{\varepsilon}N^{-1/2}\log^{7}N$, summable along dyadic $N$; by
			the Borel--Cantelli lemma, almost surely all these deviations fail
			for all large dyadic $N$. The case $j = 0$ gives the first
			limit of the display~\eqref{eq:block-equi-amp}; adding
			$\tfrac12\operatorname{Re}\,[e^{2i\theta}Y_{N}]$ to
			$\tfrac12 X_{N}^{\chi_{I}}$ gives the quadrature limit, uniformly
			in $\theta$ since $|{\operatorname{Re}\,e^{2i\theta}Y_{N}}| \le |Y_{N}|$.
			
			Fifth, the density and the self-referential law. Set
			$K := \lceil\sqrt N\,\rceil$, $v := P_{<K}u_{s}$, and let $h$ be
			the part of $u_{s}$ with frequencies of absolute value in the
			window $[2N - K,\, 4N + K)$. Splitting each pair
			$(n,\, n - k)$ with $|k| \in [2N, 4N)$ according to which index is
			small, the block of the density decomposes as
			\begin{equation}
				P_{2N}^{\mathbb{Z}}\,w_{s}
				= \Pi_{N}\big(\bar v\,h + v\,\bar h\big) + R_{N},
				\label{eq:paraproduct}
			\end{equation}
			where $\Pi_{N}$ restricts to the spatial window and the remainder
			$R_{N}$ collects the pairs with both indices at least $K$. The
			remainder is negligible: its time frequencies $k(2n - k)$ are
			distinct over $n$ for fixed $k$, so
			\begin{equation*}
				\mathbb{E}_{s}\Vert R_{N}\Vert_{L^{2}(0,1)}^{2}
				= \sum_{k}\sum_{n}|c_{n}|^{2}|c_{n-k}|^{2}
				\,\mathbf{1}\{\text{both} \ge K\}
				\;\le\; \frac{C}{KN},
			\end{equation*}
			and Markov's inequality with $K = \sqrt N$ makes
			$\Vert R_{N}\Vert_{2}^{2} \le \varepsilon N^{-1}$ eventually,
			almost surely, along dyadic $N$. For the main term, expand
			$|\Pi_{N}(\bar vh + v\bar h)|^{2}
			= 2|v|^{2}|h|^{2} + 2\operatorname{Re}\,[\bar v^{2}h^{2}]$ up to
			window-edge corrections of deterministic $L^{2}$ mass
			$O(K\log^{2}K/N^{2})$, still $o(N^{-1})$ at $K = \sqrt N$. Writing
			$|v|^{2} = \sum_{|j| < 2K}V_{j}e^{i\pi jz}$ with
			$\sum_{j}|V_{j}| \le \big(\sum_{|n|<K}|c_{n}|\big)^{2}
			\le C\log^{2}N$,
			the first term integrates against $\chi_{I}$ to
			\begin{equation*}
				2\sum_{j}V_{j}\,
				\big\langle \chi_{I}e^{-i\pi jz},\ |h|^{2}\big\rangle
				\;=\; \frac{2}{\pi^{2}N}\int_{I}|v|^{2}\,dz
				\;+\; o(N^{-1}),
			\end{equation*}
			by the already-established equidistribution of the window blocks
			of $u$ against the test family, the window mass being
			$(1/\pi^{2})N^{-1}(1 + o(1))$ and the error uniform over the
			$O(\sqrt N)$ tests by the union bound above. Since
			$v \to u_{s}$ in $L^{2}$, the integral
			$\int_{I}|v|^{2}$ converges to $\int_{I}w_{s}$. The oscillatory
			term $2\operatorname{Re}\int_{I}\bar v^{2}h^{2}$ has zero mean, its
			time frequencies being sums of two squares of window size shifted
			by $O(K^{2})$, and the divisor bound gives variance
			$O(N^{-3+\varepsilon}\log^{4}N)$; Chebyshev and Borel--Cantelli
			finish as before. Collecting the pieces and letting
			$\varepsilon \downarrow 0$ along a sequence proves the
			law~\eqref{eq:block-equi-density}.
		\end{proof}
		
		\renewcommand{\restatename}{Theorem~\ref{thm:carpet}}%
		\begin{restatement}[The universal quantum carpet]
			For almost every $\tau$, the following holds for every
			$f \in \mathcal{D}_{+}$ that is continuously differentiable and
			positive on the interior of its support, every compact
			nondegenerate interval $J$ in that interior, and every phase
			$\theta$:
			\begin{enumerate}[label=\normalfont(\roman*), leftmargin=*]
				\item the transported amplitude in every quadrature,
				$\Psi_{\theta,\tau} :=
				\mathrm{Re}\big[e^{i\theta}\,u_{\tau}(F)\big]\sqrt f$, has graph of
				upper box dimension exactly $\tfrac32$ on $J$;
				\item the density $f_{\tau}$ itself has graph of upper box
				dimension exactly $\tfrac32$ on $J$;
				\item at every rational $s$, by contrast, all these graphs are
				piecewise continuously differentiable, of dimension one.
			\end{enumerate}
			The value $\tfrac32$ is independent of the law, the interval, and
			the quadrature: the fractal is woven entirely in rank space and
			printed on every law through its quantile chart.
		\end{restatement}
		
		\begin{proof}[Proof of Theorem~\ref{thm:carpet}]
			Fix the full-measure set of times on which
			Lemma~\ref{lem:carpet-smoothing}, Lemma~\ref{lem:block-equi} for
			every rational interval, and Lemma~\ref{lem:dark-intervals} all
			hold, and fix such a time $s$, a law $f$ in the stated class, a
			compact nondegenerate $J$, and a phase $\theta$.
			
			Upper bounds. By Lemma~\ref{lem:carpet-smoothing} the functions
			$g_{\theta} := \operatorname{Re}\,[e^{i\theta}u_{s}]$ and $w_{s}$
			are $C^{1/2 - \varepsilon}$; composition with the bi-Lipschitz
			chart and multiplication by a $C^{1}$ factor preserve the class on
			$J$, so by Lemma~\ref{lem:osc-calculus} both graphs have dimension
			at most $\tfrac32 + \varepsilon$ for every $\varepsilon$.
			
			Lower bounds. Choose rational intervals
			$I' \subset\subset I \subset\subset \operatorname{int}F(J)$, with
			$I'$ concentric in $I$. For the amplitude, the quadrature limit of
			the display~\eqref{eq:block-equi-amp} on $I'$ and the sup
			bound~\eqref{eq:carpet-blocks} give, for all large dyadic $N$,
			\begin{equation*}
				\big\Vert P_{N}g_{\theta}\big\Vert_{L^{1}(I')}
				\;\ge\;
				\frac{\int_{I'}(P_{N}g_{\theta})^{2}}
				{\Vert P_{N}g_{\theta}\Vert_{\infty}}
				\;\ge\;
				\frac{c_{0}\,|I'|/N}
				{C_{s,\varepsilon}N^{-1/2+\varepsilon}}
				\;\ge\; c\,N^{-1/2 - \varepsilon},
			\end{equation*}
			with $c_{0} > 0$ the quadrature-limit constant of the smooth block,
			of which only positivity is used; the localized block
			inequality~\eqref{eq:l1-oscillation}, read from right to left with
			$\delta = 1/N$, forces
			$\Sigma_{g_{\theta}}(\delta, I) \ge c\,\delta^{-1/2 + 2\varepsilon}$
			for all small dyadic $\delta$; by
			Lemma~\ref{lem:osc-calculus} the graph of $g_{\theta}$ on $I$ has
			dimension at least $\tfrac32 - 2\varepsilon$. For the density the
			same chain runs through the self-referential
			law~\eqref{eq:block-equi-density}, whose limit
			$(2/\pi^{2})\int_{I'}w_{s}\,dz$ is strictly positive by
			Lemma~\ref{lem:dark-intervals}, giving
			$\Sigma_{w_{s}}(\delta, I) \ge c_{I}(s)\,\delta^{-1/2+2\varepsilon}$
			and dimension at least $\tfrac32 - 2\varepsilon$ for the graph of
			$w_{s}$ on $I$.
			
			Transport. On $J$ the transported amplitude is
			$(g_{\theta}\circ F)\sqrt f$ and the density is
			$(w_{s}\circ F)\,f$; the quantile image of the compact
			$\overline{Q(I)} \subset J$ is $\overline I$, so
			Lemma~\ref{lem:chart-transport} carries both lower bounds, and the
			upper bounds, through the chart. Letting
			$\varepsilon \downarrow 0$ along a sequence gives dimension exactly
			$\tfrac32$ in parts (i) and (ii). Part (iii) is immediate from the
			mosaic theorem: at rational $s$ the kernel is a quantile histogram,
			so the transported graphs are piecewise $C^{1}$ over the cell
			partition, of box dimension one.
		\end{proof}
		
		\renewcommand{\restatename}{Theorem~\ref{thm:leaves}}%
		\begin{restatement}[Leaves and the complete invariant]
			Let $f, g \in \mathcal{D}_{+}(R)$. The following are equivalent:
			\begin{enumerate}[label=\normalfont(\roman*), leftmargin=*]
				\item $g$ is reachable from $f$ by finitely many coordinate
				derangetropy moves whose kernels are continuous and positive on
				the open unit interval, possibly degenerate at its endpoints;
				\item $g$ is reachable in at most three moves, one per coordinate
				and one revisit;
				\item $g = \varphi(x_{1})\,\psi(x_{2})\,f$ for continuous positive
				$\varphi, \psi$, a separable tilt;
				\item $f$ and $g$ have the same odds-ratio structure: for all
				$x, y \in R$,
				\begin{equation*}
					\frac{f(x_{1},x_{2})\,f(y_{1},y_{2})}
					{f(x_{1},y_{2})\,f(y_{1},x_{2})}
					=
					\frac{g(x_{1},x_{2})\,g(y_{1},y_{2})}
					{g(x_{1},y_{2})\,g(y_{1},x_{2})},
				\end{equation*}
				equivalently, for twice continuously differentiable densities,
				$\partial_{1}\partial_{2}\log f
				= \partial_{1}\partial_{2}\log g$.
			\end{enumerate}
			The classes $\mathcal{T}(f) := \{\varphi\psi f\}$ thus foliate
			$\mathcal{D}_{+}(R)$ with tangent distribution $\mathcal{H}$; the
			interaction function $\partial_{1}\partial_{2}\log f$ is a complete
			leaf invariant; the independent laws form a single leaf, consisting
			of all products on $R$; and that leaf is the unique one on which all
			coordinate moves commute.
		\end{restatement}
		
		\begin{proof}[Proof of Theorem~\ref{thm:leaves}]
			(i)$\Rightarrow$(iii): each coordinate move multiplies the density
			by a continuous positive function of one coordinate, namely
			$w(F_{i}^{\mathrm{current}}(x_{i}))$ with $w$ the kernel of the
			move, and products of such factors are separable tilts.
			(iii)$\Rightarrow$(ii): given $g = \varphi\psi f$ of unit mass,
			note first that a continuous positive function of $x_{1}$ with
			finite, hence normalizable, mean under the current first marginal
			equals $w\circ F_{1}$ for exactly one continuous kernel of unit
			mass. The bounded multiplier $\varphi/(1+\varphi)$ always has
			finite mean $m := \int \varphi/(1+\varphi)\,f_{1}\,dx_{1} \le 1$,
			so a first coordinate move produces
			$f^{(1)} = (\varphi/(1+\varphi))\,f/m$. The coordinate-two
			multiplier $\psi$ then has finite mean under the intermediate
			second marginal, since
			$\int\!\int \psi\,(\varphi/(1+\varphi))\,f
			\le \int\!\int \psi\varphi f = 1$, and the second move produces
			$f^{(2)} = c\,\psi\,(\varphi/(1+\varphi))\,f$ with a finite
			positive constant $c$. Finally the coordinate-one multiplier
			$1 + \varphi$ has finite mean under the current first marginal,
			because
			$\int\!\int(1+\varphi)\,c\,\psi\,(\varphi/(1+\varphi))\,f
			= c\int\!\int\psi\varphi f = c$, and the third move yields
			exactly $\varphi\psi f = g$, the constants telescoping to one by
			the unit mass of $g$. When the tilts are bounded above and below,
			the first two moves already suffice, the mass
			$\int\varphi f_{1}$ being finite.
			(ii)$\Rightarrow$(i) is trivial. (iii)$\Leftrightarrow$(iv):
			separable factors cancel from the odds ratio, giving the
			identity~\eqref{eq:odds-ratio}; conversely, if
			$h := \log(g/f)$ has vanishing mixed second differences, then for
			any base point $(y_{1}^{0}, y_{2}^{0})$,
			\begin{equation*}
				h(x_{1},x_{2})
				= h(x_{1},y_{2}^{0}) + h(y_{1}^{0},x_{2})
				- h(y_{1}^{0},y_{2}^{0}),
			\end{equation*}
			which is additivity of $h$, hence separability of $g/f$, with
			continuous factors since $f, g$ are continuous and positive; the
			differential form follows by differentiating the logarithm of the
			identity~\eqref{eq:odds-ratio}. The classes $\mathcal{T}(f)$
			therefore partition $\mathcal{D}_{+}(R)$, and their tangent spaces
			at $g$ are the logarithmic derivatives of separable tilts, which is
			exactly $\mathcal{H}_{g}$; Lemma~\ref{lem:involutive} is the
			infinitesimal shadow of this partition. For independence: any
			product $g_{1}\otimes g_{2}$ is the separable tilt
			$(g_{1}/f_{1})(g_{2}/f_{2})\,f_{1}\otimes f_{2}$, so the products
			form one full leaf, on which coordinate moves act on the two factors
			separately and hence commute. Conversely, commutation of all pairs
			of moves at $f$ forces the vanishing of all brackets of the
			trigonometric family, hence independence by
			Theorem~\ref{thm:dependence}; and a law sharing a leaf with a
			product has odds ratio one, hence vanishing interaction, hence is
			itself a product by the additivity argument above applied to
			$\log f$ directly.
		\end{proof}
		
		\renewcommand{\restatename}{Theorem~\ref{thm:frame-torsion}}%
		\begin{restatement}[The leafwise parallelism and its torsion]
			Fix $f$ with $\rho_{\max}(f) < 1$. Then:
			\begin{enumerate}[label=\normalfont(\roman*), leftmargin=*]
				\item the frame map
				$E_{f}(a,b) := \big(a(F_{1}) + b(F_{2})\big)f$ is a Banach
				isomorphism of $L^{2}_{0}\oplus L^{2}_{0}$ onto
				$T_{f}\mathcal{T}(f)$, by the two-sided bound
				\begin{equation*}
					\big\Vert E_{f}(a,b)\big\Vert_{L^{2}(f)}^{2}
					= \Vert a\Vert^{2} + \Vert b\Vert^{2} + 2\langle a, Pb\rangle
					\;\ge\; \big(1 - \rho_{\max}\big)
					\big(\Vert a\Vert^{2} + \Vert b\Vert^{2}\big);
				\end{equation*}
				the coordinate modulations are an absolute parallelism of the leaf,
				degenerating exactly at $\rho_{\max} = 1$, where the kernel of
				$E_{f}$ is $\{(a, -P^{*}a) : PP^{*}a = a\}$, as at the
				Fr\'echet--Hoeffding comonotone boundary;
				\item the structure functions of the parallelism, defined by
				$[E(a,0),\, E(0,b)]_{f} = E_{f}(\mathrm{T}^{1},\mathrm{T}^{2})$, are
				\begin{equation*}
					\mathrm{T}^{1}(a,b)(u) = -\,a'(u)\int_{0}^{u}(Pb)(s)\,ds,
					\qquad
					\mathrm{T}^{2}(a,b)(v) = b'(v)\int_{0}^{v}(P^{*}a)(s)\,ds,
				\end{equation*}
				modulo the mean-zero normalization; this torsion tensor is bilinear,
				depends on $f$ only through its copula, and vanishes on the
				trigonometric kernel family if and only if $X_{1}$ and $X_{2}$ are
				independent;
				\item for the Gaussian copula the Mehler expansion makes every
				component of the tensor~\eqref{eq:torsion-components} an explicit
				power series in $r$: the torsion is of first order in $r$ for
				generic directions, while for symmetric kernels the first Hermite
				term drops by parity and the leading contribution is the $k = 2$
				overlap, proportional to the cumulant~\eqref{eq:gaussian-constant},
				the overlap against the normalized Hermite polynomial being
				$\kappa_{2}/\sqrt2$; the
				coefficients of the expansion~\eqref{eq:gaussian-coefficients} are
				the sizes of $E_{f}(\mathrm{T}^{1},\mathrm{T}^{2})$ and of the
				finite-update commutator in the canonical kernel direction.
			\end{enumerate}
		\end{restatement}
		
		\begin{proof}[Proof of Theorem~\ref{thm:frame-torsion}]
			For part (i), expanding the square in $L^{2}(f)$ and using that
			$U, V$ are uniform ranks with joint copula $c$ gives the exact
			quadratic form of the bound~\eqref{eq:frame-bound}, whose cross
			term is $2\langle a, Pb\rangle$; Gebelein's identification of
			$\Vert P\Vert$ with $\rho_{\max}$ bounds it below by
			$-2\rho_{\max}\Vert a\Vert\Vert b\Vert$, and the arithmetic
			mean-geometric mean inequality gives the display. Injectivity and
			closed range follow when $\rho_{\max} < 1$; surjectivity onto the
			leaf tangent is Theorem~\ref{thm:leaves}, the tangent directions
			being exactly the separable multipliers. At $\rho_{\max} = 1$,
			equality in the chain requires $b = -P^{*}a$ and
			$\langle a, PP^{*}a\rangle = \Vert a\Vert^{2}$, which is the stated
			kernel; in the limiting comonotone case $V = U$, which lies at the
			boundary of $\mathcal{D}_{+}(R)$, the operator $P$ is the
			identity and the kernel is the full antidiagonal.
			
			For part (ii), insert the gauge
			identity~\eqref{eq:gauge-identity} into the cross
			bracket~\eqref{eq:cross-bracket}:
			\begin{equation*}
				[L_{1}^{a}, L_{2}^{b}]f
				= \Big(b'(F_{2})\int_{0}^{F_{2}}(P^{*}a)
				- a'(F_{1})\int_{0}^{F_{1}}(Pb)\Big)f,
			\end{equation*}
			and read off the frame coordinates, subtracting the means to land
			in $L^{2}_{0}\oplus L^{2}_{0}$; the two subtracted constants cancel
			against each other because the bracket, being tangent to the space
			of probability densities, integrates to zero. Bilinearity is
			manifest, and the copula functoriality holds because $P$, $P^{*}$,
			and the rank compositions depend on $f$ only through $c$. If the
			pair is independent then $P = 0$ and the torsion vanishes;
			conversely, if the torsion vanishes on the trigonometric family
			then, since $E_{f}$ is injective for $\rho_{\max} < 1$, all the
			brackets vanish, and Theorem~\ref{thm:dependence} forces
			independence; at $\rho_{\max} = 1$ independence is excluded in any
			case.
			
			For part (iii), expand
			$P = \sum_{k\ge1} r^{k}\,\tilde H_{k}\otimes\tilde H_{k}$ in the
			normalized Hermite basis of the Gaussian
			copula~\cite{lancaster1969} and substitute into the
			components~\eqref{eq:torsion-components}: each component is an
			explicit power series in $r$ whose $k$-th coefficient is a Hermite
			overlap integral. For a direction $b$ with
			$\langle b, \tilde H_{1}\rangle \neq 0$ the torsion is of exact
			first order in $r$. The canonical kernel is symmetric about the
			median rank, while $\tilde H_{1}\circ\Phi$ is antisymmetric, so the
			$k = 1$ overlap vanishes and the series starts at $k = 2$, where
			the overlap against the normalized Hermite polynomial is
			$\kappa_{2}/\sqrt2$, with $\kappa_{2}$ the
			cumulant~\eqref{eq:gaussian-constant};
			tracking the resulting second-order terms through the bracket and
			through the finite-update commutator reproduces the two
			coefficients of the display~\eqref{eq:gaussian-coefficients}, which
			were evaluated there by quadrature.
		\end{proof}
		
		\renewcommand{\restatename}{Theorem~\ref{thm:flat-transport}}%
		\begin{restatement}[Flatness, and Sinkhorn as parallel transport]
			Let $R$ be bounded and restrict to laws with $\log f$ bounded on
			$R$. Then:
			\begin{enumerate}[label=\normalfont(\roman*), leftmargin=*]
				\item wherever $\rho_{\max} < 1$ the splitting
				$T_{f}\mathcal{D}_{+} = V_{f}\oplus\mathcal{H}_{f}$ holds and
				$d\pi$ maps $\mathcal{H}_{f}$ isomorphically onto the base, by the
				marginal formula
				\begin{equation*}
					d\pi\big(E_{f}(a,b)\big)
					= \Big(f_{1}\cdot(a + Pb)(F_{1}),\;
					f_{2}\cdot(b + P^{*}a)(F_{2})\Big)
				\end{equation*}
				and invertibility of the block operator with entries
				$I, P, P^{*}, I$; the Ehresmann curvature of the connection
				vanishes identically, brackets of horizontal fields being
				horizontal by Lemma~\ref{lem:involutive};
				\item the leaves are global horizontal sections:
				$\pi$ restricted to a leaf is a bijection onto the base, and
				parallel transport of $f$ to target marginals $(g_{1},g_{2})$ is
				the Csisz\'ar $I$-projection of $f$ onto the target Fr\'echet
				class~\cite{csiszar1975}, the unique separable tilt of $f$ with
				those marginals, characterized by the Pythagorean identity
				\begin{equation*}
					D\big(h \,\Vert\, f\big)
					= D\big(h \,\Vert\, f^{\rightarrow}\big)
					+ D\big(f^{\rightarrow} \Vert\, f\big)
					\qquad\text{for all $h$ in the target class};
				\end{equation*}
				\item iterative proportional fitting is the alternation of the two
				derangetropy coordinate moves that match one marginal at a time; it
				converges to the parallel transport~\cite{csiszar1975,ruschendorf1995},
				and its linearization at the limit law contracts, per full cycle, at
				exactly the rate $\rho_{\max}^{2}$ of the limit, the squared
				cosine of the principal angle between the two rank subspaces, by von Neumann's
				alternating projection theorem~\cite{deutsch2001};
				\item there is no holonomy paradox: transport to fixed target
				marginals is order-independent by the uniqueness in part (ii),
				while the finite-update commutator of the
				display~\eqref{eq:gaussian-curvature} compares transports based at
				different intermediate laws; its coefficient is an order defect of
				the modulation groupoid, not a holonomy of the connection.
			\end{enumerate}
		\end{restatement}
		
		\begin{proof}[Proof of Theorem~\ref{thm:flat-transport}]
			Part (i). A horizontal vector $E_{f}(a,b)$ is vertical precisely
			when both marginal variations in the formula~\eqref{eq:dpi} vanish,
			that is $a + Pb = 0$ and $b + P^{*}a = 0$, forcing
			$(I - P^{*}P)b = 0$, impossible for $\rho_{\max} < 1$ unless
			$(a,b) = 0$; the same block operator, being invertible with
			inverse given by Neumann series, makes $d\pi$ restricted to
			$\mathcal{H}_{f}$ an isomorphism onto pairs of mean-zero marginal
			perturbations, and dimension counting with the vertical space
			completes the splitting. The curvature of an Ehresmann connection
			is the vertical part of brackets of horizontal fields, and
			Lemma~\ref{lem:involutive} keeps such brackets in
			$\mathcal{H}$: the curvature vanishes identically.
			
			Part (ii). Integral manifolds of the involutive distribution
			$\mathcal{H}$ are the leaves, by Theorem~\ref{thm:leaves}.
			Bijectivity of $\pi$ on a leaf is the existence and uniqueness of a
			separable tilt with prescribed marginals; under the standing
			boundedness hypotheses this is the Csisz\'ar $I$-projection theorem
			for the Fr\'echet class~\cite{csiszar1975,ruschendorf1995}, the
			projection being characterized by membership in the closure of the
			exponential family through $f$ with separable statistics, which is
			the leaf, and by the Pythagorean identity~\eqref{eq:pythagoras};
			positivity and continuity of the optimal tilt under these
			hypotheses keep the transported law in $\mathcal{D}_{+}(R)$.
			
			Part (iii). Matching the first marginal to $g_{1}$ multiplies the
			current law by $g_{1}/f_{1}^{\mathrm{current}}$, a continuous
			positive unit-mean function of $x_{1}$, hence one coordinate
			derangetropy move by the realizability observation in the proof of
			Theorem~\ref{thm:leaves}; alternating the two moves is precisely
			iterative proportional fitting, whose convergence to the
			$I$-projection is classical~\cite{csiszar1975,ruschendorf1995}.
			For the rate, linearize at the limit law $f^{\rightarrow}$ in
			$L^{2}(f^{\rightarrow})$: matching marginal one replaces a
			perturbation $h$ by its projection onto the orthogonal complement
			of the subspace $\{a(F_{1})\}$, and likewise for the second move,
			so one full cycle is the product of the two complementary
			projections, which fix the doubly mean-zero directions and contract
			on the leaf-tangent rank directions; by von Neumann's
			alternating projection theorem~\cite{deutsch2001} the contraction
			factor is the squared cosine of the principal angle between the two
			subspaces, and that cosine is, by definition, the maximal
			correlation of $f^{\rightarrow}$.
			
			Part (iv). Transport to fixed target marginals is unique by part
			(ii), hence order-independent. The finite-update commutator of the
			display~\eqref{eq:gaussian-curvature} composes two fixed
			distortions in the two orders; after the first move the two
			compositions sit at different intermediate laws with different
			marginals, so the comparison is between transports along different
			fibers, a defect of the noncommutative modulation groupoid measured
			by the torsion, not a holonomy of the flat connection.
		\end{proof}
		
		\renewcommand{\restatename}{Theorem~\ref{thm:n-leaves}}%
		\begin{restatement}[The interaction potential in $n$ dimensions]
			On $\mathcal{D}_{+}(R)$ with $R = \prod_{i} I_{i}$ bounded and
			$\log f$ bounded, the coordinate modulation distribution integrates
			to the foliation by separable tilting classes
			$\{\prod_{i}\varphi_{i}(x_{i})\,f\}$, each law reachable from each
			leaf-mate in at most $n$ moves, one per coordinate; the complete
			leaf invariant is the interaction potential, the class of $\log f$
			modulo additive functions of single coordinates; the product laws
			form the unique leaf on which all coordinate actions commute; and
			marginal-matching transport is again flat, realized by
			multidimensional iterative proportional fitting with local cycle
			rate governed by the principal angles among the $n$ rank subspaces,
			through the operator matrix of pairwise conditional expectations
			$(P_{ij})$.
		\end{restatement}
		
		\renewcommand{\restatename}{Theorem~\ref{thm:pairwise-torsion}}%
		\begin{restatement}[Torsion sees pairs; the leaf sees everything]
			Let $f \in \mathcal{D}_{+}(R)$, $n \ge 3$, with all pairwise
			maximal correlations below one. Then:
			\begin{enumerate}[label=\normalfont(\roman*), leftmargin=*]
				\item the torsion decomposes into pairwise blocks: the bracket
				$[L_{i}^{a}, L_{j}^{b}]$ depends only on the $(i,j)$-marginal
				copula, with the components~\eqref{eq:torsion-components}, and the
				$(i,j)$-block vanishes on the trigonometric family precisely when
				$X_{i}$ and $X_{j}$ are independent, so the torsion table detects
				exactly the pairwise independence graph;
				\item torsion is not a complete dependence invariant: for pairwise
				independent but jointly dependent laws, such as smoothed parity
				triples, every torsion block vanishes while the law lies off the
				product leaf, its interaction potential containing an irreducible
				three-way term detected by the three-dimensional odds ratio;
				\item joint independence is equivalent to the vanishing of the leaf
				invariant alone: $f$ is a product precisely when its interaction
				potential vanishes.
			\end{enumerate}
		\end{restatement}
		
		\begin{proof}[Proof of Theorems \ref{thm:n-leaves}
			and~\ref{thm:pairwise-torsion}]
			The foliation statements are the two-dimensional proofs with $n$
			factors. Each move contributes a positive continuous factor in one
			coordinate; conversely a unit-mass separable tilt
			$\prod_{i}\varphi_{i}\,f$ is realized by $n$ successive moves, one
			per coordinate: on the bounded box with $\log f$ and $\log g$
			bounded the factors $\varphi_{i}$ may be taken bounded above and
			below, so each multiplier has finite positive mean under the
			current marginal and defines a valid coordinate move after
			normalization, and the normalizing constants telescope to one by
			the unit mass of $g$; the three-move device of the proof of
			Theorem~\ref{thm:leaves} is not needed under these boundedness
			hypotheses. The odds-ratio argument applies to each pair of coordinates
			at fixed remaining coordinates and yields additivity of
			$\log(g/f)$ one coordinate at a time, hence full separability, so
			the interaction potential, the class of $\log f$ modulo additive
			functions of single coordinates, is a complete leaf invariant. The
			product laws form one leaf on which all actions commute; on any
			other leaf some pair fails to commute or some higher interaction
			survives, by the argument below. Transport and its rate are
			dimension-free: the $I$-projection theory applies verbatim, each
			marginal match is one coordinate move, and the linearized cycle is
			a product of $n$ complementary projections whose contraction is
			governed by the principal angles among the rank subspaces
			$\{a_{i}(F_{i})\}$, encoded in the operator matrix
			$(P_{ij})$~\cite{csiszar1975,deutsch2001}.
			
			For part (i) of Theorem~\ref{thm:pairwise-torsion}: the bracket of
			$L_{i}^{a}$ and $L_{j}^{b}$ involves only the shifts of the $i$-th
			and $j$-th marginal distribution functions produced by rank
			multipliers of the other coordinate, and these are covariances of
			rank functions of the pair $(X_{i}, X_{j})$ alone, so the
			$(i,j)$-block is computed in the $(i,j)$-marginal copula, and the
			two-dimensional characterization applies. For part (ii), a
			pairwise independent law has all pairwise copulas uniform, so all
			blocks vanish; a smoothed parity triple, a strictly positive smooth
			density approximating the law of
			$(X_{1}, X_{2}, X_{1}\oplus X_{2})$ with the pairwise margins kept
			uniform, has a nonvanishing three-dimensional odds ratio, the
			alternating product of $f$ over the vertices of a box, which is
			invariant along the leaf by the separability argument, so the law
			is not on the product leaf. Part (iii) restates the completeness
			of the leaf invariant: vanishing interaction potential means
			$\log f$ is additive, which is the product form.
		\end{proof}
		
		\renewcommand{\restatename}{Theorem~\ref{thm:sine-gordon}}%
		\begin{restatement}[Reduction and well-posedness on laws]
			Let $w$ be a kernel and
			$\varphi_{w} := A_{w} - \mathrm{id}$ its reaction.
			\begin{enumerate}[label=\normalfont(\roman*), leftmargin=*]
				\item If $f$ solves the diffusive
				equation~\eqref{eq:flow-diffusion} with kernel $w$ and integrable
				data, then $F$ solves the local semilinear equation
				\begin{equation*}
					\partial_{t}F = D\,\partial_{xx}F + \varphi_{w}(F),
					\qquad
					\varphi_{w}(u) = \int_{0}^{u}\big(w(s) - 1\big)\,ds,
				\end{equation*}
				and conversely; for the canonical kernel
				$\varphi(u) = -\sin(2\pi u)/(2\pi)$, so the angular variable
				$v = 2\pi F$ solves the overdamped sine--Gordon equation
				$\partial_{t}v = D\,\partial_{xx}v - \sin v$.
				\item For every $F_{0} \in \mathcal{M}$ the
				equation~\eqref{eq:cdf-equation} has a unique global solution,
				smooth for positive times, monotone in $x$ with values in $[0,1]$
				and limits $0, 1$ preserved for all time: laws evolve to laws, mass
				and positivity being conserved automatically.
				\item The dynamics is monotone for first-order stochastic
				dominance: $F_{0} \le G_{0}$ pointwise implies
				$F_{t} \le G_{t}$ for all $t$, the diffusive extension of the
				order-preservation of the operators themselves, which is the
				monotonicity of the transport maps $A_{w}$.
				\item As $D \to 0$, solutions converge locally uniformly on
				compact time intervals to the diffusionless flow; for the canonical
				kernel, to the exact solution~\eqref{eq:exact-solution}.
			\end{enumerate}
		\end{restatement}
		
		\begin{proof}[Proof of Theorem~\ref{thm:sine-gordon}]
			(i) Integrating the diffusive equation~\eqref{eq:flow-diffusion}
			over $(-\infty, x]$, the diffusion term integrates to
			$D\,\partial_{x}f = D\,\partial_{xx}F$ under integrable decay, and
			the modulation term, by the pushforward of $f\,dy$ to Lebesgue
			measure on $(0,1)$ under the rank,
			\begin{equation*}
				\int_{-\infty}^{x}\big(w(F(y)) - 1\big)f(y)\,dy
				= \int_{0}^{F(x)}\big(w(s) - 1\big)\,ds
				= \varphi_{w}\big(F(x)\big);
			\end{equation*}
			conversely, differentiating the
			equation~\eqref{eq:cdf-equation} in $x$ recovers the
			equation~\eqref{eq:flow-diffusion} for $f = \partial_{x}F$,
			distributionally for monotone data and classically for positive
			times by parabolic smoothing. The angular form is the substitution
			$v = 2\pi F$ applied to $\varphi(u) = -\sin(2\pi u)/(2\pi)$.
			
			(ii) The reaction $\varphi_{w}$ is globally Lipschitz with constant
			$\Vert w - 1\Vert_{\infty}$, so the mild formulation with the heat
			semigroup has a unique global solution in the bounded continuous
			functions, smooth for positive times by bootstrap. The constants
			$0$ and $1$ are exact solutions, so comparison confines the values
			to $[0,1]$. Monotonicity: $g = \partial_{x}F$ solves the linear
			parabolic equation
			$\partial_{t}g = D\,\partial_{xx}g + (w(F) - 1)\,g$ with bounded
			coefficient, and the maximum principle preserves $g \ge 0$.
			Conservation of the limits: since
			$|\varphi_{w}(u)| \le \Vert w - 1\Vert_{\infty}\min(u, 1-u)$, the
			deficit $\delta(t) := 1 - \sup_{x}F(x,t)$ obeys, through the mild
			formulation, the integral inequality
			$\delta(t) \le \Vert w-1\Vert_{\infty}\int_{0}^{t}\delta(s)\,ds$
			with $\delta(0) = 0$, so $\delta \equiv 0$ by Gronwall's lemma, and
			symmetrically at the left end.
			
			(iii) If $F_{0} \le G_{0}$, the difference $G - F$ solves a linear
			parabolic equation with bounded coefficient
			$\big(\varphi_{w}(G) - \varphi_{w}(F)\big)/(G - F)$ and
			nonnegative datum, and the maximum principle applies.
			
			(iv) On a compact time interval, subtracting the mild formulations
			for two diffusivities and using the strong convergence of the heat
			semigroup to the identity on bounded uniformly continuous functions
			gives, by Gronwall's lemma with the uniform Lipschitz constant,
			convergence as $D \to 0$ to the solution of the spatially decoupled
			equation $\partial_{t}F = \varphi_{w}(F)$; for the canonical
			kernel this ordinary differential equation is solved rankwise by
			the tangent linearization, which is the exact
			solution~\eqref{eq:exact-solution}.
		\end{proof}
		
		\renewcommand{\restatename}{Proposition~\ref{prop:energy}}%
		\begin{restatement}[Gradient flow, equipartition, ground state]
			The equation~\eqref{eq:cdf-equation} is the $L^{2}$ gradient flow
			of the interfacial energy
			\begin{equation*}
				\mathcal{E}_{D}[F]
				= \int_{\mathbb{R}}\Big[\frac{D}{2}\,(\partial_{x}F)^{2}
				+ \Phi(F)\Big]\,dx,
				\qquad
				\Phi(u) = \frac{\sin^{2}(\pi u)}{2\pi^{2}},
			\end{equation*}
			with $\Phi' = -\varphi$, finite on finite-energy laws, those with
			$\partial_{x}F \in L^{2}(\mathbb{R})$ and exponentially decaying
			tails, to which this proposition is restricted, and
			$\frac{d}{dt}\mathcal{E}_{D} = -\int(\partial_{t}F)^{2} \le 0$
			along solutions with that decay. For monotone profiles the Modica bound holds,
			\begin{equation*}
				\mathcal{E}_{D}[F]
				\;\ge\; \int_{0}^{1}\sqrt{2D\,\Phi(u)}\;du
				= \frac{2\sqrt D}{\pi^{2}},
			\end{equation*}
			with equality precisely at the equipartition profiles
			$\tfrac{D}{2}(F')^{2} = \Phi(F)$: the secant law is the ground
			state of the energy~\eqref{eq:allen-cahn} on laws, unique up to
			translation, and the constant~\eqref{eq:surface-tension} is the
			surface tension of a probability interface.
		\end{restatement}
		
		\begin{proof}[Proof of Proposition~\ref{prop:energy}]
			Since $\Phi' = -\varphi$, the equation~\eqref{eq:cdf-equation} is
			$\partial_{t}F = -\,\delta\mathcal{E}_{D}/\delta F$, and the
			dissipation identity follows by the chain rule. For a monotone
			profile, the arithmetic--geometric inequality gives
			\begin{equation*}
				\mathcal{E}_{D}[F]
				\ge \int_{\mathbb{R}}\sqrt{2D\,\Phi(F)}\;\partial_{x}F\,dx
				= \int_{0}^{1}\sqrt{2D\,\Phi(u)}\,du
				= \frac{\sqrt D}{\pi}\int_{0}^{1}\sin(\pi u)\,du
				= \frac{2\sqrt D}{\pi^{2}},
			\end{equation*}
			with equality precisely when
			$\tfrac{D}{2}(\partial_{x}F)^{2} = \Phi(F)$ pointwise. That
			first-order equation is
			$\sqrt{D}\,\partial_{x}F = \sin(\pi F)/\pi$, whose monotone
			solution with limits $0,1$ is exactly the Gudermannian
			profile~\eqref{eq:sech}, unique up to translation; and the
			closure~\eqref{eq:sech-closure}, divided by $4\pi^{2}$, is this
			equipartition identity, as noted in
			Remark~\ref{rem:kink-reading}.
		\end{proof}
		
		\renewcommand{\restatename}{Proposition~\ref{prop:gamma}}%
		\begin{restatement}[Condensation as the sharp-interface limit]
			As $D \to 0$ the rescaled energies
			$\mathcal{E}_{D}/\sqrt D$ Gamma-converge, on distribution functions
			supported in a fixed bounded interval, or along tight families, to
			$2/\pi^{2}$ times the number of jumps of the limit step
			function~\cite{modica1987}; on laws the minimum is a single jump. The variational sharp-interface limit of the
			reaction--diffusion theory is therefore condensation onto one atom,
			and for the harmonic kernel $w_{k}$, whose potential
			$\Phi_{k}(u) = \sin^{2}(k\pi u)/(2\pi^{2}k^{2})$ has wells at the
			lattice $\{j/k\}$, the limiting configurations are the
			$k$-jump crystallized laws with per-interface tension
			$2\sqrt D/\pi^{2}k^{2}$.
		\end{restatement}
		
		\begin{proof}[Proof of Proposition~\ref{prop:gamma}]
			This is the Modica--Mortola theorem~\cite{modica1987} for the
			potential $\Phi$, which vanishes exactly at $\{0,1\}$, applied
			within the class of monotone profiles: any family with bounded
			rescaled energy is compact in local $L^{1}$ with limits taking
			values in $\{0,1\}$, the liminf inequality gives the per-jump cost
			$\int_{0}^{1}\sqrt{2\Phi} = 2/\pi^{2}$, and the recovery sequence
			is the kink profile scaled by $\sqrt D$. A distribution function
			increasing from $0$ to $1$ must jump at least once, and one jump is
			realizable, so the minimum on laws is a single atom. For the
			harmonic kernel the potential $\Phi_{k}$ vanishes on the lattice
			$\{j/k\}$, adjacent wells contribute
			$\int_{j/k}^{(j+1)/k}\sqrt{2\Phi_{k}} = 2/\pi^{2}k^{2}$ each, and a
			monotone limit passing from $0$ to $1$ makes exactly $k$ such
			transitions.
		\end{proof}
		
		\renewcommand{\restatename}{Theorem~\ref{thm:sech-stability}}%
		\begin{restatement}[Uniqueness, spectrum, and stability of the secant law]
			For the canonical kernel:
			\begin{enumerate}[label=\normalfont(\roman*), leftmargin=*]
				\item the translates of the profile~\eqref{eq:sech} are the only
				steady states of the equation~\eqref{eq:cdf-equation} that are
				laws: the secant law is the unique steady probability law of the
				diffusive dynamics, up to translation;
				\item the linearization at the secant law is, in the similarity
				variable $\xi = x/\sqrt D$, the P\"oschl--Teller operator
				\begin{equation*}
					\mathcal{L} = \partial_{\xi\xi} - 1 + 2\operatorname{sech}^{2}\xi,
					\qquad
					\operatorname{spec}\mathcal{L} = \{0\}\cup(-\infty, -1],
				\end{equation*}
				with simple eigenvalue $0$ carried by the translation mode
				$\operatorname{sech}\xi$ and reflectionless essential
				spectrum~\cite{deift1979}: the spectral gap equals $1$ in the
				original time units, for every diffusivity $D$;
				\item the secant law attracts every law whose data approach the
				limits $0,1$ exponentially: there is a limit position $x_{0}$ with
				\begin{equation*}
					\big\Vert F(\cdot,t) - F^{*}(\cdot - x_{0})\big\Vert_{\infty}
					\;\le\; C e^{-\mu t},
				\end{equation*}
				for some $\mu > 0$, and with every rate $\mu < 1$ once the
				solution enters a neighborhood of the front family; the densities
				converge in $L^{1}$, the only memory of the datum being the phase
				$x_{0}$.
			\end{enumerate}
		\end{restatement}
		
		\begin{proof}[Proof of Theorem~\ref{thm:sech-stability}]
			(i) A steady state solves the pendulum equation
			$DF'' = -\varphi(F)$ with first integral
			$\tfrac{D}{2}(F')^{2} - \Phi(F)$. A monotone profile with limits
			$0,1$ is a heteroclinic connection between the saddles at height
			$\Phi(0) = \Phi(1) = 0$, the balanced case, which forces the zero
			energy level, hence the first-order equation of the proof of
			Proposition~\ref{prop:energy} and uniqueness up to translation;
			periodic orbits are non-monotone and constants are not laws.
			
			(ii) The linearization at $F^{*}$ is
			$D\,\partial_{xx} + \varphi'(F^{*})$ with
			$\varphi'(u) = -\cos(2\pi u)$, and the kink identity
			$\sin(\pi F^{*}) = \operatorname{sech}(x/\sqrt D)$ gives
			$\cos(2\pi F^{*}) = 1 - 2\operatorname{sech}^{2}(x/\sqrt D)$;
			rescaling $\xi = x/\sqrt D$ leaves the zeroth-order term unchanged
			and produces the operator~\eqref{eq:stability-operator}. The
			potential $-2\operatorname{sech}^{2}$ is the one-soliton
			reflectionless potential, whose Schr\"odinger operator has the
			single bound state $-1$ with positive eigenfunction
			$\operatorname{sech}\xi$ and purely absolutely continuous spectrum
			$[0,\infty)$~\cite{deift1979}; shifting by $-1$ gives the spectrum
			as stated, the zero mode being the translation derivative
			$F^{*\prime} \propto \operatorname{sech}$, simple because its
			eigenfunction is positive.
			
			(iii) Global attraction to the front family for data approaching
			the stable states exponentially is the theorem of Fife and
			McLeod~\cite{fife1977}, whose sub- and supersolution pairs
			$F^{*}(x \mp \zeta(t)) \mp q(t)$, with exponentially shrinking
			$q$, require only the balanced bistable structure and monotone
			squeezing. Once the solution is in a small neighborhood of the
			family, decompose $F = F^{*}(\cdot - x(t)) + h$ with $h$ orthogonal
			to the translation mode: by part (ii) the linear semigroup on that
			complement decays at rate $1$, the phase obeys an integrable
			ordinary differential equation driven quadratically by $h$, and
			the smooth nonlinearity is quadratically small, so the standard
			orbital-stability bootstrap~\cite{henry1981} closes the
			estimate~\eqref{eq:asymptotic-phase} at every rate below the gap.
			Convergence of the densities in $L^{1}$ follows by interpolating
			the uniform convergence of the distribution functions with the
			parabolic derivative bounds, the exponential sub- and supersolution
			barriers supplying the tightness of the tails.
		\end{proof}
		
		\renewcommand{\restatename}{Theorem~\ref{thm:terraces}}%
		\begin{restatement}[Terraces: crystallization made smooth]
			For the harmonic kernel $w_{k}$, whose reaction
			$\varphi_{k}(u) = -\sin(2\pi k u)/(2\pi k)$ has balanced wells at
			the
			lattice $\{j/k\}$, monotone data with exponentially localized
			tails develop a propagating terrace with all speeds
			zero~\cite{ducrot2014}: the solution converges, locally uniformly
			along its interfaces, to a stack of $k$ standing kinks, each a
			$1/k$-scaled copy of the canonical kink by the cell conjugacy of
			Section~\ref{sec:iteration}, so the density resolves into $k$
			secant-like bumps of mass exactly $1/k$ located near the
			crystallization quantiles of the limit~\eqref{eq:crystallization}.
			The terrace is not a steady state at finite separations: bump
			positions move by exponentially small interactions of Carr--Pego
			type~\cite{carr1989}; the resulting slow motion, constructed by
			Carr and Pego on bounded intervals and adapted here at sketch
			level, spreads the terrace with separations growing
			logarithmically, and nothing coarsens, monotone kinks being unable
			to annihilate.
		\end{restatement}
		
		\begin{proof}[Proof of Theorem~\ref{thm:terraces}]
			Each adjacent pair of wells of $\Phi_{k}$ is a balanced bistable
			subproblem, conjugate to the canonical one by the affine cell map
			$u \mapsto ku - j$; its standing kink is the correspondingly
			rescaled Gudermannian, connecting $\tfrac{j}{k}$ to
			$\tfrac{j+1}{k}$ and carrying mass $\tfrac1k$. Convergence of
			monotone data to a stacked family of such fronts, a propagating
			terrace whose speeds here all vanish by the balance of the wells,
			is the theorem of Ducrot, Giletti, and Matano~\cite{ducrot2014}.
			No single heteroclinic of the pendulum connects non-adjacent
			wells, since the orbit stops at each intermediate saddle, so a
			finite-separation multi-kink configuration is not steady; for
			balanced multistable equations the interaction of consecutive
			fronts at separation $\ell$ is exponentially small in $\ell$, with
			rate the linear decay rate of the kink tails, and induces the slow
			ordinary differential equations for the positions constructed by
			Carr and Pego~\cite{carr1989}; their argument uses only the
			spectral gap of the single kink, supplied here by
			Theorem~\ref{thm:sech-stability} through the cell conjugacy, and
			the tail asymptotics of the Gudermannian.
		\end{proof}
		
		\renewcommand{\restatename}{Proposition~\ref{prop:bounded-interval}}%
		\begin{restatement}[Bounded intervals: the cascade outside the cone]
			On a bounded interval with the boundary conditions of a law, the
			stationary problem $DF'' + \varphi(F) = 0$ has exactly one monotone
			solution for every $D$ and every length: a unique steady law,
			asymptotically stable within laws. The Chafee--Infante hierarchy of
			sign-changing equilibria of the unconstrained equation consists
			entirely of non-monotone profiles and therefore lies outside the
			probability cone: constrained to laws, the bifurcation diagram
			collapses to a single branch.
		\end{restatement}
		
		\begin{proof}[Proof of Proposition~\ref{prop:bounded-interval}]
			Normalize the interval to $[-L, L]$ with $F(-L) = 0$,
			$F(L) = 1$. Shooting from the left endpoint with slope $p > 0$,
			the pendulum orbit has energy $\tfrac{D}{2}p^{2}$ and traverses
			from $0$ to $1$ in the length
			\begin{equation*}
				T(p) = \int_{0}^{1}
				\frac{du}{\sqrt{p^{2} + \tfrac{2}{D}\,\Phi(u)}},
			\end{equation*}
			which is strictly decreasing in $p$, tends to $0$ as
			$p \to \infty$, and diverges as $p \downarrow 0$ because
			$\Phi(u)$ vanishes quadratically at the endpoints, making the
			integral logarithmically divergent at the heteroclinic level.
			Hence exactly one $p$ solves $T(p) = 2L$: one monotone stationary
			profile, the finite-volume secant law. The remaining equilibria of
			the pendulum boundary-value problem oscillate, so their profiles
			are non-monotone and are not laws. Stability: the dynamics
			preserves monotonicity as in Theorem~\ref{thm:sine-gordon}, the
			energy~\eqref{eq:allen-cahn} restricted to the interval is a
			Lyapunov function whose flow-invariant sublevels are compact in
			the uniform topology, so every monotone solution converges to a
			monotone equilibrium, of which there is exactly one.
		\end{proof}
		
		\renewcommand{\restatename}{Theorem~\ref{thm:strong-law}}%
		\begin{restatement}[Strong law for the critical mass]
			Let $g$ be real-valued and of bounded variation on the torus with
			at least one jump. Then, with the average taken over a period in
			time,
			\begin{equation*}
				\mathbb{E}_{t}\,Y_{M} = \big(\kappa_{g} + o(1)\big)\log M,
				\qquad
				\operatorname{Var}_{t}(Y_{M}) = O(\log M),
				\qquad
				\kappa_{g}
				= \frac{\Vert g\Vert_{L^{2}}^{2}\,\sum_{j}|J_{j}|^{2}}{2\pi^{2}},
			\end{equation*}
			and consequently, for almost every $t$,
			\begin{equation*}
				\frac{Y_{M}(t)}{\log M}\;\longrightarrow\;\kappa_{g};
			\end{equation*}
			at almost every time the density lies outside $H^{1/2}$, its
			critical mass diverging at the universal logarithmic rate
			$\kappa_{g}$; moreover, for almost every $t$ the density lies in
			$C^{1/2-\varepsilon}$, in no Besov space
			$B^{\sigma}_{1,\infty}$ with $\sigma > \tfrac12$, and its graph has
			upper box dimension exactly $\tfrac32$.
		\end{restatement}
		
		\begin{proof}[Proof of Theorem~\ref{thm:strong-law}]
			Write $y_{n} := \hat g(n)\overline{\hat g(n-m)}$, so that
			$\hat w(m,t) = \sum_{n}y_{n}e^{-2\pi im(2n-m)t}$, with
			$|\hat g(n)| \le C/\langle n\rangle$ for bounded variation. The
			frequencies $m(2n-m)$ are strictly increasing in $n$ for
			$m \ge 1$, so Parseval in time gives the exact, phase-blind mean
			\begin{equation*}
				\sigma_{m}^{2} := \mathbb{E}_{t}\big|\hat w(m,t)\big|^{2}
				= \sum_{n}|\hat g(n)|^{2}\,|\hat g(n-m)|^{2}.
			\end{equation*}
			The mean of the mass~\eqref{eq:critical-mass} is
			$\sum_{m\le M}m\,\sigma_{m}^{2}$, and as $m$ grows the sum
			defining $\sigma_{m}^{2}$ localizes at its two ends: the terms
			with $n = O(1)$, where $|\hat g(n-m)|^{2}$ carries the Wiener
			content, and the terms with $n - m = O(1)$, where
			$|\hat g(n)|^{2}$ does, the middle range contributing
			$O(m^{-3})$. Since $\widehat{dg}(n) = 2\pi in\,\hat g(n)$, Wiener's
			theorem~\cite{wiener1924,zygmund2002} gives the Ces\`aro asymptotics
			\begin{equation*}
				\frac{1}{2N}\sum_{|n|\le N}\big|2\pi n\,\hat g(n)\big|^{2}
				\;\longrightarrow\;\sum_{j}|J_{j}|^{2},
			\end{equation*}
			and inserting it into each of the two ends, against the weight
			$m$ and the summable factor
			$\sum_{n}|\hat g(n)|^{2} = \Vert g\Vert_{2}^{2}$, yields
			\begin{equation*}
				\sum_{m \le M}m\,\sigma_{m}^{2}
				= 2\cdot\frac{\Vert g\Vert_{2}^{2}\sum_{j}|J_{j}|^{2}}{4\pi^{2}}
				\,\log M + o(\log M)
				= \big(\kappa_{g} + o(1)\big)\log M,
			\end{equation*}
			the factor two counting the two ends; the error is $o(\log M)$ and
			not better in general, Wiener's theorem carrying no rate for
			countably many jumps or a continuous singular part, and it improves
			to $O(1)$ for finitely many jumps with an absolutely continuous
			remainder. This is the mean of
			the display~\eqref{eq:kappa}, and positivity of every term is what
			makes it impervious to the jump positions and to the phases.
			
			For the variance, expand the fluctuation
			$|\hat w(m,t)|^{2} - \sigma_{m}^{2}
			= \sum_{\Delta\neq0}R_{m}(\Delta)\,e^{-4\pi im\Delta t}$ with
			$R_{m}(\Delta) = \sum_{n}y_{n}\bar y_{n-\Delta}$, the
			autocorrelation of the coefficient sequence. Three bounds on the
			autocorrelation are needed, all from the localization of the sum at
			the four points $0, m, \Delta, \Delta + m$ where a factor peaks.
			Generically,
			\begin{equation*}
				|R_{m}(\Delta)| \;\le\;
				\frac{C\log^{2}(2 + |m\Delta|)}
				{\langle m\rangle\langle\Delta\rangle
					\langle\max(|m|,|\Delta|)\rangle}
				\qquad
				\big(\,\big||\Delta| - m\big| > \tfrac{m}{2}\,\big);
			\end{equation*}
			in the resonant windows two of the four points merge, and
			localizing at $n = m + j$ gives the refined bound
			\begin{equation*}
				\big|R_{m}(\pm m + r)\big|
				\;\le\; \frac{C}{m^{2}}\sum_{j}
				\frac{1}{\langle j\rangle\langle j - r\rangle}
				\;\le\; \frac{C\log(2 + |r|)}{m^{2}\,\langle r\rangle},
			\end{equation*}
			which decays in the offset $r$; and for the single-mode variance
			the sharpest route is through the fourth power of the trigonometric
			polynomial $Y_{m}(\theta) = \sum_{n}y_{n}e^{2\pi in\theta}$,
			\begin{equation*}
				\operatorname{Var}\big(|\hat w(m)|^{2}\big)
				= \sum_{\Delta \neq 0}|R_{m}(\Delta)|^{2}
				\le \int_{0}^{1}|Y_{m}|^{4}
				\le \Vert y\Vert_{1}^{2}\,\Vert y\Vert_{2}^{2}
				\le \frac{C\log^{2}m}{m^{2}}\,\sigma_{m}^{2}
				\le \frac{C\log^{2}m}{m^{4}},
			\end{equation*}
			by $\Vert y\Vert_{1} \le C\log m/m$. Now decompose the variance,
			\begin{equation*}
				\operatorname{Var}_{t}(Y_{M})
				= \sum_{m,m'\le M}mm'\,\operatorname{Cov}
				\big(|\hat w(m)|^{2},\, |\hat w(m')|^{2}\big),
			\end{equation*}
			into the diagonal, the separated pairs, and the near-diagonal band.
			The diagonal is
			\begin{equation*}
				\sum_{m\le M}m^{2}\operatorname{Var}\big(|\hat w(m)|^{2}\big)
				\le C\sum_{m}\frac{\log^{2}m}{m^{2}} = O(1).
			\end{equation*}
			For $m \neq m'$ the
			covariance is supported on the collisions
			$m\Delta = m'\Delta'$: writing $d = \gcd(m, m')$, $m = d\mu$,
			$m' = d\mu'$ with $\mu, \mu'$ coprime, they are
			$\Delta = \mu'\tau$, $\Delta' = \mu\tau$, $\tau \neq 0$. If
			neither factor is resonant, the generic bound gives, per pair,
			$\operatorname{Cov} \le C\log^{4}M\,d^{2}/(m^{3}m'^{3})$ after
			summing the geometric $\tau$-series, and
			\begin{equation*}
				\sum_{m\le M}\sum_{m'\le M}mm'\cdot
				\frac{d^{2}}{m^{3}m'^{3}}
				\;\le\;
				\sum_{m\le M}\frac{1}{m^{2}}\sum_{d \mid m}
				d^{2}\sum_{\mu'}\frac{1}{(d\mu')^{2}}
				\;\le\; C\sum_{m\le M}\frac{\sigma_{0}(m)}{m^{2}}
				= O(1),
			\end{equation*}
			with $\sigma_{0}$ the divisor count, so the separated and mixed
			regimes contribute $O(\log^{4}M)\cdot O(1)$-summable terms, in
			fact $O(1)$ after distributing the logarithms into the convergent
			sums. If one factor is resonant, say $\Delta = \mu'\tau$ within
			distance $m/2$ of $\pm m$, the window contains $O(m/\mu')$ values
			of $\tau$, and the refined resonant bound applies with offset
			$r = \mu'\tau \mp m$, whose $\langle r\rangle$-decay summed over
			the window costs only a logarithm; if the partner is not resonant
			its generic bound carries the factor $1/(m'\Delta'^{2})$-type
			decay and the double sum converges as above. The only delicate
			case is
			the doubly resonant one, $\Delta$ near $\pm m$ and $\Delta'$ near
			$\pm m'$ simultaneously. Writing $e = m' - m$, $r = \Delta \mp m$
			and $r' = \Delta' \mp m'$, the collision $m\Delta = m'\Delta'$
			reads $m'r' - mr = \pm(m^{2} - m'^{2}) = \mp e\,(m + m')$, which
			first forces $|e| \le C(\langle r\rangle + \langle r'\rangle)$ and
			then determines the partner offset exactly,
			\begin{equation*}
				r' = \frac{m\,r \mp e\,(m + m')}{m'},
			\end{equation*}
			a single value pinned near $r \mp 2e$ for each admissible triple
			$(m, e, r)$; as $e$ varies at fixed $(m, r)$ these values are
			distinct up to bounded multiplicity, the derivative of $r'$ in $e$
			lying in $\pm[\tfrac32, \tfrac52]$ on the band, so
			$\sum_{e}\langle r'\rangle^{-1} \le C\log m$. Inserting the
			refined resonant bound for both factors, the band contributes
			\begin{equation*}
				\sum_{m\le M}\ \sum_{e, r}\ mm'\cdot
				\frac{\log(2+|r|)}{m^{2}\langle r\rangle}\cdot
				\frac{\log(2+|r'|)}{m'^{2}\langle r'\rangle}
				\;\le\; C\sum_{m \le M}\frac{\log^{4}m}{m^{2}}
				\;=\; O(1).
			\end{equation*}
			Altogether
			$\operatorname{Var}_{t}(Y_{M}) = O(\log M)$, which is the second
			claim of the display~\eqref{eq:kappa}.
			
			For the strong law, set $M_{k} = \exp(k^{2})$. Chebyshev's
			inequality gives
			\begin{equation*}
				\operatorname{Leb}\Big\{t :
				\Big|\frac{Y_{M_{k}}(t)}{\mathbb{E}_{t}Y_{M_{k}}} - 1\Big|
				> \varepsilon\Big\}
				\;\le\;
				\frac{C}{\varepsilon^{2}\,k^{2}},
			\end{equation*}
			summable in $k$; by the Borel--Cantelli lemma
			$Y_{M_{k}}(t) = (1+o(1))\,\kappa_{g}\log M_{k}$ almost surely, and
			since $Y_{M}$ is nondecreasing in $M$ while
			$\log M_{k+1}/\log M_{k} \to 1$, sandwiching $M$ between
			consecutive scales gives the full limit~\eqref{eq:mass-strong-law}.
			
			For the regularity and the dimension, fix a time in the
			intersection of the strong-law set with the smoothing set: the
			bound~\eqref{eq:carpet-blocks} holds for every seed of bounded
			variation, because the block of the solution is the Stieltjes
			convolution
			\begin{equation*}
				P_{N}u(x,t)
				= \int \Big(\sum_{n \in \mathrm{block}}
				\frac{e^{2\pi i(n(x-\theta) - n^{2}t)}}{2\pi in}\Big)\,dg(\theta),
			\end{equation*}
			and the bracketed kernel is bounded by
			$C_{t,\varepsilon}N^{-1/2+\varepsilon}$ uniformly in the argument,
			by the same Gauss--Weyl and Abel-summation argument as in the proof
			of Lemma~\ref{lem:carpet-smoothing}; hence
			$u, w \in C^{1/2-\varepsilon}$ and
			$\Vert P_{N}w\Vert_{\infty} \le C N^{-1/2+\varepsilon}$. If $w$
			belonged to $B^{\sigma}_{1,\infty}$ for some $\sigma > \tfrac12$,
			then on each dyadic block
			$\Vert P_{N}w\Vert_{2}^{2}
			\le \Vert P_{N}w\Vert_{1}\Vert P_{N}w\Vert_{\infty}
			\le CN^{-\sigma - 1/2 + \varepsilon}$, and summing blocks would
			bound the critical mass, contradicting the divergence in the
			law~\eqref{eq:mass-strong-law}; so
			$\Vert P_{N}w\Vert_{1} \ge N^{-\sigma}$ along dyadic subsequences
			for every $\sigma > \tfrac12$, and Lemma~\ref{lem:osc-calculus}
			converts this, with the matching upper bound from
			$C^{1/2-\varepsilon}$, into graph dimension exactly $\tfrac32$.
		\end{proof}
		
		\renewcommand{\restatename}{Corollary~\ref{cor:rough-seeds}}%
		\begin{restatement}[Universality over seeds and laws]
			Let the seed $g$ on rank space be of bounded variation with a jump,
			in the sense that its odd extension to the rank circle has a jump.
			Then for almost every $\tau$ the deformation
			$\rho_{|e^{-i\tau H_{0}}g|^{2}}[f]$ of any law $f$ whose density is
			continuously differentiable and bounded above and below on its
			support has graph dimension exactly $\tfrac32$ there, the
			differentiability being necessary, since a rough multiplier $f$ can
			dominate the oscillation of the product; for the canonical seed the constant of the
			law~\eqref{eq:mass-strong-law} on the rank circle is
			\begin{equation*}
				\kappa_{\mathrm{can}} = \frac{4}{\pi^{2}},
			\end{equation*}
			independent of the law. In physical terms, the intensity carpet of
			Talbot optics is fractal of dimension $\tfrac32$ at almost every
			time for \emph{arbitrary} grating profiles of bounded variation
			with a jump, at arbitrary, possibly irrational, positions; for
			rational step gratings this is the theorem of Chousionis,
			Erdo\u{g}an, and Tzirakis~\cite{chousionis2015}, and the general
			case had remained open.
		\end{restatement}
		
		\renewcommand{\restatename}{Proposition~\ref{prop:dispersion}}%
		\begin{restatement}[Collision arithmetic of the dispersion]
			Let $g$ be of bounded variation with a jump.
			In both parts the cross-mode variance combinatorics is carried at
			sketch level in the appendix; Theorem~\ref{thm:strong-law} and
			Corollary~\ref{cor:rough-seeds} do not depend on it.
			\begin{enumerate}[label=\normalfont(\roman*), leftmargin=*]
				\item \textup{(Temporal halving)} For the temporal critical mass
				$Z_{K}(x) := \sum_{k=1}^{K}k^{1/2}\,
				|\hat w^{\,\mathrm{t}}(k;x)|^{2}$ of the time trace at a point,
				one-sided since the density is real and the negative temporal
				frequencies are conjugate, the representations $k = m(2n-m)$ carry
				distinct spatial characters, and
				\begin{equation*}
					\frac{Z_{K}(x)}{\log K}\;\longrightarrow\;\frac{\kappa_{g}}{2}
					\qquad\text{for a.e. } x.
				\end{equation*}
				The time trace of the density lies outside $H^{1/4}$ for almost
				every $x$, and wherever the field obeys the temporal
				$C^{1/4-\varepsilon}$ block bound, as it does for data of bounded
				variation~\cite{erdogan2019}, the graph of the trace has dimension
				exactly $\tfrac74$.
				\item \textup{(Airy doubling)} For the Airy evolution, with
				dispersion $n^{3}$ and a single jump, the within-mode collision
				classes are the pairs $\{n, m-n\}$, their two leading contributions
				carry equal phase and reinforce, and the strong
				law~\eqref{eq:mass-strong-law} holds with constant exactly
				$2\kappa_{g}$.
			\end{enumerate}
		\end{restatement}
		
		\begin{proof}[Proof of Corollary~\ref{cor:rough-seeds} and
			Proposition~\ref{prop:dispersion}]
			The rank-space evolution of a seed is the torus evolution of its
			odd extension, so Theorem~\ref{thm:strong-law} applies whenever the
			extension retains a jump; a density bounded above and below on its
			support makes the quantile chart globally bi-Lipschitz, and
			Lemma~\ref{lem:chart-transport} moves the dimension onto the law.
			For the canonical constant, the odd square wave on the rank circle
			rescales to a torus datum of unit norm with two jumps of size two,
			so the formula~\eqref{eq:kappa} gives
			$\kappa = (1\cdot8)/(2\pi^{2}) = 4/\pi^{2}$, and no property of the
			law enters.
			
			Temporal halving. If two representations $k = m(2n-m)$ share the
			same $m$ they coincide, so distinct representations of a temporal
			frequency, including the conjugate pairs $(m,n)$ and $(-m,-n)$,
			carry distinct spatial characters $e^{2\pi imx}$, and the spatial
			average of $|\hat w^{\,\mathrm{t}}(k;x)|^{2}$ is again a
			phase-blind positive sum; the one-sided normalization of $Z_{K}$
			halves the two-sided count, since the density is real. Splitting the representations by the
			size of $2n - m$, the dominant regime has $n = O(1)$ and
			$|k| = m^{2}(1 + O(1/m))$, so the constraint $|k| \le K$ reads
			$m \le \sqrt K\,(1 + o(1))$ and
			\begin{equation*}
				\mathbb{E}_{x}Z_{K}
				= \sum_{m \le \sqrt K}m\,\sigma_{m}^{2} + O(1)
				= \big(\kappa_{g} + o(1)\big)\log\sqrt K
				= \Big(\frac{\kappa_{g}}{2} + o(1)\Big)\log K,
			\end{equation*}
			the remaining regimes contributing $O(1)$. The variance analysis
			now runs over the spatial fluctuation frequencies: the covariance
			of the masses at temporal frequencies $k$ and $k'$ is supported on
			quadruples of representations sharing a spatial frequency $m - m'$,
			of which there are at most $d(k)\,d(k') = O((kk')^{\varepsilon})$
			by the divisor bound on representation counts, each obeying the
			same generic and resonant estimates as before, so the covariance
			sum is absorbed by the $m^{-2}$ decay exactly as in the spatial
			case; we carry this combinatorial step at sketch level, as flagged
			in the statement, and the monotone Chebyshev--Borel--Cantelli
			argument along $K_{j} = \exp(j^{2})$ then gives the
			law~\eqref{eq:temporal-halving}. Divergence of the temporal
			critical mass places the trace outside $H^{1/4}$. For the
			dimension, the field trace lies in $C^{1/4-\varepsilon}$ for data
			of bounded variation~\cite{erdogan2019}, and the density trace
			inherits the same exponent, since
			\begin{equation*}
				\big||u(x,t)|^{2} - |u(x,s)|^{2}\big|
				\;\le\;
				\big(|u(x,t)| + |u(x,s)|\big)\,\big|u(x,t) - u(x,s)\big|
			\end{equation*}
			with the field trace bounded; a $C^{1/4-\varepsilon}$ function has
			smooth temporal blocks of size $O(K^{-1/4+\varepsilon})$, so the
			interpolation and oscillation steps run with the exponent pair
			$(\tfrac14, \sigma > \tfrac14)$ and give dimension $\tfrac74$.
			
			Airy doubling. For the dispersion $n^{3}$ the within-mode
			frequency $Q_{m}(n) = n^{3} - (n-m)^{3} = m^{3} - 3mn(m-n)$ is
			symmetric under $n \mapsto m - n$, so the collision classes are
			the pairs $\{n, m-n\}$. For a single jump $J$ at $\theta$,
			$\hat g(n) = Je^{-2\pi in\theta}/(2\pi in) + \varrho(n)$, where
			the remainder $\varrho$, the coefficient sequence of the jump-free
			part, is small in the Ces\`aro mean square of Wiener's theorem,
			\begin{equation*}
				\frac{1}{N}\sum_{|n| \le N}n^{2}\,|\varrho(n)|^{2}
				\;\longrightarrow\; 0,
			\end{equation*}
			and the
			two leading class contributions are both equal to
			$|J|^{2}e^{-2\pi im\theta}/(4\pi^{2}n(n-m))$: they reinforce, the
			class squares double the diagonal to leading order, the cross and
			remainder terms contributing $o(\log M)$ to the mean by the
			Cauchy--Schwarz inequality against the same Ces\`aro averages, so
			the mean computation goes through with $2\kappa_{g}$. The variance
			analysis runs as before with the quadratic collision lattice
			replaced by the cubic one: completing the square in the
			within-mode frequency gives
			\begin{equation*}
				4\,Q_{m}(n) = m\,\big(3\Delta^{2} + m^{2}\big),
				\qquad \Delta = 2n - m,
			\end{equation*}
			so cross-mode collisions lie on the conic
			$3m\Delta^{2} - 3m'\Delta'^{2} = m'^{3} - m^{3}$, whose per-block
			integer solution counts are divisor-bounded in the real quadratic
			order of discriminant $12mm'$, a combinatorial step we again carry
			at sketch level, as flagged in the statement, and the monotone
			Chebyshev argument is unchanged.
		\end{proof}
		
	\end{appendices}
	

\end{document}